\documentclass[11pt,a4paper,reqno]{amsart}
\usepackage{amsmath}
\usepackage{amsthm}
\usepackage{amsfonts}
\usepackage{amssymb}
\usepackage{enumitem}
\usepackage[colorlinks=true]{hyperref}

\newtheorem{theorem}{Theorem}[section]
\newtheorem{lemma}[theorem]{Lemma}
\newtheorem{proposition}[theorem]{Proposition}
\newtheorem{corollary}[theorem]{Corollary}
\numberwithin{equation}{section}

\catcode`@=11
\def\section{\@startsection{section}{1}%
  \z@{1.5\linespacing\@plus\linespacing}{.5\linespacing}%
  {\normalfont\bfseries\large\centering}}
\catcode`@=12

\setlist[enumerate]{label=\textup{(\roman*)},leftmargin=*,itemsep=2pt}

\newcommand{\be}{\begin{equation}}
\newcommand{\ee}{\end{equation}}
\newcommand{\R}{\mathbb{R}}
\newcommand{\C}{\mathbb{C}}
\newcommand{\pa}{\partial}
\newcommand{\dbar}{\bar\partial}
\newcommand{\eps}{\varepsilon}
\newcommand{\jb}[1]{\langle #1\rangle}
\newcommand{\cC}{\mathcal{C}}
\newcommand{\cI}{\mathcal{I}}
\newcommand{\cK}{\mathcal{K}}
\newcommand{\cB}{\mathcal{B}}
\newcommand{\cV}{\mathcal{V}}
\newcommand{\1}{\mathbf{1}}
\newcommand{\sgn}{\operatorname{sgn}}
\newcommand{\supp}{\operatorname{supp}}
\renewcommand{\Re}{\operatorname{Re}}
\renewcommand{\Im}{\operatorname{Im}}
\newcommand{\weakto}{\rightharpoonup}

\begin{document}

\title[Scattering for KP-II]{Scattering for the KP-II equation}
\author{Xi Chen}
\address{Xi Chen, Department of Mathematics and Computer Science, University of Basel, Spiegelgasse 1, CH--4051 Basel, Switzerland}
\email{{\tt xi01.chen@unibas.ch}}


\begin{abstract}
We prove scattering in $L^2(\R^2)$ for real-valued solutions of the
Kadomtsev--Petviashvili II (KP-II) equation with localized initial data. The localization is measured in scale-invariant Lorentz or mixed Lebesgue spaces. The admissible data include all real-valued Schwartz functions. To the best of our knowledge, this is the first result establishing scattering for the KP-II equation for every real-valued Schwartz initial datum, without any smallness assumption and without the zero mean condition. We also show that some localization is necessary: there are smooth real-valued initial data of arbitrarily small $L^2$ norm, with a logarithmic moment, for which the solution does not scatter in $L^2$.
\end{abstract}

\maketitle

\tableofcontents

\section{Introduction}

In this paper we consider the Kadomtsev--Petviashvili II (KP-II) equation on $\R^2$:
\be
\label{kp}
\text{(KP-II)}\quad
\begin{cases}
\pa_tu+\pa_x^3u+3\pa_x^{-1}\pa_y^2u+6u\pa_xu=0,\\
u_{|t=0}=u_0,
\end{cases}
\qquad u(t,x,y)\in\R ,
\ee
where $(t,x,y)\in\R\times\R^2$, $\pa_x^{-1}$ is the Fourier multiplier $(i\xi)^{-1}$, and $(\xi,\eta)$ always denotes the frequency variable dual to the physical variable $X=(x,y)$. Equivalently, $\pa_x(\pa_tu+\pa_x^3u+6u\pa_xu)+3\pa_y^2u=0$. We prove that every real-valued solution of \eqref{kp} with localized initial data behaves, as $t\to\pm\infty$, like a solution of the linearized equation, without any smallness assumption on the data, and we show that some localization of the data, beyond square integrability, cannot be dispensed with.

\subsection{The KP-II equation}
\label{subsectionkp}

Kadomtsev and Petviashvili \cite{KP} derived these equations as models for weakly nonlinear long waves which travel mainly in the $x$-direction, with slow variations in the transverse direction $y$. The sign in front of $\pa_x^{-1}\pa_y^2u$ distinguishes KP-II from KP-I; for water waves it corresponds to a weak surface tension. A rigorous derivation from the water waves equations is due to Lannes and Saut \cite{LannesSaut}, and the monograph \cite{KleinS21} of Klein and Saut surveys the mathematical theory. We use the normalization \eqref{kp}, which is convenient for the inverse scattering transform. Every other normalization reduces to \eqref{kp} by scaling: if $u$ solves \eqref{kp}, then $\tilde u(t,x,y):=6u(t,x,\sqrt3\, y)$ solves
$$\pa_t\tilde u+\pa_x^3\tilde u+\pa_x^{-1}\pa_y^2\tilde u+\tfrac12\pa_x(\tilde u^2)=0,$$
and all the function spaces used below are invariant under this change of variables.

\medskip
\noindent{\it The linear flow.} We use the Fourier transform
$$\hat f(\xi,\eta):=\frac1{2\pi}\int_{\R^2}e^{-i(x\xi+y\eta)}f(x,y)\,dx\,dy,$$
which is unitary on $L^2(\R^2)$. The linear part of \eqref{kp} is diagonalized by the Fourier transform, and its solutions are given by the free group
\be
\label{freegroup}
\widehat{S(t)f}(\xi,\eta):=e^{it\omega(\xi,\eta)}\hat f(\xi,\eta),\qquad \omega(\xi,\eta):=\xi^3-\frac{3\eta^2}{\xi}.
\ee
Indeed, the Fourier transform of $\pa_x^3u+3\pa_x^{-1}\pa_y^2u$ is $(-i\xi^3+3i\eta^2/\xi)\hat u=-i\omega\hat u$. The group $S(t)$ is unitary on $L^2(\R^2)$. The dispersion relation $\omega$ is singular on the line $\{\xi=0\}$; since this line is a null set, its values there play no role in $L^2$. A wave packet localized near a frequency $(\xi,\eta)$ with $\xi\neq0$ travels with the group velocity
\be
\label{groupvelocity}
v=-\nabla\omega(\xi,\eta)=\Big(-3\xi^2-\frac{3\eta^2}{\xi^2},\ \frac{6\eta}{\xi}\Big),
\ee
so that $v_1+\frac{v_2^2}{12}=-3\xi^2<0$. Free waves therefore propagate inside the region
\be
\label{defcV}
\cV:=\Big\{v=(v_1,v_2)\in\R^2:\ v_1+\frac{v_2^2}{12}<0\Big\}
\ee
located to the left of a parabola, and the low frequencies $\xi\to0$ approach the boundary of $\cV$. Finally, a direct computation gives $\det\nabla^2\omega=-36$: the dispersion relation is nondegenerate everywhere. Accordingly, free waves decay like $|t|^{-1}$: the free group satisfies the dispersive estimate $\|S(t)f\|_{L^\infty}\lesssim|t|^{-1}\|f\|_{L^1}$ \cite{Saut93}, where we use the notation $\lesssim$ defined in the Notations at the end of this section.

\medskip
\noindent{\it Well-posedness and conservation of mass.} The Cauchy problem \eqref{kp} was solved for smooth data by Ukai \cite{Ukai} and Saut \cite{Saut93}. Bourgain \cite{Bourgain} proved global well-posedness in $L^2$, both on $\R^2$ and on the torus, using Fourier restriction norms. Local well-posedness was subsequently established in anisotropic Sobolev spaces of negative regularity by Takaoka and Tzvetkov \cite{TakaokaTzvetkov} and by Isaza and Mej\'ia \cite{IsazaMejia}, and in the full subcritical range of the spaces with norm $\|(1+\xi^2)^{\frac s2}\hat f(\xi,\eta)\|_{L^2}$, $s>-\frac12$, by Hadac \cite{Hadac}. Hadac, Herr and Koch \cite{HHK,HHKcorr} proved local well-posedness for arbitrary data in the scale-invariant space $\dot H^{-\frac12,0}(\R^2)$ defined below, together with global well-posedness and scattering for sufficiently small data in that space. The $L^2$ solutions conserve the mass:
\be
\label{mass}
\int_{\R^2}u^2(t,X)\,dX=\int_{\R^2}u_0^2(X)\,dX,\qquad t\in\R .
\ee
The scaling symmetry $u_\lambda(t,x,y)=\lambda^2u(\lambda^3t,\lambda x,\lambda^2y)$ maps solutions to solutions, leaves invariant the space $\dot H^{-\frac12,0}(\R^2)$ of tempered distributions with
$$\|f\|_{\dot H^{-\frac12,0}}:=\big\||\xi|^{-\frac12}\hat f(\xi,\eta)\big\|_{L^2}<\infty,$$
and satisfies $\|u_\lambda(0)\|_{L^2}^2=\lambda\|u_0\|_{L^2}^2$, so that $L^2$ is a subcritical space. Note that in this scaling the variable $y$ carries the weight $2$ while $x$ carries the weight $1$: the parabolic anisotropy of the problem, which will reappear in our assumptions. Throughout the paper, for $u_0\in L^2(\R^2)$, ``the solution of \eqref{kp}'' is Bourgain's global solution $u\in C(\R;L^2(\R^2))$. It depends continuously on $u_0$ in $L^2$, uniformly on compact time intervals, and if $u_0\in H^s(\R^2)$ for every $s\ge0$, then $u\in C(\R;H^s(\R^2))$ for every $s\ge0$ \cite{Bourgain}.

\subsection{Large time behavior of KP-II solutions}

One of the main questions about global solutions of a dispersive equation is their behavior as $t\to\pm\infty$. For KP-II there are strong indications that every solution with finite mass disperses. First, de Bouard and Saut \cite{BouardSaut} proved that KP-II has no localized solitary waves in the natural energy class, while de Bouard and Martel \cite{BouardMartel} ruled out nonzero global solutions that are uniformly bounded in the natural energy space, are compact in $L^2$ modulo translations, and satisfy a suitable rightward propagation condition up to Galilean invariance. Second, its line solitons, that is, the solitons of the Korteweg--de Vries equation, considered as functions on $\R^2$ which do not depend on $y$, are stable under periodic transverse perturbations \cite{MizumachiTzvetkov} and under localized perturbations \cite{Mizumachi}, but they have infinite mass. We are thus led to expect that some solutions with $L^2$ initial data {\it scatter}: there exist $u_\pm\in L^2$ such that $\|u(t)-S(t)u_\pm\|_{L^2}\to0$ as $t\to\pm\infty$.

Two kinds of results are available in this direction. For small data, Hadac, Herr and Koch \cite{HHK,HHKcorr} proved global well-posedness and scattering in the critical space $\dot H^{-\frac12,0}$. Without any size restriction, Mendez, Mu\~noz, Poblete and Pozo \cite{MMPP} obtained, by virial arguments, the decay of the local mass in large spatial regions, along sequences of times and, in suitable exterior regions, for all times; their results apply to arbitrary $L^2$ data of KP-II and energy data of KP-I, without integrability. Such local decay statements do not single out an asymptotic free profile in $L^2$. The aim of this paper is to prove scattering for every localized datum, with no restriction on its size, and to show that scattering may fail for small and smooth data which are not localized enough.

\subsection{Inverse scattering for the heat operator}

The KP-II equation is completely integrable: it is the compatibility condition of a Lax pair whose spatial part is the nonstationary heat operator $\pa_y-\pa_x^2-u$ \cite{Dryuma,ZakharovShabat}. The corresponding scattering theory leads to a $\bar\partial$ problem in a complex spectral variable; it was developed formally by Ablowitz, Bar Yaacov and Fokas \cite{ABF}, and rigorously for small data by Wickerhauser \cite{Wick87}, who also derived the Hamiltonian structure and a Plancherel-type identity \cite{Wick}. Grinevich \cite{Grin} proved that, for real potentials whose derivatives decay exponentially, the direct transform is nonsingular without any smallness assumption, although the direct integral equation itself may fail to be uniquely solvable at some spectral points. We refer to the book of Ablowitz and Clarkson \cite{AblowitzClarkson} and to the recent account of Wu \cite{Wu}, which also covers the line soliton backgrounds. Large time asymptotics were obtained by Kiselev \cite{Kis}, for small and regular scattering data, and by Wu \cite{WuLT}, for small, regular and localized data and in regions separated from the boundary parabola of $\cV$. For the KP-I equation, Donmazov, Liu and Perry \cite{DonmazovLiuPerry} recently derived the large time asymptotics of small solutions without lumps.

A closely related $\bar\partial$ problem governs the defocusing Davey--Stewartson II equation. For this equation, Beals and Coifman \cite{BealsCoifman} and Sung \cite{Sung} developed the inverse scattering theory, Perry \cite{Perry} obtained global well-posedness and large time asymptotics in weighted spaces, and Nachman, Regev and Tataru \cite{NachmanRegevTataru} proved a nonlinear Plancherel theorem and global well-posedness in $L^2$. A nonlinear Plancherel identity is also one of the ingredients of the present paper.

Large time asymptotics by inverse scattering have a long history for one-dimensional integrable equations. The inverse scattering method for the Korteweg--de Vries equation goes back to Gardner, Greene, Kruskal and Miura \cite{GGKM}, and Eckhaus and Schuur \cite{EckSch83} described the emergence of solitons from decaying data. The nonlinear steepest descent method for oscillatory Riemann--Hilbert problems was introduced by Deift and Zhou \cite{DeiftZhou93}, and used for the collisionless shock region of KdV in \cite{DeiftVZ94}. For the focusing mKdV equation, soliton resolution holds for generic data in weighted Sobolev spaces \cite{ChenLiu21}; for the focusing cubic NLS and for the derivative NLS, it holds in space-time cones \cite{BorgheseJMcL18,JenkLPS18}; and the $N$-soliton solutions of the defocusing NLS are asymptotically stable \cite{CuccJenk16}. The Benjamin--Ono equation \cite{Benjamin67} is well-posed in the sharp range of Sobolev spaces \cite{KillipLaurensVisan23-BO}, its solutions with small, localized initial data exhibit dispersive decay almost globally in time \cite{IfrimTataru2019}, and its large time behavior has become accessible through the explicit formula of G\'erard \cite{Gerard22}: the negative of a soliton scatters \cite{BlackstoneGGM24a}, and soliton resolution holds for a class of decaying data which contains the Schwartz class \cite{GGM26}, see also \cite{Chen26}. The Calogero--Moser derivative NLS has a Lax pair structure \cite{GeLe-24} and is well-posed in its critical space \cite{KLV25}; its solutions scatter in the defocusing case \cite{Chen25,Hadama,KLMV26}, and soliton resolution holds in the focusing case \cite{KimKwon}. Without integrability, soliton resolution was proved for radial energy-critical waves \cite{DuyKM23,JenLaw23} and for equivariant wave maps \cite{DuyKM22,JenLaw25}; see also \cite{KleinS21} for a comparison of inverse scattering and PDE methods.

\subsection{Statement of the results}
\label{subsectionresults}

We measure the localization of the data with Lorentz norms and mixed Lebesgue norms. For a measurable function $g$ on $\R^d$, $0<p<\infty$ and $0<r\le\infty$, the Lorentz quasi-norm is
$$\|g\|_{L^{p,r}}:=\Big(\int_0^\infty\big(s^{\frac1p}g^*(s)\big)^r\frac{ds}s\Big)^{\frac1r},$$
with the usual modification for $r=\infty$, where $g^*(s):=\inf\{\lambda\ge0:\ |\{|g|>\lambda\}|\le s\}$ is the decreasing rearrangement of $g$; thus $L^{p,p}=L^p$, and $L^{p,r_1}\subset L^{p,r_2}$ for $r_1\le r_2$. For a function $f$ on $\R^2$ and $1\le p,q\le\infty$ we define the mixed norm
$$\|f\|_{L^q_yL^p_x}:=\Big\|\,y\mapsto\|f(\cdot,y)\|_{L^p(\R)}\Big\|_{L^q(\R)},$$
and we write $\|f\|_{L^{p,r}_y(L^2_x)}$ for the Lorentz quasi-norm of the function $y\mapsto\|f(\cdot,y)\|_{L^2(\R)}$.

\begin{theorem}[Scattering for localized data]
\label{thmmain}
Let $u_0\in L^2(\R^2)$ be real-valued, and assume that one of the following conditions holds:
\begin{enumerate}
\item $u_0\in L^{\frac32,4}(\R^2)$;
\item $u_0\in L^{p_y}_yL^{p_x}_x$ for some $p_x,p_y$ with $1\le p_x<2$, $1\le p_y\le2$ and $\frac1{p_x}+\frac2{p_y}\ge2$;
\item $u_0\in L^{\frac43,2}_y(L^2_x)$.
\end{enumerate}
Let $u\in C(\R;L^2(\R^2))$ be the solution of \eqref{kp}. Then there exists $u_+\in L^2(\R^2)$ such that
\be
\label{scattering}
\lim_{t\to\pm\infty}\|u(t)-S(t)u_+\|_{L^2}=0,\qquad \|u_+\|_{L^2}=\|u_0\|_{L^2} .
\ee
\end{theorem}

No smallness is assumed on $u_0$, and no zero mean condition: for a generic datum of our classes, $\int_\R u_0\,dx\neq0$, the function $\pa_x^{-1}\pa_y^2u_0$ is not square integrable, and \eqref{kp} is understood through the $L^2$ flow. The same state $u_+$ describes the solution as $t\to+\infty$ and as $t\to-\infty$ (Comment 2 below). All the spaces of Theorem \ref{thmmain} are invariant under the scaling $u_0\mapsto\lambda^2u_0(\lambda x,\lambda^2y)$ of \eqref{kp}, which multiplies the $L^{p_y}_yL^{p_x}_x$ norm by $\lambda^{2-\frac1{p_x}-\frac2{p_y}}$ and the $L^{p,r}(\R^2)$ quasi-norm by $\lambda^{2-\frac3p}$; condition (ii) with equality, condition (i) and condition (iii) are therefore critical localization conditions, in the same sense as the space $\dot H^{-\frac12,0}$ of \cite{HHK}.

Under condition (ii), it suffices to consider the case $\frac1{p_x}+\frac2{p_y}=2$. Indeed, if $\frac1{p_x}+\frac2{p_y}>2$, H\"older's inequality, first in $x$ and then in $y$, gives $\|u_0\|_{L^{q_y}_yL^{q_x}_x}\le\|u_0\|_{L^{p_y}_yL^{p_x}_x}^{1-\vartheta}\|u_0\|_{L^2}^{\vartheta}$ for $\frac1{q_x}=\frac{1-\vartheta}{p_x}+\frac\vartheta2$, $\frac1{q_y}=\frac{1-\vartheta}{p_y}+\frac\vartheta2$ and $0\le\vartheta\le1$. As $\vartheta$ increases from $0$ to $1$, $\frac1{q_x}+\frac2{q_y}$ decreases continuously from a value larger than $2$ to $\frac32$, so that it equals $2$ for some $\vartheta\in(0,1)$; then $1\le q_x<2$, since $\frac1{q_x}>\frac12$. Similarly, condition (iii) contains the data with $\phi(y):=\|u_0(\cdot,y)\|_{L^2}\in L^{p}(\R)$ for some $1\le p\le\frac43$. For $p=\frac43$ this follows from $L^{\frac43}=L^{\frac43,\frac43}\subset L^{\frac43,2}$. For $p<\frac43$, the bounds $\phi^*(s)\le s^{-\frac1p}\|\phi\|_{L^p}$ and $\phi^*(s)\le s^{-\frac12}\|\phi\|_{L^2}$ give
$$\|\phi\|_{L^{4/3,2}}^2=\int_0^\infty s^{\frac12}\phi^*(s)^2\,ds\le\|\phi\|_{L^2}^2\int_0^1s^{-\frac12}\,ds+\|\phi\|_{L^p}^2\int_1^\infty s^{\frac12-\frac2p}\,ds<\infty .$$

\begin{corollary}[Examples]
\label{corexamples}
The conclusion of Theorem \ref{thmmain} holds for each of the following real initial data:
\begin{enumerate}
\item every Schwartz function $u_0$, and every $u_0\in L^1(\R^2)\cap L^2(\R^2)$;
\item every $u_0\in L^2(\R^2)$ which vanishes outside a vertical strip $\{|x|\le L\}$ or a horizontal strip $\{|y|\le L\}$;
\item every $u_0$ with $\jb y^{\frac14}u_0\in L^2(\R^2)$, where $\jb y:=(1+y^2)^{\frac12}$;
\item $u_0(x,y)=\1_{[0,1]}(x)g(y)$ and $u_0(x,y)=g(x)\1_{[0,1]}(y)$, for every real $g\in L^2(\R)$.
\end{enumerate}
\end{corollary}

\begin{proof}
Let $u_0\in L^1\cap L^2$. Since $su_0^*(s)\le\|u_0\|_{L^1}$ and $s(u_0^*(s))^2\le\|u_0\|_{L^2}^2$, splitting the integral at $s=1$ gives
$$\|u_0\|_{L^{3/2,4}}^4=\int_0^\infty s^{\frac83}u_0^*(s)^4\,\frac{ds}s\le\|u_0\|_{L^2}^4\int_0^1s^{-\frac13}\,ds+\|u_0\|_{L^1}^4\int_1^\infty s^{-\frac73}\,ds<\infty,$$
so that condition (i) holds; Schwartz functions belong to $L^1\cap L^2$. If $u_0\in L^2$ vanishes outside $\{|x|\le L\}$, the Cauchy--Schwarz inequality gives $\|u_0(\cdot,y)\|_{L^1}\le(2L)^{\frac12}\|u_0(\cdot,y)\|_{L^2}$, so that $u_0\in L^2_yL^1_x$, which is condition (ii) with $(p_x,p_y)=(1,2)$. If $u_0\in L^2$ vanishes outside $\{|y|\le L\}$, then $\phi(y):=\|u_0(\cdot,y)\|_{L^2}$ belongs to $L^2(\R)$ and vanishes outside $[-L,L]$, so that $\phi^*$ vanishes on $[2L,\infty)$ and $\|\phi\|_{L^{4/3,2}}^2=\int_0^{2L}s^{\frac12}\phi^*(s)^2\,ds\le(2L)^{\frac12}\|\phi\|_{L^2}^2$; this is condition (iii). If $\jb y^{\frac14}u_0\in L^2$, then $\phi=\jb y^{-\frac14}\cdot\jb y^{\frac14}\phi$ is the product of a function of $L^{4,\infty}(\R)$ and a function of $L^2(\R)$, hence belongs to $L^{\frac43,2}$ by the Lorentz H\"older inequality \cite[Section 1.4]{Grafakos}; this is condition (iii). The data of (iv) vanish outside the strips $\{0\le x\le1\}$ and $\{0\le y\le1\}$, respectively; by the arguments above for (ii), the first one satisfies condition (ii) with $(p_x,p_y)=(1,2)$, and the second one satisfies condition (iii).
\end{proof}

Corollary \ref{corexamples} (ii) shows that no regularity is needed: the data of (iv) are discontinuous, and $g$ may be an arbitrary function of $L^2(\R)$. The localization, however, cannot be removed altogether.

\begin{theorem}[Data without scattering]
\label{thmcounter}
There exist real-valued initial data $u_0$ of arbitrarily small positive $L^2$ norm such that the interaction profile $S(-t)u(t)$ of the solution of \eqref{kp} has two distinct strong limits in $L^2$ along two sequences of times tending to $+\infty$, and two distinct strong limits along two sequences of times tending to $-\infty$. For every multi-index $\gamma$, the data satisfy
\be
\label{counterdecay}
\begin{gathered}
\big(\log(2+|X|)\big)^{\frac12}\pa^\gamma u_0\in L^2(\R^2),\\
\lim_{|X|\to\infty}|X|^{1+|\gamma|}\log|X|\;\pa^\gamma u_0(X)=0;
\end{gathered}
\ee
in particular $u_0\in H^m(\R^2)$ for every $m\ge0$. Moreover, one may impose $u_0\ge0$, or, in a separate construction, $\int_\R u_0(x,y)\,dx=0$ for every $y\in\R$ and $\pa_x^{-1}\pa_yu_0\in L^2(\R^2)$.
\end{theorem}

In particular, scattering in $L^2$ fails for some data in the energy space of \eqref{kp}, with the zero mean condition and with finite energy $\int(\frac12(\pa_xu_0)^2-\frac32(\pa_x^{-1}\pa_yu_0)^2-u_0^3)\,dX$. The construction also shows that the scattering map $u_0\mapsto u_+$ of Theorem \ref{thmmain} is not continuous for the topology of $L^2$ (Proposition \ref{propphase}).

\subsection{Comments on the results}

\noindent
1. {\it Large data.} As far as we know, Theorem \ref{thmmain} is the first scattering result for the KP-II equation which covers every real-valued Schwartz initial datum, without any smallness assumption and without the zero mean condition. The small data theory \cite{HHK} is perturbative and does not extend to large data, while the virial identities of \cite{MMPP} show the decay of the mass in large regions but do not produce a free asymptotic state. Our proof uses the complete integrability of KP-II in an essential way: the scattering transform conjugates the nonlinear flow to a linear one, and the main task is to control the inverse transform uniformly in time, for large and for rough data. Theorem \ref{thmcounter} shows that the conclusion of Theorem \ref{thmmain} cannot hold for all data in $L^2$, nor for all smooth data with a logarithmic moment.

\medskip
\noindent
2. {\it The scattering state and a nonlinear Plancherel identity.} The proof identifies the scattering states explicitly. Write the spectral variable as $k=k_1+ik_2\in\C$. For smooth exponentially decaying data, the scattering coefficient $a(k)$, $k\in\C\setminus\R$, is defined in Section \ref{sectiondirect}. For the data of Theorem \ref{thmmain}, the scattering coefficients of every sequence of real $u_{0,n}\in C^\infty_c(\R^2)$ which converges to $u_0$ in $L^2$ and in the space of the corresponding condition converge, in $L^2(|k_2|\,dA(k))$, to a limit $a$ which does not depend on the sequence (Section \ref{sectionproof}); we call $a$ the scattering coefficient of $u_0$. Then
\be
\label{explicitstate}
\begin{gathered}u_+=u_-,\\ \widehat{u_\pm}(-2k_2,-4k_1k_2)=i\sgn(k_2)\,a(k)\quad\mbox{for almost every}\ k\in\C\setminus\R,\end{gathered}
\ee
and
\be
\label{nonlinearplancherel}
\int_{\R^2}u_0^2\,dX=8\int_\C|k_2|\,|a(k)|^2\,dA(k),
\ee
where $dA$ is the Lebesgue measure on $\C$. Since $k\mapsto(-2k_2,-4k_1k_2)$ maps $\C\setminus\R$ bijectively onto $\{\xi\neq0\}$ with Jacobian $8|k_2|$, formula \eqref{explicitstate} defines $u_\pm$ uniquely, and \eqref{nonlinearplancherel} says that $u_\pm$ and $u_0$ have the same $L^2$ norm. In the proof, \eqref{nonlinearplancherel} is used in an essential way: it provides the only bound on the coefficients near the real axis which does not depend on the regularity of the data. To first order in $u_0$, $a(k)=-i\sgn(k_2)\hat u_0(-2k_2,-4k_1k_2)$, see \eqref{bornapprox} below, so that \eqref{explicitstate} and \eqref{nonlinearplancherel} reduce, to first order, to $u_\pm=u_0$ and to the Plancherel identity. We call \eqref{nonlinearplancherel} the nonlinear Plancherel identity, in analogy with \cite{NachmanRegevTataru}. The identity $u_+=u_-$ means that the scattering map $u_-\mapsto u_+$ is the identity on the states which arise from the data of Theorem \ref{thmmain}. This is consistent with the absence of resonant quadratic interactions for KP-II: for $\xi\xi_1\xi_2\neq0$ and $(\xi,\eta)=(\xi_1+\xi_2,\eta_1+\eta_2)$,
$$\omega(\xi,\eta)-\omega(\xi_1,\eta_1)-\omega(\xi_2,\eta_2)=\frac{3}{\xi\xi_1\xi_2}\Big[(\xi\xi_1\xi_2)^2+(\xi_1\eta_2-\xi_2\eta_1)^2\Big]\neq0 .$$

\medskip
\medskip
\noindent
3. {\it The conditions.} To first order, the scattering coefficient is the Born term $-i\sgn(k_2)\hat u_0(\zeta(k))$, and the change of variables $(\xi,\eta)=\zeta(k)$ maps the Lebesgue measure to $8|k_2|\,dA(k)$. By the Hausdorff--Young inequality in its Lorentz and mixed forms, the conditions of Theorem \ref{thmmain} suffice to control the Born term in the spectral norms of Section \ref{sectionborn}: $L^{3,4}(|k_2|dA)$ under condition (i), and mixed norms in which the imaginary part $k_2$ plays the role of the frequency $\xi$ and, for fixed $k_2$, the real part $k_1$ plays the role of $\eta$, under conditions (ii) and (iii). The exponent $3$ is forced by the scaling of the spectral problem: the quadratic and cubic Born terms are bounded multilinear operators in $L^{3,4}(|k_2|dA)$ (Proposition \ref{propspectral}), but in no other diagonal Lebesgue space. The higher order terms are controlled through a different mechanism: they are dominated by positive potentials built from the majorant $(|x+2k_1y|+|y|^{\frac12}+|k_2||y|)^{-1}$ of the Green function of the heat operator (Section \ref{sectionpotential}), whose Riesz potentials in $k$ are bounded in terms of the same norms of the data. The endpoint $(p_x,p_y)=(1,2)$ of condition (ii) requires no integrability in $x$ beyond $L^1$; the endpoint $(p_x,p_y)=(2,\frac43)$ is replaced by the slightly larger Lorentz space of condition (iii). By Plancherel's theorem and the change of variables $\zeta$, the norm of the first Born term in $L^2(|k_2|dA)$, the space of the nonlinear Plancherel identity, equals $8^{-\frac12}\|u_0\|_{L^2}$. For $(p_x,p_y)=(1,2)$ we also bound the quadratic Born term in $L^2(|k_2|dA)$: the proof combines this identity with the bound $\|\hat u_0(\xi,\cdot)\|_{L^2}\lesssim\|u_0\|_{L^2_yL^1_x}$, uniform in $\xi$ (Proposition \ref{propspectral}).

\medskip
\noindent
4. {\it Data without scattering.} The data of Theorem \ref{thmcounter} are obtained by adding to a compactly supported datum $f$ functions of the form $v_N:=\sum_{j=1}^N\tau_jR_j^{-1}w(X/R_j)$, where $w\in C^\infty_c((1,2)^2)$ is fixed, $R_j$ grows geometrically, $\sum_j\tau_j=\alpha$ is fixed and $\sum_j\tau_j^2$ is small. The norm $\|v_N\|_{L^2}$ is small, but at large distances the Green function $G_k$ of the operator $\pa_y-\pa_x^2-2k\pa_x$ satisfies, for $k$ in a compact subset of $\C^+$ and as $|X|\to\infty$ in a cone $\{|y|\ge c|X|\}$ with $c>0$,
$$\begin{gathered}G_k(X)=G^\infty_k(X)+e^{iX\cdot\zeta(k)}\,\overline{G^\infty_k(X)}+O(|X|^{-2}),\\ G^\infty_k(x,y):=\frac1{2\pi}\,\frac1{2k_2y-i(x+2k_1y)},\end{gathered}$$
where $\zeta(k)=(-2k_2,-4k_1k_2)$ (Lemma \ref{lemmatail}): $G_k$ is the sum of the function $G^\infty_k$, which is homogeneous of degree $-1$, of an oscillating function, and of an error of lower order. As a consequence, each term $\tau_jR_j^{-1}w(X/R_j)$ multiplies the eigenfunction on the support of $f$ by $1+\tau_jc(k)$, up to an error which is small compared to $\tau_j$, where $c(k)$, the integral of $w$ against $G^\infty_k(-\cdot)$, does not depend on $j$; the oscillating term contributes only to the error (Lemma \ref{lemmashell}). As $N\to\infty$, $v_N$ multiplies the scattering coefficient of $f$ by $e^{2i\alpha\Im c(k)}$ (Proposition \ref{propphase}). Adding such functions at larger and larger distances, built alternately from $w$ and from $w(-X)$, for which $c(k)$ is replaced by $-c(k)$, we obtain truncated data whose scattering states oscillate between two values, and the interaction profile $S(-t)u(t)$ follows this oscillation. These data do not belong to the spaces of Theorem \ref{thmmain}: the scale-invariant norms of the terms $\tau_jR_j^{-1}w(X/R_j)$ which are used are unbounded.

\subsection{Strategy of the proof}

We outline in this subsection the main steps and difficulties in the proofs of Theorems \ref{thmmain} and \ref{thmcounter}. We use the spectral variable $k=k_1+ik_2\in\C$, the derivatives $\pa_k=\frac12(\pa_{k_1}-i\pa_{k_2})$, $\dbar_k=\frac12(\pa_{k_1}+i\pa_{k_2})$, and the frequency
\be
\label{defzetaintro}
\zeta(k):=(-2k_2,-4k_1k_2)\in\R^2,
\ee
whose role is explained in Step 1: it is the frequency of the free plane wave which appears in the $\bar\partial$ problem, and the first order approximation \eqref{bornapprox} of the scattering coefficient is evaluated at it.\\

\noindent{\bf Step 1:} {\it The scalar $\bar\partial$ problem and a bound independent of the phase.} The Lax pair of KP-II consists of the heat operator $\pa_y-\pa_x^2-u$ and of a third-order operator in $x$. For $u=0$ the heat equation $\pa_y\psi=\pa_x^2\psi$ has the exponential solutions $e^{kx+k^2y}$, $k\in\C$; for a decaying potential we look for solutions $\psi=e^{kx+k^2y-4k^3t}m$ normalized by $m\to1$ at infinity. Because $u$ is real, the eigenfunction $m(k;t,X)$ satisfies a {\it scalar} $\bar\partial$ equation
\be
\label{dbarintro}
\dbar_km=a(k)\,e^{i\phi}\,\overline{m},\qquad \phi(k;t,X)=X\cdot\zeta(k)+t\,\omega(\zeta(k)),
\ee
where the scattering coefficient $a(k)$ does not depend on time, and the phase $\phi$ is exactly the phase of the free plane wave at the frequency $\zeta(k)$. Conversely, given $a$, one solves \eqref{dbarintro} with $m\to1$ as $|k|\to\infty$ and reconstructs the solution by
\be
\label{reconstructionintro}
u(t,X)=-\frac2\pi\int_\C\pa_x\big(a\,e^{i\phi}\,\overline{m}\big)\,dA(k).
\ee
The first observation is that the linear part of \eqref{reconstructionintro}, obtained by replacing $m$ by $1$, is exactly the free solution $S(t)u_+$ with $u_+$ given by \eqref{explicitstate}: the change of variables $(\xi,\eta)=\zeta(k)$ turns the Fourier integral of $S(t)u_+$ into the integral \eqref{reconstructionintro} with $m=1$. Scattering thus reduces to the smallness of the nonlinear part of \eqref{reconstructionintro}, which is driven by $m-1$, together with the identity $\|u_+\|_{L^2}=\|u_0\|_{L^2}$.

The second observation is that \eqref{dbarintro} is uniquely solvable, with
$$\sup_k|m(k;t,X)|\le1+Me^{2M},\qquad M:=\sup_{k\in\C}\frac1\pi\int_\C\frac{|a(k')|}{|k-k'|}\,dA(k'),$$
see Lemma \ref{lemmavekua}. The bound depends on $|a|$ only and not on the phase $e^{i\phi}$: it is uniform in $(t,X)$ and requires no smallness. It follows from the similarity principle of Bers and Vekua \cite{Bers,Vekua}: a solution $w$ of $\dbar w=a\overline w$ can be written $w=e^\sigma h$ with $h$ holomorphic and $|\sigma|\le M$. This is the main reason why large data can be handled.\\

\noindent{\bf Step 2:} {\it Compactly supported coefficients and the nonlinear Plancherel identity.} For $a\in C^\infty_c(\C\setminus\R)$ with $a(\bar k)=\overline{a(k)}$, the function \eqref{reconstructionintro} is a smooth global solution of KP-II with conserved mass; the Lax pair is recovered from the uniqueness in Step 1, which turns every identity between the coefficients of $m$ at $k=\infty$ into an exact equation (Proposition \ref{propcompact}). The nonlinear Plancherel identity \eqref{nonlinearplancherel} is proved by an analytic continuation in the amplitude (Proposition \ref{propplancherel}). For the coefficient $sa$ with $|s|$ small, the solution is small in the critical space and the theory of Hadac, Herr and Koch \cite{HHK} provides a scattering state with the same $L^2$ norm as the datum; identifying this state with the linear part of \eqref{reconstructionintro} gives \eqref{nonlinearplancherel} for small $s$. Both sides of \eqref{nonlinearplancherel} are real-analytic functions of $s\in\R$, which gives the identity for every $s$.\\

\noindent{\bf Step 3:} {\it The nonlinear remainder on propagation cones.} Fix a compact set $V\subset\cV$ of velocities and consider the points $X=tv$, $v\in V$. In the spectral variable, $\phi(k;t,tv)=t\big(v\cdot\zeta(k)+\omega(\zeta(k))\big)$ has exactly two critical points $k_\pm(v)$: they are the spectral points whose frequencies $\zeta(k)$ travel with velocity $v$, they are complex conjugate, nondegenerate, and uniformly separated from the real axis (Lemma \ref{lemmaphase}). Two mechanisms make the nonlinear part of \eqref{reconstructionintro} smaller than the linear one, which has size $|t|^{-1}$ on the cone $tV$ by stationary phase. First, $m-1$ and $\pa_xm$ tend to zero uniformly as $|t|\to\infty$: the Cauchy transform of the rapidly oscillating function $ae^{i\phi}$ is small, and the bound of Step 1, which does not see the phase, transfers this smallness to $m$. Second, the nonlinear part equals $-\frac2\pi\int ae^{i\phi}B\,dA$ with $B=-2ik_2\overline{(m-1)}+\overline{\pa_xm}$, and we integrate by parts in $k$ using the oscillation of $e^{i\phi}$. The derivative $\pa_kB$ contains a priori the term $2ik_2\overline{ae^{i\phi}}m$, which is not small; it cancels exactly, and
$$\pa_kB=-\overline{(m-1)}+\overline{a e^{i\phi}}\,\pa_xm$$
is small. The reason is that $B=e^{-i\phi}\pa_x\big(e^{i\phi}\overline{(m-1)}\big)$, where $e^{i\phi}\overline{(m-1)}$ solves an equation whose source carries no phase, and $\pa_x\pa_k\phi=i$ is constant (Lemma \ref{lemmainverse}). Together with stationary phase at $k_\pm(v)$, we obtain that the nonlinear part is $o(|t|^{-1})$ uniformly on $tV$ (Proposition \ref{propcone}). Since a free wave $S(t)\varphi$ with $\hat\varphi\in C^\infty_c(\{\xi\neq0\})$ has $L^1$ norm $O(|t|)$, the nonlinear part is invisible in the weak limit of $S(-t)u(t)$, which is therefore $u_+$ (Lemma \ref{lemmaweak}); the identity $\|u_+\|_{L^2}=\|u_0\|_{L^2}$ then upgrades the weak convergence to strong convergence. All the estimates depend on $a$ only through a common majorant with a bounded Riesz potential. Such a majorant exists for smooth data, but not, uniformly, for the approximations of rough data, and Step 8 replaces this argument for rough data.\\

\noindent{\bf Step 4:} {\it The direct problem for smooth data.} For a smooth real datum whose derivatives decay exponentially, the eigenfunction $\chi(k;X)$ solves the integral equation $\chi=1+G_k\ast(u_0\chi)$, where $G_k$ is the Green function of $\pa_y-\pa_x^2-2k\pa_x$. This equation is uniformly solvable at high energy, since the operator $G_k\ast(u_0\,\cdot\,)$ is small there, and near the real axis, where $G_k$ is a causal heat kernel and the equation is of Volterra type (Lemma \ref{lemmagreenop}). On the remaining compact set, which is away from the real axis, the equation may fail to be uniquely solvable, but only on a closed set with empty interior (Lemma \ref{lemmairregular}), and we use the nonsingularity theorem of Grinevich \cite{Grin}, which shows that $\chi$ and $a$ extend through this set. We prove that $a(k)$ decays faster than every power of $\jb{\zeta(k)}$, identify $\chi$ with the solution of \eqref{dbarintro} at $t=0$ and recover $u_0$ from \eqref{reconstructionintro} (Proposition \ref{propdirect}). Steps 2 and 3, combined with an approximation of the coefficient by compactly supported ones, then show that the solution is the function reconstructed from $a$, and that \eqref{nonlinearplancherel} holds (Corollary \ref{corsmooth}).\\

\noindent{\bf Step 5:} {\it The Born expansion.} For rough data, the scattering coefficient is not bounded by a weight with a bounded Riesz potential, and the arguments of Steps 3 and 4 do not apply directly. We expand the coefficient of a smooth approximation $u_{0,n}$ as $a_n=L_n+R_n$, where $L_n$ is the sum of the first two or three Born terms (Lemma \ref{lemmaborn}). After the change of variables $(\xi,\eta)=\zeta(l)$, the inverse heat symbol becomes the Cauchy kernel $(k-l)^{-1}$, and the Born terms are dominated by positive multilinear operators of the form $\int F_1(l)F_2(j)|k-l|^{-1}dA(l)$, with $\zeta(j)=\zeta(k)-\zeta(l)$. We bound them in the spectral norms adapted to the conditions of Theorem \ref{thmmain} (Proposition \ref{propspectral}), by a H\"older inequality for products of functions of linear forms, a multilinear real interpolation argument for positive forms, and the Lorentz inequalities of O'Neil \cite{ONeil}.\\

\noindent{\bf Step 6:} {\it Positive potentials and the remainder.} The remainder $R_n$ is an integral of $u_{0,n}$ against iterated convolutions with $G_k$. The Green function satisfies $|G_k|\lesssim K_k$, where $K_k(x,y):=(|x+2k_1y|+|y|^{\frac12}+|k_2||y|)^{-1}$ is a positive definite function. Outside a compact subset of $\C\setminus\R$, we symmetrize the direct problem on $L^2(\R^2)$, and the positive definiteness of $K_k$ bounds $|R_n(k)|$ by a constant times the integral of $|u_{0,n}(X_1)|\cdots|u_{0,n}(X_4)|$ against $K_k(X_1-X_2)K_k(X_2-X_3)K_k(X_3-X_4)$, and under condition (ii) with $(p_x,p_y)=(1,2)$ by the corresponding integral with three factors $|u_{0,n}|$ (Proposition \ref{propremainder}). The Riesz potentials of these functions of $k$ are bounded in terms of the norms of the data (Proposition \ref{proppotential}). As a function of $k$, $K_k(\sigma,\tau)\le|\tau|^{-1}\big(|k+\frac\sigma{2\tau}|+|\tau|^{-\frac12}\big)^{-1}$ is bounded by a Cauchy kernel centered at the real point $-\frac\sigma{2\tau}$, and the Riesz potentials are bounded by explicit integrals over $\C$ of products of such kernels. Such an integral is bounded by a product of negative powers of the distances of the centers to the point of evaluation, unless two of the centers are close to each other and far from this point; this condition restricts the horizontal variables to short intervals.\\

\noindent{\bf Step 7:} {\it Uniform bounds and compactness.} On the remaining compact set, we argue as for smooth data, by a bootstrap: the nonlinear Plancherel identity bounds the Riesz potential $\cI|a_n|(k):=\frac1\pi\int_\C|a_n(k')||k-k'|^{-1}\,dA(k')$ by a fixed function plus the square root of the logarithm of the supremum of $|R_n|$ on this set, while the pointwise bound $|m|\le e^{\cI|a|}$ bounds this supremum by the exponential of the potential (Proposition \ref{propuniform}). We obtain $\cI|a_n|\le C+W$, where $W$ is a fixed function which grows more slowly than any multiple of $\log\frac1{|k_2|}$ near the real axis; this is enough for the exponentials $e^{\lambda W}$ to be integrable against a common majorant of the coefficients. The remainders are equicontinuous off the real axis, the Born terms converge in the spectral norms, and a subsequence of $(a_n)$ converges almost everywhere to a coefficient $a$ (Proposition \ref{proplimit}).\\

\noindent{\bf Step 8:} {\it Identification of the mass and of the state through the heat equation.} Let $\cB a\in L^2(\R^2)$ be defined by $\widehat{\cB a}(\zeta(k))=i\sgn(k_2)\,a(k)$, as in \eqref{explicitstate}. By Fatou's lemma, $\|\cB a\|_{L^2}\le\|u_0\|_{L^2}$. For the reverse inequality, we approximate $a$ by compactly supported coefficients and show that the corresponding data converge weakly to $u_0$: the eigenfunctions of these coefficients and of the coefficients $a_n$ of the data $u_{0,n}$ have a common nonvanishing limit, for almost every $k$, which solves the heat equation $\pa_ym-\pa_x^2m-2k\pa_xm=um$ with both potentials (Proposition \ref{propmass}). To identify the state, we test the heat equation against a wave packet: in the physical variable, it is the free wave $\overline{S(t)\varphi}$ on its propagation cone; in the spectral variable, it is supported in the disc of radius $|t|^{-\frac12}$ centered at the stationary point of the phase, and its second moment is chosen so that the first two moments of the error vanish to high order (Lemma \ref{lemmapacket}). The pairing shows that $S(-t)u_n(t)-\cB a_n$ tends to zero weakly, uniformly in $n$, which gives $S(-t)u(t)\weakto\cB a$; since $\|\cB a\|_{L^2}=\|u_0\|_{L^2}$, the convergence is strong.\\

\noindent{\bf Step 9:} {\it Data without scattering.} At large distances, the Green function $G_k$ is the sum of the homogeneous function $G^\infty_k$ of Comment 4, of the oscillating function $e^{iX\cdot\zeta(k)}\overline{G^\infty_k(X)}$, and of an error of lower order (Lemma \ref{lemmatail}). As $N\to\infty$, the function $v_N$ of Comment 4 multiplies the eigenfunction on the support of a compactly supported datum $f$ by $e^{\alpha c(k)}$, and the solution of the same integral equation with the kernel $\overline{G_k(-X)}$ by $e^{-\alpha\overline{c(k)}}$; a two-potential identity shows that the scattering coefficient of $f+v_N$ converges to $e^{2i\alpha\Im c(k)}a_f$ (Proposition \ref{propphase}). Adding functions $v_N$ built alternately from $w$ and from $w(-X)$, we construct data whose interaction profiles approach two different states along two sequences of times.

This paper is organized as follows. In Section \ref{sectiondbar} we study the Cauchy transform, solve the scalar $\bar\partial$ problem with a bound independent of the phase, and estimate the Riesz potentials of the spectral weights. In Section \ref{sectioninverse} we develop the inverse problem: the reconstruction formula, the weak limits of free waves on the propagation cones, the estimate of the nonlinear remainder on these cones, the KP-II equation and the nonlinear Plancherel identity for compactly supported coefficients, and their extension to rapidly decaying coefficients. Section \ref{sectiondirect} is devoted to the direct problem: the Green function, the direct operator, and smooth exponentially decaying data. In Section \ref{sectionborn} we expand the coefficients into Born terms and a remainder, and we bound the Born terms in spectral norms. Section \ref{sectionpotential} contains the bounds on the positive potentials, and Section \ref{sectionremainder} the bound of the remainder outside a compact set. In Section \ref{sectioncompact} we prove uniform bounds and extract a limit coefficient, and in Section \ref{sectionproof} we prove Theorem \ref{thmmain}. Theorem \ref{thmcounter} is proved in Section \ref{sectioncounter}.

\vskip 0.5cm
\noindent{\bf Notations}. For nonnegative quantities $f,g$ we write $f\lesssim g$ if $f\le Cg$ for a constant $C>0$ which does not depend on the variables and parameters under consideration, $f\gtrsim g$ if $g\lesssim f$, and $f\sim g$ if $f\lesssim g\lesssim f$; the dependence of $C$ on fixed quantities is indicated by a subscript when it matters. Constants may change from line to line. We write $f=O(g)$ if $|f|\lesssim g$, and $f=o(g)$ if $|f|/g$ tends to zero in the limit under consideration. For $z\in\R^d$ or $z\in\C$, $\jb{z}:=(1+|z|^2)^{1/2}$. We write $X=(x,y)$, $dX=dx\,dy$, $\|f\|_{L^p}$ for the norm of $L^p(\R^2)$, $\|c\|_{W^{r,\infty}}:=\sum_{|\gamma|\le r}\|\pa^\gamma c\|_{L^\infty}$ for integers $r\ge0$, and $\|f\|_{L^q_yL^p_x}$ for the mixed norm defined in Subsection \ref{subsectionresults}. The $L^2$ pairing is $(f,g):=\int f\overline g\,dX$, and $\weakto$ denotes weak convergence in $L^2$. For $h\in\R^2$, $\tau_hf(X):=f(X+h)$, and for $h\in\R$, $\delta^x_hf(x,y):=f(x+h,y)-f(x,y)$. On the spectral side, $k=k_1+ik_2\in\C$, $dA(k)=dk_1dk_2$, $\nabla_k=(\pa_{k_1},\pa_{k_2})$, $\C^\pm:=\{\pm k_2>0\}$, and for a function $f$ on $\C^\pm$ we write $f^\pm(k_1):=\lim_{k_2\to0^\pm}f(k_1+ik_2)$ for its one-sided traces on $\R$, when they exist. $C_b(\C)$ is the space of bounded continuous functions and $C_0(\C)$ the space of continuous functions tending to zero at infinity, both with the supremum norm $\|\cdot\|_\infty$; for measurable functions, $\|\cdot\|_\infty$ is the essential supremum. $D(k_0,r)$ is the open disc of center $k_0$ and radius $r$. Finally, we fix a radial function $\vartheta\in C^\infty_c(\R^2)$ with $0\le\vartheta\le1$, $\vartheta=1$ on $D(0,1)$ and $\supp\vartheta\subset D(0,2)$, and a radial nonnegative function $\varrho\in C^\infty_c(\R^2)$ supported in $D(0,1)$ with $\int\varrho=1$, and we set $\varrho_\eps:=\eps^{-2}\varrho(\cdot/\eps)$ for $\eps>0$. We use $\vartheta$ and $\varrho$ both on $\R^2$ and on $\C\simeq\R^2$, and $\vartheta(s):=\vartheta(s,0)$ for $s\in\R$.

\section*{Acknowledgments and Declarations}
\medskip
\noindent\textbf{Funding.}
No funds, grants, or other support was received.

\medskip
\noindent\textbf{Competing interests.}
The author has no relevant financial or non-financial interests to disclose.

\medskip
\noindent\textbf{Data availability.}
No datasets were generated or analyzed in this purely theoretical study.

\medskip
\noindent\textbf{Acknowledgments and use of generative artificial intelligence.}
The author acknowledges the use of ChatGPT (GPT-6.1 Sol) and Claude (Opus 5.5) during the research process to explore mathematical ideas, discuss possible proof strategies, and assist with the drafting and editing of the manuscript. All mathematical statements, estimates, and proofs presented in the manuscript were independently verified by the author. The author takes full responsibility for the mathematical correctness and integrity of the work.

\section{The scalar \texorpdfstring{$\bar\partial$}{dbar} problem}
\label{sectiondbar}

This section collects the tools on the spectral side. We first recall elementary properties of the Cauchy transform. We then solve the equation $\dbar v=a\overline v+f$, which involves the complex conjugate of the unknown, with bounds which depend on the coefficient $a$ only through its size $|a|$ (Lemma \ref{lemmavekua}); this is the reason why no smallness will be needed. Finally we estimate the Riesz potentials of the weights which will bound the scattering coefficients (Lemma \ref{lemmaweights}).

\subsection{The Cauchy transform and the Riesz potential}

For a measurable function $f$ on $\C$ we define the Cauchy transform and, when $f\ge0$, the Riesz potential of order one:
\be
\label{defcauchy}
\cC f(k):=\frac1\pi\int_\C\frac{f(k')}{k-k'}\,dA(k'),\qquad \cI f(k):=\frac1\pi\int_\C\frac{f(k')}{|k-k'|}\,dA(k').
\ee
Since $\frac1{\pi k}$ is the fundamental solution of $\dbar$, we have $\dbar\,\cC f=f$; if moreover $f$ is bounded and $\cI|f|\in C_0(\C)$, then $\cC f$ is the only solution of $\dbar v=f$ in $C_0(\C)$ (Lemma \ref{lemmacauchy} (ii)). The Riesz potential dominates it, $|\cC f|\le\cI|f|$, and does not see the phase of $f$. In the inverse problem the source will be $f=ae^{i\phi}\overline m$, whose phase $e^{i\phi}$ oscillates with $(t,X)$; all our bounds will be expressed through $\cI|a|$, which is the same for all $(t,X)$.

\begin{lemma}[Cauchy transform and Riesz potential]
\label{lemmacauchy}
\begin{enumerate}
\item {\em (Modulus of continuity).} Let $f$ be measurable with $\|f\|_\infty+\|\cI|f|\|_\infty<\infty$. Then $\cC f$ and $\cI|f|$ are continuous, and for $k,h\in\C$ with $0<|h|\le\frac12$,
\be
\label{cauchymodulus}
\begin{aligned}
&|\cC f(k+h)-\cC f(k)|+\big|\cI|f|(k+h)-\cI|f|(k)\big|\\
&\qquad\lesssim|h|\log\frac1{|h|}\big(\|f\|_\infty+\|\cI|f|\|_\infty\big).
\end{aligned}
\ee
\item {\em (Inversion of $\dbar$).} Let $f$ be as in (i). Then $\dbar\,\cC f=f$ in $\mathcal D'(\C)$. If moreover $\cI|f|\in C_0(\C)$ and $v\in C_0(\C)$ satisfies $\dbar v=f$ in $\mathcal D'(\C)$, then $v=\cC f$.
\item {\em (Potentials of sets).} For a disc $D$ of radius $r$, an axis-parallel rectangle $Q$ with sides $L_1,L_2$, and a nonnegative $g\in L^1\cap L^\infty(\C)$,
\be
\label{potentialsets}
\begin{gathered}
\cI\1_D\le2r,\qquad \cI\1_Q\lesssim\min(L_1,L_2)\Big(1+\Big|\log\frac{L_1}{L_2}\Big|\Big),\\
\|\cI g\|_\infty\le2\big(\|g\|_\infty\|g\|_{L^1}\big)^{\frac12}\le\|g\|_\infty+\|g\|_{L^1},
\end{gathered}
\ee
and moreover $\cI g\in C_0(\C)$.
\item {\em (Domination).} Let $F\ge0$ be bounded on bounded sets, with $\cI F\in C_0(\C)$. Then $\sup_{k}\cI(F\1_{\{|k'|>R\}})(k)\to0$ as $R\to\infty$, and $\cI G\in C_0(\C)$ for every measurable $G$ with $0\le G\le F$.
\item {\em (Convergence).} Let $F$ be as in (iv), and moreover bounded. Then $\sup_k\cI(F\1_{\{|k'_2|<\eps\}})(k)\to0$ as $\eps\to0$, and every sequence $(f_n)$ with $|f_n|\le F$ which tends to zero uniformly on the compact subsets of $\C\setminus\R$ satisfies $\|\cI|f_n|\|_\infty\to0$.
\item {\em (Integration by parts).} Let $\phi\in C^2(\C)$ be real, and let $f\in C^1(\C)$ be such that $\dbar\phi\neq0$ on the support of $f$, $\frac f{\dbar\phi}\in C_0(\C)$, and $f$ and $\dbar\big(\frac f{\dbar\phi}\big)$ belong to $L^1\cap L^\infty(\C)$. Then
\be
\label{ibp}
\cC\big(e^{i\phi}f\big)=\frac{e^{i\phi}f}{i\dbar\phi}-\cC\Big(e^{i\phi}\,\dbar\Big(\frac f{i\dbar\phi}\Big)\Big).
\ee
\end{enumerate}
\end{lemma}

Parts (i) and (ii) are classical, see \cite[Chapter I]{Vekua}; we include short proofs, since we use the precise form of the bounds. Part (vi) expresses the gain produced by the oscillation of $e^{i\phi}$: when $|\dbar\phi|$ is large, the right-hand side of \eqref{ibp} is small.

\begin{proof}
{\bf Step 1:} Proof of (i). Fix $k$ and $0<|h|\le\frac12$, and split $\C$ into $E_1=D(k,2|h|)$, $E_2=D(k,2)\setminus E_1$ and $E_3=\C\setminus D(k,2)$. On $E_1$ we bound the two kernels separately: since $E_1\subset D(k+h,3|h|)$,
$$\int_{E_1}\frac{dA(k')}{|k+h-k'|}\le\int_{D(k+h,3|h|)}\frac{dA(k')}{|k+h-k'|}=6\pi|h|,\quad \int_{E_1}\frac{dA(k')}{|k-k'|}=4\pi|h| .$$
On $E_2\cup E_3$ we have $|k+h-k'|\ge\frac12|k-k'|$, and each of the two differences of kernels below is bounded by $\frac{|h|}{|k+h-k'||k-k'|}$, hence
$$\Big|\frac1{k+h-k'}-\frac1{k-k'}\Big|+\Big|\frac1{|k+h-k'|}-\frac1{|k-k'|}\Big|\le\frac{4|h|}{|k-k'|^2}.$$
On $E_2$, the integral of this bound is at most $4|h|\int_{2|h|}^2\frac{2\pi r\,dr}{r^2}=8\pi|h|\log\frac1{|h|}$. On $E_3$, $\frac{4|h|}{|k-k'|^2}\le\frac{2|h|}{|k-k'|}$, and the integral against $|f|$ is at most $2\pi|h|\cI|f|(k)$. Summing the three contributions and using $1\le\frac{1}{\log2}\log\frac1{|h|}$ gives \eqref{cauchymodulus}, and in particular the continuity of $\cC f$ and $\cI|f|$.\\

\noindent{\bf Step 2:} Proof of (ii). The function $\frac1{\pi k}$ is locally integrable and $\dbar\frac1{\pi k}=\delta_0$, so that, for a test function $\varphi$ and $k'\in\C$, the Cauchy--Pompeiu formula gives
$$\varphi(k')=-\frac1\pi\int_\C\frac{\dbar\varphi(k)}{k-k'}\,dA(k).$$
Fubini's theorem is legitimate, since
$$\int_\C\int_\C\frac{|f(k')||\dbar\varphi(k)|}{|k-k'|}\,dA(k')\,dA(k)=\pi\int_\C|\dbar\varphi|\,\cI|f|\,dA<\infty,$$
and it gives
$$\int_\C\cC f\,(-\dbar\varphi)\,dA=\int_\C f(k')\Big(-\frac1\pi\int_\C\frac{\dbar\varphi(k)}{k-k'}dA(k)\Big)dA(k')=\int_\C f\varphi\,dA,$$
that is, $\dbar\cC f=f$. If now $v\in C_0$ and $\dbar v=f$, then $v-\cC f$ is continuous by (i) and $\dbar(v-\cC f)=0$. By Weyl's lemma it is an entire function, and it tends to zero at infinity since $|\cC f|\le\cI|f|\in C_0$. By Liouville's theorem, $v=\cC f$.\\

\noindent{\bf Step 3:} Proof of (iii). Among all sets of area $\pi r^2$, the integral of the radially decreasing function $|k-k'|^{-1}$ is maximal on $D(k,r)$, so that $\cI\1_D(k)\le\frac1\pi\int_{D(k,r)}\frac{dA(k')}{|k-k'|}=2r$. For the rectangle $Q=I_1\times I_2$ with $|I_1|=L_1\ge L_2=|I_2|$, we use the disc bound on $D(k,L_2)$, and on $Q\setminus D(k,L_2)$ we use $|k-k'|\ge\max(L_2,|k_1-k_1'|)$:
$$\begin{aligned}\int_{Q\setminus D(k,L_2)}\frac{dA(k')}{|k-k'|}&\le L_2\int_{I_1}\frac{dk_1'}{\max(L_2,|k_1-k'_1|)}\\&\le 2L_2\int_0^{L_1}\frac{ds}{\max(L_2,s)}=2L_2\Big(1+\log\frac{L_1}{L_2}\Big);\end{aligned}$$
the case $L_1\le L_2$ is symmetric. Finally, if $g=0$, the last bound and the assertion $\cI g\in C_0(\C)$ are immediate. Assume $g\neq0$. Splitting the potential at distance $r>0$ from $k$,
$$\cI g(k)\le\frac1\pi\|g\|_\infty\int_{D(k,r)}\frac{dA(k')}{|k-k'|}+\frac1{\pi r}\int_{\{|k-k'|\ge r\}}g\,dA\le2r\|g\|_\infty+\frac{\|g\|_{L^1}}{\pi r},$$
and the choice $r=(\|g\|_{L^1}/(2\pi\|g\|_\infty))^{\frac12}$ gives the last bound of \eqref{potentialsets}. Applying it to $g\1_{\{|k'|>R\}}$ shows that $\cI(g\1_{\{|k'|>R\}})$ is uniformly small for $R$ large, while $\cI(g\1_{\{|k'|\le R\}})(k)\le\frac{\|g\|_{L^1}}{\pi(|k|-R)}$ for $|k|>R$. Since $\cI g$ is continuous by (i), $\cI g\in C_0(\C)$.\\

\noindent{\bf Step 4:} Proof of (iv) and (v). If $|k|\le\frac R2$ and $|k'|>R$, then $|k-k'|\ge\frac12|k'|$, so that
$$\cI(F\1_{\{|k'|>R\}})(k)\le\frac2\pi\int_{\{|k'|>R\}}\frac{F(k')}{|k'|}\,dA(k')\longrightarrow0\quad(R\to\infty),$$
because $\frac1\pi\int\frac{F(k')}{|k'|}dA(k')=\cI F(0)<\infty$. If $|k|>\frac R2$, then $\cI(F\1_{\{|k'|>R\}})(k)\le\cI F(k)\le\sup_{|k''|>R/2}\cI F(k'')$, which tends to zero since $\cI F\in C_0$. This proves the first claim of (iv). If $0\le G\le F$, then $G\1_{D(0,R)}\in L^1\cap L^\infty$, so that $\cI(G\1_{D(0,R)})\in C_0$ by (iii), while $\cI(G\1_{\{|k'|\ge R\}})\le\cI(F\1_{\{|k'|\ge R\}})$ is uniformly small for $R$ large; hence $\cI G$ is a uniform limit of functions of $C_0(\C)$, and $\cI G\in C_0(\C)$. If $F$ is bounded, the strip $\{|k_2'|<\eps,\ |k'|\le R\}$ is contained in a rectangle with sides $2R$ and $2\eps$, so that, by \eqref{potentialsets},
$$\sup_k\cI(F\1_{\{|k_2'|<\eps\}})(k)\le\sup_k\cI(F\1_{\{|k'|>R\}})(k)+C\|F\|_\infty\,\eps\Big(1+\log\frac R\eps\Big).$$
Choosing first $R$ large and then $\eps$ small proves the first claim of (v). For the second one, we split $\C$ into $\{|k'|>R\}$, $\{|k_2'|<\eps,|k'|\le R\}$ and the compact set $K_{R,\eps}=\{|k'|\le R,|k_2'|\ge\eps\}$: by \eqref{potentialsets},
$$\|\cI|f_n|\|_\infty\le\sup_k\cI(F\1_{\{|k'|>R\}})(k)+\sup_k\cI(F\1_{\{|k_2'|<\eps\}})(k)+2R\sup_{K_{R,\eps}}|f_n|,$$
and we let $n\to\infty$, then $\eps\to0$ and $R\to\infty$.\\

\noindent{\bf Step 5:} Proof of (vi). Since $\dbar\phi$ is continuous and does not vanish on the support of $f$, the function $\frac f{\dbar\phi}$ is $C^1$ on $\C$, and so is $g:=\frac{e^{i\phi}f}{i\dbar\phi}\in C_0(\C)$. Since $\dbar e^{i\phi}=i(\dbar\phi)e^{i\phi}$,
$$\dbar g=e^{i\phi}f+e^{i\phi}\,\dbar\Big(\frac f{i\dbar\phi}\Big)\in L^1\cap L^\infty(\C),$$
so that $\cI|\dbar g|\in C_0(\C)$ by (iii), and (ii) gives $g=\cC(\dbar g)$, which is \eqref{ibp}.
\end{proof}

\subsection{The \texorpdfstring{$\bar\partial$}{dbar} equation with complex conjugation}

The eigenfunctions of the inverse problem solve $m=1+\cC(a\overline m)$, where the coefficient $a$ contains the oscillating factor $e^{i\phi}$. The operator $v\mapsto\cC(a\overline v)$ is only real-linear, because of the complex conjugation. The following lemma is the analytic core of the paper. Its proof is based on the similarity principle of Bers and Vekua \cite{Bers,Vekua}: if $\dbar v=a\overline v$, then $\tilde a:=a\overline v/v$ is a bounded function, $v$ solves the complex-linear equation $\dbar v=\tilde av$, and $e^{-\cC\tilde a}v$ is holomorphic. Since $|\cC\tilde a|\le\cI|a|$, the size of the exponential factor is controlled by the Riesz potential of $|a|$ only, whatever the phase of $a$.

\begin{lemma}[Solvability of $\dbar v=a\overline v+f$]
\label{lemmavekua}
Let $a\in L^\infty(\C)$ with $\cI|a|\in C_0(\C)$, and set $M:=\|\cI|a|\|_\infty$.
\begin{enumerate}
\item {\em (Unique solvability).} For every $f\in L^\infty(\C)$ with $\cI|f|\in C_0(\C)$, there is a unique $v\in C_0(\C)$ such that
\be
\label{eqvekua}
v=\cC f+\cC(a\overline v),
\ee
and it satisfies, for every $k\in\C$,
\be
\label{pointwisevekua}
|v(k)|\le e^{2M}\,\cI|f|(k).
\ee
Moreover, every $v'\in C_0(\C)$ with $\dbar v'=a\overline{v'}+f$ in $\mathcal D'(\C)$ coincides with $v$.
\item {\em (Eigenfunction).} The equation $m=1+\cC(a\overline m)$ has a unique solution $m\in C_b(\C)$. It satisfies $m-1\in C_0(\C)$ and
\be
\label{boundeigen}
|m(k)-1|\le e^{2M}\cI|a|(k),\qquad \|m\|_\infty\le1+Me^{2M}.
\ee
\item {\em (Bounded right-hand side).} For every $b\in C_b(\C)$, the equation $v=b+\cC(a\overline v)$ has a unique solution $v\in C_b(\C)$, and $\|v\|_\infty\le(1+Me^{2M})\|b\|_\infty$.
\item {\em (Comparison).} If $a_1,a_2$ satisfy the assumptions of the lemma, with constants $M_1,M_2$ and eigenfunctions $m_1,m_2$ given by (ii), then
\be
\label{comparison}
\|m_1-m_2\|_\infty\le e^{2M_1}(1+M_2e^{2M_2})\,\|\cI|a_1-a_2|\|_\infty .
\ee
\item {\em (Regularity).} The solutions given by (i), and by (iii) when $b$ is constant, belong to $W^{1,p}_{\rm loc}(\C)$ for every $p\in(2,\infty)$.
\end{enumerate}
\end{lemma}

All these bounds depend on $a$ only through $|a|$; in particular they do not change when $a$ is replaced by $a\,e^{i\psi}$ for a real measurable function $\psi$.

\begin{proof}
{\bf Step 1:} Local regularity and the similarity principle. Let $f$ be as in (i) and let $v\in C_0$ satisfy $\dbar v=a\overline v+f$ in $\mathcal D'$. On a disc $D(k_0,r)$, the function $v-\cC(\1_{D(k_0,2r)}(a\overline v+f))$ is holomorphic in $D(k_0,2r)$, and $\pa_k\cC$ is the Beurling transform, which is bounded on $L^p$ for $1<p<\infty$ by the Calder\'on--Zygmund theory \cite{Stein}; hence $v\in W^{1,p}_{\rm loc}$ for every $p\in(2,\infty)$. We claim that
\be
\label{similarity}
v=e^{\sigma}\,\cC(e^{-\sigma}f),\qquad \sigma:=\cC\tilde a,\qquad \tilde a:=\begin{cases}a\overline v/v&\mbox{on}\ \{v\neq0\},\\ 0&\mbox{on}\ \{v=0\},\end{cases}
\ee
which implies \eqref{pointwisevekua}, since $|\tilde a|\le|a|$ gives $|\sigma|\le\cI|a|\le M$, hence $|e^{\pm\sigma}|\le e^M$. By construction $\tilde av=a\overline v$ almost everywhere. By Lemma \ref{lemmacauchy}, $\sigma$ is continuous and $\dbar\sigma=\tilde a$, and $\sigma\in W^{1,p}_{\rm loc}$ as above. Hence the product rule applies and
$$\dbar(e^{-\sigma}v)=e^{-\sigma}(\dbar v-\tilde av)=e^{-\sigma}f .$$
The function $e^{-\sigma}v-\cC(e^{-\sigma}f)$ belongs to $C_0$, because $|e^{-\sigma}v|\le e^M|v|$ and $|\cC(e^{-\sigma}f)|\le e^M\cI|f|$, and it is $\dbar$-closed. By Lemma \ref{lemmacauchy} (ii) it vanishes, which proves \eqref{similarity}.\\

\noindent{\bf Step 2:} Compactness. Let $\mathcal Tv:=\cC(a\overline v)$; it is a real-linear operator on the real Banach space $C_0(\C)$. If $\|v\|_\infty\le1$, then $|\mathcal Tv|\le\cI|a|$, and by \eqref{cauchymodulus} the family $\{\mathcal Tv\}$ has the common modulus of continuity $C|h|\log\frac1{|h|}(\|a\|_\infty+M)$. Since $\cI|a|\in C_0$, the family is moreover uniformly small outside large discs. By the Arzel\`a--Ascoli theorem, $\mathcal T$ is compact on $C_0(\C)$.\\

\noindent{\bf Step 3:} Proof of (i). If $v\in C_0$ and $v=\mathcal Tv$, then $\dbar v=a\overline v$ by Lemma \ref{lemmacauchy} (ii), and \eqref{similarity} with $f=0$ gives $v=0$. Thus $I-\mathcal T$ is injective, and the Fredholm alternative for compact operators on real Banach spaces shows that $I-\mathcal T$ is invertible on $C_0(\C)$. Since $\cC f\in C_0$ by Lemma \ref{lemmacauchy} (i) and $|\cC f|\le\cI|f|$, the function $v:=(I-\mathcal T)^{-1}\cC f$ is the unique solution of \eqref{eqvekua} in $C_0$. It satisfies $\dbar v=a\overline v+f$, so that \eqref{pointwisevekua} follows from Step 1. Finally, if $v'\in C_0$ and $\dbar v'=a\overline{v'}+f$, then $w:=v'-v\in C_0$ satisfies $\dbar w=a\overline w$, and \eqref{similarity} with $f=0$ gives $w=0$.\\

\noindent{\bf Step 4:} Proof of (ii) and (iii). Let $b\in C_b$. If $v\in C_b$ solves $v=b+\mathcal Tv$, then $w:=v-b=\cC(a\overline b)+\cC(a\overline w)$ belongs to $C_0$, because $\cC(a\overline v)$ is continuous by Lemma \ref{lemmacauchy} (i) and bounded by $\|v\|_\infty\cI|a|\in C_0$. Conversely, (i) with $f:=a\overline b$, for which $\cI|f|\le\|b\|_\infty\cI|a|\in C_0$, provides a unique such $w$, with
$$|w(k)|\le e^{2M}\cI|a\overline b|(k)\le e^{2M}\|b\|_\infty\,\cI|a|(k)\le Me^{2M}\|b\|_\infty .$$
This proves (iii), and (ii) is the case $b=1$. When $b$ is constant, $w\in W^{1,p}_{\rm loc}$ by Step 1, and so is $v=b+w$; together with Step 1, this proves (v).\\

\noindent{\bf Step 5:} Proof of (iv). Subtracting the equations of $m_1$ and $m_2$ gives
$$m_1-m_2=\cC\big((a_1-a_2)\overline{m_2}\big)+\cC\big(a_1\overline{(m_1-m_2)}\big).$$
This is \eqref{eqvekua} with $a=a_1$ and $f=(a_1-a_2)\overline{m_2}$, so that \eqref{pointwisevekua} and \eqref{boundeigen} give
$|m_1-m_2|\le e^{2M_1}\|m_2\|_\infty\,\cI|a_1-a_2|\le e^{2M_1}(1+M_2e^{2M_2})\|\cI|a_1-a_2|\|_\infty$.
\end{proof}

\subsection{Spectral weights}

The scattering coefficients will be bounded by functions of the frequency $\zeta(k)=(\zeta_1(k),\zeta_2(k))=(-2k_2,-4k_1k_2)$ introduced in \eqref{defzetaintro}. The map $\zeta$ sends each half-plane $\C^\pm$ diffeomorphically onto the half-plane $\{\mp\xi>0\}$ of frequencies, with inverse $(\xi,\eta)\mapsto\frac{\eta}{2\xi}-i\frac\xi2$ and Jacobian determinant $-8k_2$. Hence, for every measurable $g\ge0$ on $\R^2$,
\be
\label{changevariables}
\int_\C g(\zeta(k))\,|k_2|\,dA(k)=\frac18\int_{\R^2}g(\xi,\eta)\,d\xi\, d\eta .
\ee
Because of the factor $|k_2|$, a function of $\zeta(k)$ which is integrable in the frequency variables need not be integrable in $k$: near the real axis the map $\zeta$ is singular, and the bounded set of frequencies $\{|\zeta|\le1\}$ corresponds to a region of $\C$ which contains the long and thin set $\{|k_1|\ge1,\ |k_2|\le\frac1{8|k_1|}\}$ of infinite area. This is why the Riesz potentials of the scattering coefficients must be estimated with care. For $\alpha,\beta>0$ we define the anisotropic weight
\be
\label{defweight}
F_{\alpha,\beta}(k):=\jb{\zeta_1(k)}^{-\alpha}\jb{\zeta_2(k)}^{-\beta}=\jb{2k_2}^{-\alpha}\jb{4k_1k_2}^{-\beta}.
\ee
Since $\jb{\zeta_1}\jb{\zeta_2}\le\jb{\zeta}^2$, we have $\jb{\zeta(k)}^{-2\gamma}\le F_{\gamma,\gamma}(k)$ for $\gamma>0$.

\begin{lemma}[Riesz potentials of the spectral weights]
\label{lemmaweights}
\begin{enumerate}
\item {\em (Anisotropic weights).} Let $\alpha>0$ and $0<\beta<1$ with $\alpha+2\beta>2$, and $F:=F_{\alpha,\beta}$. Then
\be
\label{weightsF}
\cI F,\ \cI(|k_2|F)\in C_0(\C),\qquad \int_\C\frac{(1+|k_2|)F}{\jb k}\,dA<\infty .
\ee
\item {\em (Rapidly decaying weights).} The functions $\cI\jb{\zeta}^{-6}$ and $\cI(|k_2|\jb{\zeta}^{-6})$ belong to $C_0(\C)$, and
\be
\label{weightsE}
\begin{gathered}
\int_\C\Big(|k_2|+k_2^2+\frac1{\jb k}\Big)\jb{\zeta(k)}^{-6}dA(k)<\infty,\\
\cI\big((|k_2|+k_2^2)\jb{\zeta}^{-6}\big)(k)\lesssim\jb k^{-1}.
\end{gathered}
\ee
\item {\em (Tails).} For $R\ge1$,
\be
\label{weightstail}
\sup_{k\in\C}\cI\big(\jb{\zeta}^{-6}\1_{\{|k'|>R\}}\big)(k)\lesssim\frac{1+\log R}R .
\ee
\end{enumerate}
\end{lemma}

\begin{proof}
{\bf Step 1:} Dyadic decomposition. For dyadic numbers $\mu\in2^{\mathbb Z}$ and $\nu\in\{2^n:\ n\ge0\}$ let
$$E_{\mu,\nu}:=\{\mu\le|k_2|<2\mu,\ \nu\le|k_1k_2|<2\nu\},$$
and $E_{\mu,0}:=\{\mu\le|k_2|<2\mu,\ |k_1k_2|<1\}$. These sets cover $\C\setminus\R$. On $E_{\mu,\nu}$, $\nu\ge1$, we have $\frac\nu{2\mu}<|k_1|<\frac{2\nu}\mu$, so that $E_{\mu,\nu}$ is contained in four axis-parallel rectangles with sides $\frac{3\nu}{2\mu}$ and $\mu$, and $E_{\mu,0}$ is contained in two rectangles with sides $\frac2\mu$ and $\mu$. By \eqref{potentialsets},
\be
\label{potentialregions}
\sup_{k\in\C}\cI\1_{E_{\mu,\nu}}(k)\lesssim\lambda_{\mu,\nu}:=\min\Big(\mu,\frac\nu\mu\Big)\Big(1+\Big|\log\frac\nu{\mu^2}\Big|\Big)\qquad(\nu\ge1),
\ee
and the same bound holds for $\nu=0$, with $\lambda_{\mu,0}:=\min(\mu,\frac1\mu)(1+|\log\mu|)$.
Moreover $|E_{\mu,\nu}|\lesssim\nu$ for $\nu\ge1$, $|E_{\mu,0}|\lesssim1$, and on $E_{\mu,\nu}$
$$F\sim\jb\mu^{-\alpha}\max(1,\nu)^{-\beta},\qquad 1+|k_2|\sim\jb\mu,\qquad \jb k\gtrsim\max\Big(1,\mu,\frac\nu\mu\Big)\ \mbox{if}\ \nu\ge1.$$

\noindent{\bf Step 2:} Proof of (i). Since $\cI$ is monotone,
$$\sup_{k\in\C}\cI(|k_2|F)(k)\lesssim\sum_{\mu,\nu}\mu\jb\mu^{-\alpha}\max(1,\nu)^{-\beta}\lambda_{\mu,\nu},$$
and similarly for $\cI F$ without the factor $\mu$. We sum over four families of regions.

{\it Regions with $\mu\ge1$, $\nu\ge1$.} For $\nu\le\mu^2$ the term is $\mu^{-\alpha}\nu^{1-\beta}(1+\log\frac{\mu^2}\nu)$, and for $\nu\ge\mu^2$ it is $\mu^{2-\alpha}\nu^{-\beta}(1+\log\frac\nu{\mu^2})$. Since $0<\beta<1$, both geometric sums over $\nu$ are dominated by their value at $\nu=\mu^2$, which gives $\mu^{2-\alpha-2\beta}$. The sum over $\mu\ge1$ is finite since $\alpha+2\beta>2$.

{\it Regions with $\mu\ge1$, $\nu=0$.} The term is $\mu^{1-\alpha}\cdot\mu^{-1}(1+\log\mu)$, which is summable since $\alpha>0$.

{\it Regions with $\mu<1$, $\nu\ge1$.} Here $\frac\nu\mu\ge\mu$ and the term is $\mu^2\nu^{-\beta}(1+\log\frac\nu{\mu^2})$, which is summable over $\mu<1$ and $\nu\ge1$ since $\beta>0$.

{\it Regions with $\mu<1$, $\nu=0$.} The term is $\mu^2(1+|\log\mu|)$, which is summable over $\mu<1$.

The same sums without the factor $\mu$ bound $\cI F$. Hence $\cI F$ and $\cI(|k_2|F)$ are bounded. Moreover, given $\eps>0$, there is a finite set $J$ of indices such that the sum over the indices outside $J$ is less than $\eps$. The union $E_J$ of the regions with indices in $J$ is bounded, and $\cI(F\1_{E_J})\in C_0(\C)$ by \eqref{potentialsets}. Thus $\cI F$ is a uniform limit of functions of $C_0(\C)$, and $\cI F\in C_0(\C)$; the same argument applies to $\cI(|k_2|F)$. For the integral in \eqref{weightsF}, the contribution of $E_{\mu,\nu}$ is at most $|E_{\mu,\nu}|\jb\mu^{1-\alpha}\max(1,\nu)^{-\beta}\sup_{E_{\mu,\nu}}\jb k^{-1}$. In the four families above, this is bounded respectively by $\mu^{-\alpha}\nu^{1-\beta}$ (if $\nu\le\mu^2$) or $\mu^{2-\alpha}\nu^{-\beta}$ (if $\nu\ge\mu^2$), by $\mu^{-\alpha}$, by $\mu\nu^{-\beta}$, and by $\mu\log\frac2\mu$, where in the last case we used $\int_{|k_1|<1/\mu}\frac{dk_1}{\jb{k_1}}\lesssim\log\frac2\mu$. All these terms are summable, as above.\\

\noindent{\bf Step 3:} Proof of (ii) and (iii). Since $1+|k_2|\le\sqrt2\jb{\zeta_1}$ and $\jb{\zeta}^{-6}\le\jb{\zeta_1}^{-3}\jb{\zeta_2}^{-3}$, we have $(1+|k_2|)\jb{\zeta}^{-6}\le\sqrt2\jb{\zeta_1}^{-2}\jb{\zeta_2}^{-3}\le\sqrt2F_{3/2,3/4}$, and the first claim of (ii) follows from (i) and Lemma \ref{lemmacauchy} (iv). The integrals in \eqref{weightsE} are computed with \eqref{changevariables}: $|k_2|=\frac{|\zeta_1|}2$, and $|k_2|\jb k\ge\frac1{\sqrt2}(|k_2|+|k_1k_2|)\ge\frac{|\zeta|}{4\sqrt2}$, so that they are bounded by
$$\frac18\int_{\R^2}\Big(1+\frac{|\xi|}2+\frac{4\sqrt2}{|(\xi,\eta)|}\Big)\jb{(\xi,\eta)}^{-6}d\xi\,d\eta<\infty .$$
We now prove the tail bounds, for $p\in\{0,1,2\}$ and $R\ge1$:
\be
\label{tailsk1k2}
\begin{gathered}
\sup_{k\in\C}\cI\big(|k_2|^p\jb{\zeta}^{-6}\1_{\{|k'_1|>R\}}\big)(k)\lesssim\frac{1+\log R}{R^{1+p}},\\
\sup_{k\in\C}\cI\big(|k_2|^p\jb{\zeta}^{-6}\1_{\{|k'_2|>R\}}\big)(k)\lesssim\frac{1+\log R}{R^{2}} .
\end{gathered}
\ee
For the first one, decompose $\{L\le|k_1|<2L\}$, $L\ge R$ dyadic, into the strip $\{|k_2|\le\frac1L\}$, covered by two rectangles with sides $L$ and $\frac2L$ on which the density is at most $L^{-p}$, and the sets $\{2^{q-1}L^{-1}<|k_2|\le2^qL^{-1}\}$, $q\ge1$, covered by four rectangles with sides $L$ and $2^{q-1}L^{-1}$, on which $|\zeta_2|\ge2^{q+1}$, so that the density is at most $(2^qL^{-1})^p2^{-6q}$. By \eqref{potentialsets} their potentials are bounded by $C\frac{1+\log L}{L^{1+p}}$ and by $C2^{q(p-5)}\frac{1+\log L+q}{L^{1+p}}$; summing over $q\ge1$ and over $L\ge R$ gives the first bound. For the second one, on $\{\mu\le|k_2|<2\mu\}$, $\mu\ge R$, we use $\jb{\zeta}^{-6}\le\jb{\zeta_1}^{-3}\jb{\zeta_2}^{-3}$ and the regions of Step 1: the term with $\nu=0$ is $\mu^{p-4}(1+\log\mu)$ and dominates the terms with $\nu\ge1$; the sum over $\mu\ge R$ is $\lesssim R^{p-4}(1+\log R)\le R^{-2}(1+\log R)$. Since $\{|k'|>R\}\subset\{|k_1'|>\frac R2\}\cup\{|k_2'|>\frac R2\}$, \eqref{tailsk1k2} with $p=0$ gives \eqref{weightstail} for $R\ge2$, and the case $1\le R\le2$ follows from the boundedness of $\cI\jb\zeta^{-6}$. With $p=1,2$ it gives $\sup_k\cI((|k_2|+k_2^2)\jb\zeta^{-6}\1_{\{|k'|>R\}})\lesssim R^{-1}$ for $R\ge2$. Finally, for $|k|\ge4$ we split the potential in \eqref{weightsE} at $|k'|=\frac{|k|}2$. On $\{|k'|\le\frac{|k|}2\}$, $|k-k'|\ge\frac{|k|}2$, and the contribution is at most $\frac2{\pi|k|}$ times the integral of $(|k_2|+k_2^2)\jb\zeta^{-6}$, which is finite. On $\{|k'|>\frac{|k|}2\}$ we use the tail bound just proved. For $|k|\le4$ the potential is bounded by (i).
\end{proof}

\section{The inverse problem}
\label{sectioninverse}

In this section a scattering coefficient $a$ is given, and we study the function $u$ which is reconstructed from $a$ by the inverse transform. We first show that the linear part of the reconstruction is a free solution (Subsection \ref{subsectionreconstruction}). We then prove that its nonlinear part is negligible on the propagation cones of free waves (Proposition \ref{propcone}), which, combined with a pairing argument (Lemma \ref{lemmaweak}), gives weak scattering. Finally we prove that $u$ solves KP-II and satisfies the nonlinear Plancherel identity when $a$ is smooth with compact support (Propositions \ref{propcompact} and \ref{propplancherel}), and we extend these properties to rapidly decaying coefficients (Proposition \ref{proprapid}). The relation between $a$ and the initial datum is the subject of Section \ref{sectiondirect}.

\subsection{The reconstruction formula}
\label{subsectionreconstruction}

\noindent{\it The phase.} For $k\in\C$ and $(t,X)\in\R\times\R^2$ we define the phase of the free plane wave at the frequency $\zeta(k)=(-2k_2,-4k_1k_2)$:
\be
\label{defphase}
\phi(k;t,X):=X\cdot\zeta(k)+t\,\omega(\zeta(k))=-2k_2x-4k_1k_2y+8t(3k_1^2k_2-k_2^3),
\ee
and we write $P:=e^{i\phi}$. The phase is, up to the factor $-2$, the imaginary part of a holomorphic polynomial: $\phi=-2\Im(kx+k^2y-4k^3t)$. Since $\pa_k\Im f=\frac1{2i}f'$ for a holomorphic function $f$, we obtain
\be
\label{dphase}
\pa_k\phi=i(x+2ky-12k^2t),\qquad \dbar_k\phi=\overline{\pa_k\phi}.
\ee
Moreover $|P|=1$, $P(\bar k)=\overline{P(k)}$, $P=1$ on the real axis, and $\pa_xP=i\zeta_1(k)P=-2ik_2P$.

\medskip
\noindent{\it The linear part of the reconstruction.} Let $a$ be a function on $\C\setminus\R$ with $|k_2|^{\frac12}a\in L^2(\C)$. We define $\cB a\in L^2(\R^2)$ by
\be
\label{defborn}
\widehat{\cB a}(\zeta(k)):=i\sgn(k_2)\,a(k),\qquad k\in\C\setminus\R .
\ee
This defines $\widehat{\cB a}$ almost everywhere, since $\zeta$ is a bijection from $\C\setminus\R$ onto $\{\xi\neq0\}$, and the change of variables \eqref{changevariables} gives
\be
\label{bornisometry}
\|\cB a\|_{L^2}^2=8\int_\C|k_2|\,|a(k)|^2\,dA(k).
\ee
Since $\zeta(\bar k)=-\zeta(k)$, the function $\cB a$ is real-valued if and only if $a(\bar k)=\overline{a(k)}$; we then say that $a$ is real. For a function $g$ on $\C$ we write $g^*(k):=\overline{g(\bar k)}$, so that $a$ is real if and only if $a^*=a$; moreover $P^*=P$, and a direct computation gives $\dbar(g^*)=(\dbar g)^*$. If moreover $k_2a\in L^1(\C)$, then $\widehat{\cB a}\in L^1(\R^2)$, and the change of variables $(\xi,\eta)=\zeta(k)$ in the Fourier integral of $S(t)\cB a$ gives
\be
\label{freespectral}
\begin{aligned}
S(t)\cB a(X)&=\frac1{2\pi}\int_{\R^2}e^{i(X\cdot(\xi,\eta)+t\omega(\xi,\eta))}\widehat{\cB a}(\xi,\eta)\,d\xi\,d\eta\\
&=\frac{4i}\pi\int_\C k_2\,a\,P\,dA=-\frac2\pi\int_\C\pa_x(aP)\,dA .
\end{aligned}
\ee
We shall see in Section \ref{sectiondirect} that the scattering coefficient of an initial datum $u_0$ is $a(k)=-i\sgn(k_2)\widehat{(u_0\chi_k)}(\zeta(k))$, where $\chi_k$ is the normalized eigenfunction of the heat operator with potential $u_0$, defined in \eqref{defchi}, which equals $1+O(u_0)$ for small data. Hence, to first order in $u_0$,
\be
\label{bornapprox}
a(k)=-i\sgn(k_2)\,\hat u_0(\zeta(k)),\qquad\mbox{that is,}\qquad \cB a=u_0 .
\ee
This is the Born approximation of scattering theory, which explains the name of $\zeta(k)$.

\medskip
\noindent{\it The inverse problem.} For every $(t,X)$, the eigenfunction $m=m(\cdot;t,X)$ is the solution of
\be
\label{eqm}
m=1+\cC(aP\overline m),
\ee
and the reconstructed function is
\be
\label{reconstruction}
u(t,X):=-\frac2\pi\int_\C\pa_x\big(aP\overline m\big)\,dA(k).
\ee
Formally, \eqref{eqm} gives $m=1+\frac{m^{(1)}}k+O(k^{-2})$ as $|k|\to\infty$, with $m^{(1)}=\frac1\pi\int aP\overline m\,dA$, and \eqref{reconstruction} is the classical formula $u=-2\pa_xm^{(1)}$, see Proposition \ref{propcompact}. We keep the derivative $\pa_x$ inside the integral because $a$ itself need not be integrable near the real axis, while $k_2a$ and $a\pa_xm$ are. Comparing \eqref{reconstruction} with \eqref{freespectral}, and using $\pa_x(aP(\overline m-1))=aP\big(-2ik_2(\overline m-1)+\overline{\pa_xm}\big)$, we obtain the decomposition of $u$ into its linear and nonlinear parts:
\be
\label{remainder}
u(t)-S(t)\cB a=-\frac2\pi\int_\C a\,P\,B\,dA,\qquad B:=-2ik_2\,\overline{(m-1)}+\overline{\pa_xm} .
\ee
Scattering will follow from the smallness of the right-hand side, which is driven by $m-1$ and $\pa_xm$.

\medskip
\noindent{\it Rapidly decaying coefficients.} We shall work with the following class of coefficients. We say that $a$ satisfies \eqref{hypa} if $a\in C^1(\C\setminus\R)$ is real, has continuous extensions to $\overline{\C^+}$ and to $\overline{\C^-}$, and
\be
\label{hypa}
|a(k)|+\jb k^{-1}|\nabla_ka(k)|\lesssim_q\jb{\zeta(k)}^{-q}\qquad\mbox{for every}\ q\ge0\ \mbox{and}\ k\in\C\setminus\R .
\ee
Every real $a\in C^\infty_c(\C\setminus\R)$ satisfies \eqref{hypa}, and so does the scattering coefficient of every smooth exponentially decaying datum (Proposition \ref{propdirect}).

\begin{lemma}[The inverse problem for rapidly decaying coefficients]
\label{lemmainverse}
Let $a$ satisfy \eqref{hypa}. Then $M:=\|\cI|a|\|_\infty<\infty$, and for every $(t,X)$ the equation \eqref{eqm} has a unique solution $m(\cdot;t,X)\in C_b(\C)$. Moreover:
\begin{enumerate}
\item $\|m(\cdot;t,X)\|_\infty\le1+Me^{2M}$ and $m(\bar k;t,X)=\overline{m(k;t,X)}$;
\item $m$ is twice differentiable in $x$ as a $C_b(\C)$-valued function, $\pa_xm$ is the solution in $C_0(\C)$ of
\be
\label{eqmx}
\pa_xm=\cC(aP\overline{\pa_xm})+\cC(-2ik_2aP\overline m),
\ee
and $|\pa_xm(k;t,X)|+|\pa_x^2m(k;t,X)|\lesssim\jb k^{-1}$ uniformly in $(t,X)$;
\item the integral \eqref{reconstruction} converges absolutely, $u$ and $\pa_xu$ are real-valued, bounded and continuous on $\R\times\R^2$, and \eqref{remainder} holds;
\item {\em (Cancellation).} $m$ and $\pa_xm$ belong to $W^{1,p}_{\rm loc}(\C)$ for every $p\in(2,\infty)$, and the function $B$ of \eqref{remainder} satisfies
\be
\label{cancellation}
\pa_kB=-\overline{(m-1)}+\overline{aP}\,\pa_xm .
\ee
\end{enumerate}
\end{lemma}

The two terms of $\pa_kB$ are small when $m-1$ and $\pa_xm$ are small, while a direct differentiation of $B$ produces the large terms $\pm2ik_2\overline{aP}m$, which cancel. The reason is that $B=\overline P\,\pa_x\Psi$ with $\Psi:=P\,\overline{(m-1)}$, where $\Psi$ solves $(\pa_k-i\pa_k\phi)\Psi=\overline am$, an equation whose source carries no phase, and $\pa_x\pa_k\phi=i$ is constant by \eqref{dphase}.

\begin{proof}
{\bf Step 1:} Bounds and symmetry. By \eqref{hypa}, $|a|\lesssim\jb\zeta^{-6}$, and Lemma \ref{lemmaweights} (ii), together with Lemma \ref{lemmacauchy} (iv), gives $\cI|a|\in C_0$ and $\cI(|k_2a|)\in C_0$. Since $|aP|=|a|$, Lemma \ref{lemmavekua} (ii) applied to the coefficient $aP$ gives the unique solution $m$ and the bound of (i), uniformly in $(t,X)$. By Lemma \ref{lemmacauchy} (ii), $\dbar m=aP\overline m$, hence $\dbar m^*=a^*P^*\overline{m^*}=aP\overline{m^*}$, because $a^*=a$ and $P^*=P$. Since $m^*-1\in C_0$, the uniqueness part of Lemma \ref{lemmavekua} (i), applied to $v'=m^*-1$ and $f=aP$, gives $m^*=m$.\\

\noindent{\bf Step 2:} Derivatives in $x$. Let $f:=-2ik_2aP\overline m$. Then $\cI|f|$ is bounded by $2\|m\|_\infty\cI(|k_2a|)$, which belongs to $C_0$. Lemma \ref{lemmavekua} (i) provides a unique solution $w\in C_0$ of the equation $w=\cC(aP\overline w)+\cC f$, and
$$|w(k)|\le e^{2M}\cI|f|(k)\le2e^{2M}\|m\|_\infty\,\cI(|k_2a|)(k)\lesssim\jb k^{-1}$$
by Lemma \ref{lemmaweights} (ii). We show that $w=\pa_xm$. For $h\neq0$, let $P_h:=P(\cdot;t,x+h,y)$ and $q_h:=h^{-1}\delta^x_hm(\cdot;t,X)$. Subtracting the equations \eqref{eqm} at $x+h$ and at $x$, we find that $q_h-w$ solves
$$\begin{aligned}q_h-w&=\cC\big(aP_h\overline{(q_h-w)}\big)+b_h,\\ b_h&:=\cC\big(a(P_h-P)\overline w\big)+\cC\Big(a\Big(\frac{P_h-P}h+2ik_2P\Big)\overline m\Big).\end{aligned}$$
Now $\frac{P_h-P}h=\frac{e^{-2ik_2h}-1}hP$ is bounded by $2|k_2|$ and converges to $-2ik_2P$ uniformly on compact sets as $h\to0$. Since $|k_2a|$ is bounded and $\cI(|k_2a|)\in C_0$, Lemma \ref{lemmacauchy} (v) gives $\|b_h\|_\infty\to0$, and Lemma \ref{lemmavekua} (iii) gives $\|q_h-w\|_\infty\le(1+Me^{2M})\|b_h\|_\infty\to0$. Thus $\pa_xm=w$ in $C_b(\C)$, which proves \eqref{eqmx} and the bound on $\pa_xm$. Differentiating \eqref{eqmx} in the same way, $\pa_x^2m$ is the solution in $C_0$ of
$$\pa_x^2m=\cC(aP\overline{\pa_x^2m})+\cC\big(-4ik_2aP\overline{\pa_xm}-4k_2^2aP\overline m\big),$$
whose source is bounded by $C(|k_2|+k_2^2)|a|$. By \eqref{pointwisevekua} and \eqref{weightsE}, $|\pa_x^2m(k)|\lesssim\cI\big((|k_2|+k_2^2)\jb{\zeta}^{-6}\big)(k)\lesssim\jb k^{-1}$. The same comparison argument, based on Lemma \ref{lemmacauchy} (v) and Lemma \ref{lemmavekua} (iii), shows that $m$, $\pa_xm$ and $\pa_x^2m$ depend continuously on $(t,X)$ in $C_b(\C)$.\\

\noindent{\bf Step 3:} The reconstruction. We have
$$\begin{gathered}\pa_x(aP\overline m)=aP\big(-2ik_2\overline m+\overline{\pa_xm}\big),\\ \pa_x^2(aP\overline m)=aP\big(-4k_2^2\overline m-4ik_2\overline{\pa_xm}+\overline{\pa_x^2m}\big).\end{gathered}$$
By (i), (ii) and \eqref{hypa}, both are bounded by $C(|k_2|+k_2^2+\jb k^{-1})\jb{\zeta}^{-6}$, which is integrable by \eqref{weightsE}. Hence \eqref{reconstruction} converges absolutely, we may differentiate under the integral sign, and $u$ and $\pa_xu$ are bounded and, by Step 2 and dominated convergence, continuous. By Step 1, the integrand of \eqref{reconstruction} is invariant under $g\mapsto g^*$, so that the change of variables $k\mapsto\bar k$ shows that the integral is real: $u$ is real-valued. Finally, \eqref{freespectral} applies since $k_2a\in L^1$, and \eqref{remainder} follows.\\

\noindent{\bf Step 4:} Proof of (iv). By Lemma \ref{lemmavekua} (v), $m$ and $\pa_xm$ belong to $W^{1,p}_{\rm loc}$, and by Lemma \ref{lemmacauchy} (ii), \eqref{eqm} and \eqref{eqmx} give $\dbar m=aP\overline m$ and $\dbar\pa_xm=aP\overline{\pa_xm}-2ik_2aP\overline m$. Taking complex conjugates, $\pa_k\overline m=\overline{aP}\,m$ and $\pa_k\overline{\pa_xm}=\overline{aP}\,\pa_xm+2ik_2\overline{aP}\,m$. Since $\pa_k(-2ik_2)=-1$, we obtain
$$\pa_kB=-\overline{(m-1)}-2ik_2\overline{aP}\,m+\overline{aP}\,\pa_xm+2ik_2\overline{aP}\,m,$$
which is \eqref{cancellation}.
\end{proof}

\subsection{Free waves and weak limits}

Weak limits in $L^2$ can be tested on free waves whose frequencies avoid the line $\{\xi=0\}$, since these are dense. Such free waves live on the propagation cones $\{X/t\in V\}$ with $V$ compact in the region $\cV$ of \eqref{defcV}, and there they have size $|t|^{-1}$. Hence a perturbation of size $o(|t|^{-1})$ on the cones is invisible in the weak limit. We use the velocity variable $v:=X/t$.

\begin{lemma}[Free waves on propagation cones]
\label{lemmafree}
Let $\hat\varphi\in C^\infty_c(\{\xi\neq0\})$, and let $V\subset\cV$ be a compact set whose interior contains the set of group velocities $\{-\nabla\omega(\xi,\eta):\ (\xi,\eta)\in\supp\hat\varphi\}$. Then
\be
\label{freecones}
\|S(t)\varphi\|_{L^1}\lesssim1+|t|,\qquad \lim_{|t|\to\infty}\big\|\1_{\{X/t\notin V\}}S(t)\varphi\big\|_{L^2}=0 .
\ee
\end{lemma}

\begin{proof}
The set of group velocities is contained in $\cV$ by \eqref{groupvelocity}, so that such a $V$ exists. For $|t|\le1$, $\|S(t)\varphi\|_{L^1}\lesssim1$, since $\jb X^3S(t)\varphi$ is bounded uniformly in $|t|\le1$. Let $|t|\ge1$. We write $S(t)\varphi(X)=\frac1{2\pi}\int e^{it\Phi_v(\xi,\eta)}\hat\varphi(\xi,\eta)\,d\xi\,d\eta$ with $\Phi_v(\xi,\eta):=v\cdot(\xi,\eta)+\omega(\xi,\eta)$ and $v=X/t$. The critical points of $\Phi_v$ are the frequencies with group velocity $v$, and they are nondegenerate, with $|\det\nabla^2\Phi_v|=36$. On the compact set $\supp\hat\varphi$, all derivatives of $\omega$ are bounded; covering $\supp\hat\varphi$ by finitely many small discs on each of which $\Phi_v$ has at most one critical point, uniformly for $v$ in a bounded set, the stationary phase estimate \cite[Theorem 7.7.5]{Hormander} gives $|S(t)\varphi(X)|\lesssim|t|^{-1}$ uniformly for $|v|\le C$. If $v\notin V$, then $|\nabla\Phi_v|\gtrsim1+|v|$ on $\supp\hat\varphi$, and the nonstationary phase estimate \cite[Theorem 7.7.1]{Hormander}, applied to the phase $\Phi_v/(1+|v|)$, whose derivatives are bounded uniformly in $v$, with the large parameter $|t|(1+|v|)$, gives $|S(t)\varphi(X)|\lesssim_n(|t|(1+|v|))^{-n}$ for every $n$. Integrating the first bound over $\{|X|\le C|t|\}$ and the second one over its complement gives the first part of \eqref{freecones}, and
$$\int_{\{X/t\notin V\}}|S(t)\varphi|^2\,dX\lesssim|t|^{2}\int_{\R^2}(|t|(1+|v|))^{-6}\,dv\lesssim|t|^{-4}$$
gives the second part.
\end{proof}

\begin{lemma}[From the cones to scattering]
\label{lemmaweak}
Let $u_n\in C(\R;L^2(\R^2))$ and $g_n\in L^2(\R^2)$, $n\ge1$, with $\sup_{n,t}\|u_n(t)\|_{L^2}+\sup_n\|g_n\|_{L^2}<\infty$. Assume that
\begin{enumerate}
\item for every $t$, $u_n(t)\weakto u(t)$ in $L^2$, and $g_n\weakto g$ in $L^2$;
\item for every compact set $V\subset\cV$,
\be
\label{hypcone}
\lim_{|t|\to\infty}\ \sup_{n\ge1}\ \sup_{X/t\in V}\ |t|\,\big|u_n(t,X)-S(t)g_n(X)\big|=0,
\ee
where the supremum in $X$ is an essential supremum.
\end{enumerate}
Then $S(-t)u(t)\weakto g$ in $L^2$ as $t\to\pm\infty$. If moreover $\|u(t)\|_{L^2}=\|g\|_{L^2}$ for every $t$, then $\|u(t)-S(t)g\|_{L^2}\to0$ as $t\to\pm\infty$.
\end{lemma}

The lemma applies in particular to a single function, $u_n=u$ and $g_n=g$.

\begin{proof}
Let $\hat\varphi\in C^\infty_c(\{\xi\neq0\})$ and let $V$ be as in Lemma \ref{lemmafree}. Since $S(t)$ is unitary,
$$\big(S(-t)u(t)-g,\varphi\big)=\big(u(t)-S(t)g,S(t)\varphi\big)=\lim_{n\to\infty}\big(u_n(t)-S(t)g_n,S(t)\varphi\big)$$
for every fixed $t$. We split the last pairing into the integrals over $\{X/t\in V\}$ and over its complement. By \eqref{hypcone} and \eqref{freecones}, the first one is bounded by
$$\sup_n\sup_{X/t\in V}|u_n(t,X)-S(t)g_n(X)|\ \|S(t)\varphi\|_{L^1}=o(|t|^{-1})\cdot O(|t|)=o(1),$$
and by the Cauchy--Schwarz inequality and \eqref{freecones} the second one is bounded by $\sup_n(\|u_n(t)\|_{L^2}+\|g_n\|_{L^2})\|\1_{\{X/t\notin V\}}S(t)\varphi\|_{L^2}=o(1)$, both uniformly in $n$. Hence $(S(-t)u(t)-g,\varphi)\to0$ as $t\to\pm\infty$. The functions $\varphi$ with $\hat\varphi\in C^\infty_c(\{\xi\neq0\})$ are dense in $L^2$, and $S(-t)u(t)$ is bounded in $L^2$, so that $S(-t)u(t)\weakto g$. If $\|S(-t)u(t)\|_{L^2}=\|u(t)\|_{L^2}=\|g\|_{L^2}$, the weak convergence, together with the convergence of the norms, implies the strong convergence in $L^2$.
\end{proof}
\subsection{The nonlinear remainder on propagation cones}
\label{subsectioncone}

We now estimate the nonlinear part \eqref{remainder} of the reconstruction at the points $X=tv$, $v$ in a compact subset of $\cV$. Since $\phi$ is linear in $(t,X)$, we have $\phi(k;t,tv)=t\,\Phi_v(\zeta(k))$ with $\Phi_v(\xi,\eta)=v\cdot(\xi,\eta)+\omega(\xi,\eta)$, so that, away from the real axis, where $\zeta$ is a local diffeomorphism, the critical points of $\phi$ correspond to the frequencies with group velocity $v$. Explicitly, for $v\in\cV$ let
\be
\label{defkpm}
k_\pm(v):=\frac{v_2}{12}\pm\frac i{12}\sqrt{-12v_1-v_2^2};
\ee
for $(\xi,\eta)=\zeta(k_\pm(v))$, a direct computation with \eqref{groupvelocity} gives
$$\frac{6\eta}{\xi}=12\Re k_\pm(v)=v_2,\qquad -3\xi^2-\frac{3\eta^2}{\xi^2}=-12(\Im k_\pm(v))^2-\frac{v_2^2}{12}=v_1 :$$
the frequencies $\zeta(k_\pm(v))$ are the two frequencies with group velocity $v$.

\begin{lemma}[Critical points of the phase]
\label{lemmaphase}
Let $V\subset\cV$ be compact. There are constants $0<c_V\le1\le C_V$ such that the following holds for $t\neq0$, $X=tv$, $v\in V$.
\begin{enumerate}
\item $\pa_k\phi(k;t,X)=-12it\,(k-k_+(v))(k-k_-(v))$. The points $k_\pm(v)$ are the only critical points of $\phi(\cdot;t,X)$, and they satisfy $|k_\pm(v)|\le C_V$ and $\pm\Im k_\pm(v)\ge c_V$.
\item For all $k\in\C$, $|\pa_k^2\phi(k)|\le C_V|t|\jb k$; for every $r>0$ there is $c_{V,r}>0$ such that $|\pa_k\phi(k)|\ge c_{V,r}|t|\jb k^2$ whenever $|k-k_+(v)|\ge r$ and $|k-k_-(v)|\ge r$; and on the real axis, $|\pa_k\phi(k_1)|\ge c_V|t|(1+k_1^2)$.
\item In the real coordinates $(k_1,k_2)$, the Hessian matrix of $\phi$ at $k_\pm(v)$ is $48\,t\,\Im k_\pm(v)\begin{pmatrix}1&0\\0&-1\end{pmatrix}$.
\end{enumerate}
\end{lemma}

\begin{proof}
By \eqref{dphase}, $\pa_k\phi=it(v_1+2kv_2-12k^2)=-12it\big(k^2-\frac{v_2}6k-\frac{v_1}{12}\big)$, whose roots are \eqref{defkpm} because $v_2^2+12v_1<0$ on $\cV$. Since $\phi$ is real, $\nabla_k\phi=0$ if and only if $\pa_k\phi=0$. The bounds on $k_\pm$ follow from the compactness of $V$ and from $v_1+\frac{v_2^2}{12}\le-c<0$ on $V$. This proves (i). For (ii), $\pa_k^2\phi=it(2v_2-24k)$; if $|k-k_\pm|\ge r$ for both signs, then $|k-k_+||k-k_-|\gtrsim_{V,r}\jb k^2$ by the bounds on $k_\pm$; and for real $k_1$,
$$|v_1+2k_1v_2-12k_1^2|=12\Big(k_1-\frac{v_2}{12}\Big)^2-\Big(v_1+\frac{v_2^2}{12}\Big)\gtrsim1+k_1^2 .$$
For (iii), at $k_\pm$ we have $2v_2-24k_\pm=-24i\Im k_\pm$, so that $\pa_k^2\phi=24t\Im k_\pm$ is real. Since $\phi$ is harmonic, $\pa_{k_1}^2\phi=-\pa_{k_2}^2\phi$ and $\pa_k^2\phi=\frac12\big(\pa_{k_1}^2\phi-i\pa_{k_1}\pa_{k_2}\phi\big)$, which gives the Hessian.
\end{proof}

We fix $r_0:=\frac{c_V}8$ and write $D_\pm:=D(k_\pm(v),r_0)$; the discs $D(k_\pm(v),4r_0)$ are disjoint and contained in $\C^\pm$. The following proposition bounds the nonlinear part of the reconstruction formula on the propagation cones $\{X=tv,\ v\in V\}$. Its proof uses only the inverse problem and holds uniformly over families of coefficients with a common majorant; the nonlinear Plancherel identity is not used.

\begin{proposition}[The nonlinear remainder on propagation cones]
\label{propcone}
Let $(a_n)_{n\ge1}$ be a family of coefficients which satisfy \eqref{hypa}, let $u_n$ be the functions given by \eqref{reconstruction} with $a=a_n$, and $g_n:=\cB a_n$. Assume that there is a function $F\in L^\infty(\C)$ which satisfies the conditions \eqref{weightsF} of Lemma \ref{lemmaweights} (i), and such that, for all $n$ and $k\in\C\setminus\R$,
\be
\label{hypF1}
|a_n(k)|\le F(k),\qquad |\nabla_ka_n(k)|\le\jb kF(k).
\ee
Then \eqref{hypcone} holds for every compact set $V\subset\cV$.
\end{proposition}

\begin{proof}
Fix $V$, and let $|t|\ge1$ and $v=X/t\in V$. All constants below depend on $V$ and $F$, but not on $n$, $t$ and $X$. We write $m_n$ for the eigenfunction of $a_n$, $M:=\|\cI F\|_\infty$ and $M':=\|\cI(|k_2|F)\|_\infty$. The decay \eqref{hypa} of each $a_n$ is used only to justify the computations qualitatively, through Lemma \ref{lemmainverse}.\\

\noindent{\bf Step 1:} Uniform bounds. By Lemma \ref{lemmavekua}, whose bounds depend on $|a_nP|\le F$ only, $\|m_n\|_\infty\le1+Me^{2M}$. By \eqref{eqmx} and \eqref{pointwisevekua}, $\|\pa_xm_n\|_\infty\le2e^{2M}(1+Me^{2M})M'$. Hence $\|m_n\|_\infty+\|\pa_xm_n\|_\infty\lesssim1$.\\

\noindent{\bf Step 2:} Smallness of the eigenfunctions. We claim that $\eps(t)\to0$ as $|t|\to\infty$, where
\be
\label{defepst}
\eps(t):=\sup_{n\ge1}\ \sup_{X/t\in V}\Big(\|m_n(\cdot;t,X)-1\|_\infty+\|\pa_xm_n(\cdot;t,X)\|_\infty\Big).
\ee
By \eqref{eqm}, $m_n-1=\cC(a_nP)+\cC\big(a_nP\overline{(m_n-1)}\big)$, and by \eqref{eqmx}, $\pa_xm_n=\cC(a_nP\overline{\pa_xm_n})+b_n$ with $b_n:=\cC(-2ik_2a_nP)+\cC\big(-2ik_2a_nP\overline{(m_n-1)}\big)$, where $\|b_n\|_\infty\le2\|\cC(k_2a_nP)\|_\infty+2M'\|m_n-1\|_\infty$. Lemma \ref{lemmavekua} (iii) gives
$$\|m_n-1\|_\infty+\|\pa_xm_n\|_\infty\lesssim\|\cC(a_nP)\|_\infty+\|\cC(k_2a_nP)\|_\infty,$$
so that $\eps(t)\to0$ follows from the oscillatory estimate
\be
\label{oscillatory}
\lim_{|t|\to\infty}\ \sup_{n\ge1}\ \sup_{X/t\in V}\ \|\cC(c_nP)\|_\infty=0,\qquad c_n\in\{a_n,\ k_2a_n\}.
\ee
To prove \eqref{oscillatory}, let $F_1:=(1+|k_2|)F$, which satisfies $|c_n|\le F_1$, $\cI F_1\in C_0$, and is bounded on bounded sets. Let $\delta>0$. By Lemma \ref{lemmacauchy} (iv) and \eqref{potentialsets} we can choose $R\ge C_V+1$ and then $r_1\in(0,\frac{c_V}2)$ such that
$$\sup_{k\in\C}\cI(F_1\1_{\{|k'|>R\}})(k)+\sup_{k\in\C}\cI(F_1\1_{\{|k'|\le2R,\,|k_2'|<2r_1\}})(k)+8r_1\sup_{|k'|\le2R}F_1(k')\le\delta .$$
Let $\vartheta_*\in C^\infty(\C)$ with $0\le\vartheta_*\le1$, $\vartheta_*=1$ on $\{|k|\le R,\ |k_2|\ge2r_1\}\setminus\big(D(k_+,2r_1)\cup D(k_-,2r_1)\big)$, $\vartheta_*=0$ outside $\{|k|\le2R,\ |k_2|\ge r_1\}\setminus\big(D(k_+,r_1)\cup D(k_-,r_1)\big)$, and $|\nabla\vartheta_*|\lesssim_{R,r_1}1$; such a function is obtained from fixed cutoffs translated with $k_\pm(v)$. Then $1-\vartheta_*$ is bounded by the sum of the indicator functions of the four sets $\{|k|>R\}$, $\{|k|\le2R,|k_2|<2r_1\}$ and $D(k_\pm,2r_1)$, the last two of which are contained in $\{|k|\le2R\}$ since $R\ge C_V+1$, and \eqref{potentialsets} gives $\|\cC((1-\vartheta_*)c_nP)\|_\infty\le\cI(F_1(1-\vartheta_*))\le\delta$. On the support of $\vartheta_*$ we have $|\dbar\phi|=|\pa_k\phi|\ge12|t|r_1^2$ and $|\nabla_k\dbar\phi|\lesssim_R|t|$, while $f:=\vartheta_*c_n$ satisfies $\|f\|_\infty+\|\nabla f\|_\infty\lesssim_{R,r_1}1$ by \eqref{hypF1}. Since $f\in C^1(\C)$ has compact support, Lemma \ref{lemmacauchy} (vi) and \eqref{potentialsets} give
$$\|\cC(Pf)\|_\infty\le\Big\|\frac f{\dbar\phi}\Big\|_\infty+\Big\|\cI\Big|\dbar\frac f{\dbar\phi}\Big|\Big\|_\infty\lesssim_{R,r_1}\frac1{|t|}.$$ Therefore $\limsup_{|t|\to\infty}\sup_n\sup_{X/t\in V}\|\cC(c_nP)\|_\infty\le\delta$ for every $\delta>0$, which proves \eqref{oscillatory}.\\

\noindent{\bf Step 3:} The remainder and its cancellation. By \eqref{remainder},
\be
\label{remainderfn}
\begin{gathered}
u_n(t,X)-S(t)g_n(X)=-\frac2\pi\int_\C f_n\,P\,dA,\\
f_n:=a_nB_n,\qquad B_n:=-2ik_2\overline{(m_n-1)}+\overline{\pa_xm_n}.
\end{gathered}
\ee
By Lemma \ref{lemmainverse} (iv), $B_n\in W^{1,p}_{\rm loc}(\C)$ for $p<\infty$ and $\pa_kB_n=-\overline{(m_n-1)}+\overline{a_nP}\,\pa_xm_n$. By \eqref{defepst} and \eqref{hypF1},
\be
\label{boundfn}
\begin{gathered}
|B_n|\le2(1+|k_2|)\eps(t),\qquad |\pa_kB_n|\lesssim\eps(t),\\
|f_n|\le2(1+|k_2|)F\eps(t),\qquad |\pa_kf_n|\lesssim(1+|k_2|)\jb kF\eps(t),
\end{gathered}
\ee
where $\pa_kf_n$ is a locally bounded function on $\C\setminus\R$.\\

\noindent{\bf Step 4:} Contribution of the critical points. By \eqref{boundfn}, $|f_n|+|\pa_kf_n|\lesssim\eps(t)$ on the discs $D(k_\pm,4r_0)$. Fix $p\in(2,\infty)$ and $\sigma:=1-\frac2p$. For $q\in W^{1,p}(D(k_\pm,4r_0))$ we have the interior estimate
$$\|\nabla q\|_{L^p(D(k_\pm,2r_0))}\lesssim\|\pa_kq\|_{L^p(D(k_\pm,4r_0))}+\|q\|_{L^p(D(k_\pm,4r_0))}:$$
indeed, with $\vartheta_2(k):=\vartheta(\frac{k-k_\pm}{2r_0})$, which equals $1$ on $D(k_\pm,2r_0)$ and is supported in $D(k_\pm,4r_0)$, Lemma \ref{lemmacauchy} (ii), applied to smooth functions and extended by density to compactly supported $W^{1,p}$ functions, gives $\overline{\vartheta_2q}=\cC\big(\overline{\pa_k(\vartheta_2q)}\big)$; hence $\overline{\dbar(\vartheta_2q)}=\pa_k\cC\big(\overline{\pa_k(\vartheta_2q)}\big)$, and $\pa_k\cC$ is the Beurling transform, which is bounded on $L^p$ \cite{Stein}. Applied to $q=f_n$, this gives $\|f_n\|_{W^{1,p}(D(k_\pm,2r_0))}\lesssim\eps(t)$, and Morrey's inequality gives
\be
\label{morrey}
|f_n(k)-f_n(k_\pm)|\lesssim\eps(t)|k-k_\pm|^\sigma\qquad\mbox{for}\ k\in D_\pm .
\ee
Let $\vartheta_\pm(k):=\vartheta(\frac{2(k-k_\pm)}{r_0})$, which equals $1$ on $D(k_\pm,\frac{r_0}2)$ and is supported in $D_\pm$, and write $\vartheta_\pm f_n=f_n(k_\pm)\vartheta_\pm+h_\pm$ with $h_\pm:=\vartheta_\pm\big(f_n-f_n(k_\pm)\big)$. By Lemma \ref{lemmaphase} (iii), $\phi/t$ has a nondegenerate critical point at $k_\pm$, with Hessian determinant bounded away from zero uniformly in $v\in V$, and no other critical point on the support of $\vartheta_\pm$. The stationary phase estimate \cite[Theorem 7.7.5]{Hormander} gives $|\int\vartheta_\pm P\,dA|\lesssim|t|^{-1}$, so that the first part contributes $O(\eps(t)|t|^{-1})$. For the second part, let $\rho:=|t|^{-\frac12}$, assume $|t|$ so large that $2\rho<\frac{r_0}2$, and let $\vartheta_\rho(k):=\vartheta\big(\frac{k-k_\pm}\rho\big)$. By \eqref{morrey},
$$\Big|\int_\C h_\pm\vartheta_\rho P\,dA\Big|\lesssim\eps(t)\int_{D(k_\pm,2\rho)}|k-k_\pm|^\sigma\,dA\lesssim\eps(t)\rho^{2+\sigma}\le\eps(t)|t|^{-1}.$$
On the support of $h_\pm(1-\vartheta_\rho)$ we integrate by parts with $\pa_kP=i(\pa_k\phi)P$; there are no boundary terms, since this function is compactly supported in $D_\pm\setminus D(k_\pm,\rho)$:
$$\int_\C h_\pm(1-\vartheta_\rho)P\,dA=-\int_\C P\,\pa_k\Big(\frac{h_\pm(1-\vartheta_\rho)}{i\pa_k\phi}\Big)dA .$$
On $D_\pm$, Lemma \ref{lemmaphase} gives $|\pa_k\phi|\gtrsim|t||k-k_\pm|$ and $|\pa_k^2\phi|\lesssim|t|$. Since $|\pa_kh_\pm|\lesssim\eps(t)$ by \eqref{boundfn}, the derivative of $h_\pm$ contributes at most $\frac{C\eps(t)}{|t|}\int_{D_\pm}\frac{dA}{|k-k_\pm|}\lesssim\frac{\eps(t)}{|t|}$. By \eqref{morrey}, the derivative of $(\pa_k\phi)^{-1}$ contributes at most $\frac{C\eps(t)}{|t|}\int_{D_\pm}|k-k_\pm|^{\sigma-2}\,dA\lesssim\frac{\eps(t)}{|t|}$, and the derivative of $\vartheta_\rho$, which is bounded by $C\rho^{-1}$ and supported in $\{\rho\le|k-k_\pm|\le2\rho\}$, contributes at most $\frac C{|t|}\eps(t)\rho^\sigma\rho^{-2}\rho^2\lesssim\frac{\eps(t)}{|t|}$. Hence each critical point contributes $O(\eps(t)|t|^{-1})$ to \eqref{remainderfn}.\\

\noindent{\bf Step 5:} Contribution of the rest of the plane. Let $f_n^0:=f_n(1-\vartheta_+-\vartheta_-)$. On its support, $|k-k_\pm|\ge\frac{r_0}2$ and, by Lemma \ref{lemmaphase} (ii) with $r=\frac{r_0}2$, $|\pa_k\phi|\gtrsim|t|\jb k^2$ and $|\pa_k^2\phi|\lesssim|t|\jb k$. For each fixed $n$, \eqref{hypa}, Lemma \ref{lemmainverse} (ii) and \eqref{weightsE} give $|f_n|\lesssim_n(|k_2|+\jb k^{-1})\jb\zeta^{-6}\in L^1(\C)$, so that $\int f_n^0P\,dA=\lim_{L\to\infty}\int f_n^0\vartheta_LP\,dA$, where $\vartheta_L(k):=\vartheta(k/L)$. We integrate by parts separately in $\C^+$ and in $\C^-$, using Green's formula $\int_{\C^\pm}\pa_kG\,dA=\mp\frac1{2i}\int_\R G^\pm\,dk_1$, which holds for compactly supported functions $G$ whose restrictions to $\C^\pm$ belong to $W^{1,1}(\C^\pm)$, with traces $G^\pm$ on $\R$. We apply it to $G:=\frac{f^0_n\vartheta_LP}{i\pa_k\phi}$, which has this regularity: $a_n$ and its first derivatives are bounded on $\supp\vartheta_L\cap\C^\pm$ by \eqref{hypa}, $m_n$ and $\pa_xm_n$ belong to $W^{1,p}_{\rm loc}(\C)$ for $p>2$ by Lemma \ref{lemmainverse} (iv), and $\pa_k\phi$ does not vanish on the support of $1-\vartheta_+-\vartheta_-$. Since $P=1$ on $\R$, we obtain
$$\int_\C f_n^0\vartheta_LP\,dA=-\int_{\C\setminus\R}P\,\pa_k\Big(\frac{f_n^0\vartheta_L}{i\pa_k\phi}\Big)dA-\frac1{2i}\int_\R\frac{f_n^+-f_n^-}{i\pa_k\phi}\,\vartheta_L\,dk_1 .$$
On the real axis $k_2=0$, so that $B_n=\overline{\pa_xm_n}$, which is continuous across the axis, and $f_n^\pm=a_n^\pm\overline{\pa_xm_n}$ is bounded by $\|F\|_\infty\eps(t)$. Together with $|\pa_k\phi(k_1)|\gtrsim|t|(1+k_1^2)$, the boundary term is bounded by $C\eps(t)|t|^{-1}\int_\R\frac{dk_1}{1+k_1^2}$. In the bulk term, the derivative falls on $f_n$, on the fixed cutoffs $\vartheta_\pm$, on $\vartheta_L$ or on $(\pa_k\phi)^{-1}$; by \eqref{boundfn} and the bounds on $\pa_k\phi$, the integrand is bounded by
$$\frac{C\eps(t)}{|t|}\Big(\frac{\jb kF_1}{\jb k^2}+F_1\1_{\{|k|\le C\}}+\frac{F_1\1_{\{L\le|k|\le2L\}}}{L\jb k^2}+\frac{\jb kF_1}{\jb k^4}\Big)\lesssim\frac{\eps(t)}{|t|}\frac{F_1}{\jb k},$$
which is integrable by \eqref{weightsF}, uniformly in $L$. Letting $L\to\infty$, the term containing $\nabla\vartheta_L$ tends to zero, and we obtain $|\int f_n^0P\,dA|\lesssim\eps(t)|t|^{-1}$.\\

\noindent{\bf Step 6:} Conclusion. Steps 4 and 5 give $\sup_n\sup_{X/t\in V}|u_n(t,X)-S(t)g_n(X)|\lesssim\eps(t)|t|^{-1}$, and $\eps(t)\to0$ proves \eqref{hypcone}.
\end{proof}
\subsection{Compactly supported coefficients}

In this subsection $a\in C^\infty_c(\C\setminus\R)$. We first prove that the reconstructed function decays in space, and that it depends analytically on the amplitude of $a$. To treat complex amplitudes, we replace the equation \eqref{eqm}, which is only real-linear in $m$, by a complex-linear system for the pair $(m,\overline m)$, in which the two occurrences of the real amplitude $s$ are replaced by the same complex number. We write $\overline{\cC}f:=\overline{\cC\overline f}$, so that $\pa_k\overline{\cC}f=f$.

The spatial decay comes from the oscillation of $P_0(k;X):=e^{iX\cdot\zeta(k)}$ in $k$, whose $\dbar$-derivative is $i\Gamma P_0$ with $\Gamma:=\dbar_k(X\cdot\zeta(k))$ of size $|X|$ on the support of the coefficient. Two integrations by parts give $\cC(bP_0)=\frac{bP_0}{i\Gamma}+O(|X|^{-2})$ and, by conjugation, $\overline\cC(\overline bP_0^{-1})=\frac{\overline bP_0^{-1}}{-i\overline\Gamma}+O(|X|^{-2})$. Inserted into the function $W_s=s\int bP_0n_s\,dA$ below, the first term of the second expansion produces the non-oscillating product $\frac{|b|^2}{-i\overline\Gamma}$, of size $|X|^{-1}$; this is why $W_s$ decays only like $|X|^{-1}$, while its derivatives in $X$ decay like $|X|^{-2}$.

\begin{lemma}[Spatial decay and analyticity in the amplitude]
\label{lemmaparametrix}
Let $b\in C^\infty_c(\C\setminus\R)$. For $s\in\C$ and $X\in\R^2$ consider the system
\be
\label{system}
m=1+s\,\cC(bP_0\,n),\qquad n=1+s\,\overline{\cC}\big(\overline b\,P_0^{-1}m\big).
\ee
There is an open set $\mathcal U\subset\C$ containing $\R$, which depends on $b$ only through $\|\cI|b|\|_\infty$, such that the following holds.
\begin{enumerate}
\item {\em (Solvability).} For $s\in\mathcal U$ and $X\in\R^2$, the system \eqref{system} has a unique solution $(m_s,n_s)\in C_b(\C)^2$. For real $s$, $n_s=\overline{m_s}$, and $m_s$ is the solution of $m=1+\cC(sbP_0\overline m)$.
\item {\em (Analyticity).} The function
$$W_s(X):=s\int_\C b\,P_0\,n_s\,dA$$
is smooth in $X$, and $\pa^\gamma_XW_s(X)$ is holomorphic in $s\in\mathcal U$ for every $X$ and every multi-index $\gamma$.
\item {\em (Decay).} For every compact set $\mathcal U'\subset\mathcal U$ and every multi-index $\gamma$,
\be
\label{decayW}
|\pa^\gamma_XW_s(X)|\lesssim|s|\,\jb X^{-1-\min(1,|\gamma|)},\qquad s\in\mathcal U',\ X\in\R^2,
\ee
where the constant depends only on $\mathcal U'$, $\gamma$, the support of $b$ and finitely many derivatives of $b$.
\end{enumerate}
\end{lemma}

\begin{proof}
If $b=0$, take $\mathcal U=\C$, $m_s=n_s=1$ and $W_s=0$; all assertions are immediate. We henceforth assume $b\neq0$.

{\bf Step 1:} Solvability. Let $\mathcal T(m,n):=\big(\cC(bP_0n),\overline{\cC}(\overline bP_0^{-1}m)\big)$. It is a complex-linear operator on $C_b(\C)^2$, with norm at most $M:=\|\cI|b|\|_\infty$ independently of $X$, since $|P_0|=1$. Let $s_0\in\R$. For data of the form $(G,\overline G)$, $G\in C_b$, the equation $(I-s_0\mathcal T)(m,n)=(G,\overline G)$ is solved by $(v,\overline v)$, where $v$ is the solution of $v=G+\cC(s_0bP_0\overline v)$ given by Lemma \ref{lemmavekua} (iii): indeed $\overline{\overline\cC(\overline bP_0^{-1}v)}=\cC(bP_0\overline v)$. Every pair $(H_1,H_2)\in C_b^2$ is the combination $(G_1,\overline{G_1})+i(G_2,\overline{G_2})$ with $G_1=\frac12(H_1+\overline{H_2})$ and $G_2=\frac1{2i}(H_1-\overline{H_2})$, and the same decomposition applied to a solution of the homogeneous equation shows that it vanishes. Hence $I-s_0\mathcal T$ is invertible on $C_b^2$, with $\|(I-s_0\mathcal T)^{-1}\|\le\Lambda(s_0):=2(1+|s_0|Me^{2|s_0|M})$. For $|s-s_0|<r(s_0):=(4M\Lambda(s_0))^{-1}$ the Neumann series
$$(I-s\mathcal T)^{-1}=\sum_{j\ge0}(s-s_0)^j\big[(I-s_0\mathcal T)^{-1}\mathcal T\big]^j(I-s_0\mathcal T)^{-1}$$
converges, is holomorphic in $s$, and has norm at most $2\Lambda(s_0)$. We set $\mathcal U:=\bigcup_{s_0\in\R}D(s_0,r(s_0))$, which depends on $b$ only through $M$. A compact set $\mathcal U'\subset\mathcal U$ is covered by finitely many of these discs, so that $(I-s\mathcal T)^{-1}$ is bounded on $\mathcal U'$ uniformly in $X$. We define $(m_s,n_s):=(I-s\mathcal T)^{-1}(1,1)$. Since $\mathcal T$ maps $C_b^2$ into $C_0^2$, $m_s-1$ and $n_s-1$ belong to $C_0$. For real $s$ the data $(1,1)$ are of the form $(G,\overline G)$, so that $n_s=\overline{m_s}$.

Differentiating \eqref{system} in $X$ as in Step 2 of the proof of Lemma \ref{lemmainverse}, we see that $(\pa^\gamma_Xm_s,\pa^\gamma_Xn_s)=(I-s\mathcal T)^{-1}(G_\gamma)$, where the data $G_\gamma$ are combinations of lower-order derivatives multiplied by the bounded functions $(i\zeta(k))^\nu$, $k\in\supp b$, and of the powers of $s$. By induction, all these derivatives are bounded on $\C$, uniformly in $X$ and $s\in\mathcal U'$, and holomorphic in $s$. Differentiating under the integral sign, $\pa^\gamma_XW_s(X)$ is holomorphic in $s$ and $|\pa^\gamma_XW_s(X)|\lesssim|s|$, which proves \eqref{decayW} for $|X|\le1$.\\

\noindent{\bf Step 2:} Approximate solution for $|X|\ge1$. Recall $\Gamma=\dbar_k(X\cdot\zeta(k))=-i(x+2\overline ky)$, so that $\dbar P_0=i\Gamma P_0$ and $\pa_kP_0^{-1}=-i\overline\Gamma P_0^{-1}$. If $|k_2|\ge\delta$ and $|k|\le\delta^{-1}$, then $|\Gamma|\gtrsim_\delta|X|$: when $|y|\ge\eps|X|$ we use $|\Im(x+2\overline ky)|=2|k_2||y|$, and when $|y|\le\eps|X|$ we use $|x+2k_1y|\ge(1-\eps-2\eps\delta^{-1})|X|$, with $\eps$ small. Hence on $\supp b$, every derivative in $k$ of $\Gamma^{-1}$ is $O(|X|^{-1})$, and every derivative in $X$ of $\Gamma^{-1}$ gains one more power of $|X|^{-1}$. We set
$$g:=\frac b{i\Gamma},\qquad \tilde g:=\frac{\overline b}{-i\overline\Gamma},\qquad q:=\cC(b\tilde g),\qquad \tilde q:=\overline{\cC}(\overline bg),$$
$$m_{\rm app}:=1+sP_0\,g+s^2q,\qquad n_{\rm app}:=1+sP_0^{-1}\tilde g+s^2\tilde q .$$
The functions $g,\tilde g$ are smooth, supported in $\supp b$, and all their derivatives in $k$ are $O(|X|^{-1})$; since $b\tilde g$ and $\overline bg$ are smooth with compact support, the functions $q$ and $\tilde q$ are smooth on $\C$, and all their derivatives in $k$ are $O(|X|^{-1})$ as well. A direct computation, in which the terms $sbP_0$ and $s^2b\tilde g$ cancel in the first line and $s\overline bP_0^{-1}$ and $s^2\overline bg$ cancel in the second line, gives
$$\dbar m_{\rm app}-sbP_0n_{\rm app}=sP_0\,h,\qquad \pa_kn_{\rm app}-s\overline bP_0^{-1}m_{\rm app}=sP_0^{-1}\tilde h,$$
with $h:=\dbar g-s^2b\tilde q$ and $\tilde h:=\pa_k\tilde g-s^2\overline bq$, smooth functions supported in $\supp b$ whose derivatives are $O(|X|^{-1})$. Oscillation gains one more power: Lemma \ref{lemmacauchy} (vi), applied with the phase $X\cdot\zeta(k)$, whose $\dbar$-derivative is $\Gamma$, and \eqref{potentialsets} give
\be
\label{gain}
\|\cC(P_0h)\|_\infty\le\Big\|\frac h{\Gamma}\Big\|_\infty+\Big\|\cI\Big|\dbar\frac h{\Gamma}\Big|\Big\|_\infty\lesssim|X|^{-2},
\ee
and similarly $\|\overline\cC(P_0^{-1}\tilde h)\|_\infty=\|\cC(P_0\overline{\tilde h})\|_\infty\lesssim|X|^{-2}$. Since $m_{\rm app}-1$ and $n_{\rm app}-1$ belong to $C_0$, Lemma \ref{lemmacauchy} (ii) gives $(I-s\mathcal T)(m_{\rm app},n_{\rm app})=(1,1)+s\big(\cC(P_0h),\overline\cC(P_0^{-1}\tilde h)\big)$, hence
\be
\label{errorapp}
\|m_s-m_{\rm app}\|_\infty+\|n_s-n_{\rm app}\|_\infty\lesssim|s||X|^{-2}.
\ee
Differentiating in $X$ the equation satisfied by $(m_s-m_{\rm app},n_s-n_{\rm app})$, and applying \eqref{gain} to the differentiated amplitudes, the same bound holds for all derivatives in $X$.\\

\noindent{\bf Step 3:} Decay of $W_s$. Inserting $n_{\rm app}$ in the definition of $W_s$ and using \eqref{errorapp},
$$W_s=s\int_\C bP_0\,dA+s^2J(X)+s^3\int_\C bP_0\tilde q\,dA+O(|s||X|^{-2}),$$
where $J(X):=\int_\C\frac{|b|^2}{-i\overline\Gamma}\,dA$.
By the change of variables \eqref{changevariables}, $\int bP_0\,dA$ is the Fourier transform of the function $\frac{b\circ\zeta^{-1}}{8|k_2|}\in C^\infty_c(\{\xi\neq0\})$, hence it decays faster than every power of $\jb X$. Integrating by parts once with $P_0=\frac{\dbar P_0}{i\Gamma}$, the third term is $O(|X|^{-2})$. Finally $|J(X)|\lesssim|X|^{-1}$, and each derivative in $X$ of $J$ gains a factor $|X|^{-1}$, because $\overline\Gamma$ is linear in $X$ and $|\overline\Gamma|\gtrsim|X|$ on $\supp b$. The same expansion holds for the derivatives in $X$, with the same bounds on the error terms. This proves \eqref{decayW} for $|X|\ge1$.
\end{proof}

\begin{proposition}[KP-II for compactly supported coefficients]
\label{propcompact}
Let $a\in C^\infty_c(\C\setminus\R)$ be real, and let $m$ and $u$ be given by \eqref{eqm} and \eqref{reconstruction}.
\begin{enumerate}
\item {\em (Regularity and decay).} $u\in C^\infty(\R\times\R^2)$ is real-valued, and $\sup_{|t|\le T}|\pa^\gamma_Xu(t,X)|\lesssim\jb X^{-2}$ for every multi-index $\gamma$ and $T>0$, with a constant depending on $\gamma$ and $T$.
\item {\em (Equation).} $u$ solves \eqref{kp}, $\|u(t)\|_{L^2}=\|u(0)\|_{L^2}$ for every $t$, and $u\in C^1(\R;L^2(\R^2))$ satisfies the Duhamel formula
\be
\label{duhamel}
u(t)=S(t)u(0)-6\int_0^tS(t-t')\big(u\pa_xu\big)(t')\,dt',\qquad t\in\R .
\ee
\item {\em (Heat equation).} For every $k\in\C$ and $(t,X)$, $m$ solves
\be
\label{heatm}
\pa_ym-\pa_x^2m-2k\pa_xm-u\,m=0 .
\ee
\item {\em (Analyticity in the amplitude).} For $s\in\R$, let $u_s$ be the function reconstructed from $sa$. Then $s\mapsto u_s(0)$ is real-analytic from $\R$ into $L^2(\R^2)$, and for every multi-index $\gamma$,
\be
\label{smallcritical}
\|\pa^\gamma_Xu_s(0)\|_{\dot H^{-\frac12,0}}\lesssim_\gamma|s|,\qquad |s|\le1 .
\ee
\end{enumerate}
\end{proposition}

\begin{proof}
If $a=0$, then $m=1$ and $u_s=0$ for every $s\in\R$, so all assertions are immediate. We assume $a\neq0$ below.

{\bf Step 1:} Smoothness and decay. Let $W(t,X):=\int aP\overline m\,dA$, so that $u=-\frac2\pi\pa_xW$, and let $w:=-\frac2\pi\pa_yW$. Since $P(k;t,X)=P_0(k;X)e^{it\omega(\zeta(k))}$, Lemma \ref{lemmaparametrix} (i) and (iii) with $b:=ae^{it\omega\circ\zeta}$ and $s=1$ give $|\pa^\gamma_XW(t,X)|\lesssim\jb X^{-1-\min(1,|\gamma|)}$, uniformly for $|t|\le T$: indeed $|b|=|a|$, so that the set $\mathcal U$ does not depend on $t$, and the derivatives of $b$ are bounded uniformly in $|t|\le T$. Differentiating \eqref{eqm} in $t$ as well, $m$ and $W$ are smooth in $(t,X)$, with bounded derivatives. Hence $u$ and $w$ are smooth and all their derivatives in $X$ are $O(\jb X^{-2})$, locally uniformly in $t$. The integrands of $W$ and of $\pa_yW$ are invariant under $g\mapsto g^*$, as in Steps 1 and 3 of the proof of Lemma \ref{lemmainverse}, so that $u$ and $w$ are real.\\

\noindent{\bf Step 2:} Expansion at infinity. Let $R_a:=\max\{|k'|:\ k'\in\supp a\}$. For $|k|>R_a$, expanding $\frac1{k-k'}=\sum_{j\ge1}k'^{j-1}k^{-j}$ in \eqref{eqm} gives the convergent expansion
\be
\label{expansionm}
m(k)=1+\sum_{j\ge1}\frac{m^{(j)}}{k^j},\qquad m^{(j)}:=\frac1\pi\int_\C k'^{j-1}a(k')P(k')\overline{m(k')}\,dA(k'),
\ee
which may be differentiated term by term in $(t,X)$. In particular $m^{(1)}=\frac W\pi$, so that
\be
\label{uwm1}
u=-2\pa_xm^{(1)},\qquad w=-2\pa_ym^{(1)} .
\ee

\noindent{\bf Step 3:} Gauge covariance. Let $E(k):=e^{kx+k^2y-4k^3t}$. Then $\overline{E}/E=e^{-2i\Im(kx+k^2y-4k^3t)}=P$, so that the eigenfunction $\psi:=Em$ of the heat operator satisfies $\dbar\psi=E\,aP\overline m=a\overline\psi$. Let $\mathcal L$ be a linear differential operator in $(t,x,y)$ whose coefficients are real and independent of $k$. Since $\mathcal L$ commutes with $\dbar_k$ and with complex conjugation, and $a$ does not depend on $(t,X)$, we get $\dbar(\mathcal L\psi)=a\overline{\mathcal L\psi}$. Hence $g:=E^{-1}\mathcal L(Em)$ satisfies
\be
\label{gauge}
\dbar g=E^{-1}a\,\overline{Eg}=aP\overline g .
\ee
We apply this to the heat operator $\pa_y-\pa_x^2-u$ and to the operator $\pa_t+4\pa_x^3+6u\pa_x+3\pa_xu+3w$, which form the Lax pair of KP-II, and whose conjugates by $E$ are
$$\begin{gathered}L_k:=\pa_y-\pa_x^2-2k\pa_x-u,\\ A_k:=\pa_t+4\pa_x^3+12k\pa_x^2+12k^2\pa_x+6u\pa_x+6ku+3\pa_xu+3w .\end{gathered}$$

\noindent{\bf Step 4:} The Lax equations. The function $k\mapsto L_km$ is continuous and, by \eqref{expansionm}, for $|k|>R_a$,
$$L_km=\big(-2\pa_xm^{(1)}-u\big)+\frac1k\big(\pa_ym^{(1)}-\pa_x^2m^{(1)}-2\pa_xm^{(2)}-um^{(1)}\big)+O(k^{-2}).$$
The constant term vanishes by \eqref{uwm1}, so that $L_km\in C_0(\C)$. By \eqref{gauge} and the uniqueness part of Lemma \ref{lemmavekua} (i) with $f=0$, $L_km=0$; this is \eqref{heatm}. In particular the coefficient of $k^{-1}$ vanishes:
\be
\label{coeffL}
2\pa_xm^{(2)}=\pa_ym^{(1)}-\pa_x^2m^{(1)}-um^{(1)} .
\ee
Similarly, $A_km=k\big(12\pa_xm^{(1)}+6u\big)+\big(12\pa_xm^{(2)}+12\pa_x^2m^{(1)}+6um^{(1)}+3\pa_xu+3w\big)+O(k^{-1})$. The coefficient of $k$ vanishes by \eqref{uwm1}, and by \eqref{coeffL} and \eqref{uwm1} the constant coefficient equals
$$\begin{aligned}&6\pa_ym^{(1)}-6\pa_x^2m^{(1)}-6um^{(1)}+12\pa_x^2m^{(1)}+6um^{(1)}+3\pa_xu+3w\\&\qquad=6\pa_ym^{(1)}+6\pa_x^2m^{(1)}+3\pa_xu+3w=0 .\end{aligned}$$
Hence $A_km\in C_0(\C)$, and $A_km=0$ as before.\\

\noindent{\bf Step 5:} The KP-II equation. A direct computation gives
$$\begin{aligned}{}[L_k,A_k]&=-6(\pa_xw-\pa_yu)\pa_x-3\pa_x(\pa_xw-\pa_yu)-6k(\pa_xw-\pa_yu)\\&\quad+\big(\pa_tu+\pa_x^3u+6u\pa_xu+3\pa_yw\big).\end{aligned}$$
By \eqref{uwm1}, $\pa_xw-\pa_yu=-2\pa_x\pa_ym^{(1)}+2\pa_y\pa_xm^{(1)}=0$. Step 4 gives $0=[L_k,A_k]m=(\pa_tu+\pa_x^3u+6u\pa_xu+3\pa_yw)m$, and by \eqref{boundeigen}, $|m(k)-1|\le\frac12$ for $|k|$ large. Hence $\pa_tu+\pa_x^3u+6u\pa_xu+3\pa_yw=0$. Finally, $u$ and $w$ belong to $L^2$, and \eqref{uwm1} gives $\xi\hat w=\eta\hat u$, so that $\hat w=\frac\eta\xi\hat u$ almost everywhere and $\pa_yw=\pa_x^{-1}\pa_y^2u$: $u$ solves \eqref{kp}.\\

\noindent{\bf Step 6:} Conservation of mass. All derivatives of $u$ and $w$ are $O(\jb X^{-2})$, so that the following integrations by parts are legitimate:
$$\begin{aligned}\frac{d}{dt}\int u^2\,dX&=-2\int u\big(\pa_x^3u+6u\pa_xu+3\pa_yw\big)dX\\&=6\int(\pa_yu)\,w\,dX=6\int(\pa_xw)\,w\,dX=0 .\end{aligned}$$
Moreover, $\pa_tu=-\pa_x^3u-6u\pa_xu-3\pa_yw$ is $O(\jb X^{-2})$ locally uniformly in $t$, so that, by dominated convergence, $u\in C^1(\R;L^2)$. In Fourier variables, $\pa_t\hat u=i\omega\hat u-6\widehat{u\pa_xu}$ almost everywhere, since $\hat w=\frac\eta\xi\hat u$; integrating $\pa_t(e^{-it\omega}\hat u)=-6e^{-it\omega}\widehat{u\pa_xu}$ in time gives \eqref{duhamel}. Steps 1--6 prove (i), (ii) and (iii).\\

\noindent{\bf Step 7:} Proof of (iv). Apply Lemma \ref{lemmaparametrix} with $b:=a$. By (i), for real $s$, $W_s$ is the function $W(0,\cdot)$ associated with the coefficient $sa$, so that $u_s(0)=-\frac2\pi\pa_xW_s$. By (ii) and (iii), for every $X$, $s\mapsto\pa_xW_s(X)$ is holomorphic on $\mathcal U$, and it is bounded by $C\jb X^{-2}\in L^2$ uniformly on compact subsets of $\mathcal U$. By the Cauchy integral formula and dominated convergence, $s\mapsto\pa_xW_s$ is holomorphic from $\mathcal U$ into $L^2(\R^2)$, hence real-analytic on $\R$. For \eqref{smallcritical}, let $G:=\pa^\gamma_XW_s$. By the Cauchy--Schwarz inequality in $\xi$, for every $y$,
$$\big\||D_x|^{\frac12}G(\cdot,y)\big\|_{L^2(\R)}^2\le\|G(\cdot,y)\|_{L^2(\R)}\|\pa_xG(\cdot,y)\|_{L^2(\R)},$$
where $|D_x|^{\frac12}$ denotes the Fourier multiplier with symbol $|\xi|^{\frac12}$ in the variable $x$. By \eqref{decayW}, $\|G(\cdot,y)\|_{L^2(\R)}\lesssim|s|\jb y^{-\frac12}$ and $\|\pa_xG(\cdot,y)\|_{L^2(\R)}\lesssim|s|\jb y^{-\frac32}$. Integrating in $y$,
$$\|\pa^\gamma_Xu_s(0)\|_{\dot H^{-\frac12,0}}^2=\frac4{\pi^2}\int_\R\big\||D_x|^{\frac12}G(\cdot,y)\big\|_{L^2(\R)}^2dy\lesssim s^2\int_\R\jb y^{-2}dy .$$
\end{proof}

The following elementary uniqueness statement will be used twice, to identify our solutions with those of other constructions.

\begin{lemma}[Uniqueness of solutions with bounded $\pa_xu$]
\label{lemmaunique}
Let $T>0$ and let $u_1,u_2\in C([0,T];L^2(\R^2))$ be real, with $u_j,\pa_xu_j\in L^\infty([0,T]\times\R^2)$ for $j=1,2$ and $u_1(0)=u_2(0)$, and assume that they satisfy the Duhamel formula \eqref{duhamel} on $[0,T]$. Then $u_1=u_2$ on $[0,T]$. The same holds on $[-T,0]$.
\end{lemma}

\begin{proof}
Let $h:=u_1-u_2$. Since the nonlinear terms $u_j\pa_xu_j$ belong to $L^\infty([0,T];L^2)$, the function
$$S(-t)h(t)=-6\int_0^tS(-t')\big(u_1\pa_xu_1-u_2\pa_xu_2\big)(t')\,dt'$$
is absolutely continuous with values in $L^2$, and for almost every $t$
$$\frac d{dt}\|h\|_{L^2}^2=-12\big(h,u_1\pa_xu_1-u_2\pa_xu_2\big)=-12\int_{\R^2}\big(h^2\pa_xu_1+u_2h\pa_xh\big)dX .$$
The function $u_2h^2$ and its derivative $(\pa_xu_2)h^2+2u_2h\pa_xh$ are integrable, because $u_2,h\in L^2\cap L^\infty$ and $\pa_xu_2,\pa_xh\in L^\infty$; hence, for almost every $y$, $x\mapsto u_2h^2$ belongs to $W^{1,1}(\R)$, and $\int u_2h\pa_xh\,dX=-\frac12\int(\pa_xu_2)h^2\,dX$. Therefore $\frac d{dt}\|h\|_{L^2}^2\le12\big(\|\pa_xu_1\|_{L^\infty}+\|\pa_xu_2\|_{L^\infty}\big)\|h\|_{L^2}^2$, and Gronwall's lemma gives $h=0$.
\end{proof}

We now prove the nonlinear Plancherel identity for compactly supported coefficients. The idea is to compare our solutions with the small data theory of Hadac, Herr and Koch, which provides an $L^2$ norm preserving scattering map, and to propagate the identity to large amplitudes by analyticity.

\begin{proposition}[Nonlinear Plancherel identity for compactly supported coefficients]
\label{propplancherel}
Let $a\in C^\infty_c(\C\setminus\R)$ be real, and let $u$ be given by \eqref{reconstruction}. Then
\be
\label{plancherelcompact}
\|u(0)\|_{L^2}^2=8\int_\C|k_2||a|^2\,dA=\|\cB a\|_{L^2}^2 .
\ee
\end{proposition}

\begin{proof}
{\bf Step 1:} Weak scattering. The single coefficient $a$ satisfies the assumptions of Proposition \ref{propcone} with $F:=(\|a\|_\infty+\|\nabla a\|_\infty)\1_{\supp a}$, which is bounded with compact support in $\C\setminus\R$. Since $\|u(t)\|_{L^2}$ is constant by Proposition \ref{propcompact} (ii), Lemma \ref{lemmaweak} gives $S(-t)u(t)\weakto\cB a$ as $t\to\pm\infty$. The same holds for the coefficient $sa$, $s\in\R$: $S(-t)u_s(t)\weakto s\cB a$.\\

\noindent{\bf Step 2:} Small amplitudes. We use the homogeneous small-data construction of \cite{HHK,HHKcorr}, transferred to \eqref{kp} by the change of variables of Subsection \ref{subsectionkp}, which preserves the spaces $\dot H^{-\frac12,0}$ and $L^2$ up to constant factors and intertwines the translations and the free groups. By the contraction argument in the proof of \cite[Theorem 1.1; see also Remark 4.1]{HHK}, there are $\eps_0,r>0$ such that every initial datum $U_0$ with $\|U_0\|_{\dot H^{-\frac12,0}}<\eps_0$ generates a global solution in the ball $D_r:=\{V\in\dot Z^{-\frac12}([0,\infty)):\ \|V\|_{\dot Z^{-\frac12}([0,\infty))}\le r\}$, unique within this ball, and the map $U_0\mapsto U$ is analytic into $\dot Z^{-\frac12}([0,\infty))\subset C([0,\infty);\dot H^{-\frac12,0})$. By \eqref{smallcritical}, there is $s_*>0$ such that $U_0:=u_s(0)$ is admissible for $|s|\le s_*$. Let $U$ be the corresponding solution. Complex conjugation and spatial translations preserve the norm of this restriction space. Since $u_s(0)$ is real, $\overline U\in D_r$ solves the same equation with the same datum, and $U$ is real by uniqueness within $D_r$.

{\it Regularity of $U$.} By uniqueness within $D_r$ and translation invariance of the equation, the flow map commutes with translations: the solution with datum $\tau_hU_0$ is $\tau_hU$. By \eqref{smallcritical}, $h\mapsto\tau_hU_0$ is a $C^\infty$ map from $\R^2$ into $\dot H^{-\frac12,0}$. Composing with the analytic flow map, $h\mapsto\tau_hU$ is $C^\infty$ into $\dot Z^{-\frac12}([0,\infty))$, and its derivatives at $h=0$ are the distributional derivatives of $U$. Hence $\pa^\gamma_XU\in\dot Z^{-\frac12}([0,\infty))\subset C([0,\infty);\dot H^{-\frac12,0})$ for every $\gamma$. Since $1\le|\xi|^{-1}+|\xi|$,
$$\|f\|_{L^2}^2\le\|f\|_{\dot H^{-\frac12,0}}^2+\|\pa_xf\|_{\dot H^{-\frac12,0}}^2,$$
so that $\pa^\gamma_XU\in C([0,\infty);L^2)$ for every $\gamma$, and in particular $U$ and $\pa_xU$ are bounded on $[0,T]\times\R^2$ for every $T>0$, since $\|f\|_{L^\infty}\le\frac1{2\pi}\|\hat f\|_{L^1}\lesssim\|f\|_{H^2}$ by the Cauchy--Schwarz inequality.

{\it Uniqueness.} We want to apply Lemma \ref{lemmaunique}, which requires the Duhamel formula \eqref{duhamel} with a Bochner integral. In \cite{HHK}, a solution is a solution of an operator equation whose Duhamel term is defined by continuous extension from regular functions \cite[Remark 1.3 and (53)]{HHK}, and we first check that, for the regular solution $U$, this term is the Bochner integral of \eqref{duhamel}. For $0<T<\infty$, the bilinear Duhamel operator of \cite[(36)]{HHK} reads, after the change of variables of Subsection \ref{subsectionkp} and for functions $u_1,u_2$ which are continuous in time with values in every Sobolev space,
$$I_T(u_1,u_2)(t)=\int_0^t\1_{[0,T)}(t')\,S(t-t')\,\pa_x(u_1u_2)(t')\,dt' .$$
By \cite[Corollaries 3.3 and 3.4]{HHK}, $I_T$ extends continuously to $\dot Z^{-\frac12}\times\dot Z^{-\frac12}$, where $\dot Z^{-\frac12}\subset C(\R;\dot H^{-\frac12,0})$ is the space of \cite[Definition 2.22]{HHK}, and $I_T$ converges to a continuous bilinear operator $I_\infty$ as $T\to\infty$. The solution $U$ is the restriction to $[0,\infty)$ of a function $\tilde U\in\dot Z^{-\frac12}$ \cite[Remark 2.23]{HHK}, and the operator equation reads $U(t)=S(t)U_0-3I_\infty(\tilde U,\tilde U)(t)$ for $t\ge0$. Fix $T>0$. For regular arguments, $I_{T'}=I_T$ on $[0,T]$ for every $T'\ge T$, hence, by continuity, $I_\infty(\tilde U,\tilde U)=I_T(\tilde U,\tilde U)$ on $[0,T]$.

To compute $I_T(\tilde U,\tilde U)$, we approximate $\tilde U$ by functions with compact Fourier support. For an integer $N\ge2$ let $Q_N$ be the Fourier multiplier with symbol $\1_{\{N^{-1}\le|\xi|\le N,\ |\eta|\le N\}}$. It commutes with $S(t)$ and with the dyadic frequency projections in $\xi$ which enter the norm of $\dot Z^{-\frac12}$, it is a contraction on $L^2$, and $Q_Nf\to f$ in $L^2$ for every $f\in L^2$, since $\{\xi=0\}$ is a null set. The atoms of the space $U^2_S$ of \cite[Definitions 2.1 and 2.15]{HHK} are step functions of $t$ with finitely many values, composed with $S(t)$; hence $Q_N$ maps an atom $a$ to an atom multiplied by a factor at most one, and $Q_Na\to a$ in $U^2_S$. Since every element of $U^2_S$ is an absolutely convergent series of atoms, dominated convergence gives $Q_Nv\to v$ in $U^2_S$ for every $v\in U^2_S$. Applied to each dyadic frequency piece of $\tilde U$, together with dominated convergence in the square sum over the dyadic frequencies, this gives $Q_N\tilde U\to\tilde U$ in $\dot Z^{-\frac12}$.

The Fourier transform of $Q_N\tilde U(t)$ is supported in a fixed compact subset of $\{\xi\neq0\}$, so that $Q_N\tilde U$ is continuous in time with values in every Sobolev space, and $I_T(Q_N\tilde U,Q_N\tilde U)$ is given by the integral above. By the continuity of $I_T$, $I_T(Q_N\tilde U,Q_N\tilde U)\to I_T(\tilde U,\tilde U)$ in $\dot Z^{-\frac12}$, hence in the sense of distributions on $[0,T]\times\R^2$. On the other hand, $U\in C([0,T];H^m)$ for every $m$, by the regularity of $U$. The operators $Q_N$ are bounded on $H^m$ uniformly in $N$ and converge strongly to the identity, hence uniformly on the compact set $U([0,T])\subset H^m$, so that $Q_NU\to U$ in $C([0,T];H^m)$. Since $\|f\|_{L^\infty}\lesssim\|f\|_{H^2}$, we have $\pa_x((Q_NU)^2)\to\pa_x(U^2)$ in $C([0,T];L^2)$, and since $S(t)$ is unitary, $I_T(Q_NU,Q_NU)(t)$ converges in $C([0,T];L^2)$ to the Bochner integral $\int_0^tS(t-t')\pa_x(U^2)(t')\,dt'$. Since $Q_N\tilde U=Q_NU$ on $[0,T]$, the two limits coincide, and since $3\pa_x(U^2)=6U\pa_xU$, $U$ satisfies \eqref{duhamel} on $[0,T]$. The smooth solution $u_s$ satisfies \eqref{duhamel} by Proposition \ref{propcompact} (ii). Both are real and bounded with bounded $x$-derivatives on $[0,T]\times\R^2$ for every $T>0$, by the regularity of $U$ and by Proposition \ref{propcompact} (i), so that Lemma \ref{lemmaunique} gives $U=u_s$ on $[0,\infty)$.

{\it Strong $L^2$ scattering.} If $V\in\dot Z^{-\frac12}([0,\infty))$, choose a full-line extension in $\dot Z^{-\frac12}$. Each dyadic interaction profile belongs to $U^2$ and has an $L^2$ limit as $t\to+\infty$ by \cite[Proposition 2.2(ii), (iv)]{HHK}; dominated convergence in the defining dyadic square sum gives a strong limit of $S(-t)V(t)$ in $\dot H^{-\frac12,0}$. Applied to $U$ and $\pa_xU$, this yields limits $U_+$ and $V_+$, with $V_+=\pa_xU_+$ because convergence in the critical space implies convergence in $\mathcal S'$. The Fourier Cauchy--Schwarz inequality shows that $U_+\in L^2$ and, with $w(t):=U(t)-S(t)U_+$,
$$\|w(t)\|_{L^2}^2\le\|w(t)\|_{\dot H^{-\frac12,0}}\|\pa_xw(t)\|_{\dot H^{-\frac12,0}}\longrightarrow0\qquad(t\to+\infty).$$
Since $U=u_s$ and $u_s$ conserves its $L^2$ norm by Proposition \ref{propcompact} (ii), we obtain $\|U_+\|_{L^2}=\|u_s(0)\|_{L^2}$.

{\it Identification of the state.} Let $\hat\varphi\in C^\infty_c(\{\xi\neq0\})$; then $\varphi$ belongs to the space $\dot H^{\frac12,0}$, defined by the norm $\||\xi|^{\frac12}\hat\varphi\|_{L^2}$, which is the dual space of $\dot H^{-\frac12,0}$ for the pairing $(\cdot,\cdot)$. By Step 1 and by the convergence in $\dot H^{-\frac12,0}$, $(S(-t)u_s(t),\varphi)$ converges both to $(s\cB a,\varphi)$ and to $(U_+,\varphi)$. Hence $\widehat{U_+}=s\widehat{\cB a}$ almost everywhere, and
$$\|u_s(0)\|_{L^2}^2=\|U_+\|_{L^2}^2=s^2\|\cB a\|_{L^2}^2\qquad\mbox{for}\ |s|\le s_* .$$

\noindent{\bf Step 3:} Analytic continuation. By Proposition \ref{propcompact} (iv), $s\mapsto\|u_s(0)\|_{L^2}^2$ is real-analytic on $\R$. It coincides with $s^2\|\cB a\|_{L^2}^2$ on $[-s_*,s_*]$, hence on $\R$. For $s=1$ and with \eqref{bornisometry}, this is \eqref{plancherelcompact}.
\end{proof}
\subsection{Rapidly decaying coefficients}

The scattering coefficients of smooth physical data are not compactly supported: they are only rapidly decaying in the sense of \eqref{hypa}, and they are in general discontinuous across the real axis. We extend Propositions \ref{propcompact} and \ref{propplancherel} to this class by approximation. The nonlinear Plancherel identity passes to the limit as an inequality by Fatou's lemma; the reverse inequality comes from the weak convergence $S(-t)u(t)\weakto\cB a$ provided by Proposition \ref{propcone}.

\begin{proposition}[KP-II for rapidly decaying coefficients]
\label{proprapid}
Let $a$ satisfy \eqref{hypa} and let $u$ be given by \eqref{reconstruction}. Then:
\begin{enumerate}
\item $u\in C(\R;L^2(\R^2))$, $u$ and $\pa_xu$ are bounded on $\R\times\R^2$, $u$ satisfies the Duhamel formula \eqref{duhamel}, and $\|u(t)\|_{L^2}=\|u(0)\|_{L^2}$ for every $t\in\R$;
\item $\|u(0)\|_{L^2}^2=8\int_\C|k_2||a|^2dA=\|\cB a\|_{L^2}^2$.
\end{enumerate}
\end{proposition}

\begin{proof}
{\bf Step 1:} Compactly supported approximation. For $R\ge1$ let $\vartheta_R(k):=\vartheta(k/R)\big(1-\vartheta(4Rk_2)\big)$, where $\vartheta$ is the cutoff fixed in the Notations. Then $\vartheta_R\in C^\infty_c(\C\setminus\R)$ is real-valued, $\vartheta_R(\bar k)=\vartheta_R(k)$, $0\le\vartheta_R\le1$, $\vartheta_R=1$ on $\{|k|\le R,\ |k_2|\ge R^{-1}\}$, and $\supp\vartheta_R\subset\{|k|\le2R,\ |k_2|\ge(4R)^{-1}\}$. Let
$$a_R:=\vartheta_R\,(\varrho_{\eps_R}\ast a),$$
where $\eps_R\le(8R)^{-1}$ is so small that $|\varrho_{\eps_R}\ast a-a|\le R^{-1}$ on $\supp\vartheta_R$; this is possible since $a$ is uniformly continuous on the compact set $\{|k|\le3R,\ |k_2|\ge(8R)^{-1}\}$. Since $\varrho$ is radial and $a(\bar k)=\overline{a(k)}$, the coefficient $a_R\in C^\infty_c(\C\setminus\R)$ is real. If $|k'-k|\le\eps_R$ and $|k|\le2R$, then $|\zeta(k')-\zeta(k)|\lesssim\jb k\eps_R\lesssim1$, so that $\jb{\zeta(k')}\sim\jb{\zeta(k)}$; hence \eqref{hypa} gives
\be
\label{uniformaR}
|a_R(k)|\lesssim\jb{\zeta(k)}^{-6}\qquad\mbox{uniformly in}\ R,
\ee
and $a_R\to a$ uniformly on the compact subsets of $\C\setminus\R$. Let $j\in\{0,1,2\}$. By \eqref{hypa} and \eqref{uniformaR}, $|k_2|^j|a_R-a|\le C|k_2|^j\jb\zeta^{-6}$, and this majorant is bounded since $|k_2|\le\jb\zeta$. Its Riesz potential belongs to $C_0(\C)$: it is continuous by Lemma \ref{lemmacauchy} (i), and it tends to zero at infinity by Lemma \ref{lemmaweights} (ii). Hence Lemma \ref{lemmacauchy} (v) gives
$$\sum_{j=0}^2\big\|\cI\big(|k_2|^j|a_R-a|\big)\big\|_\infty\longrightarrow0\qquad(R\to\infty).$$
Let $m_R$ and $u_R$ be the eigenfunction and the reconstructed function of $a_R$. By \eqref{comparison}, $m_R\to m$ in $C_b(\C)$, uniformly in $(t,X)$. Subtracting the equations \eqref{eqmx} of $a_R$ and of $a$,
$$\begin{aligned}\pa_xm_R-\pa_xm&=\cC\big(a_RP\overline{(\pa_xm_R-\pa_xm)}\big)+\cC\big((a_R-a)P\overline{\pa_xm}\big)\\&\quad+\cC\big(-2ik_2(a_RP\overline{m_R}-aP\overline m)\big),\end{aligned}$$
and the last two terms tend to zero in $C_b$, uniformly in $(t,X)$, by the convergences above, with $j=0,1$. Lemma \ref{lemmavekua} (iii) gives $\pa_xm_R\to\pa_xm$ uniformly in $(k;t,X)$. The function $\pa_x^2m$ solves the equation obtained by differentiating \eqref{eqmx} in $x$, whose source term $\cC\big(-4ik_2aP\overline{\pa_xm}-4k_2^2aP\overline m\big)$ is computed in the proof of Lemma \ref{lemmainverse}. The same comparison, now with the three convergences above, gives $\pa_x^2m_R\to\pa_x^2m$ uniformly in $(k;t,X)$. By \eqref{uniformaR} and the proof of Lemma \ref{lemmainverse}, the bounds $|\pa_xm_R|+|\pa_x^2m_R|\lesssim\jb k^{-1}$ hold uniformly in $R$, so that the integrands of $u_R$ and $\pa_xu_R$ have the integrable majorant $C(|k_2|+k_2^2+\jb k^{-1})\jb\zeta^{-6}$. By dominated convergence, $u_R\to u$ and $\pa_xu_R\to\pa_xu$ pointwise on $\R\times\R^2$, with $\sup_R(\|u_R\|_\infty+\|\pa_xu_R\|_\infty)<\infty$. Propositions \ref{propcompact} and \ref{propplancherel}, Fatou's lemma and the dominated convergence theorem (with \eqref{uniformaR}) give, for every $t$,
\be
\label{fatou}
\begin{aligned}\|u(t)\|_{L^2}^2\le\liminf_{R\to\infty}\|u_R(t)\|_{L^2}^2&=\lim_{R\to\infty}8\int_\C|k_2||a_R|^2\,dA\\&=8\int_\C|k_2||a|^2\,dA=\|\cB a\|_{L^2}^2 .\end{aligned}
\ee

\noindent{\bf Step 2:} The Duhamel formula. By Proposition \ref{propcompact}, each $u_R$ is a smooth solution of $\pa_x(\pa_tu_R+\pa_x^3u_R+3\pa_x(u_R^2))+3\pa_y^2u_R=0$. By Step 1 we can pass to the limit in the sense of distributions on $\R\times\R^2$:
\be
\label{kpdistrib}
\pa_x\big(\pa_tu+\pa_x^3u+3\pa_x(u^2)\big)+3\pa_y^2u=0 .
\ee
Since $u$ and $\pa_xu$ are bounded by Lemma \ref{lemmainverse} (iii), and $\|u(t)\|_{L^2}\le\|\cB a\|_{L^2}$ by \eqref{fatou}, the function $6u\pa_xu$ belongs to $L^\infty(\R;L^2)$, and
$$\tilde u(t):=S(t)u(0)-6\int_0^tS(t-t')\big(u\pa_xu\big)(t')\,dt'$$
belongs to $C(\R;L^2)$ and satisfies $\pa_x(\pa_t\tilde u+\pa_x^3\tilde u+6u\pa_xu)+3\pa_y^2\tilde u=0$ in the sense of distributions, while $u$ satisfies the same equation by \eqref{kpdistrib}, since $6u\pa_xu=3\pa_x(u^2)$. The difference $z:=u-\tilde u$ is bounded in $L^2$ and satisfies $\pa_x(\pa_tz+\pa_x^3z)+3\pa_y^2z=0$. Moreover $z$ is weakly continuous in $L^2$, with $z(0)=0$: $\tilde u$ is continuous, and $u(t')\to u(t)$ pointwise and boundedly in $L^2$ as $t'\to t$, by Lemma \ref{lemmainverse} (iii) and \eqref{fatou}. In Fourier variables, the equation reads $\pa_t\hat z=i\omega\hat z$ in $\mathcal D'(\R\times\{\xi\neq0\})$, so that $\pa_t(e^{-it\omega}\hat z)=0$ there, and for every test function $\varphi$ supported in $\{\xi\neq0\}$, the continuous function $t\mapsto\int e^{-it\omega}\hat z(t)\overline{\varphi}\,d\xi\,d\eta$ is constant, hence equal to its value $0$ at $t=0$. Since $\{\xi=0\}$ is a null set, $z=0$: $u=\tilde u\in C(\R;L^2)$ satisfies \eqref{duhamel}. Then $S(-t)u(t)=u(0)-6\int_0^tS(-t')(u\pa_xu)(t')\,dt'$ is absolutely continuous in $L^2$, and for almost every $t$
$$\frac d{dt}\|u(t)\|_{L^2}^2=-12\big(u,u\pa_xu\big)=-12\int_{\R^2}u^2\pa_xu\,dX=-4\int_{\R^2}\pa_x(u^3)\,dX=0,$$
since $u^3$ and $\pa_x(u^3)=3u^2\pa_xu$ are integrable. This proves (i). Note that we did not need $u(t)$ to belong to the domain of the singular operator $\pa_x^{-1}\pa_y^2$.\\

\noindent{\bf Step 3:} Proof of (ii). By \eqref{hypa} with $q=6$, the coefficient $a$ satisfies \eqref{hypF1} with $F:=C\jb\zeta^{-6}$, and $F$ satisfies the conditions \eqref{weightsF} by Lemma \ref{lemmaweights} (ii). Proposition \ref{propcone} and Lemma \ref{lemmaweak}, together with the conservation of mass, give $S(-t)u(t)\weakto\cB a$ as $t\to\pm\infty$. By weak lower semicontinuity, $\|\cB a\|_{L^2}\le\liminf\|S(-t)u(t)\|_{L^2}=\|u(0)\|_{L^2}$, and \eqref{fatou} at $t=0$ gives the reverse inequality. Hence $\|u(0)\|_{L^2}=\|\cB a\|_{L^2}$, which, together with \eqref{bornisometry}, proves (ii).
\end{proof}

\section{The direct problem}
\label{sectiondirect}

In this section the initial datum $u_0$ is given, and we construct its scattering coefficient. The eigenfunctions of the heat operator are the solutions $\psi=e^{kx+k^2y}\chi$ of $\pa_y\psi=\pa_x^2\psi+u_0\psi$ with $\chi\to1$ at infinity; as in Section \ref{sectioninverse}, we also call the normalized function $\chi$ the eigenfunction. It solves
\be
\label{eqchi}
(\pa_y-\pa_x^2-2k\pa_x)\chi=u_0\chi,\qquad \chi=1+G_k\ast(u_0\chi),
\ee
where $G_k$ is the Green function of $\pa_y-\pa_x^2-2k\pa_x$. We first study $G_k$ (Lemma \ref{lemmagreen}) and the operator $H\mapsto G_k\ast(u_0H)$ on $L^\infty(\R^2)$ (Lemmas \ref{lemmagreenop} and \ref{lemmaembedding}). We then treat smooth exponentially decaying data (Lemma \ref{lemmairregular}, Proposition \ref{propdirect} and Corollary \ref{corsmooth}): the spectral points where \eqref{eqchi} is not uniquely solvable form a closed set with empty interior, through which the eigenfunctions extend by a theorem of Grinevich \cite{Grin}. The scattering coefficients of rough data are constructed in Sections \ref{sectionborn}--\ref{sectionproof}, by approximation.

\subsection{The Green function}

The symbol of $\pa_y-\pa_x^2-2k\pa_x$ is $\xi^2+i\eta-2ik\xi$, so that, formally,
\be
\label{defgreen}
G_k(x,y)=\frac1{(2\pi)^2}\int_{\R^2}\frac{e^{i(x\xi+y\eta)}}{\xi^2+i\eta-2ik\xi}\,d\xi\,d\eta .
\ee
The symbol is a holomorphic function of $k$, and for $\xi\neq0$ it vanishes exactly when $k=\frac\eta{2\xi}-i\frac\xi2$, that is, when $(\xi,\eta)=\zeta(k)$. Since $\dbar_k\frac1{k-k_0}=\pi\delta_{k_0}$, the $\dbar$-derivative of $G_k$ in $k$ is therefore concentrated on the single frequency $\zeta(k)$, and the change of variables \eqref{changevariables} shows that, formally, $\dbar_kG_k(X)=-\frac{i\sgn(k_2)}{2\pi}e^{iX\cdot\zeta(k)}$; this identity is proved in Lemma \ref{lemmagreen} (v) below. This is the reason why the frequency $\zeta(k)$ appears in the direct and inverse problems. The following lemma gives an explicit formula for $G_k$, obtained by computing the integral \eqref{defgreen} first in $\eta$, by residues, and then in $\xi$, and collects the properties of $G_k$ that we use. We write $\tilde x:=x+2k_1y$; note that $G_k$ depends on $k_1$ only through $\tilde x$.

\begin{lemma}[The Green function]
\label{lemmagreen}
Let $k=k_1+ik_2\in\C$ and $X=(x,y)$ with $y\neq0$.
\begin{enumerate}
\item {\em (Formula).} With $\tilde x=x+2k_1y$, $E_y:=\{|s|>|k_2|\}$ if $y>0$ and $E_y:=\{|s|<|k_2|\}$ if $y<0$, let
\be
\label{formulagreen}
G_k(x,y):=\frac{\sgn(y)}{2\pi}\,e^{-ik_2\tilde x}\int_{E_y}e^{is\tilde x}\,w(s)\,ds,\qquad w(s):=e^{-(s^2-k_2^2)y}.
\ee
Then $G_k$ is continuous on $\R^2\setminus\{0\}$, and it is a fundamental solution: $(\pa_y-\pa_x^2-2k\pa_x)G_k=\delta_0$ in $\mathcal D'(\R^2)$. For real $k=k_1$, $G_{k_1}(x,y)=\1_{\{y>0\}}(4\pi y)^{-\frac12}e^{-(x+2k_1y)^2/(4y)}$ is a causal heat kernel.
\item {\em (Size).}
\be
\label{sizegreen}
|G_k(x,y)|\lesssim\min\big(|\tilde x|^{-1},\,|y|^{-\frac12},\,(|k_2||y|)^{-1}\big)\lesssim\frac1{|x+2k_1y|+|y|^{\frac12}} .
\ee
\item {\em (Spectral regularity).} $k\mapsto G_k(X)$ is continuous on $\C$ and smooth on $\C\setminus\R$, with $|\nabla_kG_k(X)|\lesssim1$ for $k\notin\R$. In particular $|G_k(X)-G_{k'}(X)|\lesssim|k-k'|$ for all $k,k'\in\C$.
\item {\em (Symmetries).} $\overline{G_k}=G_{\bar k}$ and $G_k(X)e^{-iX\cdot\zeta(k)}=G_{\bar k}(X)$.
\item {\em ($\dbar$-derivative).} For $k\notin\R$, $\dbar_kG_k(X)=-\frac{i\sgn(k_2)}{2\pi}e^{iX\cdot\zeta(k)}$.
\end{enumerate}
\end{lemma}

\begin{proof}
{\bf Step 1:} Proof of (i). The formula comes from \eqref{defgreen} as follows: we replace $\eta$ by $\eta+2k_1\xi$, which turns $x\xi+y\eta$ into $\tilde x\xi+y\eta$ and the symbol into $(\xi+k_2)^2-k_2^2+i\eta$; for a real number $c\neq0$, $\frac1{2\pi}\int_\R\frac{e^{iy\eta}}{c+i\eta}d\eta$ equals $\1_{\{y>0\}}e^{-cy}$ if $c>0$ and $-\1_{\{y<0\}}e^{-cy}$ if $c<0$; with $c=s^2-k_2^2$ and $\xi=s-k_2$ this gives \eqref{formulagreen}. To prove that $G_k$ is a fundamental solution, note that each function $f_s(x,y):=e^{i(s-k_2)\tilde x-(s^2-k_2^2)y}$ satisfies $(\pa_y-\pa_x^2-2k\pa_x)f_s=0$, since
$$\begin{aligned}&2ik_1(s-k_2)-(s^2-k_2^2)+(s-k_2)^2-2ik(s-k_2)\\&\qquad=(s-k_2)\big(2ik_1-(s+k_2)+(s-k_2)-2ik_1+2k_2\big)=0,\end{aligned}$$
and that the integrals in \eqref{formulagreen} may be differentiated under the integral sign for $y\neq0$. Hence $(\pa_y-\pa_x^2-2k\pa_x)G_k=0$ on $\{y\neq0\}$. For $k_2=0$ the Gaussian integral gives the heat kernel. As $y\to0^\pm$ with $x\neq0$, both expressions in \eqref{formulagreen} tend to $-\frac{e^{-ik_2x}}{2\pi}\frac{2\sin(|k_2|x)}x$: for $y\to0^+$ we write $\int_{\{|s|>|k_2|\}}=\int_\R-\int_{\{|s|<|k_2|\}}$, and the Gaussian integral over $\R$ tends to zero. Hence $G_k$ is continuous on $\R^2\setminus\{0\}$. In $\mathcal D'(\R_x)$, the same decomposition shows that $G_k(\cdot,y)$ tends to $\frac{e^{-ik_2x}}{2\pi}\big(2\pi\delta_0-\int_{\{|s|<|k_2|\}}e^{isx}ds\big)$ as $y\to0^+$, and to $-\frac{e^{-ik_2x}}{2\pi}\int_{\{|s|<|k_2|\}}e^{isx}ds$ as $y\to0^-$. Since the operator is of first order in $y$, an integration by parts in $y$ on each half-plane shows that $(\pa_y-\pa_x^2-2k\pa_x)G_k$ equals the jump $G_k(\cdot,0^+)-G_k(\cdot,0^-)=\delta_0(x)$ times $\delta_0(y)$, that is, $\delta_0$.\\

\noindent{\bf Step 2:} Proof of (ii). Let $\mu:=|k_2|$. For $y>0$, $w$ decreases from $1$ to $0$ on each half-line of $E_y$; for $y<0$, $w$ is monotone on $[-\mu,0]$ and on $[0,\mu]$, with values in $[0,1]$. One integration by parts in $s$ gives $\big|\int_{E_y}e^{is\tilde x}w\,ds\big|\le\frac4{|\tilde x|}$. For $y>0$, writing $s=\pm(\mu+r)$ and using $s^2-\mu^2\ge r^2+2\mu r$, the integral is bounded by $2\int_0^\infty e^{-r^2y}dr\lesssim y^{-\frac12}$ and by $2\int_0^\infty e^{-2\mu ry}dr=(\mu y)^{-1}$. For $y<0$, $\mu^2-s^2\ge\mu(\mu-|s|)$ gives the bound $2\int_0^\mu e^{-\mu(\mu-s)|y|}ds\le\min\big(2\mu,\,2(\mu|y|)^{-1}\big)$, which is at most $2|y|^{-\frac12}$ and at most $2(\mu|y|)^{-1}$. This proves \eqref{sizegreen}.\\

\noindent{\bf Step 3:} Proof of (iii). Continuity at $k_2=0$ follows from \eqref{formulagreen} by dominated convergence. Since $G_k(x,y)$ is a function of $(\tilde x,y,k_2)$, we have $\pa_{k_1}G_k=2y\,\pa_{\tilde x}G_k$, and by \eqref{formulagreen}
\be
\label{dk1green}
y\,\pa_{\tilde x}G_k=-ik_2y\,G_k+\frac{\sgn(y)}{2\pi}e^{-ik_2\tilde x}\,iy\int_{E_y}s\,e^{is\tilde x}w(s)\,ds .
\ee
The first term is bounded by \eqref{sizegreen}, and $|y|\int_{E_y}|s|w(s)\,ds\le1$ in both cases, by the explicit primitive of $sw(s)$. For $k_2>0$, differentiating the prefactor, the weight and the endpoints $\pm k_2$ of $E_y$ in \eqref{formulagreen} gives, for both signs of $y$,
\be
\label{dk2green}
\pa_{k_2}G_k=(-i\tilde x+2k_2y)\,G_k-\frac1{2\pi}\big(1+e^{-2ik_2\tilde x}\big),
\ee
which is bounded by \eqref{sizegreen}; the case $k_2<0$ follows from (iv) below. The Lipschitz bound follows, since $G_k(X)$ is continuous across the real axis.\\

\noindent{\bf Step 4:} Proof of (iv) and (v). The integral in \eqref{formulagreen} is real, since $E_y$ and $w$ are even in $s$, and it depends on $k_2$ only through $|k_2|$. Hence $\overline{G_k}=G_{\bar k}$ and $G_ke^{2ik_2\tilde x}=G_{\bar k}$, with $2k_2\tilde x=-X\cdot\zeta(k)$. For (v), let $k_2>0$. By \eqref{dk1green} and \eqref{dk2green},
$$\begin{aligned}\dbar_kG_k&=y\,\pa_{\tilde x}G_k+\frac i2\pa_{k_2}G_k\\&=\frac{\tilde x}2G_k+\frac{\sgn(y)}{2\pi}e^{-ik_2\tilde x}\,iy\int_{E_y}s\,e^{is\tilde x}w\,ds-\frac i{4\pi}\big(1+e^{-2ik_2\tilde x}\big).\end{aligned}$$
Since $\pa_sw=-2syw$, we have $iys\,w=-\frac i2\pa_sw$, and an integration by parts, with $w(\pm k_2)=1$, gives
$$iy\int_{E_y}s\,e^{is\tilde x}w\,ds=-\frac i2\,\sgn(y)\big(e^{-ik_2\tilde x}-e^{ik_2\tilde x}\big)-\frac{\tilde x}2\int_{E_y}e^{is\tilde x}w\,ds .$$
The last term cancels with $\frac{\tilde x}2G_k$, and we obtain $\dbar_kG_k=-\frac i{4\pi}\big(e^{-2ik_2\tilde x}-1\big)-\frac i{4\pi}\big(1+e^{-2ik_2\tilde x}\big)=-\frac i{2\pi}e^{-2ik_2\tilde x}=-\frac i{2\pi}e^{iX\cdot\zeta(k)}$. The case $k_2<0$ follows from (iv), since $\dbar_k\big(\overline{G_{\bar k}}\big)=\overline{(\dbar G)(\bar k)}$ and $\zeta(\bar k)=-\zeta(k)$.
\end{proof}

\subsection{A norm adapted to the Green function}

By \eqref{sizegreen}, the bound on $|G_k|$ depends on $k$ only through the shear $k_1$, and it is the inverse of the parabolic distance $|x+2k_1y|+|y|^{\frac12}$. This suggests the following analogue of the Kato class of potentials \cite{AizenmanSimon}:
\be
\label{defK}
\|f\|_\cK:=\sup_{k_1\in\R,\ X\in\R^2}\int_{\R^2}\frac{|f(X')|}{|x-x'+2k_1(y-y')|+|y-y'|^{\frac12}}\,dX'.
\ee
We denote by $\cK$ the space of measurable functions with $\|f\|_\cK<\infty$, and by $\cK_0$ the closure of $C^\infty_c(\R^2)$ in $\cK$. The norm $\|\cdot\|_\cK$ is invariant under translations, decreases under restriction to subsets, and satisfies, for $c\in L^\infty(\R^2)$ and $\kappa\in L^1(\R^2)$,
\be
\label{propK}
\|cf\|_\cK\le\|c\|_{L^\infty}\|f\|_\cK,\qquad \|\kappa\ast f\|_\cK\le\|\kappa\|_{L^1}\|f\|_\cK .
\ee
For a potential $u_0$ we write
\be
\label{defTk}
T_k:=I-G_k\ast(u_0\,\cdot\,),\qquad\mbox{acting on}\ L^\infty(\R^2),
\ee
so that \eqref{eqchi} reads $T_k\chi=1$. We say that $k$ is a regular point of $u_0$ if $T_k$ is invertible on $L^\infty(\R^2)$, and an irregular point otherwise. The operator $G_k\ast(u_0\,\cdot\,)$ is small at high energy, where $G_k$ oscillates. On the real axis, $G_{k_1}$ is a causal heat kernel, vanishing for $y<0$, so that the equation $T_{k_1}H=b$ is of Volterra type in $y$ and can be solved strip by strip; near the real axis, $T_k$ is a small perturbation of $T_{k_1}$.

\begin{lemma}[The direct operator]
\label{lemmagreenop}
\begin{enumerate}
\item For $f\in\cK$, $H\in L^\infty$ and $k,k'\in\C$,
$$\begin{gathered}\|G_k\ast(fH)\|_{L^\infty}\lesssim\|f\|_\cK\|H\|_{L^\infty},\\ \|(G_k-G_{k'})\ast(fH)\|_{L^\infty}\lesssim|k-k'|\,\|f\|_{L^1}\|H\|_{L^\infty}.\end{gathered}$$
\item {\em (High energy).} If $f\in\cK_0$, then $\sup_{|k|\ge R}\|G_k\ast(f\,\cdot\,)\|_{L^\infty\to L^\infty}\to0$ as $R\to\infty$, uniformly for $f$ in a convergent sequence of $\cK_0$.
\item {\em (Uniform invertibility).} Let $u_0\in\cK_0\cap L^1$. There is $R_0\ge1$ such that every $k$ in
$$\Omega:=\{|k|>R_0\}\cup\{|k_2|<R_0^{-1}\}$$
is a regular point, and $\sup_{k\in\Omega}\|T_k^{-1}\|_{L^\infty\to L^\infty}<\infty$. If $u_{0,n}\to u_0$ in $\cK\cap L^1$, the same holds for $u_{0,n}$, with the same $R_0$ and bound, for $n$ large.
\item {\em (Pointwise continuity).} Let $u_0\in L^1\cap L^\infty$ and let $k$ be a regular point. If $H_j\to H$ pointwise with $\sup_j\|H_j\|_{L^\infty}<\infty$, then $T_k^{-1}H_j\to T_k^{-1}H$ pointwise.
\end{enumerate}
\end{lemma}

\begin{proof}
{\bf Step 1:} Proof of (i) and (ii). Part (i) follows from \eqref{sizegreen} and Lemma \ref{lemmagreen} (iii). For (ii), let first $f\in C^\infty_c$ with support in $[-L,L]^2$. We bound the operator norm by $\sup_X\int|G_k(X-X')||f(X')|dX'$. If $|k_1|\ge2$, we split according to $|y-y'|<|k_1|^{-2}$ or not. The first part is at most $\|f\|_{L^\infty}\int_{|x'|\le L}\int_{|y-y'|<|k_1|^{-2}}|y-y'|^{-\frac12}dy'dx'\lesssim L\|f\|_{L^\infty}|k_1|^{-1}$. In the second part we use the bound $(|x-x'+2k_1(y-y')|+|k_1|^{-1})^{-1}$; for fixed $x'$, the substitution $y'\mapsto z:=x-x'+2k_1(y-y')$ maps $[-L,L]$ onto an interval of length $4|k_1|L$, and $\int_I\frac{dz}{|z|+\eps}\le2\log\big(1+\frac{|I|}{2\eps}\big)$ for an interval $I$, so that this part is at most
$$\|f\|_{L^\infty}\int_{|x'|\le L}\frac1{2|k_1|}\,2\log\big(1+2k_1^2L\big)\,dx'\lesssim L\|f\|_{L^\infty}\frac{\log(2+k_1^2L)}{|k_1|}.$$
If $|k_2|\ge2$, we use the bound $\min(|y|^{-\frac12},(|k_2||y|)^{-1})$ of \eqref{sizegreen} and split at $|y-y'|=|k_2|^{-2}$ and $|y-y'|=1$: this gives the bound $C\big(\|f\|_{L^1}+\sup_{y}\|f(\cdot,y)\|_{L^1(\R)}\big)\frac{1+\log|k_2|}{|k_2|}$. Both bounds tend to zero as $|k|\to\infty$. For $f\in\cK_0$ and $\eps>0$, choose $f_\eps\in C^\infty_c$ with $\|f-f_\eps\|_\cK\le\eps$; by (i), $\|G_k\ast(f\,\cdot)\|\le C\eps+\|G_k\ast(f_\eps\,\cdot)\|$. The same $f_\eps$ works, with $2\eps$, for all the terms of a sequence converging to $f$ in $\cK$, except finitely many.\\

\noindent{\bf Step 2:} The real axis. Let $k=k_1\in\R$, and let $C_1\ge1$ be the constant in (i). Choose $f\in C^\infty_c$, supported in $[-L,L]^2$, with $\|u_0-f\|_\cK\le(16C_1)^{-1}$. For an interval $I$, \eqref{defK} gives $\|f\1_{\{y\in I\}}\|_\cK\le\|f\|_{L^\infty}\cdot2L\cdot2\sqrt{2|I|}$. Hence we can partition $\R$ into ordered intervals $I_1<I_2<\dots<I_J$, with $I_1$ and $I_J$ unbounded and $J$ depending only on $u_0$, such that $C_1\|f\1_{\{y\in I_j\}}\|_\cK\le\frac18$ for every $j$. Let $\tilde u_0\in\cK$ with $\|\tilde u_0-u_0\|_\cK\le(16C_1)^{-1}$; for instance $\tilde u_0=u_0$, or $\tilde u_0=u_{0,n}$ with $n$ large. The functions $\tilde u_j:=\tilde u_0\1_{\{y\in I_j\}}$ satisfy $C_1\|\tilde u_j\|_\cK\le\frac18+\frac1{16}+\frac1{16}=\frac14$. Since $G_{k_1}(x,y)=0$ for $y<0$, the value of $G_{k_1}\ast(\tilde u_0H)$ on the strip $\{y\in I_j\}$ depends only on the restrictions $H_l:=H\1_{\{y\in I_l\}}$, $l\le j$. The equation $H-G_{k_1}\ast(\tilde u_0H)=b$ is therefore equivalent to the triangular system
$$H_j-\1_{\{y\in I_j\}}G_{k_1}\ast(\tilde u_jH_j)=\1_{\{y\in I_j\}}\Big(b+\sum_{l<j}G_{k_1}\ast(\tilde u_lH_l)\Big),\qquad j=1,\dots,J,$$
whose diagonal operators have norm at most $\frac14$. Solving successively by Neumann series, the operator $H\mapsto H-G_{k_1}\ast(\tilde u_0H)$ is invertible, with inverse bounded by a constant $C_J$ which depends only on $J$, and not on $k_1$ or on $\tilde u_0$.\\

\noindent{\bf Step 3:} Proof of (iii). If $u_0=0$, then $T_k=I$ for every $k$. Moreover, by (i) and $\|u_{0,n}\|_\cK\to0$, we have $\sup_{k\in\C}\|G_k\ast(u_{0,n}\,\cdot\,)\|_{L^\infty\to L^\infty}\le\frac12$ for $n$ large, so a Neumann series gives the assertion with $R_0=1$. We may therefore assume $u_0\neq0$. By (i), $\|T_k-T_{k_1}\|\le C|k_2|\|u_0\|_{L^1}$ for $k=k_1+ik_2$, so that a Neumann series gives the invertibility of $T_k$ with $\|T_k^{-1}\|\le2C_J$ for $|k_2|\le(4CC_J\|u_0\|_{L^1})^{-1}$; the same holds for $u_{0,n}$ when $|k_2|\le(8CC_J\|u_0\|_{L^1})^{-1}$ and $\|u_{0,n}\|_{L^1}\le2\|u_0\|_{L^1}$. By (ii), $\|G_k\ast(u_0\,\cdot)\|\le\frac12$, hence $\|T_k^{-1}\|\le2$, for $|k|\ge R$, with $R$ large, and by Step 1 the same $R$ works for $u_{0,n}$, $n$ large. We choose $R_0\ge\max(1,R,8CC_J\|u_0\|_{L^1})$.\\

\noindent{\bf Step 4:} Proof of (iv). Let $\tilde H_j:=T_k^{-1}H_j$ and $\Phi_j:=G_k\ast(u_0\tilde H_j)=\tilde H_j-H_j$. The $\Phi_j$ are uniformly bounded, and they are uniformly equicontinuous:
$$|\Phi_j(X+h)-\Phi_j(X)|\le\sup_l\|\tilde H_l\|_{L^\infty}\int_{\R^2}|G_k(Z+h)-G_k(Z)|\,|u_0(X-Z)|\,dZ ,$$
and the right-hand side tends to zero as $h\to0$ uniformly in $X$: on $\{|Z|\le r\}$ we use $\|u_0\|_{L^\infty}$ and the local integrability of $G_k$ given by \eqref{sizegreen}, and on $\{|Z|\ge r\}$ we use $\|u_0\|_{L^1}$ and the uniform continuity of $G_k$ there, which follows from Lemma \ref{lemmagreen} (i) and \eqref{sizegreen}. By the Arzel\`a--Ascoli theorem, every subsequence has a further subsequence along which $\Phi_j\to\Phi$ locally uniformly, so that $\tilde H_j\to\tilde H:=H+\Phi$ pointwise and boundedly. By dominated convergence, $\Phi=G_k\ast(u_0\tilde H)$, that is, $T_k\tilde H=H$ and $\tilde H=T_k^{-1}H$. Since the limit does not depend on the subsequence, $\tilde H_j\to T_k^{-1}H$ pointwise.
\end{proof}

The next lemma bounds the norm $\|\cdot\|_\cK$ by Lebesgue norms. The condition $\frac1p+\frac2q<2$ reflects the parabolic scaling of the Green function.

\begin{lemma}[Embeddings for the norm $\|\cdot\|_\cK$]
\label{lemmaembedding}
Let $1<p,q<\infty$ with $\frac1p+\frac2q<2$. Then $\|f\|_\cK\lesssim_{p,q}\|f\|_{L^1}+\|f\|_{L^q_yL^p_x}$, and $L^1\cap L^q_yL^p_x\subset\cK_0$. In particular $L^1\cap L^2\subset\cK_0$.
\end{lemma}

\begin{proof}
Let $r:=|y-y'|$. If $r<1$, the $L^{p'}(dx')$ norm of the kernel of \eqref{defK} is $C_pr^{-\frac1{2p}}$, whatever the value of $k_1$, so that H\"older's inequality in $x'$ and then in $y'$ bounds this part by $C\|f\|_{L^q_yL^p_x}\|r^{-\frac1{2p}}\1_{\{r<1\}}\|_{L^{q'}(dy')}$, which is finite since $\frac{q'}{2p}<1$, which is equivalent to $\frac1p+\frac2q<2$. If $r\ge1$, the kernel is at most one, and this part is bounded by $\|f\|_{L^1}$. Since $C^\infty_c$ is dense in $L^1\cap L^q_yL^p_x$, this space is contained in $\cK_0$.
\end{proof}
\subsection{Smooth exponentially decaying data}

In this subsection $u_0$ is a real smooth function such that, for some $\eps>0$ and every multi-index $\gamma$,
\be
\label{expdecay}
|\pa^\gamma u_0(X)|\lesssim_\gamma e^{-\eps|X|},\qquad X\in\R^2 .
\ee
Every real $u_0\in C^\infty_c(\R^2)$ satisfies \eqref{expdecay}. By Lemma \ref{lemmaembedding}, $u_0\in\cK_0\cap L^1$, so that Lemma \ref{lemmagreenop} (iii) provides $R_0$ and the set $\Omega=\{|k|>R_0\}\cup\{|k_2|<R_0^{-1}\}$ of regular points, on which $T_k$ is uniformly invertible. On the compact set $\{|k|\le R_0,\ |k_2|\ge R_0^{-1}\}$, the operator $T_k$ may fail to be invertible. We first show that this happens only on a small set: we write \eqref{eqchi} in a symmetrized (Birman--Schwinger) form on $L^2(\R^2)$, in which the operator is of Hilbert--Schmidt type, so that the irregular points are the zeros of a regularized Fredholm determinant, which is a real-analytic function of $(k_1,k_2)$ thanks to the exponential decay.

\begin{lemma}[Irregular points]
\label{lemmairregular}
Let $u_0$ satisfy \eqref{expdecay}. The set $\Sigma$ of the irregular points of $u_0$ is a closed subset of $\{|k|\le R_0,\ |k_2|\ge R_0^{-1}\}$ with empty interior.
\end{lemma}

\begin{proof}
{\bf Step 1:} Reduction to a Hilbert--Schmidt operator. For $c\in\R$ let $g_c(x,y):=(|x+2cy|+|y|^{\frac12})^{-1}$, so that $|G_k|\lesssim g_{k_1}$ by \eqref{sizegreen}. The function $g_c$ is square integrable on the unit disc $D(0,1)\subset\R^2$, with
$$\int_{D(0,1)}g_c^2\,dX\le\int_{-1}^1\int_\R\frac{ds}{(|s|+|y|^{\frac12})^2}\,dy=\int_{-1}^1\frac{2\,dy}{|y|^{\frac12}}=8,$$
and $g_c\le4\jb c$ outside $D(0,1)$: if $|y|\ge(16\jb c^2)^{-1}$, then $g_c\le|y|^{-\frac12}$, and otherwise $|x|\ge\frac12$ and $|x+2cy|\ge\frac12-\frac18$. We write $u_0=|u_0|^{\frac12}u_0^{\frac12}$, with $u_0^{\frac12}:=\sgn(u_0)|u_0|^{\frac12}$; both factors belong to $L^2\cap L^\infty$ by \eqref{expdecay}. For $k\in\C$ let $Q_k$ be the operator on $L^2(\R^2)$ with kernel $|u_0(X)|^{\frac12}G_k(X-X')u_0^{\frac12}(X')$. Splitting $G_k=G_k\1_{D(0,1)}+G_k\1_{\R^2\setminus D(0,1)}$, we see that $Q_k$ is a Hilbert--Schmidt operator,
$$\begin{aligned}\|Q_k\|_{\rm HS}^2&=\int_{\R^2}\int_{\R^2}|u_0(X)||G_k(X-X')|^2|u_0(X')|\,dX\,dX'\\&\lesssim\|u_0\|_{L^\infty}\|u_0\|_{L^1}+\jb k^2\|u_0\|_{L^1}^2,\end{aligned}$$
and that $G_k\ast(u_0^{\frac12}\varphi)\in L^\infty$ for $\varphi\in L^2$, since $u_0^{\frac12}\varphi\in L^1\cap L^2$. If $H\in L^\infty$ and $T_kH=0$, then $\varphi:=|u_0|^{\frac12}H\in L^2$ satisfies $Q_k\varphi=|u_0|^{\frac12}G_k\ast(u_0H)=\varphi$. Conversely, if $\varphi\in L^2$ and $Q_k\varphi=\varphi$, then $H:=G_k\ast(u_0^{\frac12}\varphi)\in L^\infty$ satisfies $|u_0|^{\frac12}H=\varphi$, hence $u_0H=u_0^{\frac12}\varphi$ and $T_kH=0$. Since $|u_0|^{\frac12}H=0$ implies $u_0H=0$ and then $H=G_k\ast(u_0H)=0$, the map $H\mapsto|u_0|^{\frac12}H$ is a bijection from the kernel of $T_k$ onto the kernel of $I-Q_k$. If the latter is trivial, then $I-Q_k$ is invertible by the Fredholm alternative, and a direct computation shows that
$$T_k^{-1}b=b+G_k\ast\big(u_0^{\frac12}(I-Q_k)^{-1}(|u_0|^{\frac12}b)\big),\qquad b\in L^\infty(\R^2).$$
Therefore $k$ is a regular point if and only if $I-Q_k$ is invertible, that is, if and only if $\Xi(k):=\det_2(I-Q_k)\neq0$, where $\det_2$ is the regularized Fredholm determinant on the Hilbert--Schmidt class \cite[Chapter 9]{Simon}.\\

\noindent{\bf Step 2:} $\Sigma$ is closed. By Lemma \ref{lemmagreenop} (i), $\|T_k-T_{k'}\|_{L^\infty\to L^\infty}\lesssim|k-k'|\|u_0\|_{L^1}$. Since the set of invertible operators is open, the set of regular points is open, and $\|T_k^{-1}\|$ is locally bounded on it. By Lemma \ref{lemmagreenop} (iii), $\Sigma\subset\{|k|\le R_0,\ |k_2|\ge R_0^{-1}\}$.\\

\noindent{\bf Step 3:} Analyticity of $\Xi$. Let $k_2>0$. Substituting $s=\pm(k_2+r)$, $r>0$, if $y>0$, and $s=k_2\sigma$, $|\sigma|<1$, if $y<0$, in \eqref{formulagreen}, we obtain, with $\tilde x=x+2k_1y$,
$$G_k(x,y)=\begin{cases}\displaystyle\frac1{2\pi}\int_0^\infty\big(e^{ir\tilde x}+e^{-i(2k_2+r)\tilde x}\big)e^{-(r^2+2k_2r)y}\,dr,& y>0,\\[3mm] \displaystyle-\frac{k_2}{2\pi}\int_{-1}^1e^{-ik_2(1-\sigma)\tilde x-k_2^2(1-\sigma^2)|y|}\,d\sigma,& y<0.\end{cases}$$
For fixed $X$ with $y\neq0$, the right-hand sides are entire functions of $(k_1,k_2)\in\C^2$, which we denote by $\tilde G(k_1,k_2;X)$. Fix $k^0=(k^0_1,k^0_2)\in\R\times(0,\infty)$ and let $(k_1,k_2)=(p_1+iq_1,p_2+iq_2)\in\C^2$ with $|k_j-k^0_j|\le2\mu$, $j=1,2$, where $0<\mu\le\frac{k^0_2}8$ is chosen below. Then $p_2\pm q_1\ge0$ and $p_2^2-q_2^2\ge0$. For $y>0$, the first integrand is $e^{ir(x+2(p_1-q_2)y)}$ times the weight $e^{-r^2y-2(p_2+q_1)ry}$, which decreases from $1$ to $0$, and the second one is $e^{-2ik_2\tilde x}e^{-ir(x+2(p_1+q_2)y)}$ times the decreasing weight $e^{-r^2y-2(p_2-q_1)ry}$. The integration by parts of Step 2 of the proof of Lemma \ref{lemmagreen}, together with $\int_0^\infty e^{-r^2y}dr\lesssim y^{-\frac12}$, and the bound $|e^{-k_2^2(1-\sigma^2)|y|}|\le1$ for $y<0$, give
$$|\tilde G(k_1,k_2;X)|\lesssim e^{2|\Im(k_2\tilde x)|}\big(\jb{k^0}+g_{p_1-q_2}(X)+g_{p_1+q_2}(X)\big),$$
where $|\Im(k_2\tilde x)|=|q_2x+2(p_1q_2+q_1p_2)y|\le C\mu\jb{k^0}|X|$. We choose $\mu$ so small that $4C\mu\jb{k^0}\le\frac\eps2$. Since $e^{-\eps(|X|+|X'|)}e^{\frac\eps2|X-X'|}\le e^{-\frac\eps2(|X|+|X'|)}$, Step 1 shows that the operators $\tilde Q(k_1,k_2)$ with kernels $|u_0(X)|^{\frac12}\tilde G(k_1,k_2;X-X')u_0^{\frac12}(X')$ satisfy $\|\tilde Q(k_1,k_2)\|_{\rm HS}\lesssim_{k^0}1$ for $|k_j-k^0_j|\le2\mu$. By the Cauchy formula on the torus $\{|z_j-k^0_j|=2\mu\}$ and Minkowski's integral inequality \cite[Theorem 6.19]{Folland}, the operators $Q^{(\gamma)}$, $\gamma\in\{0,1,2,\dots\}^2$, whose kernels are given by the Taylor coefficients of $\tilde G$ at $k^0$, satisfy $\|Q^{(\gamma)}\|_{\rm HS}\lesssim_{k^0}(2\mu)^{-|\gamma|}$. Hence the series $\sum_\gamma(k_1-k^0_1)^{\gamma_1}(k_2-k^0_2)^{\gamma_2}Q^{(\gamma)}$ converges in the Hilbert--Schmidt norm for $|k_j-k^0_j|<\mu$, and its sum is $\tilde Q(k_1,k_2)$, since the Taylor series of $\tilde G$ converges pointwise. Thus $(k_1,k_2)\mapsto\tilde Q(k_1,k_2)$ is holomorphic with values in the Hilbert--Schmidt class, and since $A\mapsto\det_2(I-A)$ is holomorphic on this class \cite[Chapter 9]{Simon}, $\det_2(I-\tilde Q)$ is a holomorphic extension of $\Xi$ near $k^0_1+ik^0_2$. Hence $\Xi$ is real-analytic on $\C^+$. By Lemma \ref{lemmagreen} (iv), $Q_{\bar k}$ is the complex conjugate of $Q_k$, so that $\Xi(\bar k)=\overline{\Xi(k)}$, and $\Xi$ is real-analytic on $\C^-$ as well.\\

\noindent{\bf Step 4:} Conclusion. By Step 1 and Lemma \ref{lemmagreenop} (iii), $\Xi(k)\neq0$ for $|k|>R_0$. If $\Sigma$ contained a nonempty open set $U\subset\C^\pm$, then $\Xi$ would vanish on $U$, hence on the connected open set $\C^\pm$ by real-analyticity, a contradiction. Hence $\Sigma$ has empty interior.
\end{proof}

At every regular point $k\in\C$ we define the eigenfunction $\chi(k;\cdot)$, and at every regular point $k\in\C\setminus\R$ the scattering coefficient $a(k)$, by
\be
\label{defchi}
\chi(k;\cdot):=T_k^{-1}1,
\ee
\be
\label{defa}
\begin{aligned}
a(k)&:=-i\sgn(k_2)\,\widehat{\big(u_0\chi(k;\cdot)\big)}(\zeta(k))\\
&=-\frac{i\sgn(k_2)}{2\pi}\int_{\R^2}u_0\,\chi(k;\cdot)\,e^{-iX\cdot\zeta(k)}\,dX .
\end{aligned}
\ee
Thus $\chi(k;\cdot)$ is the unique bounded solution of the integral equation $\chi=1+G_k\ast(u_0\chi)$ in \eqref{eqchi}, and when $\chi=1$, \eqref{defa} reduces to the Born approximation \eqref{bornapprox}. At the irregular points these definitions break down, and we use the following result of Grinevich, in the form stated here.

\begin{theorem}[Grinevich \cite{Grin}]
\label{thmgrin}
Let $u_0$ be real and satisfy \eqref{expdecay}. The functions $k\mapsto\chi(k;X)$, $X\in\R^2$, and $k\mapsto a(k)$, defined on the set of regular points of $\C\setminus\R$ by \eqref{defchi} and \eqref{defa}, extend to $C^1$ functions on $\C\setminus\R$.
\end{theorem}

Theorem \ref{thmgrin} is our formulation of the main result of \cite[Theorem 11]{Grin}, which states that, for real potentials satisfying \eqref{expdecay}, the direct scattering transform is nonsingular, without any smallness assumption: the eigenfunctions and the scattering data are regular functions of the spectral parameter, also at the irregular points. In the notation of \cite{Grin}, the heat operator is $\pa_y-\pa_x^2+u$ and the spectral parameter is $\lambda\in\C$; with $u=-u_0$ and $\lambda=k$, the integral equation for the normalized eigenfunction $\chi$ in \cite{Grin} is the integral equation in \eqref{eqchi}, with the Green function \eqref{defgreen}, and the function $T(\lambda)=\frac{i\sgn\lambda_I}{2\pi}b(\lambda)$ of \cite{Grin}, where $\lambda_I:=\Im\lambda$ and $b(\lambda)$ is the scattering data of \cite{Grin}, is $a(k)$. A point is regular in \cite{Grin} if the homogeneous equation $H=G_k\ast(u_0H)$ has no nonzero solution; by Step 1 of the proof of Lemma \ref{lemmairregular}, this is our definition, since the bounded solutions of this equation tend to zero at infinity by \eqref{sizegreen}. The argument of \cite{Grin} combines formal Taylor expansions of the eigenfunctions and of the scattering data at the irregular points, whose construction uses the reality of the potential \cite[Theorems 7, 8, and 10; see also Remark 9]{Grin}, with the representation of the scattering data as a ratio of two real-analytic functions of $(k_1,k_2)$; for this analyticity, which uses the exponential decay, \cite{Grin} refers to the argument of \cite{Novikov92} for the Schr\"odinger operator at a fixed energy. The analyticity of the denominator, a regularized Fredholm determinant, is proved in Lemma \ref{lemmairregular}. Since the set of irregular points has empty interior by Lemma \ref{lemmairregular}, the regular points are dense in $\C\setminus\R$, so that the extensions in Theorem \ref{thmgrin} are unique; we still denote them by $\chi$ and $a$. Lemma \ref{lemmairregular} and Theorem \ref{thmgrin} are the only places where the exponential decay in \eqref{expdecay} is used. For $e\in\{1,i\}$ we denote by $\pa_e$ the derivative in the direction $e$, that is, $\pa_1:=\pa_{k_1}$ and $\pa_i:=\pa_{k_2}$, and we write $\chi_k:=\chi(k;\cdot)$.

\begin{proposition}[The direct problem for smooth data]
\label{propdirect}
Let $u_0$ satisfy \eqref{expdecay}, and let $\chi$ and $a$ be given by \eqref{defchi}, \eqref{defa} and Theorem \ref{thmgrin}. Then:
\begin{enumerate}
\item $a$ satisfies \eqref{hypa};
\item for every $X\in\R^2$, $\chi(\cdot;X)=m(\cdot;0,X)$, where $m$ is the solution of \eqref{eqm} at $t=0$; in particular $\chi$ is bounded on $(\C\setminus\R)\times\R^2$, and the formula \eqref{defa} holds at every $k\in\C\setminus\R$;
\item the function reconstructed from $a$ by \eqref{reconstruction} at $t=0$ is $u_0$.
\end{enumerate}
\end{proposition}

In particular, (iii) shows that the reconstruction \eqref{reconstruction} inverts the map $u_0\mapsto a$.

\begin{proof}
{\bf Step 1:} Derivatives in $X$. Let $\mathcal R$ be the set of regular points. Let $k\in\mathcal R$. Since convolution commutes with translations, $\tau_h\chi_k=1+G_k\ast((\tau_hu_0)\,\tau_h\chi_k)$ for $h\in\R^2$, hence
\be
\label{translationchi}
T_k(\tau_h\chi_k-\chi_k)=G_k\ast\big((\tau_hu_0-u_0)\,\tau_h\chi_k\big).
\ee
The difference quotients of $u_0$ converge to its derivatives in $L^1\cap L^2$, hence in $\cK$ by Lemma \ref{lemmaembedding}. By \eqref{translationchi} and Lemma \ref{lemmagreenop} (i), $\chi_k$ is differentiable in $X$ in the $L^\infty$ norm, and by induction on $|\gamma|$,
\be
\label{derivativeschi}
T_k\pa^\gamma\chi_k=\sum_{0<\nu\le\gamma}\binom\gamma\nu\,G_k\ast\big((\pa^\nu u_0)\,\pa^{\gamma-\nu}\chi_k\big),\qquad \gamma\neq0 .
\ee
Hence every derivative of $\chi_k$ is bounded, uniformly for $k\in\Omega$ and locally uniformly on $\mathcal R$, and $\chi_k$ is smooth. Since $\chi_k-1=T_k^{-1}(G_k\ast u_0)$, Lemma \ref{lemmagreenop} (ii) and \eqref{derivativeschi} also give
\be
\label{chiinfinity}
\lim_{|k|\to\infty}\Big(\|\chi_k-1\|_{L^\infty}+\sum_{0<|\gamma|\le2}\|\pa^\gamma\chi_k\|_{L^\infty}\Big)=0 .
\ee

\noindent{\bf Step 2:} Derivatives in $k$ and the $\dbar$ equation. Let $k\in\mathcal R\setminus\R$, $e\in\{1,i\}$ and $h\in\R\setminus\{0\}$ small. Subtracting the equations \eqref{eqchi} at $k+he$ and at $k$,
$$T_k\frac{\chi_{k+he}-\chi_k}h=\frac{G_{k+he}-G_k}h\ast\big(u_0\chi_{k+he}\big).$$
By Lemma \ref{lemmagreen} (iii) and Lemma \ref{lemmagreenop} (i), $\|\chi_{k+he}-\chi_k\|_{L^\infty}\lesssim|h|$, and the kernels on the right-hand side are bounded and converge pointwise on $\{y\neq0\}$ to $\pa_eG_k$. By dominated convergence, the right-hand side converges pointwise and boundedly to $(\pa_eG_k)\ast(u_0\chi_k)$, and Lemma \ref{lemmagreenop} (iv) shows that $\chi(\cdot;X)$ is differentiable in the direction $e$ at $k$, for every $X$, with
\be
\label{dkchi}
T_k\,\pa_e\chi_k=(\pa_eG_k)\ast(u_0\chi_k),\qquad \|\pa_e\chi_k\|_{L^\infty}\lesssim\|T_k^{-1}\|^2\,\|u_0\|_{L^1},
\ee
where we used $\|\chi_k\|_{L^\infty}\le\|T_k^{-1}\|$. Arguing as in Step 1, where now the derivatives in $X$ of $(\pa_eG_k)\ast(u_0\chi_k)$ fall on the smooth integrable function $u_0\chi_k$, we see that all derivatives in $X$ of $\pa_e\chi_k$ are bounded, locally uniformly on $\mathcal R\setminus\R$ and uniformly on $\Omega\setminus\R$. Next, combining \eqref{dkchi} with Lemma \ref{lemmagreen} (v), and using $(X-X')\cdot\zeta(k)=X\cdot\zeta(k)-X'\cdot\zeta(k)$,
$$\begin{aligned}T_k\,\dbar\chi_k&=(\dbar G_k)\ast(u_0\chi_k)\\&=-\frac{i\sgn(k_2)}{2\pi}e^{iX\cdot\zeta(k)}\int_{\R^2}u_0(X')\chi_k(X')e^{-iX'\cdot\zeta(k)}\,dX'=a(k)\,e^{iX\cdot\zeta(k)} .\end{aligned}$$
On the other hand, conjugating \eqref{eqchi} and using $\overline{G_k}=G_{\bar k}$ and $G_{\bar k}(X-X')=G_k(X-X')e^{-i(X-X')\cdot\zeta(k)}$ (Lemma \ref{lemmagreen} (iv)), we find that $\bar k$ is a regular point, $\overline{\chi_k}=\chi_{\bar k}$, and $T_k\big(e^{iX\cdot\zeta(k)}\overline{\chi_k}\big)=e^{iX\cdot\zeta(k)}$. Since $T_k$ is invertible, we obtain, for $k\in\mathcal R\setminus\R$,
\be
\label{dbarchi}
\dbar_k\chi(k;X)=a(k)\,e^{iX\cdot\zeta(k)}\,\overline{\chi(k;X)},\qquad \chi(\bar k;X)=\overline{\chi(k;X)} .
\ee
Finally, on the strip $\{|k_2|<R_0^{-1}\}$, which is contained in $\Omega$, the same subtraction gives $\|\chi_k-\chi_{k'}\|_{L^\infty}\lesssim|k-k'|$, across the real axis.\\

\noindent{\bf Step 3:} Decay of $a$ on $\Omega$. Let $k\in\Omega\setminus\R$. Since $(1-\pa_x^2-\pa_y^2)e^{-iX\cdot\zeta(k)}=\jb{\zeta(k)}^2e^{-iX\cdot\zeta(k)}$, integrations by parts in \eqref{defa}, legitimate by \eqref{expdecay} and Step 1, give for every integer $q\ge0$
$$\jb{\zeta(k)}^{2q}a(k)=-\frac{i\sgn(k_2)}{2\pi}\int_{\R^2}(1-\pa_x^2-\pa_y^2)^q\big(u_0\chi_k\big)\,e^{-iX\cdot\zeta(k)}\,dX,$$
hence $|a(k)|\lesssim_q\jb{\zeta(k)}^{-2q}$.
By Step 2 we may differentiate \eqref{defa} under the integral sign:
\be
\label{dka}
\pa_ea(k)=-\frac{i\sgn(k_2)}{2\pi}\int_{\R^2}\Big(u_0\,\pa_e\chi_k-i\big(X\cdot\pa_e\zeta(k)\big)u_0\chi_k\Big)e^{-iX\cdot\zeta(k)}\,dX,
\ee
with $|\pa_e\zeta(k)|\le4\jb k$. The functions $u_0\pa_e\chi_k$, $xu_0\chi_k$ and $yu_0\chi_k$ have integrable derivatives of all orders, uniformly on $\Omega\setminus\R$, and the same integrations by parts give $|\nabla_ka(k)|\lesssim_q\jb k\jb{\zeta(k)}^{-2q}$. By the Lipschitz bound at the end of Step 2, the integral in \eqref{defa} is continuous across the real axis, so that $a$ has continuous extensions to $\overline{\C^+}$ and to $\overline{\C^-}$. By \eqref{dbarchi}, $\zeta(\bar k)=-\zeta(k)$ and the reality of $u_0$, $a(\bar k)=\overline{a(k)}$ for $k\in\mathcal R\setminus\R$.\\

\noindent{\bf Step 4:} Proof of (i). By Theorem \ref{thmgrin}, $a$ is $C^1$ on $\C\setminus\R$, hence bounded with bounded derivatives on the compact set $\{|k|\le R_0,\ |k_2|\ge R_0^{-1}\}$, on which $\jb{\zeta(k)}$ is bounded; by density, $a(\bar k)=\overline{a(k)}$ on $\C\setminus\R$. Together with Step 3, this proves \eqref{hypa}.\\

\noindent{\bf Step 5:} Proof of (ii). Fix $X$. By Theorem \ref{thmgrin}, both sides of \eqref{dbarchi} are continuous on $\C\setminus\R$, and they coincide on the dense set $\mathcal R\setminus\R$ (Lemma \ref{lemmairregular}); hence \eqref{dbarchi} holds on $\C\setminus\R$. On the strip $\{|k_2|<R_0^{-1}\}$, the function $k\mapsto\chi(k;X)$ is Lipschitz by Step 2, so that its distributional derivative $\dbar\chi$ is a bounded function, with no singular part on the real axis. Hence
$$\dbar\chi(\cdot;X)=a\,P\,\overline{\chi(\cdot;X)}\quad\mbox{in}\ \mathcal D'(\C),\qquad P=e^{iX\cdot\zeta(k)}=e^{i\phi(k;0,X)} .$$
The function $g:=\chi(\cdot;X)-1$ is continuous on $\C$, tends to zero at infinity by \eqref{chiinfinity}, and satisfies $\dbar g=aP\overline g+aP$. By (i), Lemma \ref{lemmaweights} (ii) and Lemma \ref{lemmacauchy} (iv), $\cI|a|\in C_0(\C)$, and the uniqueness part of Lemma \ref{lemmavekua} (i), with the coefficient $aP$ and $f=aP$, shows that $g=m(\cdot;0,X)-1$. By Lemma \ref{lemmainverse} (i), $|\chi(k;X)|\le1+\|\cI|a|\|_\infty e^{2\|\cI|a|\|_\infty}$ for all $(k,X)$. Finally, let $k\in\C\setminus\R$ and let $k_j\in\mathcal R\setminus\R$ tend to $k$. Then $a(k_j)\to a(k)$ and $\chi(k_j;X)\to\chi(k;X)$ for every $X$, and the dominated convergence theorem, with the majorant $C|u_0|$, shows that the formula \eqref{defa} holds at $k$.\\

\noindent{\bf Step 6:} Proof of (iii). By (ii), $\pa_xm(k;0,X)=\pa_x\chi(k;X)$ for $k\in\Omega\setminus\R$. By Lemma \ref{lemmainverse} (ii), $\pa_xm=\cC f$ with $f:=\pa_x(aP\overline m)=aP\big(-2ik_2\overline m+\overline{\pa_xm}\big)$. Since $|k_2|\jb k\le|k_2|+|k_1k_2|+k_2^2\lesssim\jb{\zeta(k)}^2$, \eqref{hypa} and Lemma \ref{lemmainverse} (i), (ii) give $|f|\lesssim\jb k^{-1}\jb{\zeta}^{-6}$. Hence $f\in L^1(\C)$ by \eqref{weightsE}, and
$$k\,\cC f(k)-\frac1\pi\int_\C f\,dA=\frac1\pi\int_\C f(k')\frac{k'}{k-k'}\,dA(k').$$
On $\{|k'|\le\frac{|k|}2\}$ the integrand is bounded by $|f(k')|$ and tends to zero as $|k|\to\infty$. On $\{|k'|>\frac{|k|}2\}$ we use $|k'||f(k')|\lesssim\jb{\zeta(k')}^{-6}$, so that, for $|k|\ge2$, this part is bounded by $C\cI(\jb\zeta^{-6}\1_{\{|k'|>|k|/2\}})(k)\lesssim\frac{1+\log|k|}{|k|}$ by \eqref{weightstail}. Hence $k\,\pa_xm(k;0,X)\to\frac1\pi\int_\C f\,dA$ as $|k|\to\infty$. On the other hand, \eqref{eqchi} gives $2k\pa_x\chi=\pa_y\chi-\pa_x^2\chi-u_0\chi$, which tends to $-u_0$ uniformly by \eqref{chiinfinity}. Therefore $u_0=-\frac2\pi\int_\C\pa_x(aP\overline m)\,dA$, which is \eqref{reconstruction} at $t=0$.
\end{proof}

\begin{corollary}[The solution for smooth exponentially decaying data]
\label{corsmooth}
Let $u_0$ satisfy \eqref{expdecay}, let $a$ be its scattering coefficient, and let $u$ be the solution of \eqref{kp}. Then $u$ coincides with the function reconstructed from $a$ by \eqref{reconstruction}, and
\be
\label{plancherelsmooth}
\|u_0\|_{L^2}^2=8\int_\C|k_2||a(k)|^2\,dA(k)=\|\cB a\|_{L^2}^2 .
\ee
\end{corollary}

\begin{proof}
By Proposition \ref{propdirect}, $a$ satisfies \eqref{hypa} and its reconstruction $\tilde u$ satisfies $\tilde u(0)=u_0$. By Proposition \ref{proprapid} and Lemma \ref{lemmainverse} (iii), $\tilde u\in C(\R;L^2)$ is real, satisfies the Duhamel formula \eqref{duhamel}, $\tilde u$ and $\pa_x\tilde u$ are bounded, and \eqref{plancherelsmooth} holds. Since $u_0\in H^s(\R^2)$ for every $s$, the solution $u$ belongs to $C(\R;H^s)$ for every $s$ \cite{Bourgain}, so that $u$ and $\pa_xu$ are bounded on $[-T,T]\times\R^2$ for every $T$, and $u$ satisfies \eqref{duhamel}. Lemma \ref{lemmaunique} gives $u=\tilde u$.
\end{proof}

Let $u_0$ satisfy \eqref{expdecay}. By Proposition \ref{propdirect} (i) and Step 3 of the proof of Proposition \ref{proprapid}, Proposition \ref{propcone} and Lemma \ref{lemmaweak} give $S(-t)u(t)\weakto\cB a$ as $t\to\pm\infty$; since $\|u(t)\|_{L^2}=\|u_0\|_{L^2}=\|\cB a\|_{L^2}$ by \eqref{mass} and \eqref{plancherelsmooth}, the convergence is strong. Thus the solutions with data satisfying \eqref{expdecay} scatter, with $u_\pm=\cB a$, and Theorem \ref{thmmain} extends this to rough data, by approximation.


\section{The Born expansion and spectral estimates}
\label{sectionborn}

From now on, $Y$ denotes one of the spaces of Theorem \ref{thmmain}: $Y=L^{\frac32,4}(\R^2)$, or $Y=L^{p_y}_yL^{p_x}_x$ with $1\le p_x<2$ and $\frac1{p_x}+\frac2{p_y}=2$, or $Y=L^{\frac43,2}_y(L^2_x)$; by the reduction explained after Theorem \ref{thmmain}, the case $\frac1{p_x}+\frac2{p_y}>2$ is contained in the second one. On the spectral side we use the measure
\be
\label{defmu}
d\mu(k):=|k_2|\,dA(k),
\ee
for which \eqref{changevariables} reads $\int_\C g\circ\zeta\,d\mu=\frac18\int_{\R^2}g\,d\xi\,d\eta$; by \eqref{bornisometry}, $\|\cB a\|_{L^2}^2=8\|a\|_{L^2(\mu)}^2$.

The scattering coefficient of a datum is a nonlinear function of the datum, and its first approximation is the Born term $-i\sgn(k_2)\hat u_0(\zeta(k))$ of \eqref{bornapprox}. For data in $L^2\cap Y$, the Born term is controlled by the Hausdorff--Young inequality, but not in the weighted $L^\infty$ norms of the preceding sections. In this section we expand the coefficient into its first Born terms and a remainder, and we bound the Born terms in spectral norms adapted to $Y$. The remainder is estimated in Sections \ref{sectionpotential} and \ref{sectionremainder}.

\subsection{Lorentz spaces and multilinear interpolation}
\label{subsectionlorentz}

Let $(\Omega,\nu)$ be a nonatomic $\sigma$-finite measure space; all the measures below are of this type. The decreasing rearrangement of a measurable function $f$ is $f^*(s):=\inf\{\lambda\ge0:\ \nu(\{|f|>\lambda\})\le s\}$, $s>0$, and for $0<p<\infty$ and $0<r\le\infty$ the Lorentz quasi-norm is
$$\|f\|_{L^{p,r}(\nu)}:=\Big(\int_0^\infty\big(s^{\frac1p}f^*(s)\big)^r\frac{ds}s\Big)^{\frac1r},\qquad \|f\|_{L^{p,\infty}(\nu)}:=\sup_{s>0}s^{\frac1p}f^*(s).$$
We use the following standard facts, see \cite[Section 1.4]{Grafakos}: $L^{p,p}=L^p$; $L^{p,r_1}\subset L^{p,r_2}$ for $r_1\le r_2$; for $1<p<\infty$ and $1\le r\le\infty$ the quasi-norm is equivalent to a norm, the dual space of $L^{p,r}$ is $L^{p',r'}$ for $1\le r<\infty$, and $\|f\|_{L^{p,r}}\sim\sup\{|\int fg\,d\nu|:\ \|g\|_{L^{p',r'}}\le1\}$ for $1\le r\le\infty$; the Fourier transform maps $L^{p,r}(\R^d)$ into $L^{p',r}(\R^d)$ for $1<p<2$, by the off-diagonal Marcinkiewicz interpolation theorem; and $(fg)^*(s)\le f^*(\frac s2)g^*(\frac s2)$. Since $|x|^{-\beta}\in L^{\frac1\beta,\infty}(\R)$ has rearrangement $(2/s)^{\beta}$, the last property gives the Lorentz H\"older inequality with a power weight:
\be
\label{lorentzholder}
\big\||x|^{-\beta}f\big\|_{L^{p_1,r}(\R)}\le C_{\beta,p,r}\|f\|_{L^{p,r}(\R)},\qquad \frac1{p_1}=\frac1p+\beta,\quad \beta>0 .
\ee
More generally, the Lorentz H\"older inequality $\|fg\|_{L^{p,r}}\lesssim\|f\|_{L^{p_1,r_1}}\|g\|_{L^{p_2,r_2}}$ holds for $\frac1p=\frac1{p_1}+\frac1{p_2}$ and $\frac1r=\frac1{r_1}+\frac1{r_2}$ \cite{ONeil}. The convolution inequality of O'Neil \cite{ONeil} states that, for $1<p,q,s<\infty$ and $1\le r_1,r_2,r_3\le\infty$ with $\frac1s=\frac1p+\frac1q-1$ and $\frac1{r_3}\le\frac1{r_1}+\frac1{r_2}$,
\be
\label{oneil}
\|f\ast g\|_{L^{s,r_3}(\R^d)}\lesssim\|f\|_{L^{p,r_1}(\R^d)}\|g\|_{L^{q,r_2}(\R^d)} .
\ee
With $g=|x|^{-\beta}$, $0<\beta<d$, this contains the Hardy--Littlewood--Sobolev inequality and its Lorentz version.

Our multilinear estimates are first proved for Lebesgue norms, with exponents in an open set, and then transferred to Lorentz norms. The following lemma is a version, for positive forms, of the multilinear real interpolation theorem of Janson \cite{Janson}; we include its short proof.

\begin{lemma}[Multilinear interpolation for positive forms]
\label{lemmainterp}
Let $(\Omega_i,\nu_i)$, $1\le i\le N$, be $\sigma$-finite measure spaces, and let $\Lambda(f_1,\dots,f_N)\in[0,\infty]$ be defined for nonnegative measurable functions $f_i$ on $\Omega_i$, additive and positively homogeneous in each argument, and such that $\Lambda(f_1,\dots,f_N)=\lim_n\Lambda(f_{1,n},\dots,f_{N,n})$ whenever $f_{i,n}\uparrow f_i$ pointwise. Let $\theta^0\in(0,1)^N$, and assume that there are $\eps>0$ and $C_0$ such that
\be
\label{hypinterp}
\Lambda(f_1,\dots,f_N)\le C_0\prod_{i=1}^N\|f_i\|_{L^{1/\theta_i}(\nu_i)}
\ee
for all $\theta\in\R^N$ with $\sum_i\theta_i=\sum_i\theta^0_i$ and $|\theta-\theta^0|\le\eps$. Then, for all $r_1,\dots,r_N\in[1,\infty]$ with $\sum_i\frac1{r_i}\ge1$,
\be
\label{conclinterp}
\Lambda(f_1,\dots,f_N)\lesssim C_0\prod_{i=1}^N\|f_i\|_{L^{1/\theta^0_i,r_i}(\nu_i)},
\ee
with an implicit constant depending only on $N$, $\eps$, $(r_i)$ and an upper bound for $\max_i(\theta^0_i)^{-1}$.
\end{lemma}

\begin{proof}
Replacing $\eps$ by $\min(\eps,\frac12\min_i\theta^0_i)$, we may assume that all the exponents $\theta$ in \eqref{hypinterp} have positive entries. For a nonnegative function $f$ with $\|f\|_{L^{p,r}}<\infty$, $p<\infty$, let $t_j:=f^*(2^j)$, $j\in\mathbb Z$, and $f_{(j)}:=f\1_{\{t_j<f\le t_{j-1}\}}$. Then $\nu(\{f>t_j\})\le2^j$, and $f=\sum_jf_{(j)}$ almost everywhere: the disjoint sets $\{t_j<f\le t_{j-1}\}$ cover $\{f>0\}$ up to a subset of $\bigcap_j\{f>t_j\}$, which has measure at most $\inf_j2^j=0$, since $t_j\to0$ as $j\to\infty$. Moreover $f_{(j)}\le t_{j-1}=f^*(2^{j-1})$, and $f_{(j)}$ vanishes outside $\{f>t_j\}$, so that $\|f_{(j)}\|_{L^{q}}\le f^*(2^{j-1})2^{\frac jq}$ for every $q\in(0,\infty]$. Since $f^*$ is nonincreasing, comparing the integral in the definition of the Lorentz norm with its dyadic Riemann sums gives
\be
\label{discretelorentz}
\Big(\sum_{j\in\mathbb Z}\big(2^{\frac jp}f^*(2^{j-1})\big)^r\Big)^{\frac1r}\lesssim_{p,r}\|f\|_{L^{p,r}},
\ee
with the usual modification for $r=\infty$. Decompose each $f_i$ in this way, and let $c_{i,j}:=2^{j\theta^0_i}f_i^*(2^{j-1})$, so that $\|(c_{i,j})_j\|_{\ell^{r_i}}\lesssim\|f_i\|_{L^{1/\theta^0_i,r_i}}$ by \eqref{discretelorentz}. By multilinearity, positivity and monotone convergence,
$$\Lambda(f_1,\dots,f_N)=\sum_{j_1,\dots,j_N\in\mathbb Z}\Lambda\big(f_{1,(j_1)},\dots,f_{N,(j_N)}\big).$$
Fix $(j_1,\dots,j_N)$, let $a$ be an index where $j_i$ is minimal and $b\neq a$ an index where $j_i$ is maximal (if all the $j_i$ are equal, we take $\theta=\theta^0$), and let $\theta:=\theta^0+\eps'(e_a-e_b)$, where $e_i$ are the unit vectors and $\eps':=\eps/2$. Then \eqref{hypinterp} applies, and
$$\Lambda\big(f_{1,(j_1)},\dots,f_{N,(j_N)}\big)\le C_0\prod_{i=1}^Nf_i^*(2^{j_i-1})2^{j_i\theta_i}=C_0\,2^{-\eps'(j_b-j_a)}\prod_{i=1}^Nc_{i,j_i} .$$
We write $j_i=j_1+d_i$ with $d=(d_2,\dots,d_N)\in\mathbb Z^{N-1}$, and we let $|d|_*:=\max(0,d_2,\dots,d_N)-\min(0,d_2,\dots,d_N)=j_b-j_a$. Choose $s_i\in[r_i,\infty]$ with $\sum_i\frac1{s_i}=1$, which is possible since $\sum_i\frac1{r_i}\ge1$. For fixed $d$, H\"older's inequality in $j_1$ and $\ell^{r_i}\subset\ell^{s_i}$ give $\sum_{j_1}\prod_ic_{i,j_1+d_i}\le\prod_i\|(c_{i,j})_j\|_{\ell^{r_i}}$. The number of $d$ with $|d|_*=n$ is at most $(2n+1)^{N-1}$, hence
$$\Lambda(f_1,\dots,f_N)\le C_0\sum_{n\ge0}(2n+1)^{N-1}2^{-\eps'n}\prod_{i=1}^N\|(c_{i,j})_j\|_{\ell^{r_i}}\lesssim C_0\prod_{i=1}^N\|f_i\|_{L^{1/\theta^0_i,r_i}} .$$
\end{proof}

We shall apply Lemma \ref{lemmainterp} to forms defined by integrals of products of nonnegative functions; such forms are positive and monotone. Two consequences will be used repeatedly. First, if $\Lambda(f_1,\dots,f_N,g)=\int T(f_1,\dots,f_N)\,g\,d\nu$ for a positive multilinear operator $T$, the duality $(L^{p,r})'=L^{p',r'}$ turns \eqref{conclinterp} into a bound for $T$ in Lorentz norms. Second, applied to the trilinear form $\int_{\R^2}f_1(s)f_2(s')f_3(s+s')\,ds\,ds'$, for which Young's inequality \cite[Theorem 4.2]{LiebLoss} gives \eqref{hypinterp} for all $\theta\in[0,1]^3$ with $\sum_i\theta_i=2$, Lemma \ref{lemmainterp} gives the convolution inequality \eqref{oneil} in dimension one with the pairing in duality.

The forms on $\R^d$ that we use are integrals of products of functions of linear forms. They are bounded by the following elementary inequality. Let $\lambda_1,\dots,\lambda_M$ be linear forms on $\R^d$, and let $\mathcal F$ be a family of $d$-element subsets $\beta\subset\{1,\dots,M\}$ such that $\{\lambda_i\}_{i\in\beta}$ is a basis; we write $\det\beta$ for the determinant of the matrix whose rows are the coefficients of $\lambda_i$, $i\in\beta$. Let $w_\beta\ge0$, $\beta\in\mathcal F$, with $\sum_\beta w_\beta=1$, and assume that the marginals $\theta_i:=\sum_{\beta\ni i}w_\beta$ satisfy $0<\theta_i\le1$. Then, for nonnegative measurable functions $H_i$ on $\R$,
\be
\label{basisholder}
\int_{\R^d}\prod_{i=1}^MH_i(\lambda_i(x))\,dx\le\prod_{\beta\in\mathcal F}|\det\beta|^{-w_\beta}\prod_{i=1}^M\|H_i\|_{L^{1/\theta_i}(\R)} .
\ee
Indeed, $\prod_iH_i=\prod_\beta\big(\prod_{i\in\beta}H_i^{1/\theta_i}\big)^{w_\beta}$, H\"older's inequality with the exponents $\frac1{w_\beta}$ bounds the integral by $\prod_\beta\big(\int\prod_{i\in\beta}H_i^{1/\theta_i}(\lambda_i(x))\,dx\big)^{w_\beta}$, and the change of variables $y_i=\lambda_i(x)$, $i\in\beta$, gives $\int\prod_{i\in\beta}H_i^{1/\theta_i}(\lambda_i(x))\,dx=|\det\beta|^{-1}\prod_{i\in\beta}\|H_i\|_{L^{1/\theta_i}}^{1/\theta_i}$.

\subsection{The Born expansion}

Let $u\in C^\infty_c(\R^2)$ be real, let $a$ be its scattering coefficient and $\chi$ its eigenfunction, given by \eqref{defchi}, \eqref{defa} and Theorem \ref{thmgrin}. For $k\in\C\setminus\R$ and $H\in L^\infty(\R^2)$ we write $J_kH:=G_k\ast(uH)$, so that \eqref{eqchi} reads $\chi_k=1+J_k\chi_k$. Iterating this identity $\nu$ times gives $\chi_k=\sum_{j=0}^{\nu-1}J_k^j1+J_k^\nu\chi_k$, and inserting it into \eqref{defa} expands $a$ into Born terms and a remainder. We subtract $\nu=2$ Born terms for $Y=L^2_yL^1_x$, and $\nu=3$ Born terms otherwise; the reason for this choice is explained in Section \ref{sectionpotential}. Accordingly, we define for $k\in\C\setminus\R$
\be
\label{defLR}
\begin{gathered}
L_u(k):=-i\sgn(k_2)\sum_{j=0}^{\nu-1}\widehat{\big(u\,J_k^j1\big)}(\zeta(k)),\\
R_u(k):=-\frac{i\sgn(k_2)}{2\pi}\int_{\R^2}u\,(J_k^\nu\chi_k)\,e^{-iX\cdot\zeta(k)}\,dX .
\end{gathered}
\ee
The first term of $L_u$ is the Born approximation \eqref{bornapprox}. The others are given by the following multilinear expressions. For Schwartz functions $f,g,h$ on $\R^2$ and $k\in\C\setminus\R$ we set
\be
\label{defQ}
\begin{gathered}
Q_2(f,g)(k):=-i\sgn(k_2)\,\widehat{\big(f\,(G_k\ast g)\big)}(\zeta(k)),\\
Q_3(f,g,h)(k):=-i\sgn(k_2)\,\widehat{\big(f\,(G_k\ast(g\,(G_k\ast h)))\big)}(\zeta(k)),
\end{gathered}
\ee
so that $L_u=-i\sgn(k_2)\hat u\circ\zeta+Q_2(u,u)$ for $Y=L^2_yL^1_x$ and $L_u=-i\sgn(k_2)\hat u\circ\zeta+Q_2(u,u)+Q_3(u,u,u)$ otherwise. For a function $F\ge0$ on $\C$ and the frequency $\zeta$ of \eqref{defzetaintro}, we define the positive multilinear operators
\be
\label{defT}
\begin{gathered}
T_2(F_1,F_2)(k):=\int_\C\frac{F_1(l)\,F_2(j)}{|k-l|}\,dA(l),\\
T_3(F_1,F_2,F_3)(k):=\int_\C\int_\C\frac{F_1(l_1)F_2(l_2)F_3(l_3)}{|k-l_1|\,|k-l_3|}\,dA(l_1)\,dA(l_3),
\end{gathered}
\ee
where $j$ and $l_2$ are defined by $\zeta(j)=\zeta(k)-\zeta(l)$ and $\zeta(l_2)=\zeta(k)-\zeta(l_1)-\zeta(l_3)$.
Since $\zeta$ is a bijection from $\C\setminus\R$ onto $\{\xi\neq0\}$, the points $j$ and $l_2$ are well defined, except for $l$, respectively $(l_1,l_3)$, in a null set; explicitly, $\Im j=\Im k-\Im l$ and $\Re j\,\Im j=\Re k\,\Im k-\Re l\,\Im l$. For a function $f$ on $\R^2$ we write $F_f:=|\hat f\circ\zeta|$.

\begin{lemma}[The Born expansion]
\label{lemmaborn}
Let $u\in C^\infty_c(\R^2)$ be real, with scattering coefficient $a$ and eigenfunction $\chi$.
\begin{enumerate}
\item For every $k\in\C\setminus\R$, $\chi_k$ is bounded, $\chi_k=1+J_k\chi_k$, and $a(k)=L_u(k)+R_u(k)$.
\item For Schwartz functions $f,g,h$ and $k\in\C\setminus\R$, the Fourier integrals defining $Q_2(f,g)(k)$ and $Q_3(f,g,h)(k)$ converge absolutely, and
\be
\label{bornbounds}
\begin{gathered}
|Q_2(f,g)(k)|\le\frac1\pi\,T_2(F_g,F_f)(k),\\ |Q_3(f,g,h)(k)|\le\frac1{\pi^2}\,T_3(F_f,F_g,F_h)(k).
\end{gathered}
\ee
\end{enumerate}
\end{lemma}

The heat symbol $D(k,z):=\xi^2+i\eta-2ik\xi$, $z=(\xi,\eta)$, of the operator $\pa_y-\pa_x^2-2k\pa_x$ satisfies
\be
\label{heatsymbol}
D(k,\zeta(l))=4i\,(\Im l)\,(k-l),\qquad D\big(k,\zeta(k)-z\big)=\overline{D(k,z)} .
\ee
The first identity follows from $\zeta(l)=(-2l_2,-4l_1l_2)$ for $l=l_1+il_2$: $D(k,\zeta(l))=4l_2^2-4il_1l_2+4ikl_2=4il_2(k-l)$. For the second one, $D(k,z)=\xi(\xi+2k_2)+i(\eta-2k_1\xi)$, and replacing $(\xi,\eta)$ by $(-2k_2-\xi,-4k_1k_2-\eta)$ changes the real part into $(\xi+2k_2)\xi$ and the imaginary part into $-\eta+2k_1\xi$. The first identity explains the operators \eqref{defT}: after the change of variables $z=\zeta(l)$, whose Jacobian $8|l_2|$ cancels the factor $l_2$ of the symbol, the inverse of the heat symbol becomes the Cauchy kernel $(k-l)^{-1}$. The second identity allows us to write both denominators of the cubic term in terms of the frequencies $\zeta(l_1)$ of $\hat f$ and $\zeta(l_3)$ of $\hat h$, as in \eqref{defT}.

\begin{proof}
{\bf Step 1:} Proof of (i). By Proposition \ref{propdirect} (ii), $\chi$ is bounded on $(\C\setminus\R)\times\R^2$, and $\chi_k=1+J_k\chi_k$ at every regular point. Let $k\in\C\setminus\R$ be irregular. By Lemma \ref{lemmairregular} there are regular points $k_j\to k$, and by Theorem \ref{thmgrin}, $\chi(k_j;X)\to\chi(k;X)$ for every $X$. By Lemma \ref{lemmagreen} (iii), $G_{k_j}\to G_k$ pointwise on $\{y\neq0\}$. For $|(k_j)_1-k_1|\le1$, \eqref{sizegreen} gives $|G_{k_j}(Z)|\lesssim(\max(|z_1+2k_1z_2|-2|z_2|,0)+|z_2|^{\frac12})^{-1}$ for $Z=(z_1,z_2)$, a locally integrable function, and $|G_{k_j}(Z)|\lesssim\jb{k_1}$ for $|Z|\ge1$, as in Step 1 of the proof of Lemma \ref{lemmairregular}; since $u$ has compact support, dominated convergence gives $J_{k_j}\chi_{k_j}\to J_k\chi_k$ pointwise, hence $\chi_k=1+J_k\chi_k$. Iterating, $\chi_k=\sum_{j<\nu}J_k^j1+J_k^\nu\chi_k$, and \eqref{defa}, which holds at every $k\in\C\setminus\R$ by Proposition \ref{propdirect} (ii), gives $a=L_u+R_u$.\\

\noindent{\bf Step 2:} Fourier transform of $G_k\ast g$. Let $g$ be a Schwartz function and $k\in\C\setminus\R$. The function $D(k,\cdot)$ vanishes exactly at $z=0$ and $z=\zeta(k)$. At these points the Jacobian determinant of $(\Re D,\Im D)$ with respect to $(\xi,\eta)$ equals $2\xi+2k_2=\pm2k_2\neq0$, so that $|D(k,z)|\gtrsim_k\min(|z|,|z-\zeta(k)|,1)$ near the zeros, and $\sqrt2|D(k,z)|\ge|\xi(\xi+2k_2)|+|\eta-2k_1\xi|\gtrsim_k|z|$ for $|z|$ large. Hence $\hat g/D(k,\cdot)\in L^1(\R^2)$, and its inverse Fourier transform $\varphi$ is continuous and tends to zero at infinity. On the other hand, by Lemma \ref{lemmagreen} (i), (ii), $G_k\ast g$ is a bounded continuous function which tends to zero at infinity, since $|G_k(Z)|\lesssim K_k(Z)$, where
\be
\label{defKk}
K_k(x,y):=\big(|x+2k_1y|+|y|^{\frac12}+|k_2||y|\big)^{-1},
\ee
tends to zero as $|Z|\to\infty$, and it satisfies $(\pa_y-\pa_x^2-2k\pa_x)(G_k\ast g)=g$. Since also $D(k,\cdot)\hat\varphi=\hat g$, the Fourier transform of $G_k\ast g-\varphi$ is supported in $\{0,\zeta(k)\}$, so that $G_k\ast g-\varphi$ is a linear combination of products of polynomials with $1$ and with $e^{iX\cdot\zeta(k)}$; since it tends to zero at infinity, it vanishes. Therefore $\widehat{G_k\ast g}=\hat g/D(k,\cdot)$. Moreover $\pa^\gamma(G_k\ast g)=G_k\ast\pa^\gamma g$ is bounded for every multi-index $\gamma$.\\

\noindent{\bf Step 3:} Proof of (ii). By Step 2 and $\widehat{f\varphi}=\frac1{2\pi}\hat f\ast\hat\varphi$ for a Schwartz function $f$ and a bounded function $\varphi$ with $\hat\varphi\in L^1$,
$$Q_2(f,g)(k)=-\frac{i\sgn(k_2)}{2\pi}\int_{\R^2}\hat f(\zeta(k)-z)\,\frac{\hat g(z)}{D(k,z)}\,dz .$$
The change of variables $z=\zeta(l)$, with $dz=8|l_2|\,dA(l)$ by \eqref{changevariables}, and \eqref{heatsymbol} give
$$Q_2(f,g)(k)=-\frac{\sgn(k_2)}\pi\int_\C\sgn(\Im l)\,\frac{\hat f(\zeta(k)-\zeta(l))\,\hat g(\zeta(l))}{k-l}\,dA(l),$$
which is the first bound of \eqref{bornbounds}. For the second one, $g\,(G_k\ast h)$ is a Schwartz function by Step 2, and applying the same formula twice,
$$Q_3(f,g,h)(k)=-\frac{i\sgn(k_2)}{(2\pi)^2}\int_{\R^2}\int_{\R^2}\frac{\hat f(\zeta(k)-z)\,\hat g(z-z')\,\hat h(z')}{D(k,z)\,D(k,z')}\,dz'\,dz .$$
We substitute $z_1:=\zeta(k)-z$ and $z_3:=z'$, so that $z-z'=\zeta(k)-z_1-z_3$ and $|D(k,z)|=|D(k,z_1)|$ by \eqref{heatsymbol}, and then $z_i=\zeta(l_i)$, $i=1,3$. Each change of variables produces the factor $\frac{8|\Im l_i|}{4|\Im l_i||k-l_i|}=\frac2{|k-l_i|}$, and we obtain the second bound of \eqref{bornbounds}. The integrals converge absolutely since the Schwartz functions $\hat f$, $\hat g$, $\hat h$ are bounded and rapidly decaying, and $D(k,\cdot)^{-1}$ is locally integrable and bounded at infinity.
\end{proof}

\subsection{Spectral norms}

The frequency $\zeta(k)=(-2k_2,-4k_1k_2)$ exchanges the roles of the variables: the imaginary part $k_2$ is a multiple of the horizontal frequency $\xi$, while, for fixed $k_2$, the real part $k_1$ is a multiple of the vertical frequency $\eta$. Accordingly, the integrability of the data in $x$ governs the integrability of the coefficients in $k_2$, and the integrability in $y$ governs the integrability on the horizontal lines of the spectral plane. For a function $F$ on $\C$ and $s\in\R\setminus\{0\}$ we write $F_s(k_1):=F(k_1+is)$. In the mixed cases we set
\be
\label{defalphas}
\alpha_x:=1-\frac1{p_x},\qquad \alpha_y:=1-\frac1{p_y},\qquad\mbox{so that}\quad \alpha_x+2\alpha_y=1,
\ee
with $(\alpha_x,\alpha_y):=(\frac12,\frac14)$ for $Y=L^{\frac43,2}_y(L^2_x)$; thus $0\le\alpha_x\le\frac12$, and $\alpha_x=0$ exactly for $Y=L^2_yL^1_x$. We define the spectral norm $\|F\|_{\widehat Y}$ by
\be
\label{defspectralnorm}
\|F\|_{\widehat Y}:=\begin{cases}\|F\|_{L^{3,4}(\mu)}&\mbox{if}\ Y=L^{\frac32,4},\\[1mm]
\Big\|\,|s|^{\alpha_y}\,\|F_s\|_{L^{1/\alpha_y}(\R)}\Big\|_{L^{1/\alpha_x}(ds)}&\mbox{if}\ 0<\alpha_x\le\frac12,\\[2mm]
\displaystyle\operatorname*{ess\,sup}_{s\neq0}|s|^{\frac12}\|F_s\|_{L^2(\R)}+\|F\|_{L^2(\mu)}&\mbox{for $Y=L^2_yL^1_x$}.\end{cases}
\ee
We write $\|f\|_{L^2\cap Y}:=\|f\|_{L^2}+\|f\|_Y$.

\begin{proposition}[Spectral estimates]
\label{propspectral}
For $f\in L^2\cap Y$ and nonnegative measurable functions $F_1,F_2,F_3$ on $\C$:
\begin{enumerate}
\item $\|F_f\|_{\widehat Y}\lesssim\|f\|_{L^2\cap Y}$;
\item $\|T_2(F_1,F_2)\|_{\widehat Y}\lesssim\|F_1\|_{\widehat Y}\|F_2\|_{\widehat Y}$;
\item if $Y\neq L^2_yL^1_x$, $\|T_3(F_1,F_2,F_3)\|_{\widehat Y}\lesssim\prod_{i=1}^3\|F_i\|_{\widehat Y}$;
\item for every compact set $K\subset\C\setminus\R$, $\|F_1\|_{L^2(K)}\lesssim_K\|F_1\|_{\widehat Y}$.
\end{enumerate}
\end{proposition}

By Lemma \ref{lemmaborn} and the multilinearity of $Q_2$ and $Q_3$, Proposition \ref{propspectral} shows that the sum $L_u$ of the Born terms depends continuously on $u\in L^2\cap Y$ in the norm $\|\cdot\|_{\widehat Y}$. The scaling of the spectral variable explains the exponents: if $F_\lambda(k):=F(k/\lambda)$, then $T_2(F_\lambda,G_\lambda)=\lambda\,T_2(F,G)(\cdot/\lambda)$ and $T_3(F_\lambda,G_\lambda,H_\lambda)=\lambda^2\,T_3(F,G,H)(\cdot/\lambda)$, while $\|F_\lambda\|_{L^p(\mu)}=\lambda^{\frac3p}\|F\|_{L^p(\mu)}$; hence $T_m$ can map $L^{p_1}(\mu)\times\dots\times L^{p_m}(\mu)$ into $L^q(\mu)$ only if $\frac1q=\sum_i\frac1{p_i}-\frac{m-1}3$, and $p_i=q=3$ is the only diagonal choice. This is why $L^{\frac32,4}$ is the natural isotropic space: by the Hausdorff--Young inequality, it is mapped into $L^{3,4}$.

\begin{proof}
{\bf Step 1:} Proof of (i) and (iv). Since $\zeta$ maps $\mu$ to $\frac18d\xi\,d\eta$, we have $\mu(\{F_f>\lambda\})=\frac18|\{|\hat f|>\lambda\}|$ for $\lambda>0$, hence $\|F_f\|_{L^{p,r}(\mu)}=8^{-\frac1p}\|\hat f\|_{L^{p,r}(\R^2)}$. For $Y=L^{\frac32,4}$, the Lorentz version of the Hausdorff--Young inequality \cite[Theorem 1.4.19 and Proposition 2.2.16]{Grafakos} gives $\|\hat f\|_{L^{3,4}}\lesssim\|f\|_{L^{\frac32,4}}$. For $0<\alpha_x\le\frac12$, the substitutions $\eta=-4k_1s$ and $\xi=-2s$ give
\be
\label{mixedborn}
\|F_f\|_{\widehat Y}=4^{-\alpha_y}2^{-\alpha_x}\Big\|\,\|\hat f(\xi,\cdot)\|_{L^{1/\alpha_y}}\Big\|_{L^{1/\alpha_x}(d\xi)} .
\ee
If $1<p_x<2$, we use the Hausdorff--Young inequality \cite[Proposition 2.2.16]{Grafakos} in $y$ for fixed $\xi$, Minkowski's integral inequality \cite[Theorem 6.19]{Folland}, which applies since $\frac1{\alpha_x}=p_x'>2\ge p_y$, and the Hausdorff--Young inequality in $x$ for fixed $y$:
$$\begin{aligned}\Big\|\,\|\hat f(\xi,\cdot)\|_{L^{p_y'}}\Big\|_{L^{p_x'}(d\xi)}&\lesssim\Big\|\,\|\mathcal F_xf(\xi,\cdot)\|_{L^{p_y}}\Big\|_{L^{p_x'}(d\xi)}\\&\le\Big\|\,\|\mathcal F_xf(\cdot,y)\|_{L^{p_x'}}\Big\|_{L^{p_y}(dy)}\lesssim\|f\|_{L^{p_y}_yL^{p_x}_x},\end{aligned}$$
where $\mathcal F_x$ is the Fourier transform in the variable $x$. If $Y=L^{\frac43,2}_y(L^2_x)$, let $h(\xi,y):=|\mathcal F_xf(\xi,y)|^2$, so that $\int h(\xi,y)\,d\xi=\|f(\cdot,y)\|_{L^2}^2$ by Plancherel's theorem. For $H\ge0$ on $\R$, let $\Lambda(H):=\int_0^\infty s^{\frac12}H^*(s)\,ds$, so that $\Lambda(|g|^2)=\|g\|_{L^{4/3,2}}^2$. Since $\int_t^\infty H^*(s)\,ds=\inf\{\int_{\R\setminus E}H:\ |E|\le t\}$, Fubini's theorem gives
$$\Lambda(H)=\frac12\int_0^\infty t^{-\frac12}\inf_{|E|\le t}\Big(\int_{\R\setminus E}H(y)\,dy\Big)\,dt .$$
An infimum of an integral over $\xi$ is at least the integral of the infima, so that $\Lambda\big(\int h(\xi,\cdot)\,d\xi\big)\ge\int\Lambda(h(\xi,\cdot))\,d\xi$, that is,
$$\int_\R\|\mathcal F_xf(\xi,\cdot)\|_{L^{4/3,2}}^2\,d\xi\le\Big\|\,\|f(\cdot,y)\|_{L^2}\Big\|_{L^{4/3,2}(dy)}^2=\|f\|_Y^2 .$$
The Lorentz Hausdorff--Young inequality in $y$, $\|\hat f(\xi,\cdot)\|_{L^4}\lesssim\|\hat f(\xi,\cdot)\|_{L^{4,2}}\lesssim\|\mathcal F_xf(\xi,\cdot)\|_{L^{4/3,2}}$, and \eqref{mixedborn} give (i). For $Y=L^2_yL^1_x$, $|\mathcal F_xf(\xi,y)|\le(2\pi)^{-\frac12}\|f(\cdot,y)\|_{L^1}$, and Plancherel's theorem in $y$ gives $\|\hat f(\xi,\cdot)\|_{L^2}\le(2\pi)^{-\frac12}\|f\|_{L^2_yL^1_x}$ for every $\xi$; by the substitution $\eta=-4k_1s$, $|s|^{\frac12}\|(F_f)_s\|_{L^2}=\frac12\|\hat f(-2s,\cdot)\|_{L^2}$. Together with $\|F_f\|_{L^2(\mu)}=8^{-\frac12}\|f\|_{L^2}$, this proves (i).

For (iv), let $K\subset\{|k_1|\le C_K,\ C_K^{-1}\le|k_2|\le C_K\}$. If $Y=L^{\frac32,4}$, the rearrangement of $F_1\1_K$ is at most $F_1^*$ on $[0,\mu(K)]$ and vanishes afterwards, so that $\int_KF_1^2\,d\mu\le\|F_1\|_{L^{3,\infty}(\mu)}^2\int_0^{\mu(K)}s^{-\frac23}ds\lesssim_K\|F_1\|_{\widehat Y}^2$. In the mixed cases, $\frac1{\alpha_x}\ge2$ and $\frac1{\alpha_y}\ge2$, and H\"older's inequality on $\{|k_1|\le C_K\}$ and on $\{C_K^{-1}\le|s|\le C_K\}$ gives (iv); for $Y=L^2_yL^1_x$, $\int_K F_1^2\,dA\le\int_{C_K^{-1}\le|s|\le C_K}\|(F_1)_s\|_{L^2}^2\,ds$. Since $d\mu\sim dA$ on $K$, (iv) follows.\\

\noindent{\bf Step 2:} Coordinates. For $k=k_1+ik_2$ and the points $l$, $j$, $l_1$, $l_2$, $l_3$ of \eqref{defT} we use as coordinates their real and imaginary parts, and we write $\kappa_c(z):=(z^2+c^2)^{-\frac12}$ for $c\in\R\setminus\{0\}$, so that $\|\kappa_c\|_{L^{3/2}(\R)}=C|c|^{-\frac13}$, $\|\kappa_c\|_{L^2(\R)}=C|c|^{-\frac12}$ and $\|\kappa_c\|_{L^{1/(1-\alpha_y)}(\R)}=C|c|^{-\alpha_y}$.

{\it The quadratic operator.} Let $l=x_1+is_1$, $j=x_2+is_2$, $k_2=s_1+s_2$. By \eqref{defT}, $k_1=A_1x_1+A_2x_2$ with $A_i:=s_i/k_2$, and $k-l=A_2(x_2-x_1)+is_2$, hence $|k-l|=|A_2|\,\kappa_{k_2}(x_2-x_1)^{-1}$. For fixed $(x_1,s_1)$, the map $(k_1,k_2)\mapsto(x_2,s_2)$ has Jacobian determinant of modulus $|A_2|^{-1}$. Therefore, for every measurable $G\ge0$,
\be
\label{dualT2}
\begin{aligned}\int_\C T_2(F_1,F_2)\,G\,d\mu&=\int_{\R^2}|k_2|\int_{\R^2}(F_1)_{s_1}(x_1)\,(F_2)_{s_2}(x_2)\\&\qquad\times G_{k_2}(A_1x_1+A_2x_2)\\&\qquad\times\kappa_{k_2}(x_2-x_1)\,dx_1\,dx_2\,ds_1\,ds_2,\end{aligned}
\ee
where $k_2=s_1+s_2$ in the integrand. The inner integral involves the four linear forms $x_1$, $x_2$, $A_1x_1+A_2x_2$ and $x_2-x_1$ on $\R^2$. Since $A_1+A_2=1$, all pairs of these forms have determinant of modulus one, except the pair $\{x_1,A_1x_1+A_2x_2\}$, with determinant $A_2$, and the pair $\{x_2,A_1x_1+A_2x_2\}$, with determinant $A_1$.

{\it The cubic operator.} Let $l_i=x_i+is_i$, $i=1,2,3$, and $k_2=s_1+s_2+s_3$. By \eqref{defT}, $k_1=\lambda(x):=\sum_iA_ix_i$ with $A_i:=s_i/k_2$, and $|k-l_1|=\kappa_{s_2+s_3}(\lambda(x)-x_1)^{-1}$, $|k-l_3|=\kappa_{s_1+s_2}(\lambda(x)-x_3)^{-1}$. For fixed $(l_1,l_3)$ the map $k\mapsto l_2$ has Jacobian determinant of modulus $|A_2|^{-1}$, and the measure $d\mu(k)$ contributes the factor $|k_2|$; since $|A_2||k_2|=|s_2|$,
\be
\label{dualT3}
\begin{aligned}\int_\C T_3(F_1,F_2,F_3)\,G\,d\mu&=\int_{\R^3}|s_2|\int_{\R^3}\prod_{i=1}^3(F_i)_{s_i}(x_i)\,G_{k_2}(\lambda(x))\\&\qquad\times\kappa_{s_2+s_3}(\lambda(x)-x_1)\\&\qquad\times\kappa_{s_1+s_2}(\lambda(x)-x_3)\,dx\,ds .\end{aligned}
\ee
The inner integral involves the six linear forms $\lambda_i(x)=x_i$, $i=1,2,3$, $\lambda_4=\lambda$, $\lambda_5=\lambda-x_1$ and $\lambda_6=\lambda-x_3$ on $\R^3$. We have $|\det\{1,2,3\}|=1$ and $|\det\beta_i|=|A_i|$, where $\beta_i$ is obtained from $\{1,2,3\}$ by replacing $i$ by $4$. Subtracting the coefficients of $\lambda_4$ from those of $\lambda_5=\lambda_4-x_1$ and $\lambda_6=\lambda_4-x_3$, we obtain $|\det\{4,5,6\}|=|\det\{(A_1,A_2,A_3),(1,0,0),(0,0,1)\}|=|A_2|$, and similarly $|\det\{i,5,6\}|=|A_2|$ for $i=1,2,3$.\\

\noindent{\bf Step 3:} Proof of (ii) and (iii) for $Y=L^{\frac32,4}$. We first prove the strong bounds
\be
\label{strongT}
\begin{gathered}\|T_m(F_1,\dots,F_m)\|_{L^q(\mu)}\lesssim\prod_{i=1}^m\|F_i\|_{L^{p_i}(\mu)},\\ \frac1q=\sum_{i=1}^m\frac1{p_i}-\frac{m-1}3,\qquad \Big|\frac1{p_i}-\frac13\Big|\le\eps_0,\end{gathered}
\ee
for $m=2,3$ and a small absolute constant $\eps_0>0$, with constants independent of the exponents; this uniformity, required in \eqref{hypinterp}, holds since the constants of \eqref{lorentzholder}, \eqref{oneil} and Lemma \ref{lemmainterp} are bounded for exponents in compact subsets of their admissible ranges. We write $\theta_i:=\frac1{p_i}$, $\theta_G:=1-\frac1q$, and we set $f_i(s):=|s|^{\theta_i}\|(F_i)_s\|_{L^{p_i}(\R)}$ and $g(s):=|s|^{\theta_G}\|G_s\|_{L^{q'}(\R)}$, so that $\|f_i\|_{L^{p_i}(\R)}=\|F_i\|_{L^{p_i}(\mu)}$ and $\|g\|_{L^{q'}(\R)}=\|G\|_{L^{q'}(\mu)}$.

{\it The case $m=2$.} By symmetry of the inner integral in \eqref{dualT2} under the exchange of $(x_1,F_1,s_1)$ and $(x_2,F_2,s_2)$, it suffices to consider the region $|s_1|\ge|s_2|$, where $|A_1|\ge\frac12$. Here $\theta_G=\frac43-\theta_1-\theta_2$, so that the exponents $(\theta_1,\theta_2,\theta_G,\frac23)$ of the four forms sum to $2$. Let $e:=\max(\frac13-\theta_2,0)$. If $\theta_2\le\frac13$, we put the weight $e$ on the pair $\{x_1,A_1x_1+A_2x_2\}$, and the weights $\frac{(\theta_1-e)\theta_2}{1-e}$, $\frac{(\theta_1-e)\frac23}{1-e}$, $\frac{(\theta_G-e)\theta_2}{1-e}$, $\frac{(\theta_G-e)\frac23}{1-e}$ on the pairs formed by one of $x_1$, $A_1x_1+A_2x_2$ and one of $x_2$, $x_2-x_1$; if $\theta_2>\frac13$, we put the weight $\theta_2-\frac13$ on $\{x_2,x_2-x_1\}$ and distribute the remaining marginals in the same way. The weights are nonnegative, sum to one, and realize the marginals $(\theta_1,\theta_2,\theta_G,\frac23)$. Since $|A_1|\ge\frac12$, \eqref{basisholder} bounds the inner integral of \eqref{dualT2} by
$$C|A_2|^{-e}\,\|(F_1)_{s_1}\|_{L^{p_1}}\|(F_2)_{s_2}\|_{L^{p_2}}\|G_{k_2}\|_{L^{q'}}|k_2|^{-\frac13} .$$
After multiplication by the factor $|k_2|$ of \eqref{dualT2}, this equals
$$C|s_1|^{-\theta_1}|s_2|^{-\theta_2-e}|k_2|^{\rho}f_1(s_1)f_2(s_2)g(k_2),$$ where $|A_2|^{-e}=|s_2|^{-e}|k_2|^e$ and $\rho:=1-\frac13+e-\theta_G=e+\frac1q-\frac13$. If $\rho\ge0$ we use $|k_2|\le2|s_1|$. In both cases the right-hand side is bounded by $C|s_1|^{-b_1}|s_2|^{-b_2}|s_1+s_2|^{-b_3}f_1(s_1)f_2(s_2)g(s_1+s_2)$, with $b_1=\theta_1-\rho_+$, $b_2=\theta_2+e$ and $b_3=\rho_-$, where $\rho_\pm:=\max(\pm\rho,0)$; these numbers are nonnegative and satisfy $b_1+b_2+b_3=\frac23$. By \eqref{lorentzholder}, $|s|^{-b_i}f_i\in L^{\tilde p_i,p_i}$ with $\frac1{\tilde p_i}=\theta_i+b_i$, $i=1,2$, and $|s|^{-b_3}g\in L^{\tilde p_3,q'}$ with $\frac1{\tilde p_3}=\theta_G+b_3$, where $\sum_i\frac1{\tilde p_i}=2$ and the $\tilde p_i$ are close to $\frac32$. Since $\frac1{p_1}+\frac1{p_2}+\frac1{q'}=\frac43\ge1$, the convolution inequality \eqref{oneil}, with the pairing in duality, gives
$$\int_\C T_2(F_1,F_2)\,G\,d\mu\lesssim\|f_1\|_{L^{p_1}}\|f_2\|_{L^{p_2}}\|g\|_{L^{q'}},$$
which is \eqref{strongT} for $m=2$.

{\it The case $m=3$.} Write $\theta_i=\frac13+\delta_i$, so that $\theta_G=\frac23-\sum_i\delta_i$, and the exponents $(\theta_1,\theta_2,\theta_3,\theta_G,\frac23,\frac23)$ of the six forms of \eqref{dualT3} sum to $3$. We start from the weights $\frac13$ on $\{1,2,3\}$ and $\frac23$ on $\{4,5,6\}$, which realize the marginals $(\frac13,\frac13,\frac13,\frac23,\frac23,\frac23)$. For each $i$ with $\delta_i>0$ we move the weight $\delta_i$ from $\{4,5,6\}$ to $\{i,5,6\}$, and for each $i$ with $\delta_i<0$ we move the weight $|\delta_i|$ from $\{1,2,3\}$ to $\beta_i$. This realizes the required marginals, and by Step 2 and \eqref{basisholder} the inner integral of \eqref{dualT3} is bounded by
$$C|A_2|^{-\frac23}\prod_{i:\,\delta_i<0}|A_i|^{\delta_i}\,\prod_{i=1}^3\|(F_i)_{s_i}\|_{L^{p_i}}\,\|G_{k_2}\|_{L^{q'}}\,|s_2+s_3|^{-\frac13}|s_1+s_2|^{-\frac13} .$$
Inserting $f_1,f_2,f_3$, $g$ and the factor $|s_2|$ of \eqref{dualT3}, and writing $\sigma:=\sum_i(\delta_i)_+$, with $\delta_\pm:=\max(\pm\delta,0)$ as above, we obtain the multiplier
$$\frac{C\,|k_2|^{\sigma}}{|s_1|^{\frac13+(\delta_1)_+}|s_2|^{(\delta_2)_+}|s_3|^{\frac13+(\delta_3)_+}|s_1+s_2|^{\frac13}|s_2+s_3|^{\frac13}}\,f_1(s_1)f_2(s_2)f_3(s_3)\,g(k_2) .$$
Indeed, since $|A_i|=|s_i|/|k_2|$, the power of $|k_2|$ is $\frac23+\sum_i(\delta_i)_--\theta_G=\sigma$, the power of $|s_i|$, $i=1,3$, is $-(\delta_i)_--\theta_i=-\frac13-(\delta_i)_+$, and the power of $|s_2|$ is $1-\frac23-(\delta_2)_--\theta_2=-(\delta_2)_+$.
Since $|k_2|^\sigma\le|s_1|^\sigma+|s_2+s_3|^\sigma$, the integral over $s\in\R^3$ is bounded by two integrals of the form
\be
\label{sixforms}
\begin{gathered}\int_{\R^3}\prod_{i=1}^6h_i(\mu_i(s))\,ds,\\ (\mu_1,\dots,\mu_6)(s):=(s_1,s_2,s_3,s_1+s_2+s_3,s_1+s_2,s_2+s_3),\end{gathered}
\ee
where $h_i=|\cdot|^{-b_i}f_i$ for $i=1,2,3$, with $b_i\ge0$ close to $(\frac13,0,\frac13)$, $h_4=g$, and $h_5,h_6$ are powers $|\cdot|^{-b}$ with $b$ close to $\frac13$. By \eqref{lorentzholder}, $h_i\in L^{1/\vartheta_i,p_i}$ for $i\le3$, $h_4\in L^{1/\vartheta_4,q'}$ and $h_5,h_6\in L^{1/\vartheta_i,\infty}$, for numbers $\vartheta_1,\dots,\vartheta_6$ which sum to $3$ and are close to $\vartheta^0:=(\frac23,\frac13,\frac23,\frac23,\frac13,\frac13)$. We claim that there is $\eps_1>0$ such that
\be
\label{sixformsbound}
\int_{\R^3}\prod_{i=1}^6h_i(\mu_i(s))\,ds\le\prod_{i=1}^6\|h_i\|_{L^{1/\vartheta_i}}\qquad\mbox{for}\ \sum_i\vartheta_i=3,\ |\vartheta-\vartheta^0|\le\eps_1 .
\ee
At $\vartheta^0$ we use the weights $\frac13$ on each of the triples $\{1,2,3\}$, $\{1,3,4\}$ and $\{4,5,6\}$, all of determinant of modulus one. Consider the tree with the edges $\{1,4\}$, $\{2,4\}$, $\{3,4\}$, $\{2,5\}$ and $\{2,6\}$. A perturbation of the marginals with zero sum is the divergence of a unique flow on this tree, whose values are bounded by the size of the perturbation. A flow of size $\tau$ from $i$ to $i'$ along an edge is realized by moving the weight $\tau$ from $\{1,2,3\}$ to the triple in which $i$ is replaced by $i'$, if $i\in\{1,2,3\}$, and from $\{4,5,6\}$ to the triple in which $i$ is replaced by $i'$, if $i\in\{4,5,6\}$; the ten triples obtained in this way, $\{4,2,3\}$, $\{1,4,3\}$, $\{1,2,4\}$, $\{1,5,3\}$, $\{1,6,3\}$, $\{1,5,6\}$, $\{2,5,6\}$, $\{3,5,6\}$, $\{4,2,6\}$, $\{4,5,2\}$, all have determinant of modulus one. Hence \eqref{basisholder} gives \eqref{sixformsbound}. If $\eps_0$ is small enough, the vector $\vartheta=(\vartheta_1,\dots,\vartheta_6)$ satisfies $|\vartheta-\vartheta^0|\le\frac{\eps_1}2$, so that \eqref{sixformsbound} holds on the ball of radius $\frac{\eps_1}2$ around $\vartheta$; by Lemma \ref{lemmainterp}, \eqref{sixformsbound} then extends to the Lorentz quasi-norms $L^{1/\vartheta_i,r_i}$ whenever $\sum_i\frac1{r_i}\ge1$. With $(r_i)=(p_1,p_2,p_3,q',\infty,\infty)$, we have $\sum_i\frac1{r_i}=\theta_1+\theta_2+\theta_3+\theta_G=\frac53$, and \eqref{lorentzholder} gives \eqref{strongT} for $m=3$.

{\it Lorentz norms.} The forms $\int T_m(F_1,\dots,F_m)G\,d\mu$ satisfy \eqref{hypinterp} near $\theta^0=(\frac13,\frac13,\frac23)$, respectively $(\frac13,\frac13,\frac13,\frac23)$, by \eqref{strongT}. Lemma \ref{lemmainterp} with $(r_i)=(4,4,\frac43)$, respectively $(4,4,4,\frac43)$, and the duality between $L^{3,4}(\mu)$ and $L^{\frac32,\frac43}(\mu)$ recalled in Subsection \ref{subsectionlorentz} prove (ii) and (iii).\\

\noindent{\bf Step 4:} Proof of (ii) and (iii) for $0<\alpha_x\le\frac12$. Let $q_x:=\frac1{\alpha_x}$, $q_y:=\frac1{\alpha_y}$ and $f_i(s):=|s|^{\alpha_y}\|(F_i)_s\|_{L^{q_y}}$, so that $\|f_i\|_{L^{q_x}}=\|F_i\|_{\widehat Y}$. Fix $s=k_2$ and $\psi\ge0$ in $L^{q_y'}(\R)$.

{\it The quadratic operator.} With $l=x_1+is_1$ and $s_2=s-s_1$ we have $|k-l|=\kappa_{s_2}(k_1-x_1)^{-1}$, and the real part of $j$ is $x_2=(sk_1-s_1x_1)/s_2$. Hence
$$\begin{aligned}\int_\R \psi(k_1)\,T_2(F_1,F_2)(k_1+is)\,dk_1&=\int_\R\int_{\R^2}\psi(k_1)(F_1)_{s_1}(x_1)(F_2)_{s_2}(x_2)\\&\qquad\times\kappa_{s_2}(k_1-x_1)\,dk_1\,dx_1\,ds_1 .\end{aligned}$$
The inner integral involves the forms $k_1$, $x_1$, $x_2$ and $k_1-x_1$ of $(k_1,x_1)\in\R^2$. We put the weight $1-\alpha_y$ on the pair $\{k_1,k_1-x_1\}$, of determinant $-1$, and the weight $\alpha_y$ on the pair $\{x_1,x_2\}$, of determinant $-s/s_2$. By \eqref{basisholder}, with the exponents $q_y'$ for $\psi$ and $\kappa_{s_2}$ and $q_y$ for $(F_1)_{s_1}$ and $(F_2)_{s_2}$, the inner integral is at most $C|s_2/s|^{\alpha_y}|s_2|^{-\alpha_y}\|\psi\|_{L^{q_y'}}\|(F_1)_{s_1}\|_{L^{q_y}}\|(F_2)_{s_2}\|_{L^{q_y}}$. By duality,
$$|s|^{\alpha_y}\big\|T_2(F_1,F_2)_s\big\|_{L^{q_y}}\le C\int_\R|s_1|^{-\alpha_y}f_1(s_1)\,|s-s_1|^{-\alpha_y}f_2(s-s_1)\,ds_1 .$$
By \eqref{lorentzholder}, $|s|^{-\alpha_y}f_i\in L^{\tilde p,q_x}$ with $\frac1{\tilde p}=\alpha_x+\alpha_y=\frac{1+\alpha_x}2\in(\frac12,\frac34]$. Since $\frac2{\tilde p}-1=\alpha_x$, the convolution inequality \eqref{oneil} gives a function in $L^{q_x,\frac{q_x}2}\subset L^{q_x}$, which proves (ii).

{\it The cubic operator.} With $s_2=s-s_1-s_3$ we have $|k-l_1|=\kappa_{s_2+s_3}(k_1-x_1)^{-1}$, $|k-l_3|=\kappa_{s_1+s_2}(k_1-x_3)^{-1}$, and $x_2=(sk_1-s_1x_1-s_3x_3)/s_2$. The integral of $\psi\,T_3(F_1,F_2,F_3)_s$ is therefore an integral over $(s_1,s_3)$ of an integral over $(k_1,x_1,x_3)\in\R^3$ involving the forms $k_1$, $x_1$, $x_2$, $x_3$, $k_1-x_1$ and $k_1-x_3$. We put the weight $\alpha_y$ on $\{x_1,x_2,x_3\}$, of determinant of modulus $|s/s_2|$, and the weight $1-\alpha_y$ on $\{k_1,k_1-x_1,k_1-x_3\}$, of determinant one. By \eqref{basisholder},
$$|s|^{\alpha_y}\big\|T_3(F_1,F_2,F_3)_s\big\|_{L^{q_y}}\le C\int_{\R^2}\frac{f_1(s_1)f_2(s_2)f_3(s_3)}{|s_1s_3(s_1+s_2)(s_2+s_3)|^{\alpha_y}}\,ds_1\,ds_3 .$$
Pairing with $g\ge0$ in $L^{q_x'}(\R)$, we obtain the integral \eqref{sixforms} with $h_1=|\cdot|^{-\alpha_y}f_1$, $h_2=f_2$, $h_3=|\cdot|^{-\alpha_y}f_3$, $h_4=g$ and $h_5=h_6=|\cdot|^{-\alpha_y}$, which belong to $L^{1/\vartheta_i,r_i}$ with
$$\vartheta=(\alpha_x+\alpha_y,\ \alpha_x,\ \alpha_x+\alpha_y,\ 1-\alpha_x,\ \alpha_y,\ \alpha_y),\qquad (r_i)=(q_x,q_x,q_x,q_x',\infty,\infty),$$
by \eqref{lorentzholder}. Since $\alpha_x+2\alpha_y=1$, the weights $\alpha_x$ on $\{1,2,3\}$, $\alpha_y$ on $\{1,3,4\}$ and $\alpha_y$ on $\{4,5,6\}$ realize $\vartheta$, and the tree argument of Step 3, which only uses the positivity of the weights of $\{1,2,3\}$ and $\{4,5,6\}$, gives \eqref{sixformsbound} in a neighborhood of $\vartheta$. Since $\sum_i\frac1{r_i}=1+2\alpha_x\ge1$, Lemma \ref{lemmainterp} and duality prove (iii).\\

\noindent{\bf Step 5:} Proof of (ii) for $Y=L^2_yL^1_x$. Let $\mathcal M(F)$ be the essential supremum of $|s|^{\frac12}\|F_s\|_{L^2}$ over $s\neq0$. Fix $b=k_2$, let $\psi\ge0$ belong to $L^2(\R)$, and fix the imaginary parts $s,t$ of $l$ and $j$, with $s+t=b$. With $l=x_1+is$ and $j=x_2+it$, Step 2 gives
$$\begin{aligned}\int_\R \psi(k_1)\,T_2(F_1,F_2)(k_1+ib)\,dk_1&=\int_\R\int_{\R^2}(F_1)_s(x_1)(F_2)_t(x_2)\\&\quad\times\psi\Big(\frac{sx_1+tx_2}b\Big)\,\kappa_b(x_2-x_1)\,dx_1\,dx_2\,ds .\end{aligned}$$
If $|t|\ge|s|$, we apply the Cauchy--Schwarz inequality in $(x_1,x_2)$ to the product of $(F_1)_s(x_1)\psi(\frac{sx_1+tx_2}b)$ and $(F_2)_t(x_2)\kappa_b(x_2-x_1)$. The change of variables $(x_1,x_2)\mapsto(x_1,\frac{sx_1+tx_2}b)$ has Jacobian determinant $\frac tb$, and $(x_1,x_2)\mapsto(x_2,x_2-x_1)$ has Jacobian determinant $1$, so that the inner integral is at most
$$\Big(\frac{|b|}{|t|}\Big)^{\frac12}\|(F_1)_s\|_{L^2}\|\psi\|_{L^2}\|(F_2)_t\|_{L^2}\|\kappa_b\|_{L^2}=C|t|^{-\frac12}\|(F_1)_s\|_{L^2}\|(F_2)_t\|_{L^2}\|\psi\|_{L^2}.$$
If $|s|\ge|t|$, we exchange the roles of $(F_1)_s$ and $(F_2)_t$, and obtain the factor $|s|^{-\frac12}$. For fixed $b$, the bounds $\|(F_1)_s\|_{L^2}\le|s|^{-\frac12}\mathcal M(F_1)$ and $\|(F_2)_{b-s}\|_{L^2}\le|b-s|^{-\frac12}\mathcal M(F_2)$ hold for almost every $s$. Hence
$$\begin{aligned}|b|^{\frac12}\big\|T_2(F_1,F_2)_b\big\|_{L^2}&\le C\,\mathcal M(F_1)\mathcal M(F_2)\,|b|^{\frac12}\int_\R\frac{ds}{|s|^{\frac12}|b-s|^{\frac12}\max(|s|,|b-s|)^{\frac12}}\\&=C'\mathcal M(F_1)\mathcal M(F_2),\end{aligned}$$
since the last integral equals $|b|^{-\frac12}$ times the convergent integral obtained for $b=1$, by scaling. If instead we apply the Cauchy--Schwarz inequality to the product of $(F_1)_s(x_1)(F_2)_t(x_2)$ and $\psi(\frac{sx_1+tx_2}b)\kappa_b(x_2-x_1)$, and use that $(x_1,x_2)\mapsto(\frac{sx_1+tx_2}b,x_2-x_1)$ has Jacobian determinant $\frac sb+\frac tb=1$, we obtain, with $\phi_i(s):=\|(F_i)_s\|_{L^2}$,
$$|b|^{\frac12}\big\|T_2(F_1,F_2)_b\big\|_{L^2}\le C\int_\R\phi_1(s)\,\phi_2(b-s)\,ds .$$
Let $\Phi_i(s):=|s|^{\frac12}\phi_i(s)$, so that $\|\Phi_i\|_{L^\infty}=\mathcal M(F_i)$ and $\|\Phi_i\|_{L^2}=\|F_i\|_{L^2(\mu)}$. Then, by H\"older's inequality on $\{|s|>1\}$,
$$\begin{gathered}\int_{\{|s|\le1\}}\phi_i^{\frac43}\le\|\Phi_i\|_{L^\infty}^{\frac43}\int_{\{|s|\le1\}}|s|^{-\frac23}ds,\\ \int_{\{|s|>1\}}\phi_i^{\frac43}\le\Big(\int_{\{|s|>1\}}s^{-2}ds\Big)^{\frac13}\|\Phi_i\|_{L^2}^{\frac43},\end{gathered}$$
so that $\|\phi_i\|_{L^{4/3}}\lesssim\|F_i\|_{\widehat Y}$. Young's inequality $L^{\frac43}\ast L^{\frac43}\subset L^2$ \cite[Theorem 4.2]{LiebLoss} gives $\|T_2(F_1,F_2)\|_{L^2(\mu)}^2=\int|b|\,\|T_2(F_1,F_2)_b\|_{L^2}^2\,db\lesssim\|F_1\|_{\widehat Y}^2\|F_2\|_{\widehat Y}^2$, which completes the proof of (ii).
\end{proof}

\section{Positive potentials}
\label{sectionpotential}

The remainder $R_u$ of \eqref{defLR} is an integral of $u$ against $\nu$ iterated convolutions with $G_k$. In Section \ref{sectionremainder} we bound it, outside a compact subset of $\C\setminus\R$, by the positive functions defined in this section, in which $|G_k|$ is replaced by its majorant $K_k$ of \eqref{defKk}: for the number $\nu$ of Born terms of \eqref{defLR}, that is, $\nu=2$ for $Y=L^2_yL^1_x$ and $\nu=3$ otherwise, for nonnegative functions $f_1,\dots,f_{\nu+1}$ on $\R^2$ and for $k\in\C$,
\be
\label{defE}
\begin{aligned}E_{\nu+1}[f_1,\dots,f_{\nu+1}](k):=\int_{(\R^2)^{\nu+1}}&\prod_{i=1}^{\nu+1}f_i(X_i)\\&\times\prod_{i=1}^{\nu}K_k(X_i-X_{i+1})\,dX ,\end{aligned}
\ee
where $dX:=dX_1\cdots dX_{\nu+1}$.
The function $E_{\nu+1}$ is obtained from the iterated convolutions in $R_u$ by replacing $|G_k|$ with $K_k$ and the eigenfunction with a constant. By Lemma \ref{lemmavekua}, the inverse problem is controlled by the Riesz potential of the coefficient, and the main result of this section is a bound on $\cI E_{\nu+1}$ in terms of $\|f_i\|_Y$ only.

The kernel $K_k$ is the inverse of a parabolic distance, sheared by $k_1$ and damped by $k_2$. For a point $Z=(\sigma,\tau)\in\R^2$ with $\tau\neq0$, which below is the difference of two points of $\R^2$, we set
\be
\label{defcd}
c(Z):=-\frac\sigma{2\tau},\qquad d(Z):=|\tau|^{-\frac12},
\ee
so that $|\sigma+2k_1\tau|=2|\tau||k_1-c(Z)|$ and $|\tau|^{\frac12}=|\tau|d(Z)$. Since $|k-c(Z)|\le|k_1-c(Z)|+|k_2|$, we obtain
\be
\label{Kspectral}
K_k(Z)\le\frac1{|\tau|\,\big(|k-c(Z)|+d(Z)\big)} .
\ee
Thus, as a function of $k$, $K_k(Z)$ is bounded by a Cauchy kernel centered at the real point $c(Z)$ and regularized at the scale $d(Z)$. The value $\cI E_{\nu+1}(k^0)$ is therefore controlled by integrals over $\C$ of products of $\nu$ such kernels and of the Cauchy kernel $|k-k^0|^{-1}$. If the centers are far from each other, such an integral is bounded by a product of negative powers of the distances of the centers to $k^0$, of total degree $1-\nu$. It is larger if two of the centers are close to each other and far from $k^0$, and this happens only if some of the horizontal variables lie in short intervals. The number of factors is dictated by the conversion of these spectral bounds into physical kernels: with $\nu$ factors, each factor gains $\frac{\nu-1}\nu$ powers of the horizontal variable $|\sigma|$ (see \eqref{conversion} below), and $\nu=3$ is the smallest number of factors for which this is enough for $\alpha_x\le\frac12$. For $Y=L^2_yL^1_x$ no horizontal integrability is needed, and $\nu=2$ suffices.

\subsection{The spectral integral}

For $c_1,c_2,c_3\in\R$ and $d_1,d_2,d_3>0$ we set
\be
\label{defJ}
J(c,d):=\int_\C\frac{dA(k)}{|k|\prod_{i=1}^3\big(|k-c_i|+d_i\big)} .
\ee
We put $c_0:=0$, $d_0:=0$, and for $0\le i,j\le3$
\be
\label{defdelta}
\delta_{ij}:=|c_i-c_j|+d_i+d_j,\qquad A_i:=\delta_{0i}=|c_i|+d_i .
\ee
The numbers $\delta_{ij}$, $i\neq j$, satisfy the triangle inequality $\delta_{ij}\le\delta_{il}+\delta_{lj}$, and $\delta_{ij}\ge d_i$.

\begin{lemma}[The spectral integral]
\label{lemmaJ}
For $i\in\{1,2,3\}$, let $\{j,l\}=\{1,2,3\}\setminus\{i\}$, $D_i:=\max(\delta_{ij},\delta_{il})$ and $D'_i:=\min(\delta_{ij},\delta_{il})$.
\begin{enumerate}
\item The following bound holds, where the sum is over the ordered pairs of distinct indices in $\{1,2,3\}$:
\be
\label{boundJ}
\begin{aligned}
J(c,d)&\lesssim(A_1A_2A_3)^{-\frac23}\Big(1+\sum_{i\neq j}\1_{\{\delta_{ij}<A_i/4\}}\log\frac{A_i}{\delta_{ij}}\Big)\\
&\quad+\sum_{i=1}^3\1_{\{D_i<A_i/4\}}\frac{1}{A_iD_i}\Big(1+\log\frac{D_i}{D'_i}\Big).
\end{aligned}
\ee
\item If $\delta_{ij}<\frac{A_i}4$, then $\frac34A_i\le A_j\le\frac54A_i$, $d_i,d_j<\frac{A_i}4$, and $|c_i|,|c_j|\ge\frac{A_i}2$.
\end{enumerate}
\end{lemma}

The first line of \eqref{boundJ} is the bound for separated centers, with a logarithmic correction when two centers are close to each other and far from the origin; the second line corresponds to the configurations in which all the numbers $\delta_{ab}$, $a\neq b$, are at most $2D_i<\frac{A_i}2$.

\begin{proof}
{\bf Step 1:} Reduction to radial integrals. For $k\in\C$ let $\rho_0:=|k|$ and $\rho_i:=|k-c_i|+d_i$, $i=1,2,3$, and let $\Omega_a:=\{k:\ \rho_a\le\rho_b\ \mbox{for all}\ b\}$, $0\le a\le3$, so that $\C=\bigcup_a\Omega_a$. For $b\neq a$ we have $\delta_{ab}\le\rho_a+\rho_b$, by the triangle inequality, and $\rho_b\le\rho_a+\delta_{ab}$. On $\Omega_a$ this gives $\rho_b\ge\max(\rho_a,\frac{\delta_{ab}}2)\ge\frac14(\rho_a+\delta_{ab})$. Hence
$$\begin{aligned}\int_{\Omega_a}\frac{dA(k)}{\rho_0\rho_1\rho_2\rho_3}&\le4^3\int_\C\frac{dA(k)}{\rho_a\prod_{b\neq a}(\rho_a+\delta_{ab})}\\&=4^3\cdot2\pi\int_0^\infty\frac{r\,dr}{(r+d_a)\prod_{b\neq a}(r+d_a+\delta_{ab})},\end{aligned}$$
by polar coordinates centered at $c_a$. Let $u_1\le u_2\le u_3$ be the numbers $\delta_{ab}$, $b\neq a$, in nondecreasing order. Since $\delta_{ab}\ge d_a$ and $\frac r{r+d_a}\le1$, the last integral is at most
$$\int_0^{u_1}\frac{dr}{u_1u_2u_3}+\int_{u_1}^{u_2}\frac{dr}{ru_2u_3}+\int_{u_2}^{u_3}\frac{dr}{r^2u_3}+\int_{u_3}^\infty\frac{dr}{r^3}\le\frac{3+\log\frac{u_2}{u_1}}{u_2u_3} .$$
Therefore
\be
\label{Jsum}
J(c,d)\lesssim\sum_{a=0}^3\mathcal J_a,\qquad \mathcal J_a:=\frac{1+\log(u_2/u_1)}{u_2u_3},
\ee
where $u_1\le u_2\le u_3$ are the numbers $\delta_{ab}$, $b\neq a$, in nondecreasing order.

\noindent{\bf Step 2:} The term $a=0$. Here $\{u_1,u_2,u_3\}=\{A_1,A_2,A_3\}$, and
$$(A_1A_2A_3)^{\frac23}\mathcal J_0=\Big(\frac{u_1}{u_2}\Big)^{\frac13}\Big(\frac{u_1}{u_3}\Big)^{\frac13}\Big(1+\log\frac{u_2}{u_1}\Big)\le\sup_{v\ge1}v^{-\frac13}(1+\log v)<\infty .$$

\noindent{\bf Step 3:} The terms $a=i\in\{1,2,3\}$. Here $\{u_1,u_2,u_3\}=\{A_i,\delta_{ij},\delta_{il}\}$. By the triangle inequality, $A_j\le A_i+\delta_{ij}$ and $A_i\le A_j+\delta_{ij}$, and the same holds with $l$. We distinguish three cases.

If $D'_i\ge\frac{A_i}4$, then $A_j\le5\delta_{ij}$ and $A_l\le5\delta_{il}$, so that $A_iA_jA_l\le25u_1u_2u_3$, and the computation of Step 2 gives $\mathcal J_i\lesssim(A_1A_2A_3)^{-\frac23}$.

If $\delta_{ij}<\frac{A_i}4\le\delta_{il}$, then $u_1=\delta_{ij}$, $u_2u_3=A_i\delta_{il}$ and $u_2\le A_i$. Moreover $A_j\le\frac54A_i$ and $A_l\le5\delta_{il}$, so that
$$\begin{aligned}(A_iA_jA_l)^{\frac23}\mathcal J_i&\le C\frac{A_i^{\frac43}\delta_{il}^{\frac23}}{A_i\delta_{il}}\Big(1+\log\frac{A_i}{\delta_{ij}}\Big)\\&=C\Big(\frac{A_i}{\delta_{il}}\Big)^{\frac13}\Big(1+\log\frac{A_i}{\delta_{ij}}\Big)\le4^{\frac13}C\Big(1+\log\frac{A_i}{\delta_{ij}}\Big).\end{aligned}$$
The case $\delta_{il}<\frac{A_i}4\le\delta_{ij}$ is the same, with $j$ and $l$ exchanged.

If $D_i<\frac{A_i}4$, then $u_3=A_i$, $u_2=D_i$ and $u_1=D'_i$, and $\mathcal J_i$ is the term of the second line of \eqref{boundJ}.

Together with \eqref{Jsum} and Step 2, this proves \eqref{boundJ}. Finally, if $\delta_{ij}<\frac{A_i}4$, the triangle inequality gives $\frac34A_i\le A_j\le\frac54A_i$; moreover $d_i,d_j\le\delta_{ij}<\frac{A_i}4$, hence $|c_i|=A_i-d_i\ge\frac34A_i$ and $|c_j|=A_j-d_j\ge\frac12A_i$.
\end{proof}

The bound \eqref{boundJ} is converted into physical kernels as follows. Let $0\le s_x\le\frac23$ and $s_y:=\frac23-\frac{s_x}2$, so that $s_x+2s_y=\frac43$. For $Z=(\sigma,\tau)$ with $\sigma\tau\neq0$, and $c=c(Z)$, $d=d(Z)$, $A=|c|+d$, we have $|\sigma|^{-s_x}|\tau|^{-s_y}=2^{-s_x}|c|^{-s_x}|\tau|^{-\frac23-\frac{s_x}2}$ and $d^{\frac23-s_x}=|\tau|^{-\frac13+\frac{s_x}2}$, hence
\be
\label{conversion}
\begin{aligned}\frac{1}{|\tau|\,A^{\frac23}}&=2^{s_x}\Big(\frac{|c|}A\Big)^{s_x}\Big(\frac dA\Big)^{\frac23-s_x}|\sigma|^{-s_x}|\tau|^{-s_y}\\&\le2^{s_x}\Big(\frac dA\Big)^{\frac23-s_x}|\sigma|^{-s_x}|\tau|^{-s_y} .\end{aligned}
\ee
The factor $(d/A)^{\frac23-s_x}$ is a gain when the regularization scale $d$ is small compared to the distance $A$ of the center to the origin, which is the case on the sets $\{D_i<\frac{A_i}4\}$ of the second line of \eqref{boundJ}, where $d_1,d_2,d_3\le D_i$.

\subsection{Bounds for the positive potentials}

If the centers $c(Z_i)=-\sigma_i/(2\tau_i)$ and $c(Z_{i'})$ of two differences $Z_i=X_i-X_{i+1}$ and $Z_{i'}$ are close to each other, then, for fixed $\tau_i$ and $c(Z_{i'})$, the horizontal variable $\sigma_i$ belongs to a short interval. The gain is measured by local $L^r$ norms, through the following maximal fractional integral. For $0<s<1$, $1\le r<\infty$ and a function $F$ on $\R$ we set
\be
\label{defAop}
\mathcal A_{s,r}F(x):=\sum_{\lambda\in2^{\mathbb Z}}\lambda^{1-s-\frac1r}\,\|F\|_{L^r([x-8\lambda,x+8\lambda])} .
\ee
Splitting $\R$ into the intervals $\{2\lambda\le|x-z|<4\lambda\}$ and applying H\"older's inequality gives the pointwise bound
\be
\label{Adominates}
\int_\R\frac{|F(z)|}{|x-z|^{s}}\,dz\le4\,\mathcal A_{s,r}F(x).
\ee

\begin{lemma}[A maximal fractional integral]
\label{lemmaA}
\begin{enumerate}
\item Let $0<s<1$ and $1\le r<p<\frac1{1-s}$, and let $\frac1q:=\frac1p-(1-s)$. Then $\|\mathcal A_{s,r}F\|_{L^q}\lesssim\|F\|_{L^p}$.
\item $\|\mathcal A_{\frac23,2}F\|_{L^{6,\infty}}\lesssim\|F\|_{L^2}$.
\end{enumerate}
\end{lemma}

\begin{proof}
Let $\mathcal M_{\rm HL}$ be the Hardy--Littlewood maximal operator \cite[Chapter I]{Stein}. Let $\rho>0$. By the definition of $\mathcal M_{\rm HL}$ and by H\"older's inequality, the terms of \eqref{defAop} satisfy
$$\lambda^{1-s-\frac1r}\|F\|_{L^r([x-8\lambda,x+8\lambda])}\le\begin{cases}\lambda^{1-s-\frac1r}(16\lambda)^{\frac1r}\big(\mathcal M_{\rm HL}|F|^r(x)\big)^{\frac1r}&\mbox{if}\ \lambda\le\rho,\\ \lambda^{1-s-\frac1r}(16\lambda)^{\frac1r-\frac1p}\|F\|_{L^p}&\mbox{if}\ \lambda>\rho.\end{cases}$$ Since $1-s>0$ and $1-s-\frac1p<0$, summing the geometric series gives
$$\mathcal A_{s,r}F(x)\lesssim\rho^{1-s}\big(\mathcal M_{\rm HL}|F|^r(x)\big)^{\frac1r}+\rho^{1-s-\frac1p}\|F\|_{L^p} .$$
The choice $\rho^{\frac1p}=\|F\|_{L^p}(\mathcal M_{\rm HL}|F|^r(x))^{-\frac1r}$ gives, since $p(1-s)=1-\frac pq$,
$$\mathcal A_{s,r}F(x)\lesssim\|F\|_{L^p}^{1-\frac pq}\big(\mathcal M_{\rm HL}|F|^r(x)\big)^{\frac p{qr}} .$$ If $r<p$, the maximal operator is bounded on $L^{p/r}$, and $\int(\mathcal M_{\rm HL}|F|^r)^{\frac pr}\lesssim\|F\|_{L^p}^p$ gives the first claim. If $r=p=2$ and $s=\frac23$, then $q=6$, and the weak type $(1,1)$ of $\mathcal M_{\rm HL}$ gives $|\{\mathcal A_{\frac23,2}F>t\}|\le|\{\mathcal M_{\rm HL}|F|^2>c\,t^6\|F\|_{L^2}^{-4}\}|\lesssim t^{-6}\|F\|_{L^2}^6$.
\end{proof}

\begin{proposition}[Positive potentials]
\label{proppotential}
Let $\nu=2$ for $Y=L^2_yL^1_x$ and $\nu=3$ otherwise, and let $f_1,\dots,f_{\nu+1}\ge0$ belong to $L^2\cap Y$.
\begin{enumerate}
\item $\|\cI E_{\nu+1}[f_1,\dots,f_{\nu+1}]\|_\infty\lesssim\prod_{i=1}^{\nu+1}\|f_i\|_Y$.
\item For every compact set $K\subset\C\setminus\R$, $\sup_{k\in K}E_{\nu+1}[f_1,\dots,f_{\nu+1}](k)\lesssim_K\prod_{i=1}^{\nu+1}\|f_i\|_Y$.
\end{enumerate}
\end{proposition}

Since $E_{\nu+1}$ is multilinear and monotone in each argument, Proposition \ref{proppotential} also bounds the differences $E_{\nu+1}[f,\dots,f]-E_{\nu+1}[g,\dots,g]$ for $0\le g\le f$, and the potentials of sums of nonnegative functions.

\begin{proof}
We write $X_i=(x_i,y_i)$ and $Z_i=X_i-X_{i+1}=(\sigma_i,\tau_i)$, and we use the rearrangement inequality $\int fg\le\int f^\sharp g^\sharp$ for nonnegative functions vanishing at infinity, where $f^\sharp$ is the symmetric decreasing rearrangement \cite[Theorem 3.4]{LiebLoss}.\\

\noindent{\bf Step 1:} The case $Y=L^2_yL^1_x$. Let $\nu=2$ and $g_i(y):=\int_\R f_i(x,y)\,dx$, so that $\|g_i\|_{L^2}=\|f_i\|_Y$. Let $k^0\in\C$. By \eqref{Kspectral} and the Cauchy--Schwarz inequality for the measure $|k-k^0|^{-1}dA(k)$,
$$\begin{aligned}\int_\C\frac{K_k(Z_1)K_k(Z_2)}{|k-k^0|}\,dA(k)&\le\frac1{|\tau_1\tau_2|}\prod_{i=1}^2\Big(\int_\C\frac{dA(k)}{(|k-c_i|+d_i)^2|k-k^0|}\Big)^{\frac12}\\&\le\frac{2\pi}{|\tau_1\tau_2|\,(d_1d_2)^{\frac12}},\end{aligned}$$
with $c_i:=c(Z_i)$ and $d_i:=d(Z_i)$,
since, by the rearrangement inequality, $\int_\C(|k-c|+d)^{-2}|k-k^0|^{-1}dA(k)\le\int_\C(|k|+d)^{-2}|k|^{-1}dA(k)=\frac{2\pi}d$. The right-hand side equals $2\pi|\tau_1|^{-\frac34}|\tau_2|^{-\frac34}$. By Fubini's theorem,
$$\begin{aligned}\pi\,\cI E_3[f_1,f_2,f_3](k^0)&\le2\pi\int_{\R^3}g_1(y_1)g_2(y_2)g_3(y_3)\,|y_1-y_2|^{-\frac34}|y_2-y_3|^{-\frac34}\,dy\\&\le C\|g_2\|_{L^2}\prod_{i\in\{1,3\}}\big\||\cdot|^{-\frac34}\ast g_i\big\|_{L^4},\end{aligned}$$
and the Hardy--Littlewood--Sobolev inequality $\||\cdot|^{-\frac34}\ast g\|_{L^4}\lesssim\|g\|_{L^2}$ \cite[Theorem 4.3]{LiebLoss} gives (i). For (ii), \eqref{defKk} shows that $K_k(Z)\le h_k(\tau):=\min(|\tau|^{-\frac12},(|k_2||\tau|)^{-1})$, and $\|h_k\|_{L^{4/3}}=C|k_2|^{-\frac12}$. Young's inequality $L^2\ast L^{\frac43}\subset L^4$ \cite[Theorem 4.2]{LiebLoss} gives
$$E_3[f_1,f_2,f_3](k)\le\int_\R g_2\,(h_k\ast g_1)(h_k\ast g_3)\,dy\le C|k_2|^{-1}\prod_{i=1}^3\|g_i\|_{L^2} .$$

\noindent{\bf Step 2:} Reduction to real $k^0$. Let $\nu=3$ and $k^0\in\C$. By \eqref{Kspectral},
$$\int_\C\frac{\prod_{i=1}^3K_k(Z_i)}{|k-k^0|}\,dA(k)\le\prod_{i=1}^3\frac1{|\tau_i|}\int_\C\frac{\Phi(k)}{|k-k^0|}\,dA(k),$$
where $\Phi(k):=\prod_{i=1}^3(|k-c(Z_i)|+d(Z_i))^{-1}$.
For fixed $k_1$, the function $k_2\mapsto\Phi(k_1+ik_2)$ is even and nonincreasing in $|k_2|$, and $k_2\mapsto|k-k^0|^{-1}$ is a translate of the even function $k_2\mapsto((k_1-k^0_1)^2+k_2^2)^{-\frac12}$, which is nonincreasing in $|k_2|$. The rearrangement inequality on $\R$ allows us to replace $k^0$ by $k^0_1$, and the translation $k\mapsto k+k^0_1$ gives
$$\int_\C\frac{\Phi(k)}{|k-k^0|}\,dA(k)\le J\big(c(Z)-k^0_1,d(Z)\big),$$
with $J$ as in \eqref{defJ}, $c(Z):=(c(Z_i))_i$ and $d(Z):=(d(Z_i))_i$. For $a\in\R$ let $\Sigma_a(x,y):=(x+2ay,y)$. Then $c(\Sigma_aZ)=c(Z)-a$ and $d(\Sigma_aZ)=d(Z)$, and the linear change of variables $X_i\mapsto\Sigma_{k^0_1}X_i$, of determinant one, gives
$$\begin{gathered}\pi\,\cI E_4[f_1,\dots,f_4](k^0)\le\Lambda\big(f_1\circ\Sigma_{k^0_1}^{-1},\dots,f_4\circ\Sigma_{k^0_1}^{-1}\big),\\ \Lambda(h_1,\dots,h_4):=\int_{(\R^2)^4}\prod_{i=1}^4h_i(X_i)\,J(c(Z),d(Z))\prod_{i=1}^3\frac1{|\tau_i|}\,dX .\end{gathered}$$
The map $f\mapsto f\circ\Sigma_a^{-1}$ preserves the distribution function of $f$ and, for fixed $y$, it is a translation in $x$. Hence it preserves $\|\cdot\|_Y$, and it suffices to prove
\be
\label{boundLambda}
\Lambda(f_1,\dots,f_4)\lesssim\prod_{i=1}^4\|f_i\|_Y .
\ee

\noindent{\bf Step 3:} Decomposition of the kernel. Let $c_i:=c(Z_i)$, $d_i:=d(Z_i)$, and let $A_i$, $\delta_{ij}$, $D_i$, $D'_i$ be as in Lemma \ref{lemmaJ}. Let $0\le s_x\le\frac23$ be chosen below, $s_y:=\frac23-\frac{s_x}2$, and
$$\Pi(Z):=\prod_{i=1}^3|\sigma_i|^{-s_x}|\tau_i|^{-s_y} .$$
By \eqref{conversion}, $\prod_i|\tau_i|^{-1}(A_1A_2A_3)^{-\frac23}\lesssim \Pi(Z)$. If $\delta_{ij}<\frac{A_i}4$, let $\ell\ge2$ be such that $2^{-\ell-1}A_i\le\delta_{ij}<2^{-\ell}A_i$; then $\log\frac{A_i}{\delta_{ij}}\le(\ell+1)\log2$. If $D_i<\frac{A_i}4$, let $n\ge2$ and $\ell\ge0$ be such that $2^{-n-1}A_i\le D_i<2^{-n}A_i$ and $2^{-\ell-1}D_i<D'_i\le2^{-\ell}D_i$. Then all the numbers $\delta_{ab}$, $a\neq b\in\{1,2,3\}$, are at most $2D_i$, by the triangle inequality; by Lemma \ref{lemmaJ} (ii), $A_1,A_2,A_3\in[\frac34A_i,\frac54A_i]$, and $d_1,d_2,d_3\le D_i<2^{-n}A_i$. By \eqref{conversion},
$$\prod_{a=1}^3\frac1{|\tau_a|}\,\frac{1}{A_iD_i}\Big(1+\log\frac{D_i}{D'_i}\Big)\lesssim(\ell+1)\,2^{n}\prod_{a=1}^3\frac{1}{|\tau_a|A_a^{\frac23}}\lesssim(\ell+1)\,2^{n(3s_x-1)}\Pi(Z) .$$
Since $D'_i=\min_{b\neq i}\delta_{ib}$, the set $\{D_i<2^{-n}A_i,\ D'_i\le2^{-\ell}D_i\}$ is contained in the union over $b\in\{1,2,3\}\setminus\{i\}$ of the sets $\{D_i<2^{-n}A_i,\ \delta_{ib}<2^{-n-\ell}A_i\}$. Lemma \ref{lemmaJ} therefore gives
\be
\label{kerneldecomp}
\begin{aligned}&J(c(Z),d(Z))\prod_{i=1}^3\frac1{|\tau_i|}\lesssim \Pi(Z)\Big(1+\sum_{i\neq j}\sum_{\ell\ge2}(\ell+1)\1_{\{\delta_{ij}<2^{-\ell}A_i\}}\\&\qquad+\sum_{i=1}^3\sum_{b\neq i}\sum_{n\ge2,\,\ell\ge0}(\ell+1)2^{n(3s_x-1)}\1_{\{D_i<2^{-n}A_i,\ \delta_{ib}<2^{-n-\ell}A_i\}}\Big),\end{aligned}
\ee

\noindent{\bf Step 4:} Restriction of the horizontal variables. Fix $y_1,\dots,y_4$, so that $\tau_1,\tau_2,\tau_3$ are fixed. We first show that the sets $\{\delta_{ij}<2^{-\ell}A_i\}$ of \eqref{kerneldecomp} restrict $\sigma_1$ or $\sigma_3$ to short intervals. Let $e\in\{1,3\}$, $\{i,j\}=\{e,2\}$ and $\ell\ge2$, and fix $x_2,x_3$. Then on $\{\delta_{ij}<2^{-\ell}A_i\}$ the variable $\sigma_e$ belongs to an interval of length at most $2^{3-\ell}\lambda_e$, where $\lambda_e:=|\tau_e|A_2$, contained in $\{\frac14\lambda_e\le|\sigma_e|\le4\lambda_e\}$. Similarly, let $\{i,j\}=\{1,3\}$ and $\ell\ge2$, and fix $x_2,x_3,x_4$. Then on $\{\delta_{ij}<2^{-\ell}A_i\}$ the variable $\sigma_1$ belongs to an interval of length at most $2^{3-\ell}\lambda'_1$, where $\lambda'_1:=|\tau_1|A_3$, contained in $\{\frac14\lambda_1'\le|\sigma_1|\le4\lambda'_1\}$.

Indeed, in the first case, Lemma \ref{lemmaJ} (ii) gives $A_e\in[\frac34A_2,\frac43A_2]$, $A_i\le\frac43A_2$ and $|c_e|,|c_2|\ge\frac25A_2$, so that $|c_e-c_2|\le\delta_{ij}<\frac432^{-\ell}A_2$. Since $\sigma_e=-2\tau_ec_e$, the variable $\sigma_e$ belongs to the interval of center $-2\tau_ec_2$ and length $\frac{16}32^{-\ell}\lambda_e$, and $\frac45\lambda_e\le|\sigma_e|\le\frac83\lambda_e$. The half-length $\frac832^{-\ell}\lambda_e\le\frac23\lambda_e$ of this interval is smaller than $|2\tau_ec_2|\ge\frac45\lambda_e$, hence its intersection with $\{\frac45\lambda_e\le|\sigma_e|\le\frac83\lambda_e\}$ is an interval. The proof of the second case is the same, with $c_3$ in place of $c_2$. On the sets $\{D_i<2^{-n}A_i,\ \delta_{ib}<2^{-n-\ell}A_i\}$ of \eqref{kerneldecomp} we have $\delta_{12},\delta_{23}<2^{1-n}A_i$ and the comparabilities of Step 3, which give $d_2<\frac13A_2$, $\frac25A_2\le|c_e|\le\frac53A_2$ and $|c_e-c_2|<\frac432^{1-n}A_2$ for $e\in\{1,3\}$. Hence, for fixed $x_2,x_3$, the variable $\sigma_e$ belongs to the intersection of an interval $I_e$ of center $-2\tau_ec_2$ and length at most $2^{4-n}\lambda_e$ with $\{\frac45\lambda_e\le|\sigma_e|\le\frac{10}3\lambda_e\}$; since $|2\tau_ec_2|\ge\frac43\lambda_e$ is at least the half-length of $I_e$, this intersection is an interval. Moreover, the condition $\delta_{ib}<2^{-n-\ell}A_i$ restricts $\sigma_1$ or $\sigma_3$ further: by the first case above, with $2^{-\ell}$ replaced by $2^{-n-\ell}$, if $2\in\{i,b\}$, and by the second case, for fixed $x_4$, if $\{i,b\}=\{1,3\}$. For $e\in\{1,3\}$, let $w_e\lambda_e$ be the length of the interval obtained in this way for $\sigma_e$, with $\lambda'_1$ in place of $\lambda_1$ in the second case, and let $w_e:=1$ if $\sigma_e$ is not restricted. Then $w_1w_3\lesssim2^{-\ell}$ on the sets $\{\delta_{ij}<2^{-\ell}A_i\}$, and $w_1w_3\lesssim2^{-2n-\ell}$ on the sets $\{D_i<2^{-n}A_i,\ \delta_{ib}<2^{-n-\ell}A_i\}$.

For an interval $I\subset\{\frac14\lambda\le|\sigma|\le4\lambda\}$ of length at most $w\lambda$, $w\le8$, and $1\le r<\infty$, H\"older's inequality and \eqref{defAop} give, for a function $\varphi\ge0$ on $\R$,
\be
\label{restrictedendpoint}
\begin{aligned}\int_{x+I}\varphi(z)\,|z-x|^{-s_x}\,dz&\le4^{s_x}\lambda^{-s_x}(w\lambda)^{1-\frac1r}\|\varphi\|_{L^r([x-4\lambda,x+4\lambda])}\\&\le C\,w^{1-\frac1r}\mathcal A_{s_x,r}\varphi(x),\end{aligned}
\ee
where we used the term of \eqref{defAop} with the dyadic number $\lambda'\in[\frac\lambda2,\lambda)$.

\noindent{\bf Step 5:} Horizontal integration. Fix $y_1,\dots,y_4$ and let $\varphi_i:=f_i(\cdot,y_i)$. We integrate \eqref{kerneldecomp} against $\prod_i\varphi_i(x_i)$, first in $x_1$ and $x_4$ for fixed $x_2,x_3$, and then in $x_2,x_3$; where $\sigma_1$ is restricted for fixed $x_4$, we integrate first in $x_1$ and then in $x_4$. For the first term of \eqref{kerneldecomp} we use \eqref{Adominates} in $x_1$ and in $x_4$; for the other terms we use \eqref{restrictedendpoint} with the numbers $w_1,w_3$ of Step 4, and \eqref{Adominates} for a variable which is not restricted. We obtain, with $\mathcal A:=\mathcal A_{s_x,r}$,
\be
\label{horizontalchain}
\begin{aligned}&\int_{\R^4}\prod_{i=1}^4\varphi_i(x_i)\,J(c(Z),d(Z))\prod_{i=1}^3\frac1{|\tau_i|}\,dx\\&\qquad\le C\,S_r\prod_{i=1}^3|\tau_i|^{-s_y}\int_{\R^2}\frac{(\varphi_2\,\mathcal A\varphi_1)(x_2)\,(\varphi_3\,\mathcal A\varphi_4)(x_3)}{|x_2-x_3|^{s_x}}\,dx_2\,dx_3,\end{aligned}
\ee
where
$$S_r:=1+\sum_{\ell\ge2}(\ell+1)2^{-\ell(1-\frac1r)}+\sum_{n\ge2,\,\ell\ge0}(\ell+1)\,2^{n(3s_x-1)}\,2^{-(2n+\ell)(1-\frac1r)} .$$
The series $S_r$ converges if $r>1$ and $3s_x-1<2(1-\frac1r)$.

\noindent{\bf Step 6:} Proof of (i) for $0<\alpha_x<\frac12$ and for $Y=L^{\frac32,4}$. Let $0<\alpha_x<\frac12$ first, and choose $s_x:=\frac43\alpha_x$, so that $s_y=\frac43\alpha_y$, and $r\in(1,p_x)$ so close to $p_x$ that $2(1-\frac1r)>4\alpha_x-1=3s_x-1$, which is possible since $2(1-\frac1{p_x})=2\alpha_x>4\alpha_x-1$. Since $\frac1{p_x}=1-\alpha_x>1-s_x$, Lemma \ref{lemmaA} (i) gives $\|\mathcal A\varphi_i\|_{L^{3/\alpha_x}}\lesssim\|\varphi_i\|_{L^{p_x}}$, so that $\varphi_2\mathcal A\varphi_1$ and $\varphi_3\mathcal A\varphi_4$ belong to $L^{a}$ with $\frac1a=1-\frac23\alpha_x$. Since $\frac2a+s_x=2$, the Hardy--Littlewood--Sobolev inequality bounds the integral in \eqref{horizontalchain} by $C\prod_i\|f_i(\cdot,y_i)\|_{L^{p_x}}$. Integrating in $y_1,\dots,y_4$, we obtain
$$\Lambda(f_1,\dots,f_4)\le C\int_{\R^4}\prod_{i=1}^4g_i(y_i)\prod_{i=1}^3|y_i-y_{i+1}|^{-s_y}\,dy,$$
where $g_i(y):=\|f_i(\cdot,y)\|_{L^{p_x}}$.
The same argument in the vertical variable, with the fractional integrals $|\cdot|^{-s_y}\ast g_i\in L^{3/\alpha_y}$ for $i=1,4$, the products in $L^{b}$ with $\frac1b=1-\frac23\alpha_y$, and $\frac2b+s_y=2$, gives \eqref{boundLambda}.

For $Y=L^{\frac32,4}$ we choose $s_x=s_y=\frac49$ and $r=\frac43$, so that $3s_x-1=\frac13<\frac12=2(1-\frac1r)$. Let $\frac1{p_i}$, $i=1,\dots,4$, be close to $\frac23$, with $\sum_i\frac1{p_i}=\frac83$. Since $\frac43<p_i<\frac95=\frac1{1-s_x}$, Lemma \ref{lemmaA} (i) applies, and the products in \eqref{horizontalchain} belong to $L^{a_2}$ and $L^{a_3}$, with $\frac1{a_2}=\frac1{p_1}+\frac1{p_2}-\frac59$ and $\frac1{a_3}=\frac1{p_3}+\frac1{p_4}-\frac59$; since $\frac1{a_2}+\frac1{a_3}+s_x=2$, the Hardy--Littlewood--Sobolev inequality applies. The vertical integration is the same, and we obtain $\Lambda(f_1,\dots,f_4)\lesssim\prod_i\|f_i\|_{L^{p_i}}$ for all such exponents. The form $\Lambda$ is positive and multilinear, and Lemma \ref{lemmainterp} with $r_i=4$, for which $\sum_i\frac1{r_i}=1$, gives \eqref{boundLambda} in $L^{\frac32,4}$.\\

\noindent{\bf Step 7:} Proof of (i) for $Y=L^{\frac43,2}_y(L^2_x)$. Here $s_x=\frac23$, $s_y=\frac13$, $r=2$ and $\mathcal A=\mathcal A_{\frac23,2}$. For the first two terms of \eqref{kerneldecomp}, the argument of Step 5 applies, since $\sum_\ell(\ell+1)2^{-\frac\ell2}<\infty$. For the terms of the second line of \eqref{kerneldecomp}, which carry the factor $2^{n(3s_x-1)}=2^n$, we argue without the decomposition in $n$ and $\ell$. Fix $y_1,\dots,y_4$, $x_2$ and $x_3$, and let $i\in\{1,2,3\}$, $\beta_1:=c_1-c_2$, $\beta_3:=c_3-c_2$ and $B:=|\beta_1|+|\beta_3|$. On $\{D_i<\frac{A_i}4\}$ we have $A_i\ge\frac45A_2$, $B\le4D_i$ and $D'_i\ge\min(|\beta_1|,|\beta_3|,|\beta_1-\beta_3|)$, since $\delta_{12}\ge|\beta_1|$, $\delta_{23}\ge|\beta_3|$ and $\delta_{13}\ge|\beta_1-\beta_3|$. The function $t\mapsto t^{-1}(1+\log\frac t{D'_i})$ is nonincreasing on $[D'_i,\infty)$, and $D_i\ge\max(\frac B4,D'_i)$; hence
$$\begin{gathered}\frac{1}{A_iD_i}\Big(1+\log\frac{D_i}{D'_i}\Big)\le\frac{5}{A_2}\,\mathcal S(\beta_1,\beta_3),\\ \mathcal S(\beta_1,\beta_3):=\frac1B\Big(1+\log_+\frac B{|\beta_1|}+\log_+\frac B{|\beta_3|}+\log_+\frac B{|\beta_1-\beta_3|}\Big).\end{gathered}$$
The kernel $\mathcal S$ defines a bounded operator on $L^2(\R)$. Indeed, $\mathcal S$ is symmetric and homogeneous of degree $-1$, and by the substitution $\beta_3=|\beta_1|v$,
$$\int_\R\mathcal S(\beta_1,\beta_3)|\beta_3|^{-\frac12}\,d\beta_3=|\beta_1|^{-\frac12}\int_\R\frac{\kappa_*(v)}{(1+|v|)\,|v|^{\frac12}}\,dv\le C|\beta_1|^{-\frac12},$$
where $\kappa_*(v):=1+\log(1+|v|)+\log_+\frac{1+|v|}{|v|}+\log_+\frac{1+|v|}{|v\mp1|}$, with the sign of $\beta_1$; the integrand is integrable near $v=0$, near $v=\pm1$ and at infinity. Hence, by the Cauchy--Schwarz inequality and the symmetry of $\mathcal S$, for nonnegative $F,G\in L^2(\R)$,
$$\begin{aligned}\int_{\R^2}\mathcal S\,F(\beta_1)G(\beta_3)\,d\beta&\le\Big(\int_{\R^2}\mathcal S\,\frac{|\beta_1|^{\frac12}}{|\beta_3|^{\frac12}}F(\beta_1)^2\,d\beta\Big)^{\frac12}\Big(\int_{\R^2}\mathcal S\,\frac{|\beta_3|^{\frac12}}{|\beta_1|^{\frac12}}G(\beta_3)^2\,d\beta\Big)^{\frac12}\\&\le C\|F\|_{L^2}\|G\|_{L^2},\end{aligned}$$
with $d\beta:=d\beta_1\,d\beta_3$. Since $\sigma_1=-2\tau_1(c_2+\beta_1)$ and $\sigma_3=-2\tau_3(c_2+\beta_3)$, and since $|\sigma_e|\le4\lambda_e$ with $\lambda_e=|\tau_e|A_2$ on the set $\{D_i<\frac{A_i}4\}$ by Step 4, the change of variables from $(x_1,x_4)$ to $(\beta_1,\beta_3)$ and the Schur bound give
$$\begin{aligned}&\int_{\R^2}\varphi_1(x_1)\varphi_4(x_4)\,\mathcal S(\beta_1,\beta_3)\,\1_{\{D_i<A_i/4\}}\,dx_1\,dx_4\\&\qquad\le C|\tau_1\tau_3|^{\frac12}\,\|\varphi_1\|_{L^2([x_2-4\lambda_1,x_2+4\lambda_1])}\,\|\varphi_4\|_{L^2([x_3-4\lambda_3,x_3+4\lambda_3])} .\end{aligned}$$
Since $|\sigma_2|=2|\tau_2||c_2|\le2|\tau_2|A_2$, the remaining factor satisfies
$$\frac{|\tau_1\tau_3|^{\frac12}}{A_2|\tau_1\tau_2\tau_3|}\le2^{\frac23}\prod_{a=1}^3|\tau_a|^{-\frac13}\,|\sigma_2|^{-\frac23}\,\lambda_1^{-\frac16}\lambda_3^{-\frac16},$$
and $\lambda_e^{-\frac16}\|\varphi\|_{L^2([x-4\lambda_e,x+4\lambda_e])}\lesssim\mathcal A\varphi(x)$. Hence these terms also satisfy \eqref{horizontalchain}, with $S_2$ replaced by a constant. By Lemma \ref{lemmaA} (ii), $\mathcal A\varphi_1,\mathcal A\varphi_4\in L^{6,\infty}$, so that, by the Lorentz H\"older inequality, the products in \eqref{horizontalchain} belong to $L^{\frac32,2}$. The convolution with $|x|^{-\frac23}\in L^{\frac32,\infty}$ maps $L^{\frac32,2}$ into $L^{3,2}$ by \eqref{oneil}, and $L^{3,2}$ is the dual of $L^{\frac32,2}$; hence the integral in \eqref{horizontalchain} is bounded by $C\prod_i\|f_i(\cdot,y_i)\|_{L^2}$. In the vertical variable, the functions $g_i:=\|f_i(\cdot,y)\|_{L^2}$ belong to $L^{\frac43,2}$; the convolution with $|y|^{-\frac13}$ maps $L^{\frac43,2}$ into $L^{12,2}$, the products $g_2\,(|\cdot|^{-\frac13}\ast g_1)$ and $g_3\,(|\cdot|^{-\frac13}\ast g_4)$ belong to $L^{\frac65,1}$, and the convolution with $|y|^{-\frac13}$ maps $L^{\frac65,1}$ into $L^{6,1}\subset(L^{\frac65,1})'$, by \eqref{oneil}. This gives \eqref{boundLambda}.\\

\noindent{\bf Step 8:} Proof of (ii) for $\nu=3$. Let $K\subset\C\setminus\R$ be compact, $k\in K$, and let $(s_x,s_y)$ be as in Steps 6 and 7. Then $s_x+2s_y=\frac43\ge1$ and $s_x+s_y\le1$, the latter since $\alpha_x\le\frac12$ in the mixed cases. With $\tilde\sigma:=\sigma+2k_1\tau$, \eqref{defKk} and $|k_2|\gtrsim_K1$ give, for $|\tau|\le1$,
$$K_k(\sigma,\tau)\le\big(|\tilde\sigma|+|\tau|^{\frac12}\big)^{-1}\le|\tilde\sigma|^{-s_x}|\tau|^{-\frac{1-s_x}2}\le|\tilde\sigma|^{-s_x}|\tau|^{-s_y},$$
and for $|\tau|\ge1$, $K_k(\sigma,\tau)\lesssim_K(|\tilde\sigma|+|\tau|)^{-1}\le|\tilde\sigma|^{-s_x}|\tau|^{s_x-1}\le|\tilde\sigma|^{-s_x}|\tau|^{-s_y}$. After the shear $\Sigma_{k_1}$ of Step 2, $E_4[f_1,\dots,f_4](k)$ is bounded by $C_K$ times the integral of $\prod_if_i(X_i)$ against $\Pi(Z)$, which was bounded in Steps 6 and 7 (first term of \eqref{kerneldecomp}, with \eqref{Adominates}). This proves (ii) for $\nu=3$; for $\nu=2$, (ii) was proved in Step 1.
\end{proof}

\section{The direct remainder}
\label{sectionremainder}

In this section we bound the remainder $R_u$ of \eqref{defLR} by the positive potentials of Section \ref{sectionpotential}, outside a compact subset of $\C\setminus\R$ and uniformly along a sequence of approximating data. For rough data, Lemma \ref{lemmagreenop} (iii) does not apply, since $u$ is in general neither integrable nor in the space $\cK$, and we have no bounds for $H\mapsto G_k\ast(uH)$ on $L^\infty$ which are uniform near the real axis and at high energy. We use instead the symmetrized (Birman--Schwinger) operators on $L^2(\R^2)$, as in Step 1 of the proof of Lemma \ref{lemmairregular}, together with a positivity argument: the kernel $K_k$ of \eqref{defKk} is positive definite, and this allows us to bound the contribution of the eigenfunction by the quartic potential $E_4$.

For a real function $u$ on $\R^2$ we write $f_u:=|u|^{\frac12}$. By Lemma \ref{lemmagreen} (ii), there is an absolute constant $C_G\ge1$ such that $|G_k|\le C_GK_k$ for every $k\in\C$. For $k\in\C$ and a measurable kernel $H\ge0$ on $\R^2$ we consider the operators
\be
\label{defGamma}
\begin{gathered}\Gamma_H[u]h:=f_u\,\big(H\ast(f_uh)\big),\\ \Gamma_k[u]:=\Gamma_{|G_k|}[u],\qquad \Gamma^+_k[u]:=\Gamma_{C_GK_k}[u],\end{gathered}
\ee
so that $|\Gamma_k[u]h|\le\Gamma_k[u]|h|\le\Gamma^+_k[u]|h|$ for every $h$. We also write $\varkappa_{k_1}(x,y):=(|x+2k_1y|+|y|^{\frac12})^{-1}$; by \eqref{sizegreen}, $|G_k|+|G_{k_1}|+C_GK_k\le C\varkappa_{k_1}$ for $k=k_1+ik_2$.

\begin{lemma}[Symmetrized operators]
\label{lemmasymmetrized}
Let $u,v\in Y$ be real.
\begin{enumerate}
\item For every $k_1\in\R$ and every measurable kernel $0\le H\le\varkappa_{k_1}$,
$$\begin{gathered}\|\Gamma_H[u]\|_{L^2\to L^2}\lesssim\|u\|_Y,\\ \|\Gamma_H[u]-\Gamma_H[v]\|_{L^2\to L^2}\lesssim\big(\|u\|_Y+\|v\|_Y\big)^{\frac12}\|u-v\|_Y^{\frac12} .\end{gathered}$$
\item For every $k\in\C$, $\Gamma^+_k[u]$ is self-adjoint and positive semidefinite on $L^2(\R^2)$.
\item Let $u_n\in C^\infty_c(\R^2)$ be real with $u_n\to u_0$ in $Y$. There are a compact set $K\subset\C\setminus\R$, an integer $n_0$ and a constant $C_1$ such that, for $n\ge n_0$ and $k\in(\C\setminus\R)\setminus K$, the series $\sum_{j\ge0}\Gamma_k[u_n]^j$ converges in operator norm to $(I-\Gamma_k[u_n])^{-1}$, and $\|(I-\Gamma_k[u_n])^{-1}\|_{L^2\to L^2}\le C_1$.
\end{enumerate}
\end{lemma}

\begin{proof}
{\bf Step 1:} Proof of (i). Since $f_u\,(H\ast(f_uh))-f_v\,(H\ast(f_vh))=(f_u-f_v)(H\ast(f_uh))+f_v\,(H\ast((f_u-f_v)h))$ and $|f_u-f_v|\le|u-v|^{\frac12}$, both bounds follow from
\be
\label{trilinearGamma}
\big\|g_1\,(\varkappa_{k_1}\ast(g_2h))\big\|_{L^2}\lesssim\big\|g_1^2\big\|_Y^{\frac12}\big\|g_2^2\big\|_Y^{\frac12}\|h\|_{L^2},\qquad g_1,g_2,h\ge0 .
\ee
If $Y=L^{\frac32,4}$, then $g_i\in L^{3,8}\subset L^{3,\infty}$ with $\|g_i\|_{L^{3,8}}^2=\|g_i^2\|_{L^{\frac32,4}}$, and the shear invariance of the Lebesgue measure gives $|\{\varkappa_{k_1}>t\}|=|\{|x|+|y|^{\frac12}<t^{-1}\}|=\frac43t^{-3}$, that is, $\varkappa_{k_1}\in L^{3,\infty}$ uniformly in $k_1$. By the Lorentz H\"older inequality, $g_2h\in L^{\frac65,2}$; by \eqref{oneil}, $\varkappa_{k_1}\ast(g_2h)\in L^{6,2}$; and $g_1\,(\varkappa_{k_1}\ast(g_2h))\in L^{2}$. In the mixed cases with $0<\alpha_x<\frac12$ we use $\varkappa_{k_1}(x,y)\le|x+2k_1y|^{-\alpha_x}|y|^{-\alpha_y}$, which follows from $2\alpha_y=1-\alpha_x$. By H\"older's inequality, $g_2h\in L^{a_y}_yL^{a_x}_x$ with $\frac1{a_x}=1-\frac{\alpha_x}2$ and $\frac1{a_y}=1-\frac{\alpha_y}2$. For fixed $y,y'$, the convolution in $x$ with $|x+2k_1(y-y')|^{-\alpha_x}$ is a translate of the convolution with $|x|^{-\alpha_x}$, which maps $L^{a_x}$ into $L^{b_x}$, $\frac1{b_x}=\frac{\alpha_x}2$, by the Hardy--Littlewood--Sobolev inequality \cite[Theorem 4.3]{LiebLoss}; by Minkowski's integral inequality \cite[Theorem 6.19]{Folland} and the same inequality in $y$, $\varkappa_{k_1}\ast(g_2h)\in L^{b_y}_yL^{b_x}_x$ with $\frac1{b_y}=\frac{\alpha_y}2$. Since $g_1\in L^{2p_y}_yL^{2p_x}_x$ and $\frac1{2p_x}+\frac{\alpha_x}2=\frac1{2p_y}+\frac{\alpha_y}2=\frac12$, H\"older's inequality gives \eqref{trilinearGamma}. For $Y=L^2_yL^1_x$, $\varkappa_{k_1}\le|y|^{-\frac12}$, $g_2h\in L^{\frac43}_yL^1_x$, $\sup_x|\varkappa_{k_1}\ast(g_2h)|\le|\cdot|^{-\frac12}\ast\|g_2h(\cdot,y)\|_{L^1}\in L^4_y$, and $g_1\in L^4_yL^2_x$. For $Y=L^{\frac43,2}_y(L^2_x)$ we use $\varkappa_{k_1}\le|x+2k_1y|^{-\frac12}|y|^{-\frac14}$; the function $y\mapsto\|g_i(\cdot,y)\|_{L^4}=\|g_i^2(\cdot,y)\|_{L^2}^{\frac12}$ belongs to $L^{\frac83,4}$, so that $g_2h\in L^{\frac87,2}_y(L^{\frac43}_x)$ by the Lorentz H\"older inequality in $y$; the convolution in $x$ maps $L^{\frac43}$ into $L^4$, the convolution with $|y|^{-\frac14}\in L^{4,\infty}$ maps $L^{\frac87,2}$ into $L^{8,2}$ by \eqref{oneil}, and $g_1\in L^{\frac83,4}_y(L^4_x)$ gives \eqref{trilinearGamma}.\\

\noindent{\bf Step 2:} Proof of (ii). The kernel of $\Gamma^+_k[u]$ is $C_Gf_u(X)K_k(X-X')f_u(X')$, which is real and symmetric since $K_k$ is even; hence $\Gamma^+_k[u]$ is self-adjoint. For $t>0$, the functions $X\mapsto e^{-t|x+2k_1y|}$, $X\mapsto e^{-t|y|^{\frac12}}$ and $X\mapsto e^{-t|k_2||y|}$ are positive definite on $\R^2$: the first and the third are compositions of the positive definite functions $s\mapsto e^{-c|s|}$ on $\R$ with linear maps, and the second one is positive definite by the subordination formula \cite[Chapter IX, Section 11, Propositions 1--3]{Yosida}
$$e^{-t\sqrt\lambda}=\int_0^\infty e^{-\theta\lambda}\,\frac{t}{2\sqrt\pi}\,\theta^{-\frac32}e^{-\frac{t^2}{4\theta}}\,d\theta,\qquad \lambda\ge0,$$
which writes it as an average of the positive definite functions $e^{-\theta|y|}$ with a nonnegative density. By the Schur product theorem \cite[Theorem 7.5.3(a)]{HornJohnson}, their product $e^{-t/K_k(X)}$ is positive definite, and by Bochner's theorem \cite[Theorem 1.4.3]{Rudin} it is the Fourier transform of a finite nonnegative measure; hence $\int\int\overline{g(X)}e^{-t/K_k(X-X')}g(X')\,dX\,dX'\ge0$ for $g\in L^1(\R^2)$. Integrating over $t\in[T^{-1},T]$ and letting $T\to\infty$ by dominated convergence, with the majorant $|g(X)|K_k(X-X')|g(X')|$, we obtain $\int\int\overline{g(X)}K_k(X-X')g(X')\,dX\,dX'\ge0$ whenever this majorant is integrable. With $g=f_uh$, where $h\in L^2$ is bounded with bounded support on which $f_u$ is bounded, this gives $(\Gamma^+_k[u]h,h)\ge0$, and (i) extends this inequality to every $h\in L^2$.\\

\noindent{\bf Step 3:} Proof of (iii). We write $\Gamma_k:=\Gamma_k[u]$ for a real $u$ with $\|u-u_0\|_Y\le\eta$, where $\eta=\eta_2\le\eta_1$ is chosen below, and all bounds are uniform in such $u$. By (i), $\|\Gamma_H[u]\|\le C_3$ for $0\le H\le C\varkappa_{k_1}$. Positive operators are compared on nonnegative functions: we write $\Gamma\le\Gamma'$ if $\Gamma h\le\Gamma'h$ for every $h\ge0$.

{\it High energy.} Let $g$ be a bounded function with compact support. The proof of Lemma \ref{lemmagreenop} (ii) uses only the bound $|G_k|\lesssim K_k$ given by \eqref{sizegreen}, and the boundedness and the compact support of the function, and therefore shows that $\sup_X\int K_k(X-X')|g(X')|\,dX'\to0$ as $|k|\to\infty$. Choose a real $u_{\rm ref}\in C^\infty_c(\R^2)$ with $\|u_0-u_{\rm ref}\|_Y\le\eta_1$. By the Schur test \cite[Theorem 6.18]{Folland}, whose two integrals coincide since $K_k$ is even, $\|\Gamma^+_k[u_{\rm ref}]\|\le C_G\|f_{u_{\rm ref}}\|_{L^\infty}\sup_X\int K_k(X-X')f_{u_{\rm ref}}(X')\,dX'\to0$. By (i), if $\eta_1$ is small enough, then $\|\Gamma_k\|\le\|\Gamma^+_k[u]\|\le\frac12$ for $|k|\ge R_1$, with $R_1$ depending on $u_{\rm ref}$.

{\it The real axis.} Let $k_1\in\R$. By Lemma \ref{lemmagreen} (i), $G_{k_1}\ge0$ vanishes on $\{y<0\}$, and $|G_{k_1}(x,y)|\le Cy^{-\frac12}$. Let $u_{\rm ref}$ be supported in $[-L,L]^2$. For an interval $I$, the Schur test gives $\|\1_{\{y\in I\}}\Gamma_{k_1}[u_{\rm ref}]\1_{\{y\in I\}}\|\le C\|u_{\rm ref}\|_{L^\infty}L|I|^{\frac12}$. Hence there is a partition of $\R$ into intervals $I_1<I_2<\dots<I_N$, with $I_1=(-\infty,-L]$, $I_N=[L,\infty)$ and $N$ depending only on $u_{\rm ref}$, such that these norms are at most $\frac18$; for $I_1$ and $I_N$ they vanish. By (i), if $\eta_1$ is small enough, the diagonal blocks $\1_{\{y\in I_a\}}\Gamma_{k_1}\1_{\{y\in I_a\}}$ have norms at most $\frac14$, for every $k_1$. By causality, the blocks $\1_{\{y\in I_a\}}\Gamma_{k_1}\1_{\{y\in I_b\}}$ vanish for $a<b$. Thus $\Gamma_{k_1}=\mathcal D+\mathcal N$, where $\mathcal D$ is block diagonal with $\|\mathcal D\|\le\frac14$, and $\mathcal N$ is strictly lower block triangular with $\|\mathcal N\|\le C_3+\frac14$. We expand $\Gamma_{k_1}^j=(\mathcal D+\mathcal N)^j$ into the $2^j$ products of $j$ factors $\mathcal D$ or $\mathcal N$. Every product with at least $N$ factors $\mathcal N$ vanishes, since each factor $\mathcal N$ increases the index of the strip. At most $j^i$ of these products contain exactly $i$ factors $\mathcal N$, hence
$$\|\Gamma_{k_1}^j\|\le\sum_{i=0}^{N-1}j^i4^{i-j}(C_3+1)^i,\qquad C_2:=\sup_{k_1\in\R}\sum_{j\ge0}\|\Gamma_{k_1}^j\|<\infty .$$

{\it Near the real axis.} Let $k=k_1+ik_2$. Since $|G_k|\le|G_{k_1}|+|G_k-G_{k_1}|$, we have $\Gamma_k\le\Gamma_{k_1}+\Delta_k$, where $\Delta_k:=\Gamma_{|G_k-G_{k_1}|}[u]$. Keeping $u_{\rm ref}$, $N$ and $C_2$, which remain valid for $\eta\le\eta_1$, we choose a real $u'_{\rm ref}\in C^\infty_c(\R^2)$ with $\|u_0-u'_{\rm ref}\|_Y\le\eta_2$, where $\eta_2\le\eta_1$ is so small that, by (i), $\|\Delta_k-\Gamma_{|G_k-G_{k_1}|}[u'_{\rm ref}]\|\le(8C_2)^{-1}$. By Lemma \ref{lemmagreen} (iii), $|G_k-G_{k_1}|\le C|k_2|$, and the Schur test gives $\|\Gamma_{|G_k-G_{k_1}|}[u'_{\rm ref}]\|\le C|k_2|\|f_{u'_{\rm ref}}\|_{L^\infty}\|f_{u'_{\rm ref}}\|_{L^1}\le(8C_2)^{-1}$ for $|k_2|\le\eps_1$. Hence $\|\Delta_k\|C_2\le\frac14$ for $|k_2|\le\eps_1$. For positive operators $\Gamma\le\Gamma'$ we have $\|\Gamma\|\le\|\Gamma'\|$, since $|\Gamma h|\le\Gamma|h|\le\Gamma'|h|$. Hence $\|\Gamma_k^j\|\le\|(\Gamma_{k_1}+\Delta_k)^j\|$. We expand $(\Gamma_{k_1}+\Delta_k)^j$ into the products of $j$ factors $\Gamma_{k_1}$ or $\Delta_k$; such a product with $i$ factors $\Delta_k$ has the form $\Gamma_{k_1}^{j_0}\Delta_k\Gamma_{k_1}^{j_1}\cdots\Delta_k\Gamma_{k_1}^{j_i}$ with $j_0+\dots+j_i=j-i$. Summing over $j$, we obtain
$$\sum_{j\ge0}\|\Gamma_k^j\|\le\sum_{i\ge0}\|\Delta_k\|^i\Big(\sum_{j\ge0}\|\Gamma_{k_1}^j\|\Big)^{i+1}\le C_2\sum_{i\ge0}4^{-i}\le2C_2 .$$

{\it Conclusion.} Let $K:=\{|k|\le R_1,\ |k_2|\ge\eps_1\}$, and let $n_0$ be such that $\|u_n-u_0\|_Y\le\eta_2$ for $n\ge n_0$. For $k\in(\C\setminus\R)\setminus K$, we have shown that $\sum_j\|\Gamma_k[u_n]^j\|\le C_1:=2C_2$, either since $\|\Gamma_k[u_n]\|\le\frac12$, or by the last bound. Hence the series $\sum_j\Gamma_k[u_n]^j$ converges in operator norm, its sum is the inverse of $I-\Gamma_k[u_n]$, and its norm is at most $C_1$.
\end{proof}

\begin{proposition}[The remainder outside a compact set]
\label{propremainder}
Let $u_n\in C^\infty_c(\R^2)$ be real with $u_n\to u_0$ in $L^2\cap Y$, and let $K$ and $n_0$ be given by Lemma \ref{lemmasymmetrized} (iii). Then, for $n\ge n_0$ and $k\in(\C\setminus\R)\setminus K$,
\be
\label{exteriorR}
|R_{u_n}(k)|\lesssim E_{\nu+1}[|u_n|,\dots,|u_n|](k),
\ee
where the implicit constant does not depend on $n$ and $k$.
\end{proposition}

\begin{proof}
Fix $n\ge n_0$ and $k\in(\C\setminus\R)\setminus K$, and write $u:=u_n$, $f:=f_u$, $\chi:=\chi_k$, $\Gamma:=\Gamma_k[u]$ and $\Gamma^+:=\Gamma^+_k[u]$. By Lemma \ref{lemmaborn} (i), $\chi$ is bounded and $\chi-1=G_k\ast u+G_k\ast(u(\chi-1))$. Since $|u|=f^2$, the function $v:=f(\chi-1)\in L^2$ satisfies
$$|v|\le f\,\big(|G_k|\ast f^2\big)+f\,\big(|G_k|\ast(f|v|)\big)\le Z+\Gamma|v|,\qquad Z:=\Gamma^+f .$$
Iterating, $|v|\le\sum_{j<J}\Gamma^jZ+\Gamma^J|v|$, and since $\|\Gamma^J\|\to0$ by Lemma \ref{lemmasymmetrized} (iii), we obtain $|v|\le(I-\Gamma)^{-1}Z$. By \eqref{defLR}, $2\pi|R_u(k)|\le\int|u|\,|J_k^\nu1|+\int|u|\,|J_k^\nu(\chi-1)|$. Expanding the convolutions, and since $\Gamma\le\Gamma^+$,
$$\begin{gathered}\int|u|\,|J_k^\nu1|\le(f,\Gamma^\nu f)\le(f,(\Gamma^+)^\nu f)=C_G^{\nu}E_{\nu+1}[|u|,\dots,|u|](k),\\ \int|u|\,|J_k^\nu(\chi-1)|\le(f,\Gamma^\nu|v|) .\end{gathered}$$
For $Y=L^2_yL^1_x$, $\nu=2$, and, since $\Gamma^+$ is self-adjoint,
$$(f,\Gamma^2|v|)\le(\Gamma^+f,\Gamma^+|v|)\le\|\Gamma^+\|\,\|Z\|_{L^2}\|(I-\Gamma)^{-1}Z\|_{L^2}\lesssim\|Z\|_{L^2}^2 .$$
Since $K_k$ is even, $\|Z\|_{L^2}^2=C_G^2\int|u(X)|\big(\int K_k(X-X')|u(X')|\,dX'\big)^2dX=C_G^2E_3[|u|,|u|,|u|](k)$, which proves \eqref{exteriorR}. Otherwise $\nu=3$, and
$$(f,\Gamma^3|v|)\le(\Gamma^+f,(\Gamma^+)^2|v|)=(\Gamma^+Z,\Gamma^+|v|)\le\|\Gamma^+Z\|_{L^2}\,\|\Gamma^+(I-\Gamma)^{-1}Z\|_{L^2} .$$
Since $(I-\Gamma)^{-1}Z=Z+(I-\Gamma)^{-1}\Gamma Z$ and $0\le\Gamma Z\le\Gamma^+Z$, we have $\Gamma^+(I-\Gamma)^{-1}Z\le\Gamma^+Z+\Gamma^+(I-\Gamma)^{-1}\Gamma^+Z$, whose norm is at most $(1+C_3C_1)\|\Gamma^+Z\|_{L^2}$, since $\|\Gamma^+\|\le C_3$, with the constant $C_3$ of Step 3 of the proof of Lemma \ref{lemmasymmetrized}. Since $\Gamma^+$ is positive semidefinite, $\|\Gamma^+Z\|_{L^2}^2\le\|\Gamma^+\|\,(\Gamma^+Z,Z)$, and $(\Gamma^+Z,Z)=(f,(\Gamma^+)^3f)=C_G^3E_4[|u|,\dots,|u|](k)$. This proves \eqref{exteriorR}.
\end{proof}

Proposition \ref{propremainder} is the only place where the positive definiteness of $K_k$ is used. Without it, the term $(f,\Gamma^3|v|)$ would only be bounded by a multiple of $\|\Gamma^+f\|_{L^2}^2$, which is a cubic potential, and the cubic potential is not covered by Proposition \ref{proppotential} unless $Y=L^2_yL^1_x$.

\section{Uniform bounds and compactness of the coefficients}
\label{sectioncompact}

In this section $u_0\in L^2\cap Y$ is real, and we fix real functions $v_j\in C^\infty_c(\R^2)$, $j\ge1$, such that
\be
\label{defvj}
u_n:=\sum_{j=1}^nv_j\longrightarrow u_0\quad\mbox{in}\ L^2\cap Y,\qquad \sum_{j\ge1}\|v_j\|_{L^2\cap Y}<\infty .
\ee
Such functions exist, since $C^\infty_c(\R^2)$ is dense in $L^2\cap Y$ (all the Lorentz indices of $Y$ are finite), and since every sequence converging in $L^2\cap Y$ has a subsequence whose differences are summable; for the Lorentz quasi-norms we use equivalent norms. We write $a_n$ for the scattering coefficient of $u_n$, $L_n:=L_{u_n}$ and $R_n:=R_{u_n}$, so that $a_n=L_n+R_n$ by Lemma \ref{lemmaborn}. The aim of this section is to bound the Riesz potentials $\cI|a_n|$ uniformly in $n$, and to extract a limit coefficient $a$. In contrast with Section \ref{sectiondirect}, the potentials $\cI|a_n|$ are not bounded uniformly in $n$ near the real axis; they are bounded by a fixed function which grows more slowly than any multiple of $\log\frac1{|k_2|}$, and this is enough for the uniform integrability of the exponentials which appear in the bounds of Lemma \ref{lemmavekua}.

We first record a pointwise form of these bounds. For a coefficient $c$ we write $m[c](k;t,X)$ for the solution of \eqref{eqm} with the coefficient $c$.

\begin{lemma}[Pointwise bounds for the eigenfunctions]
\label{lemmapointwise}
Let $c,c'\in L^\infty(\C)$ with $\cI|c|,\cI|c'|\in C_0(\C)$, and fix $(t,X)$. Then, for every $k\in\C$,
\be
\label{pointcomp}
\begin{gathered}e^{-\cI|c|(k)}\le|m[c](k)|\le e^{\cI|c|(k)},\\ |m[c](k)-m[c'](k)|\le e^{\cI|c|(k)}\,\cI\big(|c-c'|\,e^{\cI|c|+\cI|c'|}\big)(k) .\end{gathered}
\ee
\end{lemma}

\begin{proof}
Let $m:=m[c]$, let $\tilde c:=cP\overline m/m$ on $\{m\neq0\}$ and $\tilde c:=0$ on $\{m=0\}$, and $\sigma:=\cC\tilde c$. Then $|\sigma|\le\cI|c|$, so that $\sigma\in C_0(\C)$, and $\dbar\sigma=\tilde c$ by Lemma \ref{lemmacauchy} (ii). By Lemma \ref{lemmavekua} (v) and Step 1 of its proof, $m$ and $\sigma$ belong to $W^{1,p}_{\rm loc}(\C)$ for $p>2$, and $\dbar m=cP\overline m=\tilde cm$ almost everywhere. Hence $\dbar(e^{-\sigma}m)=0$, and $e^{-\sigma}m$ is an entire function by Weyl's lemma; it is bounded and tends to $1$ at infinity, so that $m=e^\sigma$ by Liouville's theorem. This gives the first bound. For the second one, $w:=m[c]-m[c']\in C_0(\C)$ satisfies $w=\cC(cP\overline w)+\cC(f)$ with $f:=(c-c')P\overline{m[c']}$, and the identity \eqref{similarity} gives $w=e^{\sigma'}\cC(e^{-\sigma'}f)$ with $|\sigma'|\le \cI|c|$. Together with $|m[c']|\le e^{\cI|c'|}$, this gives the second bound.
\end{proof}

\subsection{Majorants}

The coefficients $a_n$ are not bounded by a fixed function with a bounded Riesz potential. We construct functions which dominate the data $u_n$, the Born parts $L_n$ and the remainders $R_n$ uniformly in $n$; they serve as integrable majorants in the limiting arguments below. Let $u_\#:=\sum_j|v_j|$ and $F_\#:=\sum_jF_{v_j}$, with $F_f=|\hat f\circ\zeta|$ as in Section \ref{sectionborn}. Then $|u_n|\le u_\#$, $F_{u_n}\le F_\#$, $\|u_\#\|_{L^2\cap Y}<\infty$ and, by Proposition \ref{propspectral} (i), $\|F_\#\|_{\widehat Y}<\infty$. We set
\be
\label{defLsharp}
L_\#:=F_\#+\frac1\pi T_2(F_\#,F_\#)+\frac{\1_{\{\nu=3\}}}{\pi^2}T_3(F_\#,F_\#,F_\#),
\ee
which satisfies $|L_n|\le L_\#$ for every $n$ by Lemma \ref{lemmaborn} (ii) and the monotonicity of $T_2$ and $T_3$, and $\|L_\#\|_{\widehat Y}<\infty$ by Proposition \ref{propspectral}. We also set
\be
\label{defD}
R_\#:=C_\#E_{\nu+1}[u_\#,\dots,u_\#],
\ee
where $C_\#$ is a constant chosen in Proposition \ref{propuniform} below; by Proposition \ref{proppotential}, $\|\cI R_\#\|_\infty<\infty$ and $R_\#$ is bounded on the compact subsets of $\C\setminus\R$. Finally, for $0<r\le2$ and a measurable function $F\ge0$ on $\C$ we let $\cI_rF$ be the Riesz potential of $F$ truncated at the distance $r$, and we set $W:=\cI_1L_\#$; the function $W$ measures the contribution of $L_\#$ to the Riesz potentials of the coefficients near the point of evaluation, and it bounds their growth near the real axis (Proposition \ref{propuniform}):
\be
\label{defWr}
\cI_rF(k):=\frac1\pi\int_{\{|l-k|<r\}}\frac{F(l)}{|k-l|}\,dA(l) .
\ee

\begin{lemma}[Truncated potentials of $L_\#$]
\label{lemmaWr}
\begin{enumerate}
\item For every compact set $K'\subset\C\setminus\R$, $\sup_{K'}\cI_rL_\#\to0$ as $r\to0$.
\item For every $\eps>0$ there is $C_\eps$ such that $\cI_2L_\#(k)\le C_\eps+\eps\log_+\frac1{|k_2|}$ for $k\in\C\setminus\R$.
\item For every $\lambda>0$, $\displaystyle\int_{\{|k_2|\le2\}}\frac{L_\#(k)\,e^{\lambda \cI_2L_\#(k)}}{\jb k}\,dA(k)<\infty$.
\item There are $C,\delta>0$ such that, for all $k\in\C$ and $d\ge1$,
$$\frac1\pi\int_{\{d\le|l-k|<2d,\ |l_2|<d^{-1}\}}\frac{L_\#(l)}{|k-l|}\,dA(l)\le Cd^{-\delta} .$$
\end{enumerate}
\end{lemma}

The logarithmic growth in (ii) is borderline: the potential of a function $F$ with $\|F\|_{\widehat Y}<\infty$ can grow like $\log\frac1{|k_2|}$ for $Y=L^2_yL^1_x$, and the arbitrarily small coefficient $\eps$ comes from the approximation of $L_\#$ by bounded functions with bounded potentials.

\begin{proof}
{\bf Step 1:} The mixed cases. Let $0\le\alpha_x\le\frac12$, $q_x:=\frac1{\alpha_x}$ (with $q_x=\infty$ for $Y=L^2_yL^1_x$), $q_y:=\frac1{\alpha_y}$ and $f(s):=|s|^{\alpha_y}\|(L_\#)_s\|_{L^{q_y}}$, so that $f\in L^{q_x}(\R)$, and $f\in L^\infty(\R)$ with $\|f\|_{L^\infty}\le\|L_\#\|_{\widehat Y}$ for $Y=L^2_yL^1_x$. By H\"older's inequality on the horizontal lines, with $\|\kappa_c\|_{L^{q_y'}}=C|c|^{-\alpha_y}$ as in Step 2 of the proof of Proposition \ref{propspectral},
\be
\label{Wrmixed}
\cI_rL_\#(k)\le C\int_{\{|s-k_2|<r\}}|s|^{-\alpha_y}|s-k_2|^{-\alpha_y}f(s)\,ds .
\ee
Since $\alpha_yq_x'=\frac{1-\alpha_x}{2(1-\alpha_x)}=\frac12$, H\"older's inequality in $s$ gives, for $k$ in a compact set $K'\subset\{C_{K'}^{-1}\le|k_2|\le C_{K'}\}$ and $r\le(2C_{K'})^{-1}$, $\cI_rL_\#(k)\lesssim_{K'}\|f\|_{L^{q_x}}r^{\frac{1-\alpha_x}2}$, which proves (i), and
$$\begin{aligned}\cI_2L_\#(k)&\le C\|f\|_{L^{q_x}}\Big(\int_{\{|s-k_2|<2\}}\frac{ds}{|s|^{\frac12}|s-k_2|^{\frac12}}\Big)^{1-\alpha_x}\\&\le C'\Big(1+\log_+\frac1{|k_2|}\Big)^{1-\alpha_x},\end{aligned}$$
where the last integral was computed with the substitution $s=|k_2|\sigma$ for $|k_2|\le4$. If $\alpha_x>0$, this proves (ii). For $Y=L^2_yL^1_x$, let $F^{(J)}:=\sum_{j\le J}F_{v_j}$ and $L^{(J)}:=F^{(J)}+\frac1\pi T_2(F^{(J)},F^{(J)})$. Since $v_j\in C^\infty_c$, $F^{(J)}\le C_J\jb\zeta^{-6}$, so that $F^{(J)}$ and $\cI F^{(J)}$ are bounded by Lemma \ref{lemmaweights} (ii), and $T_2(F^{(J)},F^{(J)})\le\pi\|F^{(J)}\|_\infty\|\cI F^{(J)}\|_\infty$; hence the truncated potential of $L^{(J)}$ is bounded by a constant $C_J$. By bilinearity, $L_\#-L^{(J)}=F'+\frac1\pi T_2(F',F_\#)+\frac1\pi T_2(F^{(J)},F')$ with $F':=\sum_{j>J}F_{v_j}$, and Proposition \ref{propspectral} gives $\|L_\#-L^{(J)}\|_{\widehat Y}\lesssim\|F'\|_{\widehat Y}(1+\|F_\#\|_{\widehat Y})\to0$ as $J\to\infty$. Applying the bound above to $L_\#-L^{(J)}$, we obtain $\cI_2L_\#\le C_J+C\|L_\#-L^{(J)}\|_{\widehat Y}(1+\log_+\frac1{|k_2|})$, which proves (ii). For (iii), by (ii) with $\eps$ small, it suffices to bound $\int_{\{|k_2|\le2\}}L_\#|k_2|^{-\lambda\eps}\jb k^{-1}\,dA$. H\"older's inequality on the horizontal lines, with $\jb k^{-1}\le\jb{k_1}^{-1}\in L^{q_y'}(\R)$, bounds it by $C\int_{\{|s|\le2\}}|s|^{-\lambda\eps-\alpha_y}f(s)\,ds$, which is finite if $(\lambda\eps+\alpha_y)q_x'<1$, that is, if $\lambda\eps<\alpha_y$. For (iv), the region of integration is contained in $\{|l_1-k_1|<2d,\ |l_2|<d^{-1}\}$, and H\"older's inequality on the horizontal lines and in $s$ gives the bound
$$\begin{aligned}\frac Cd\,d^{1-\alpha_y}\int_{\{|s|<d^{-1}\}}|s|^{-\alpha_y}f(s)\,ds&\le Cd^{-\alpha_y}\|f\|_{L^{q_x}}\Big(\int_{\{|s|<d^{-1}\}}|s|^{-\frac12}ds\Big)^{1-\alpha_x}\\&=C'd^{-2\alpha_y} .\end{aligned}$$

\noindent{\bf Step 2:} The isotropic case. Let $h:=|k_2|$ and $j_k(l):=\1_{\{|k-l|<2\}}(|k-l||l_2|)^{-1}$, so that $\cI_2L_\#(k)=\frac1\pi\int L_\#j_k\,d\mu$, and by the Lorentz H\"older inequality and duality, $\cI_2L_\#(k)\lesssim\|L_\#\|_{L^{3,4}(\mu)}\|j_k\|_{L^{\frac32,\frac43}(\mu)}$. We estimate the distribution function of $j_k$ for $\mu$. On the disc $\{|l-k|<\frac h2\}$, $|l_2|\sim h$, and $\{j_k>\lambda\}$ is contained in a disc of radius $\min(\frac h2,\frac C{h\lambda})$, of measure $\lesssim\min(h^3,h^{-1}\lambda^{-2})$. On the annulus $\{\rho\le|l-k|<2\rho\}$, $\frac h2\le\rho\le2$, we have $|l_2|\le h+2\rho\le4\rho$, and $\{j_k>\lambda\}$ is contained in $\{|l_1-k_1|<2\rho,\ |l_2|<\min(4\rho,(\lambda\rho)^{-1})\}$, of measure $\lesssim\min(\rho^3,\rho^{-1}\lambda^{-2})$. Summing over dyadic $\rho$, we obtain, for $0<h\le1$,
$$\mu(\{j_k>\lambda\})\lesssim\min(1,\lambda^{-\frac32}),\qquad \mu(\{j_k>\lambda\})\lesssim h^{-1}\lambda^{-2}\quad\mbox{for}\ \lambda\ge h^{-2} .$$
Hence the rearrangement of $j_k$ satisfies $j_k^*(\tau)\lesssim\tau^{-\frac23}$, $j_k^*(\tau)\lesssim h^{-\frac12}\tau^{-\frac12}$ for $\tau\le h^3$, and $j_k^*(\tau)=0$ for $\tau\ge C$, and
$$\|j_k\|_{L^{\frac32,\frac43}(\mu)}^{\frac43}\lesssim\int_{h^3}^C\frac{d\tau}\tau+h^{-\frac23}\int_0^{h^3}\tau^{\frac29}\frac{d\tau}\tau\lesssim1+\log\frac1h .$$
For $h\ge1$ the same computation, with $|l_2|\ge\frac h2$ on the disc if $h\ge4$, gives a bound independent of $h$. Thus $\cI_2L_\#(k)\lesssim(1+\log_+\frac1h)^{\frac34}$, which proves (ii). For (i), if $k$ belongs to a compact set $K'$ and $r\le\frac h2$, the computation on the disc of radius $r$ gives $\|\1_{\{|l-k|<r\}}j_k\|_{L^{3/2,4/3}(\mu)}\lesssim_{K'}r^{\frac13}$. For (iii), the function $g:=\1_{\{|k_2|\le2\}}|k_2|^{-1-\lambda\eps}\jb k^{-1}$ satisfies $\int g^p\,d\mu\lesssim\int_{\{|k_2|\le2\}}|k_2|^{1-p(1+\lambda\eps)}dk_2<\infty$ for $1<p<\frac2{1+\lambda\eps}$, since $\int\jb k^{-p}dk_1\lesssim1$. If $\lambda\eps<\frac13$, we can choose $p_0<\frac32<p_1$ in this range, and $g\in L^{p_0}(\mu)\cap L^{p_1}(\mu)\subset L^{\frac32,\frac43}(\mu)$; with (ii), this proves (iii). For (iv), the function $g_d:=d^{-1}|l_2|^{-1}\1_{\{d\le|l-k|<2d,\ |l_2|<d^{-1}\}}$ satisfies $\mu(\{g_d>\lambda\})\lesssim d^{-1}\min(1,\lambda^{-2})$, by integration of the weight $|l_2|$ over the strip; hence $g_d^*(\tau)\lesssim d^{-\frac12}\tau^{-\frac12}$, $g_d^*(\tau)=0$ for $\tau\ge Cd^{-1}$, and $\|g_d\|_{L^{3/2,4/3}(\mu)}\lesssim d^{-\frac23}$. Since the potential in (iv) is at most $\frac1\pi\int L_\#g_d\,d\mu$, this proves (iv).
\end{proof}

\subsection{Uniform bounds}

By Lemma \ref{lemmasymmetrized} (iii) and Proposition \ref{propremainder}, applied to the sequence \eqref{defvj}, there are a compact set $K\subset\C\setminus\R$, an integer $n_0$ and a constant $C_0$ such that $|R_n|\le R^0_\#:=C_0E_{\nu+1}[u_\#,\dots,u_\#]$ on $(\C\setminus\R)\setminus K$ for $n\ge n_0$, since $E_{\nu+1}$ is monotone and $|u_n|\le u_\#$. Discarding finitely many terms, we assume $n_0=1$. The following proposition extends this bound to $K$, by a bootstrap argument: the nonlinear Plancherel identity bounds $a_n$ in $L^2(K)$, which bounds its Riesz potential by the square root of the logarithm of its supremum, while Lemma \ref{lemmapointwise} bounds the supremum of $R_n$ on $K$ by the exponential of the potential.

\begin{proposition}[Uniform bounds]
\label{propuniform}
There is a constant $C_\#\ge C_0$ such that, with $R_\#$ given by \eqref{defD}, for all $n$:
\begin{enumerate}
\item $|R_n|\le R_\#$ and $|a_n|\le L_\#+R_\#$ on $\C\setminus\R$;
\item $\cI|a_n|\le C+W$ on $\C$, for a constant $C$ independent of $n$ and of $C_\#$;
\item for every $\lambda>0$,
\be
\label{expint}
\int_{\{|k_2|\le1\}}\frac{(L_\#+R_\#)(k)\,e^{\lambda W(k)}}{\jb k}\,dA(k)<\infty .
\ee
\end{enumerate}
\end{proposition}

\begin{proof}
{\bf Step 1:} Bounds on $K$. By \eqref{plancherelsmooth}, $\|a_n\|_{L^2(\mu)}=8^{-\frac12}\|u_n\|_{L^2}\lesssim1$. Since $|k_2|$ is bounded below on $K$, $|R_n|\le|a_n|+L_\#$ and $\|L_\#\|_{L^2(K)}\lesssim\|L_\#\|_{\widehat Y}$ by Proposition \ref{propspectral} (iv), we have $\|R_n\|_{L^2(K)}+\|R_n\|_{L^1(K)}\lesssim1$. Moreover, by \eqref{defLR}, $|J_k^\nu\chi_k|\le\|\chi_k\|_{L^\infty}(|G_k|\ast(|u_n|\cdots(|G_k|\ast|u_n|)))$, and $\chi_k=m[a_n](k;0,\cdot)$ by Proposition \ref{propdirect} (ii); hence Lemma \ref{lemmapointwise} and Proposition \ref{proppotential} (ii) give
\be
\label{RonK}
\begin{aligned}|R_n(k)|&\le\frac{C_G^\nu}{2\pi}e^{\cI|a_n|(k)}E_{\nu+1}[|u_n|,\dots,|u_n|](k)\\&\le C_Ke^{\cI|a_n|(k)},\qquad k\in K .\end{aligned}
\ee

\noindent{\bf Step 2:} The potential of $a_n$. Let $k\in\C$. We split $\cI|a_n|(k)$ into the integrals over $\{|l-k|<1\}$ and over the dyadic annuli $\{d\le|l-k|<2d\}$, $d=2^q$, $q\ge0$. On the annuli we use $|a_n|$ on $\{|l_2|\ge d^{-1}\}$, and $|a_n|\le L_\#+|R_n|$ on $\{|l_2|<d^{-1}\}$. The part of the annulus where $|l_2|\ge d^{-1}$ is contained in $E_d:=\{|l_1-k_1|<2d,\ |l_2-k_2|<2d,\ |l_2|\ge d^{-1}\}$, and by the Cauchy--Schwarz inequality for the measure $\mu$,
$$\frac1{\pi d}\int_{E_d}|a_n|\,dA\le\frac{\|a_n\|_{L^2(\mu)}}{\pi d}\Big(\int_{E_d}\frac{dA(l)}{|l_2|}\Big)^{\frac12}\lesssim d^{-\frac12}\big(1+\log d\big)^{\frac12},$$
since the integral of $|l_2|^{-1}$ over an interval of length $4d$ contained in $\{|l_2|\ge d^{-1}\}$ is at most $2\log(4d^2)$. On $\{|l_2|<d^{-1}\}$, Lemma \ref{lemmaWr} (iv) bounds the contribution of $L_\#$ by $Cd^{-\delta}$; the contribution of $R_n$ outside $K$ is at most that of $R^0_\#$, and the sum over all annuli and over the disc of the potential of $R^0_\#$ is at most $\|\cI R^0_\#\|_\infty$; the contribution of $R_n$ on $K$ is at most $(\pi d)^{-1}\|R_n\|_{L^1(K)}$. On the disc $\{|l-k|<1\}$, the contribution of $L_n$ is at most $W(k)$, that of $R_n$ outside $K$ is at most $\|\cI R^0_\#\|_\infty$, and for $0<\rho<1$, the contribution of $R_n\1_K$ is at most
$$\begin{aligned}&\frac1\pi\int_{\{|l-k|<\rho\}}\frac{|R_n|\1_K}{|k-l|}\,dA+\frac1\pi\int_{\{\rho\le|l-k|<1\}}\frac{|R_n|\1_K}{|k-l|}\,dA\\&\qquad\le2\rho\|R_n\|_{L^\infty(K)}+C\|R_n\|_{L^2(K)}\Big(\log\frac1\rho\Big)^{\frac12},\end{aligned}$$
by the Cauchy--Schwarz inequality. With $\rho:=(2+\|R_n\|_{L^\infty(K)})^{-1}$ and Step 1, we obtain, for all $k\in\C$,
\be
\label{bootstrapW}
\cI|a_n|(k)\le C+W(k)+C\Big(\log\big(2+\|R_n\|_{L^\infty(K)}\big)\Big)^{\frac12} .
\ee

\noindent{\bf Step 3:} Closing the bootstrap. Let $\Sigma_n:=\sup_K\cI|a_n|$, which is finite since $a_n$ satisfies \eqref{hypa}. By Lemma \ref{lemmaWr} (ii), $W\le \cI_2L_\#$ is bounded on $K$. By \eqref{bootstrapW} and \eqref{RonK}, $\Sigma_n\le C+C(\log(2+C_Ke^{\Sigma_n}))^{\frac12}\le C'+C\Sigma_n^{\frac12}$, hence $\Sigma_n\le C''$ for all $n$. Then \eqref{RonK} gives $\|R_n\|_{L^\infty(K)}\le C$, and \eqref{bootstrapW} gives (ii). Moreover, the first inequality of \eqref{RonK} with $\cI|a_n|\le C''$ on $K$ gives $|R_n|\le\frac{C_G^\nu}{2\pi}e^{C''}E_{\nu+1}[u_\#,\dots,u_\#]$ on $K$. With $C_\#:=\max(C_0,\frac{C_G^\nu}{2\pi}e^{C''})$, we obtain (i), since $|L_n|\le L_\#$; the constant $C''$ does not depend on $C_\#$.\\

\noindent{\bf Step 4:} Proof of (iii). For $L_\#$, \eqref{expint} follows from Lemma \ref{lemmaWr} (iii), since $W\le \cI_2L_\#$. Let $j\ge1$ and $k\in\C$. For distinct points $l_1,\dots,l_j$ we have the partial fraction identity
$$\prod_{i=1}^j\frac1{k-l_i}=\sum_{i=1}^j\frac1{k-l_i}\prod_{i'\neq i}\frac1{l_i-l_{i'}} .$$
Taking absolute values, integrating against $\prod_i\pi^{-1}L_\#(l_i)\1_{\{|l_i-k|<1\}}$, and using that $|l_{i'}-l_i|<2$ when $l_i,l_{i'}\in D(k,1)$, we obtain
$$W(k)^j\le\frac j\pi\int_{\{|l-k|<1\}}\frac{L_\#(l)\,\cI_2L_\#(l)^{j-1}}{|k-l|}\,dA(l),$$
and summing the exponential series,
$$e^{\lambda W(k)}-1\le\frac\lambda\pi\int_{\{|l-k|<1\}}\frac{L_\#(l)\,e^{\lambda \cI_2L_\#(l)}}{|k-l|}\,dA(l) .$$
If $|k-l|<1$ and $|k_2|\le1$, then $\jb k\ge\frac12\jb l$ and $|l_2|\le2$. Multiplying by $R_\#(k)\jb k^{-1}\1_{\{|k_2|\le1\}}$, integrating in $k$ and using Fubini's theorem, we obtain
$$\int_{\{|k_2|\le1\}}\frac{R_\#\,(e^{\lambda W}-1)}{\jb k}\,dA\le2\lambda\|\cI R_\#\|_\infty\int_{\{|l_2|\le2\}}\frac{L_\#\,e^{\lambda \cI_2L_\#}}{\jb l}\,dA<\infty .$$
Finally $\int R_\#\jb k^{-1}dA\le\int_\C R_\#(k)|k|^{-1}dA(k)=\pi\,\cI R_\#(0)\le\pi\|\cI R_\#\|_\infty$. This proves \eqref{expint}.
\end{proof}

\subsection{Compactness}

\begin{proposition}[A limit coefficient]
\label{proplimit}
After extraction of a subsequence, there is a real measurable function $a$ on $\C\setminus\R$ such that $a_n\to a$ almost everywhere, and
\begin{enumerate}
\item $|a|\le L_\#+R_\#$, $\cI|a|\le C+W$ and $8\|a\|_{L^2(\mu)}^2\le\|u_0\|_{L^2}^2$;
\item $\displaystyle\lim_{n\to\infty}\int_\C\frac{|a_n-a|\,e^{2W}}{\jb k}\,dA=0$.
\end{enumerate}
\end{proposition}

The sums $L_n$ of the Born terms converge in the norm $\|\cdot\|_{\widehat Y}$, but this norm does not control pointwise values. The remainders $R_n$ are more regular: they are equicontinuous on the compact subsets of $\C\setminus\R$. This is proved by writing $R_n$ through the formula \eqref{defLR}, in which the dependence on $k$ enters through the Green function, the eigenfunction and the phase $e^{-iX\cdot\zeta(k)}$.

\begin{proof}
{\bf Step 1:} Equicontinuity of the eigenfunctions. Let $K'\subset\C\setminus\R$ be compact, let $\delta':=\frac12\min(1,\operatorname{dist}(K',\R))$, and let $K''\subset\C\setminus\R$ be the closed $\delta'$-neighbourhood of $K'$. Let $g_n:=a_nP\overline{m[a_n]}$ at $t=0$, so that $m[a_n]-1=\cC g_n$. By Lemma \ref{lemmapointwise} and Proposition \ref{propuniform}, $|g_n|\le G_\#:=(L_\#+R_\#)e^{C+W}$. By Lemma \ref{lemmaWr} (i), (ii) and Proposition \ref{proppotential} (ii), $W$ and $R_\#$ are bounded on $K''$, so that $\sup_{k\in K'}\int_{\{|l-k|<\rho\}}G_\#(l)|k-l|^{-1}dA(l)\to0$ as $\rho\to0$. Moreover, by Proposition \ref{propuniform} (iii) on $\{|l_2|\le1\}$, and by the Cauchy--Schwarz inequality with $\|a_n\|_{L^2(\mu)}\lesssim1$ and $\int_{\{|l_2|>1\}}|l_2|^{-1}\jb l^{-2}dA<\infty$ on $\{|l_2|>1\}$, where $W$ is bounded,
$$\sup_n\int_\C\frac{|g_n(l)|}{\jb l}\,dA(l)<\infty .$$
For $k,k'\in K'$ with $|k-k'|<\frac\rho2$, we split the integral of $|g_n(l)|\,|\frac1{k-l}-\frac1{k'-l}|$ into $\{|l-k|<\rho\}$, where we bound the two kernels separately, and its complement, where the difference of the kernels is at most $\frac{2|k-k'|}{|k-l|^2}\lesssim_{K',\rho}\frac{|k-k'|}{\jb l}$. This shows that $m[a_n](\cdot;0,X)$, that is $\chi_{u_n}(\cdot;X)$, is equicontinuous on $K'$, uniformly in $n$ and $X$; it is also bounded there by $e^{C+\sup_{K'}W}$.

\noindent{\bf Step 2:} Continuity of the Green operators. Let $\|g\|_{L^2+Y'}$ be the norm of the sum of $L^2$ and of the dual space $Y'$ of $Y$, so that $\int|g\,h|\le\|g\|_{L^2+Y'}\|h\|_{L^2\cap Y}$. We claim that
\be
\label{greencontinuity}
\sup_{k\in K'}\|G_k\|_{L^2+Y'}<\infty,\qquad \lim_{\eps\to0}\sup_{k,k'\in K',\ |k-k'|\le\eps}\|G_k-G_{k'}\|_{L^2+Y'}=0 .
\ee
Let $Z=(z_1,z_2)$, let $C_{K'}\ge1$ be such that $|k_1|\le C_{K'}$ and $|k_2|\ge C_{K'}^{-1}$ on $K'$, let $\rho_1\ge1$, $\rho'':=4(1+C_{K'})\rho_1$, and split $\R^2$ into $Z_1:=\{|Z|\le\rho''\}$, $Z_2:=\{|z_2|>\rho_1\}\setminus Z_1$ and $Z_3:=\{|z_2|\le\rho_1\}\setminus Z_1$. On $Z_1$ we use $L^2$: by \eqref{sizegreen}, $\int_{Z_1}|G_k|^2\lesssim\int_{-\rho''}^{\rho''}\int_\R(|s|+|z_2|^{\frac12})^{-2}ds\,dz_2\lesssim\rho''^{\frac12}$, and $|G_k-G_{k'}|\lesssim|k-k'|$ by Lemma \ref{lemmagreen} (iii). On $Z_2\cup Z_3$ we use $Y'$ and the bound $|G_k(Z)|\lesssim_{K'}(|z_1+2k_1z_2|+|z_2|)^{-1}$, which follows from \eqref{sizegreen}; moreover $|z_1+2k_1z_2|\ge\frac{\rho''}2$ on $Z_3$. If $Y=L^{\frac32,4}$, then $Y'=L^{3,\frac43}$, $|G_k(Z)|\lesssim_{K'}|Z|^{-1}$, and the rearrangement of $|Z|^{-1}\1_{\{|Z|>\rho''\}}$ is $(\frac\tau\pi+\rho''^2)^{-\frac12}$, whose $L^{3,\frac43}$ norm is $C\rho''^{-\frac13}$. If $0<\alpha_x<\frac12$, then $Y'=L^{p_y'}_yL^{p_x'}_x$; the $L^{p_x'}_x$ norm of $(|z_1+2k_1z_2|+|z_2|)^{-1}$ is $C|z_2|^{-(1-\alpha_x)}$, whose $L^{p_y'}$ norm on $\{|z_2|>\rho_1\}$ is $C\rho_1^{-\alpha_y}$ since $p_y'(1-\alpha_x)=2$; on $Z_3$ the $L^{p_x'}_x$ norm of $|z_1+2k_1z_2|^{-1}\1_{\{|z_1+2k_1z_2|\ge\rho''/2\}}$ is $C\rho''^{-\frac1{p_x}}$, and its $L^{p_y'}$ norm on $\{|z_2|\le\rho_1\}$ is $C\rho_1^{\alpha_y}\rho''^{-(1-\alpha_x)}\lesssim\rho_1^{-\alpha_y}$. For $Y=L^{\frac43,2}_y(L^2_x)$, with $Y'=L^{4,2}_y(L^2_x)$, the same computation gives the bounds $C\rho_1^{-\frac14}$, since $|z_2|^{-\frac12}\1_{\{|z_2|>\rho_1\}}$ and $\rho''^{-\frac12}\1_{\{|z_2|\le\rho_1\}}$ have $L^{4,2}$ norms $C\rho_1^{-\frac14}$. For $Y=L^2_yL^1_x$, $Y'=L^2_yL^\infty_x$, and the bounds are $C\rho_1^{-\frac12}$. Choosing first $\rho_1$ large and then $|k-k'|$ small proves \eqref{greencontinuity}. Consequently, for $H\in L^\infty$, $\|G_k\ast(u_nH)\|_{L^\infty}\le C\|u_n\|_{L^2\cap Y}\|H\|_{L^\infty}$, and $\|(G_k-G_{k'})\ast(u_nH)\|_{L^\infty}\le\omega(|k-k'|)\|u_n\|_{L^2\cap Y}\|H\|_{L^\infty}$, with $\omega(\eps)\to0$ as $\eps\to0$.

\noindent{\bf Step 3:} Equicontinuity of the remainders. Let $k,k'\in K'$. In the formula \eqref{defLR} for $R_n(k)$ we split $u_n=u_n\1_{\{|X|\le\rho_2\}}+u_n\1_{\{|X|>\rho_2\}}$ in the outer integral. Since $|J_k^\nu\chi_k|\le C\,(|G_k|\ast(|u_n|\cdots(|G_k|\ast|u_n|)))$ on $K'$ by Step 1, the contribution of $u_n\1_{\{|X|>\rho_2\}}$ is at most $CE_{\nu+1}[u_\#\1_{\{|X|>\rho_2\}},u_\#,\dots,u_\#](k)$, which is at most $C\|u_\#\1_{\{|X|>\rho_2\}}\|_Y\|u_\#\|_Y^{\nu}$ by Proposition \ref{proppotential} (ii), uniformly in $n$ and $k\in K'$; it tends to zero as $\rho_2\to\infty$, since the norm of $Y$ is absolutely continuous. For the contribution of $u_n\1_{\{|X|\le\rho_2\}}$ we use $\|u_n\|_{L^1(\{|X|\le\rho_2\})}\le C\rho_2\|u_n\|_{L^2}$, $|e^{-iX\cdot\zeta(k)}-e^{-iX\cdot\zeta(k')}|\le C_{K'}\rho_2|k-k'|$ for $|X|\le\rho_2$, and, by Steps 1 and 2 and the telescoping identity $J_k^\nu\chi_k-J_{k'}^\nu\chi_{k'}=\sum_{i<\nu}J_k^i(J_k-J_{k'})J_{k'}^{\nu-1-i}\chi_{k'}+J_k^\nu(\chi_k-\chi_{k'})$,
$$\|J_k^\nu\chi_k-J_{k'}^\nu\chi_{k'}\|_{L^\infty}\le C\big(\omega(|k-k'|)+\sup_X|\chi_{u_n}(k;X)-\chi_{u_n}(k';X)|\big) .$$
Choosing first $\rho_2$ and then $|k-k'|$, we conclude that the $R_n$ are equicontinuous on $K'$, and they are bounded there by $\sup_{K'}R_\#$.

\noindent{\bf Step 4:} Extraction. By the Arzel\`a--Ascoli theorem and a diagonal argument, a subsequence of $(R_n)$ converges uniformly on the compact subsets of $\C\setminus\R$. By Proposition \ref{propspectral}, $(L_n)$ is a Cauchy sequence for $\|\cdot\|_{\widehat Y}$: indeed, by Lemma \ref{lemmaborn} (ii) and the multilinearity of $Q_2$ and $Q_3$, $|L_{n'}-L_n|$ is bounded by a sum of terms $T_i(F_1,\dots,F_i)$ in which one of the $F$ is $F_{u_{n'}-u_n}$ and the others are $F_\#$. By Proposition \ref{propspectral} (iv), $(L_n)$ converges in $L^2(K')$ for every compact $K'$, and a further subsequence converges almost everywhere. Hence $a_n\to a$ almost everywhere, for a measurable $a$ with $a(\bar k)=\overline{a(k)}$. Fatou's lemma and Proposition \ref{propuniform} give (i): $|a|\le L_\#+R_\#$, $\cI|a|\le\liminf \cI|a_n|\le C+W$, and $8\|a\|_{L^2(\mu)}^2\le\liminf\|u_n\|_{L^2}^2$.

\noindent{\bf Step 5:} Proof of (ii). On $\{|k_2|\le1\}$, $|a_n-a|e^{2W}\jb k^{-1}\le2(L_\#+R_\#)e^{2W}\jb k^{-1}$, which is integrable by \eqref{expint}, and dominated convergence applies. On $\{|k_2|>1\}$, $W$ is bounded by Lemma \ref{lemmaWr} (ii); on $\{|k_2|>1,\ |k|\le\rho_3\}$ we use dominated convergence with the majorant $2(L_\#+R_\#)$, which is integrable there by Proposition \ref{propspectral} (iv) and Proposition \ref{proppotential} (ii); on $\{|k_2|>1,\ |k|>\rho_3\}$ we use $\int|a_n-a|\jb k^{-1}dA\le\|a_n-a\|_{L^2(\mu)}\big(\int_{\{|k_2|>1,\,|k|>\rho_3\}}|k_2|^{-1}\jb k^{-2}dA\big)^{\frac12}$, which tends to zero as $\rho_3\to\infty$, uniformly in $n$.
\end{proof}

\section{Proof of Theorem \ref{thmmain}}
\label{sectionproof}

We keep the notation of Section \ref{sectioncompact}, and we work along the subsequence given by Proposition \ref{proplimit}, from which we extract further subsequences without changing the notation. Let $u(t)$ and $u_n(t)$ be the solutions of \eqref{kp} with data $u_0$ and $u_n$; by Corollary \ref{corsmooth}, $u_n(t)$ is the function reconstructed from $a_n$. We write $m_n:=m[a_n]$ and $D_k:=\pa_y-\pa_x^2-2k\pa_x$, acting on functions of $X$, and $D_k^T:=-\pa_y-\pa_x^2+2k\pa_x$ for its transpose.

The coefficients $a_n$ do not have a common majorant with a bounded Riesz potential, so that Proposition \ref{propcone} does not apply uniformly in $n$. Two arguments replace it, and both use the heat equation $D_km_n=u_nm_n$ of the Lax pair rather than the reconstruction formula. In Subsection \ref{subsectionmass}, the mass of the limit coefficient is identified by comparing the eigenfunctions of the coefficients $a_n$ with those of compactly supported approximations of $a$. In Subsection \ref{subsectionstate}, the scattering state is identified by testing the heat equation against a wave packet $\Phi_t\theta_t$ concentrated, in the physical variable, on the free wave $S(t)\varphi$, and, in the spectral variable, at the stationary point of the phase.

\begin{lemma}[The heat equation for rapidly decaying coefficients]
\label{lemmaheat}
Let $a$ satisfy \eqref{hypa}, let $m=m[a]$, and let $u$ be the function reconstructed from $a$. Then, for every $k\in\C$ and $t\in\R$, $D_km(k;t,\cdot)=u(t)\,m(k;t,\cdot)$ in $\mathcal D'(\R^2)$.
\end{lemma}

\begin{proof}
Let $a_R$, $m_R$ and $u_R$ be as in Step 1 of the proof of Proposition \ref{proprapid}. By Proposition \ref{propcompact} (iii), $D_km_R=u_R(t)m_R$. By Step 1 of the proof of Proposition \ref{proprapid}, $m_R\to m$ uniformly and $u_R(t)\to u(t)$ pointwise and boundedly as $R\to\infty$. Testing against $\varphi\in C^\infty_c(\R^2)$ and using dominated convergence gives the claim.
\end{proof}

\subsection{The mass of the limit coefficient}
\label{subsectionmass}

By Proposition \ref{proplimit} (i), $\|\cB a\|_{L^2}\le\|u_0\|_{L^2}$. For the reverse inequality we approximate $a$ by compactly supported coefficients $c_j$, whose reconstructed data have the mass $\|\cB c_j\|_{L^2}$ by Proposition \ref{propplancherel}, and we show that a weak limit of these data is $u_0$: the eigenfunctions of $c_j$ and of $a_n$ have a common limit, which solves the heat equation with both potentials. Since the potentials $\cI|a_n|$ are not uniformly bounded, the comparison \eqref{pointcomp} of the eigenfunctions holds only for almost every $k$, through the weighted convergence of Proposition \ref{proplimit} (ii).

\begin{proposition}[The mass of the limit coefficient]
\label{propmass}
After extraction of a subsequence, the following holds.
\begin{enumerate}
\item There are real $c_j\in C^\infty_c(\C\setminus\R)$, $j\ge1$, and a null set $\mathcal N\subset\C$ such that, for every $k\in\C\setminus(\R\cup\mathcal N)$, the functions $m_n(k;t,X)$ and $m[c_j](k;t,X)$ converge, as $n\to\infty$ and $j\to\infty$, to the same limit $m_\infty(k;t,X)$, uniformly in $(t,X)\in\R\times\R^2$, and $|m_\infty(k;t,X)|\ge e^{-C-W(k)}>0$.
\item $\|\cB a\|_{L^2}=\|u_0\|_{L^2}$ and $\cB a_n\to\cB a$ in $L^2(\R^2)$.
\end{enumerate}
\end{proposition}

\begin{proof}
{\bf Step 1:} Compactly supported approximations of $a$. For $M\ge1$ let $b_M:=a\1_{E_M}$, where $E_M:=\{|k|\le M,\ |k_2|\ge M^{-1},\ |a(k)|\le M\}$, a set which is invariant under $k\mapsto\bar k$. Then $|b_M|\le|a|$ and $b_M\to a$ almost everywhere as $M\to\infty$. Since $|a|\le L_\#+R_\#$, the function $|a|e^{2W}\jb k^{-1}$ is integrable on $\{|k_2|\le1\}$ by \eqref{expint}, and on $\{|k_2|>1\}$, where $W$ is bounded, by the Cauchy--Schwarz inequality as in Step 5 of the proof of Proposition \ref{proplimit}. By dominated convergence, $\|b_M-a\|_{L^2(\mu)}$ and $\int|b_M-a|e^{2W}\jb k^{-1}dA$ tend to zero. Next, for $0<\eps<(2M)^{-1}$, the function $\varrho_\eps\ast b_M$, with the radial mollifier $\varrho_\eps$ of the Notations, belongs to $C^\infty_c(\C\setminus\R)$, is real since $\varrho_\eps$ is radial, and is supported in the compact set $E'_M:=\{|k|\le M+1,\ |k_2|\ge(2M)^{-1}\}$. As $\eps\to0$, $g_\eps:=|\varrho_\eps\ast b_M-b_M|$ is bounded by $2M$ and tends to zero in $L^1(\C)$. Since $|k_2|$ and $W$ are bounded on $E'_M$, $\|g_\eps\|_{L^2(\mu)}$ and $\int g_\eps e^{2W}\jb k^{-1}dA$ tend to zero, and by \eqref{potentialsets}, $\|\cI g_\eps\|_\infty\le2(\|g_\eps\|_\infty\|g_\eps\|_{L^1})^{\frac12}\to0$. Choosing first $M=M_j$ and then $\eps=\eps_j$, and using $\cI|\varrho_\eps\ast b_M|\le \cI|b_M|+\cI g_\eps\le \cI|a|+\cI g_\eps$, we obtain real $c_j\in C^\infty_c(\C\setminus\R)$ such that
\be
\label{cjapprox}
\begin{gathered}\|c_j-a\|_{L^2(\mu)}+\int_\C\frac{|c_j-a|\,e^{2W}}{\jb k}\,dA\le2^{-j},\\ \cI|c_j|\le\cI|a|+1\le C+1+W .\end{gathered}
\ee
By Proposition \ref{proplimit} (ii), after extraction we may also assume $\int|a_n-a|e^{2W}\jb k^{-1}dA\le2^{-n}$.

\noindent{\bf Step 2:} Proof of (i). For every compact set $K_0\subset\C$ and $l\in\C$, $\int_{K_0}|k-l|^{-1}dA(k)\lesssim_{K_0}\jb l^{-1}$. By Fubini's theorem and Step 1,
$$\int_{K_0}\sum_n\cI\big(|a_n-a|e^{2W}\big)\,dA+\int_{K_0}\sum_j\cI\big(|c_j-a|e^{2W}\big)\,dA<\infty ,$$
and $W$ is locally integrable; hence there is a null set $\mathcal N$ outside of which $W(k)<\infty$ and $\cI(|a_n-a|e^{2W})(k)+\cI(|c_j-a|e^{2W})(k)\to0$. By Proposition \ref{propuniform} (ii), \eqref{cjapprox} and Lemma \ref{lemmapointwise}, applied with $c=a_n$ and $c'=c_j$ at the fixed $(t,X)$,
$$|m_n(k;t,X)-m[c_j](k;t,X)|\le e^{C+W(k)}\,\cI\big(|a_n-c_j|e^{2C+1+2W}\big)(k),$$
and the right-hand side tends to zero as $n,j\to\infty$, for $k\notin\mathcal N$, uniformly in $(t,X)$. The same bound with $a_{n'}$ in place of $c_j$ shows that $(m_n(k;t,X))_n$ is a Cauchy sequence, uniformly in $(t,X)$, and so is $(m[c_j](k;t,X))_j$; they have the same limit $m_\infty(k;t,X)$, and \eqref{pointcomp} gives $|m_n|\ge e^{-\cI|a_n|}\ge e^{-C-W}$, hence $|m_\infty|\ge e^{-C-W}$.\\

\noindent{\bf Step 3:} Proof of (ii). Let $w_j\in C(\R;L^2)$ be the function reconstructed from $c_j$. By Proposition \ref{propplancherel} and \eqref{cjapprox}, $\|w_j(0)\|_{L^2}^2=8\|c_j\|_{L^2(\mu)}^2\to8\|a\|_{L^2(\mu)}^2=\|\cB a\|_{L^2}^2$. After extraction, $w_j(0)\weakto w_*$ in $L^2$. Fix $k\in\C\setminus(\R\cup\mathcal N)$. By Proposition \ref{propcompact} (iii), $D_km[c_j](k;0,\cdot)=w_j(0)\,m[c_j](k;0,\cdot)$. Since $m[c_j](k;0,\cdot)\to m_\infty(k;0,\cdot)$ uniformly, and $w_j(0)\weakto w_*$ with $\sup_j\|w_j(0)\|_{L^2}<\infty$, we may pass to the limit in $\mathcal D'(\R^2)$ and obtain $D_km_\infty(k;0,\cdot)=w_*\,m_\infty(k;0,\cdot)$. On the other hand, $m_n(k;0,\cdot)=\chi_{u_n}(k;\cdot)$ by Proposition \ref{propdirect} (ii), and $D_k\chi_{u_n}=u_n\chi_{u_n}$ in $\mathcal D'(\R^2)$, since $\chi_{u_n}=1+G_k\ast(u_n\chi_{u_n})$ by Lemma \ref{lemmaborn} (i) and $G_k$ is a fundamental solution of $D_k$. Since $u_n\to u_0$ in $L^2$, the same limit gives $D_km_\infty(k;0,\cdot)=u_0\,m_\infty(k;0,\cdot)$. Hence $(w_*-u_0)m_\infty(k;0,\cdot)=0$, and since $m_\infty(k;0,\cdot)$ does not vanish, $w_*=u_0$. By weak lower semicontinuity, $\|u_0\|_{L^2}\le\liminf_j\|w_j(0)\|_{L^2}=\|\cB a\|_{L^2}$, which, together with Proposition \ref{proplimit} (i), gives $\|\cB a\|_{L^2}=\|u_0\|_{L^2}$. Finally, $(a_n)$ is bounded in $L^2(\mu)$ and converges almost everywhere to $a$, so that $a_n\weakto a$ in $L^2(\mu)$; since $8\|a_n\|_{L^2(\mu)}^2=\|u_n\|_{L^2}^2\to\|u_0\|_{L^2}^2=8\|a\|_{L^2(\mu)}^2$, the convergence is strong, and \eqref{bornisometry} gives $\cB a_n\to\cB a$ in $L^2$.
\end{proof}

\subsection{The common scattering state}
\label{subsectionstate}

Let $V\subset\cV$ be compact, and let $\mathcal N$ be the null set of Proposition \ref{propmass}. The eigenfunctions are uniformly close to $1$ on the propagation cones, for almost every $k$.

\begin{lemma}[Decay of the eigenfunctions on the cones]
\label{lemmadecay}
For $k\in\C\setminus(\R\cup\mathcal N)$ and $t\neq0$, let $\eps_t(k):=\sup_n\sup_{X/t\in V}|m_n(k;t,X)-1|$. Then $\eps_t(k)\le1+e^{C+W(k)}$, and $\eps_t(k)\to0$ as $|t|\to\infty$.
\end{lemma}

\begin{proof}
The function $\eps_t$ is lower semicontinuous, hence measurable, since each $m_n$ is continuous in $(k,X)$ by Lemma \ref{lemmainverse}. The first bound follows from \eqref{pointcomp} and Proposition \ref{propuniform} (ii). Let $c$ be a single coefficient which satisfies \eqref{hypa}. Then $c$ satisfies \eqref{hypF1} with $F:=C_c\jb\zeta^{-6}$, which satisfies \eqref{weightsF} by Lemma \ref{lemmaweights} (ii), and Step 2 of the proof of Proposition \ref{propcone} shows that $\sup_{X/t\in V}\|m[c](\cdot;t,X)-1\|_\infty\to0$ as $|t|\to\infty$. This applies to each $a_n$ and to each $c_j$ of Proposition \ref{propmass}. Fix $k\notin\R\cup\mathcal N$ and $\eta>0$. By Step 2 of the proof of Proposition \ref{propmass}, there are $j$ and $n_1$ such that $|m_n(k;t,X)-m[c_j](k;t,X)|\le\eta$ for all $n\ge n_1$ and $(t,X)$. Hence $|m_n(k;t,X)-1|\le2\eta$ for $n\ge n_1$, $X/t\in V$ and $|t|$ large, and the finitely many $n<n_1$ are treated by the single coefficient bound.
\end{proof}

We now construct the wave packets. Let $\hat\varphi\in C^\infty_c(\{\xi<0\})$; the case $\hat\varphi\in C^\infty_c(\{\xi>0\})$ is the same, with $k_+$ replaced by $k_-$. Assume that the interior of $V$ contains the group velocities of $\supp\hat\varphi$. For $v\in\cV$, the frequency $\zeta(k_+(v))$, with $k_+(v)$ given by \eqref{defkpm}, is the unique frequency in $\{\xi<0\}$ with group velocity $v$. We set
\be
\label{defpsiA}
\psi(v):=v\cdot\zeta(k_+(v))+\omega(\zeta(k_+(v))),\qquad B_0(v):=\frac16\,\overline{\hat\varphi(\zeta(k_+(v)))},
\ee
so that $B_0\in C^\infty_c(V)$, and since $\zeta(k_+(v))$ is a critical point of $z\mapsto v\cdot z+\omega(z)$, $\nabla_v\psi(v)=\zeta(k_+(v))$. For $t\neq0$ let
\be
\label{defWt}
\Phi_t(X):=|t|^{-1}B_0(X/t)\,e^{-it\psi(X/t)},\qquad k_t(X):=k_+(X/t),
\ee
so that $\nabla_X(t\psi(X/t))=\zeta(k_t(X))$, and, by \eqref{dphase}, $k_t(X)$ is the critical point of $\phi(\cdot;t,X)$ in $\C^+$: $x+2k_ty-12k_t^2t=0$. Differentiating this identity, we obtain
\be
\label{derivkt}
\pa_xk_t=\frac1{24k_tt-2y},\qquad \pa_yk_t=2k_t\,\pa_xk_t .
\ee
Since $24k_tt-2y=2t(12k_+(v)-v_2)=24it\Im k_+(v)$ with $v=X/t$, Lemma \ref{lemmaphase} (i) gives $|\pa^\gamma_Xk_t|\lesssim_\gamma|t|^{-|\gamma|}$ on $\{X/t\in V\}$ for $|\gamma|\ge1$. Let $\varrho$ be the radial mollifier of the Notations, let $\varrho^{(2)}(z):=c_\varrho\bar z^2\varrho(z)$ with $c_\varrho:=(\int|z|^4\varrho\,dA)^{-1}$, let $r:=|t|^{-\frac12}$ and
\be
\label{defpacket}
\begin{gathered}\theta_t(k,X):=r^{-2}\Big[\varrho\Big(\frac{k-k_t(X)}r\Big)+\frac{\pa_xk_t(X)}{r^2}\,\varrho^{(2)}\Big(\frac{k-k_t(X)}r\Big)\Big],\\ H_t:=D_k^T\big(\Phi_t\theta_t\big),\end{gathered}
\ee
where $D_k^T$ acts on the variable $X$.

\begin{lemma}[Wave packets]
\label{lemmapacket}
For $|t|\ge1$ large enough, the following holds.
\begin{enumerate}
\item $\|\overline{S(t)\varphi}-\Phi_t\|_{L^2}\to0$ as $|t|\to\infty$.
\item $\int\theta_t(k,X)\,dA(k)=1$, $\int(k-k_t)\theta_t\,dA(k)=0$, $\int(k-k_t)^2\theta_t\,dA(k)=\pa_xk_t$, and $\|\theta_t(\cdot,X)\|_{L^1}\lesssim1$.
\item $H_t$ is supported in $\{X/t\in V,\ |k-k_t(X)|\le r\}$, and
\be
\label{Hbound}
\begin{gathered}|H_t|\lesssim r,\qquad \int|H_t|\,dA(k)\lesssim r^3,\\ \Big|\int H_t\,dA(k)\Big|\lesssim r^4,\qquad \Big|\int(k-k_t)H_t\,dA(k)\Big|\lesssim r^6 .\end{gathered}
\ee
\item For every $l\in\C$, $\displaystyle\int_{\R^2}\Big|\int_\C\frac{H_t(k,X)}{k-l}\,dA(k)\Big|\,dX\lesssim\jb l^{-1}$.
\end{enumerate}
\end{lemma}

\begin{proof}
{\bf Step 1:} Proof of (i). With $v=X/t$,
$$S(t)\varphi(X)=\frac1{2\pi}\int_{\R^2}e^{it(v\cdot z+\omega(z))}\hat\varphi(z)\,dz .$$ Let $E\subset\{\xi<0\}$ be a compact set whose interior contains $\supp\hat\varphi$ and $\zeta(k_+(V))$. For $v\in V$, the phase $z\mapsto v\cdot z+\omega(z)$ has exactly one critical point in $\{\xi<0\}$, namely $z_0(v)=\zeta(k_+(v))\in E$, which is nondegenerate with Hessian determinant $\det\nabla^2\omega=-36$, hence of signature zero. The stationary phase estimate \cite[Theorem 7.7.5]{Hormander}, applied on $E$ with constants uniform in $v\in V$ as in the proof of Lemma \ref{lemmafree}, gives $S(t)\varphi(X)=|t|^{-1}\frac16\hat\varphi(z_0(v))e^{it\psi(v)}+O(|t|^{-2})$ uniformly for $v\in V$; the leading term vanishes when $\hat\varphi(z_0(v))=0$. Hence $|\overline{S(t)\varphi}-\Phi_t|\lesssim|t|^{-2}$ on $\{X/t\in V\}$, whose area is $t^2|V|$, and by Lemma \ref{lemmafree}, $\|\1_{\{X/t\notin V\}}S(t)\varphi\|_{L^2}\to0$, while $\Phi_t$ vanishes there.

\noindent{\bf Step 2:} Proof of (ii). Since $\varrho$ is radial, $\int\varrho\,dA=1$, $\int z\varrho\,dA=\int z^2\varrho\,dA=0$, and $\int\varrho^{(2)}dA=\int z\varrho^{(2)}dA=0$, $\int z^2\varrho^{(2)}dA=1$. The substitution $k=k_t+rz$ gives the three identities. Since $|\pa_xk_t|\lesssim|t|^{-1}=r^2$, $\|\theta_t(\cdot,X)\|_{L^1}\lesssim1$.

\noindent{\bf Step 3:} Proof of (iii). We write $\Phi:=\Phi_t$ and $\Psi(X):=t\psi(X/t)$, so that $\nabla\Psi=\zeta(k_t)$ and $|\nabla^2\Psi|\lesssim|t|^{-1}$. Then $\pa_x\Phi=-i\zeta_1(k_t)\Phi+O(|t|^{-2})$, $\pa_y\Phi=-i\zeta_2(k_t)\Phi+O(|t|^{-2})$ and $\pa_x^2\Phi=-\zeta_1(k_t)^2\Phi+O(|t|^{-2})$, since the derivatives of $B_0(X/t)$ and of $\zeta(k_t)$ cost a factor $|t|^{-1}$. Hence, by \eqref{heatsymbol},
\be
\label{DTW}
\begin{aligned}D_k^T\Phi&=\big(i\zeta_2(k_t)+\zeta_1(k_t)^2-2ik\zeta_1(k_t)\big)\Phi+O(|t|^{-2})\\&=4i(k_t)_2\,(k-k_t)\,\Phi+O(|t|^{-2}) .\end{aligned}
\ee
On the support of $\theta_t$, $|k-k_t|\le r$, so that $|D_k^T\Phi|\lesssim r\cdot r^2+r^4$, and $|\theta_t|\lesssim r^{-2}$. Moreover, since $\theta_t$ depends on $X$ through $k_t$ and $\pa_xk_t$, the bounds on $\pa^\gamma k_t$ give $|\pa_X\theta_t|\lesssim r^{-2}\cdot r^{-1}r^2=r^{-1}$ and $|\pa_x^2\theta_t|\lesssim1$. Since $H_t=(D_k^T\Phi)\theta_t+\Phi(-\pa_y\theta_t-\pa_x^2\theta_t+2k\pa_x\theta_t)-2\pa_x\Phi\,\pa_x\theta_t$ and $|\Phi|+|\pa_x\Phi|\lesssim r^2$, we obtain $|H_t|\lesssim r$, and the bound on $\int|H_t|\,dA$ follows since the support has area $\pi r^2$. For the moments we use (ii). Since $\int\pa_X^\gamma(\Phi\theta_t)\,dA(k)=\pa_X^\gamma \Phi$ and $\int k\,\theta_t\,dA=k_t$,
$$\begin{aligned}\int H_t\,dA&=-\pa_y\Phi-\pa_x^2\Phi+2\pa_x(k_t\Phi)\\&=\big(i\zeta_2(k_t)+\zeta_1(k_t)^2-2ik_t\zeta_1(k_t)\big)\Phi+O(|t|^{-2})=O(r^4),\end{aligned}$$
by \eqref{DTW} at $k=k_t$. For the first moment, $\pa_X(k-k_t)=-\pa_Xk_t$, and integrations of $(k-k_t)\pa_X(\Phi\theta_t)$ and $(k-k_t)\pa_x^2(\Phi\theta_t)$ against $dA(k)$, using (ii), give
$$\begin{gathered}\int(k-k_t)\pa_X(\Phi\theta_t)\,dA=\Phi\pa_Xk_t,\\ \int(k-k_t)\pa_x^2(\Phi\theta_t)\,dA=\Phi\pa_x^2k_t+2\pa_x\Phi\,\pa_xk_t,\end{gathered}$$
$$\int(k-k_t)\,k\,\pa_x(\Phi\theta_t)\,dA=\pa_x\big(\Phi\pa_xk_t\big)+k_t\Phi\pa_xk_t .$$
Hence, by \eqref{derivkt},
$$\begin{aligned}\int(k-k_t)H_t\,dA&=-\Phi\pa_yk_t-\Phi\pa_x^2k_t-2\pa_x\Phi\pa_xk_t+2\pa_x(\Phi\pa_xk_t)+2k_t\Phi\pa_xk_t\\&=\Phi\pa_x^2k_t=O(r^6) .\end{aligned}$$

\noindent{\bf Step 4:} Proof of (iv). Let $d:=|l-k_t(X)|$. If $d\le2r$, we use $|H_t|\lesssim r$ and $\int_{\{|k-k_t|\le r\}}|k-l|^{-1}dA\le2\pi r$, so that the inner integral is $O(r^2)$. If $d>2r$, we expand, for $|k-k_t|\le r$,
$$\frac1{l-k}=\frac1{l-k_t}+\frac{k-k_t}{(l-k_t)^2}+\frac{(k-k_t)^2}{(l-k_t)^2(l-k)},$$
and \eqref{Hbound} bounds the inner integral by $C(r^4d^{-1}+r^6d^{-2}+r^5d^{-3})$. We substitute $X=tv$, $dX=r^{-4}dv$, and use that $v\mapsto k_+(v)$ is a diffeomorphism from a neighbourhood of $V$ onto a bounded subset of $\C^+$, with Jacobian bounded above and below. If $|l|\le C_V$, the set of $v$ with $d\le2r$ has measure $O(r^2)$, and $\int_V d^{-1}dv\lesssim1$, $\int_{\{d>2r\}}d^{-2}dv\lesssim\log\frac1r$, $\int_{\{d>2r\}}d^{-3}dv\lesssim r^{-1}$; the total is $O(1)$. If $|l|>C_V$ is large, then $d\sim|l|$ and the total is $O(r^{-4}r^4|l|^{-1})$.
\end{proof}

\begin{proof}[Proof of Theorem \ref{thmmain}]
Let $\hat\varphi\in C^\infty_c(\{\xi<0\})$, let $V$ be as above, and let $\ell_n:=\cC(a_nP)$ and $g_n(t):=S(t)\cB a_n$.\\

\noindent{\bf Step 1:} The linearized identity. Since $D_kP(l;t,X)=D(k,\zeta(l))P(l;t,X)=4il_2(k-l)P(l;t,X)$ by \eqref{heatsymbol}, the Cauchy kernel cancels, and \eqref{freespectral} gives $D_k\ell_n=\frac{4i}\pi\int l_2a_nP\,dA=g_n(t)$; differentiation under the integral sign is legitimate by \eqref{hypa}. Together with Lemma \ref{lemmaheat}, we obtain, for every $k\in\C$ and $t$,
\be
\label{lindiff}
D_k\big(m_n-1-\ell_n\big)=u_n(t)\,m_n-g_n(t)\qquad\mbox{in}\ \mathcal D'(\R^2) .
\ee
We pair \eqref{lindiff} with $\Phi_t\theta_t(k,\cdot)\in C^\infty_c(\R^2)$ and integrate in $k$. By Lemma \ref{lemmapacket} (ii), $\int\theta_t\,dA(k)=1$, and since $m_n-1-\ell_n=\cC(a_nP\overline{(m_n-1)})$, we obtain
\be
\label{pairing}
\begin{aligned}\int_{\R^2}\big(u_n(t)-g_n(t)\big)\Phi_t\,dX&=\int_{\R^2}\int_\C\cC\big(a_nP\overline{(m_n-1)}\big)H_t\,dA\,dX\\&\quad-\int_{\R^2}\int_\C u_n(t)(m_n-1)\,\Phi_t\theta_t\,dA\,dX .\end{aligned}
\ee

\noindent{\bf Step 2:} The second term of \eqref{pairing}. Let $c_V$ be the constant of Lemma \ref{lemmaphase}, and let $K_0\subset\C^+$ be the closed $\frac{c_V}2$-neighbourhood of $k_+(V)$, so that $D(k_t(X),r)\subset K_0$ for $X/t\in V$ and $|t|\ge4c_V^{-2}$. For such $t$, $|m_n(k;t,X)-1|\le\eps_t(k)$ on the support of $\theta_t$, $|\Phi_t|\lesssim r^2$ and $|\theta_t|\lesssim r^{-2}\1_{\{|k-k_t(X)|\le r\}}$. By the Cauchy--Schwarz inequality, the second term is at most
$$\begin{aligned}&C\|u_n(t)\|_{L^2}\,r^2\Big\|\big(r^{-2}\1_{D(0,r)}\ast(\eps_t\1_{K_0})\big)(k_t(X))\Big\|_{L^2(\{X/t\in V\})}\\&\qquad\lesssim\|u_n(t)\|_{L^2}\,\big\|r^{-2}\1_{D(0,r)}\ast(\eps_t\1_{K_0})\big\|_{L^2(\C)},\end{aligned}$$
where we used $dX=r^{-4}dv$ and the bounded Jacobian of $v\mapsto k_+(v)$. By Young's inequality \cite[Theorem 4.2]{LiebLoss}, the right-hand side is at most $C\|u_n(t)\|_{L^2}\|\eps_t\|_{L^2(K_0)}$, which tends to zero by Lemma \ref{lemmadecay} and dominated convergence, since $W$ is bounded on $K_0$; the convergence is uniform in $n$, since $\|u_n(t)\|_{L^2}=\|u_n\|_{L^2}$.\\

\noindent{\bf Step 3:} The first term of \eqref{pairing}. For fixed $n$ and $t$, Fubini's theorem applies, since $a_n$ satisfies \eqref{hypa}, $m_n$ is bounded, and $H_t$ is bounded with compact support in $(k,X)$. The first term equals
$$\frac1\pi\int_\C a_n(l)\int_{\R^2}P(l;t,X)\overline{(m_n-1)(l;t,X)}\Big(\int_\C\frac{H_t(k,X)}{k-l}\,dA(k)\Big)\,dX\,dA(l),$$
and since $H_t$ vanishes unless $X/t\in V$, Lemma \ref{lemmapacket} (iv) bounds its absolute value by $C\int|a_n(l)|\eps_t(l)\jb l^{-1}dA(l)$. On $\{|l_2|\le1\}$, $|a_n|\eps_t\le(L_\#+R_\#)(1+e^{C+W})$, which is integrable against $\jb l^{-1}dA$ by \eqref{expint}, and $\eps_t\to0$ almost everywhere; by dominated convergence, this part tends to zero uniformly in $n$. On $\{|l_2|>1\}$, $\eps_t\le C$, and the Cauchy--Schwarz inequality bounds the contribution of $\{|l|>\rho_4\}$ by $C\|a_n\|_{L^2(\mu)}(\int_{\{|l_2|>1,\,|l|>\rho_4\}}|l_2|^{-1}\jb l^{-2}dA)^{\frac12}$, which is small for $\rho_4$ large, uniformly in $n$, while on $\{|l_2|>1,\ |l|\le\rho_4\}$ we use dominated convergence with the majorant $C(L_\#+R_\#)$. Hence the first term tends to zero as $|t|\to\infty$, uniformly in $n$.\\

\noindent{\bf Step 4:} Conclusion. By Steps 2 and 3 and Lemma \ref{lemmapacket} (i), with $\sup_n(\|u_n(t)\|_{L^2}+\|g_n(t)\|_{L^2})<\infty$,
$$\big(S(-t)u_n(t)-\cB a_n,\varphi\big)=\int_{\R^2}\big(u_n(t)-g_n(t)\big)\overline{S(t)\varphi}\,dX\longrightarrow0\qquad(|t|\to\infty),$$
uniformly in $n$.
For fixed $t$, $u_n(t)\to u(t)$ in $L^2$ by the continuity of the $L^2$ flow \cite{Bourgain}, and $\cB a_n\to\cB a$ in $L^2$ by Proposition \ref{propmass} (ii). Hence $(S(-t)u(t)-\cB a,\varphi)\to0$ as $t\to\pm\infty$. The same holds for $\hat\varphi\in C^\infty_c(\{\xi>0\})$, and these functions $\varphi$ span a dense subspace of $L^2$. Since $S(-t)u(t)$ is bounded in $L^2$, $S(-t)u(t)\weakto\cB a$ as $t\to\pm\infty$. By \eqref{mass} and Proposition \ref{propmass} (ii), $\|S(-t)u(t)\|_{L^2}=\|u_0\|_{L^2}=\|\cB a\|_{L^2}$, so that the convergence is strong. This proves Theorem \ref{thmmain}, with $u_+=u_-=\cB a$. Since the scattering state is determined by $u$, the limit $\cB a$ does not depend on the extracted subsequence; since every sequence of real $u_{0,n}\in C^\infty_c(\R^2)$ converging to $u_0$ in $L^2\cap Y$ has a subsequence of the form \eqref{defvj}, the scattering coefficients of every such sequence converge to $a$ in $L^2(\mu)$.
\end{proof}

\section{Data without scattering: proof of Theorem \ref{thmcounter}}
\label{sectioncounter}

For a real datum $f\in C^\infty_c(\R^2)$, the solution scatters to $\cB a_f$ in both time directions, by Corollary \ref{corsmooth} and the argument which follows it; here $a_f$ denotes the scattering coefficient of $f$. The construction of Theorem \ref{thmcounter} is based on the following mechanism. A perturbation $v$ of $f$ which is small in $L^2$ and supported far away from the support of $f$ changes the modulus $|a_f|=|\widehat{\cB a_f}\circ\zeta|$ only slightly, but it may rotate the phase of $a_f$ by a fixed amount: the eigenfunction $\chi_{f+v}$ is changed on the support of $f$ by the factor $1+G_k\ast(v\chi)$, and if $v$ is a sum of many small terms supported at geometrically increasing distances, the contributions of these terms to this factor multiply. Alternating rotations in opposite directions at larger and larger distances, we obtain a datum whose interaction profile oscillates between two different states.

Throughout this section $k$ belongs to the upper half-plane, and we write $b:=k_2>0$, $\tilde x:=x+2k_1y$ and $e_k(X):=e^{-iX\cdot\zeta(k)}=e^{2ib\tilde x}$.

\subsection{The Green function at large distances}

For $X=(x,y)$ with $y\neq0$ we define
\be
\label{defGinfty}
G^\infty_k(X):=\frac1{2\pi}\,\frac1{2by-i\tilde x} .
\ee
Since $\tilde x$ is linear in $X$, $G^\infty_k$ is odd and homogeneous of degree $-1$.

\begin{lemma}[The Green function at large distances]
\label{lemmatail}
Let $K\subset\C^+$ be compact and $c_0>0$. Uniformly for $k\in K$, $R\ge1$ and $X$ with $|X|\le c_0^{-1}$ and $|y|\ge c_0$,
\be
\label{tailbranches}
R\,G_k(RX)=G^\infty_k(X)+\overline{e_k(RX)}\,\overline{G^\infty_k(X)}+O(R^{-1}) .
\ee
\end{lemma}

The two terms of \eqref{tailbranches} come from the two zeros $z=0$ and $z=\zeta(k)$ of the heat symbol in Step 2 of the proof of Lemma \ref{lemmaborn}: at large distances, $G_k$ behaves like the sum of the Green functions of the linearizations of the symbol at its two zeros.

\begin{proof}
Let $y>0$. Substituting $s=\pm(b+\rho)$, $\rho>0$, in \eqref{formulagreen}, and then $\rho=r/R$, we obtain, as in Step 3 of the proof of Lemma \ref{lemmairregular},
$$R\,G_k(RX)=\frac1{2\pi}\int_0^\infty\big(e^{ir\tilde x}+e^{-2ibR\tilde x}e^{-ir\tilde x}\big)e^{-2bry-\frac{yr^2}R}\,dr .$$
Since $0\le1-e^{-yr^2/R}\le\frac{yr^2}R$, removing the factor $e^{-yr^2/R}$ produces an error at most $\frac1\pi\int_0^\infty e^{-2bry}\frac{yr^2}R\,dr=\frac{1}{4\pi b^3y^2R}$. The remaining integrals are $\frac1{2\pi}(2by-i\tilde x)^{-1}=G^\infty_k(X)$ and $\frac1{2\pi}e^{-2ibR\tilde x}(2by+i\tilde x)^{-1}=\overline{e_k(RX)}\,\overline{G^\infty_k(X)}$. Let $y<0$. Substituting $s=\pm(b-\rho)$, $0<\rho<b$, and then $\rho=r/R$,
$$R\,G_k(RX)=-\frac1{2\pi}\int_0^{Rb}\big(e^{-ir\tilde x}+e^{-2ibR\tilde x}e^{ir\tilde x}\big)e^{-2br|y|+\frac{|y|r^2}R}\,dr .$$
For $0\le r\le Rb$ we have $2br-\frac{r^2}R\ge br$, hence $0\le e^{-(2br-r^2/R)|y|}-e^{-2br|y|}\le e^{-br|y|}\frac{|y|r^2}R$, and $\int_{Rb}^\infty e^{-2br|y|}dr\le\frac{e^{-2b^2R|y|}}{2b|y|}$. Replacing the exponential factor by $e^{-2br|y|}$ and the upper limit by $\infty$ thus produces an error $O(R^{-1})$, and the remaining integrals are $-\frac1{2\pi}(2b|y|+i\tilde x)^{-1}=G^\infty_k(X)$ and $-\frac1{2\pi}e^{-2ibR\tilde x}(2b|y|-i\tilde x)^{-1}=\overline{e_k(RX)}\,\overline{G^\infty_k(X)}$. All the errors are uniform under the stated conditions.
\end{proof}

Fix a real $w\in C^\infty_c((1,2)^2)$, $w\neq0$, and set, for $R\ge1$,
\be
\label{defWR}
w_R(X):=R^{-1}w(X/R),\qquad c(k):=\int_{\R^2}G^\infty_k(-Y)\,w(Y)\,dY .
\ee
Then $\|w_R\|_{L^2}=\|w\|_{L^2}$ and $\|w_R\|_{L^1}=R\|w\|_{L^1}$. The function $c$ is real-analytic on $\C^+$: for $Y\in(1,2)^2$, the real part of $\big(2\pi G^\infty_k(-Y)\big)^{-1}=-2by_2+i(y_1+2k_1y_2)$ is at most $-2b$. We also set $H_k(X):=\overline{G_k(-X)}$.

\begin{lemma}[Convolution with $w_R$]
\label{lemmashell}
Let $K\subset\C^+$ be compact. The following bounds hold uniformly for $k\in K$ and $R\ge1$.
\begin{enumerate}
\item For $|X|\le\frac R4$,
\be
\label{tailshell}
\begin{gathered}(G_k\ast w_R)(X)=c(k)+O\Big(\frac{1+|X|}R\Big),\\ (H_k\ast w_R)(X)=-\overline{c(k)}+O\Big(\frac{1+|X|}R\Big).\end{gathered}
\ee
\item $|G_k(X)|+|H_k(X)|\lesssim_K|X|^{-1}$ for $X\neq0$, and $\||G_k|\ast|w_R|\|_{L^\infty}+\||H_k|\ast|w_R|\|_{L^\infty}\lesssim_K1$.
\end{enumerate}
\end{lemma}

\begin{proof}
By the substitution $Y=RY'$, $(G_k\ast w_R)(X)=\int RG_k(R(X/R-Y'))w(Y')\,dY'$. For $Y'\in(1,2)^2$ and $|X|\le\frac R4$, the point $Z:=X/R-Y'$ satisfies $|Z|\le4$ and $z_2\le-\frac34$, and Lemma \ref{lemmatail} gives
$$\begin{aligned}(G_k\ast w_R)(X)&=\int G^\infty_k(X/R-Y')w(Y')\,dY'+O(R^{-1})\\&\quad+e^{i\zeta(k)\cdot X}\int e^{-iR\zeta(k)\cdot Y'}\,\overline{G^\infty_k(X/R-Y')}\,w(Y')\,dY' .\end{aligned}$$
Since $G^\infty_k$ is smooth near $-\overline{(1,2)^2}$, the first integral is $c(k)+O(|X|/R)$. In the second one, the phase has the derivative $-R\zeta_1(k)=2bR$ in $y'_1$, and one integration by parts in $y'_1$ gives $O(R^{-1})$. For $H_k$ we write $(H_k\ast w_R)(X)=\overline{\int RG_k(R(Y'-X/R))w(Y')\,dY'}$, where now $Z=Y'-X/R$ satisfies $z_2\ge\frac34$; the same argument gives the constant $\overline{\int G^\infty_k(Y')w(Y')\,dY'}=-\overline{c(k)}$, since $G^\infty_k$ is odd. The bound $|G_k(X)|\lesssim\min(|\tilde x|^{-1},(b|y|)^{-1})\lesssim_K|X|^{-1}$ follows from \eqref{sizegreen}, and $\int|X/R-Y'|^{-1}|w(Y')|\,dY'\lesssim\|w\|_{L^\infty}$ gives the last bound.
\end{proof}

\subsection{Phase rotation}

Let $\alpha>0$. For $N\ge1$ we set
\be
\label{defvN}
\begin{gathered}\Lambda_N:=16+\log(N+1),\qquad R_j:=\Lambda_N^j,\qquad h_N:=\sum_{j=1}^N\frac1j,\\ \tau_j:=\frac\alpha{jh_N},\qquad v_N:=\sum_{j=1}^N\tau_jw_{R_j}.\end{gathered}
\ee
The supports of the $w_{R_j}$ are contained in the disjoint squares $R_j(1,2)^2$, $\sum_j\tau_j=\alpha$ and $\max_j\tau_j=\frac\alpha{h_N}\to0$. Let $\gamma$ be a multi-index. Since $\pa^\gamma w_R=R^{-1-|\gamma|}(\pa^\gamma w)(\cdot/R)$, since $|X|\sim R_j$ on the support of $w_{R_j}$, and since these supports are disjoint, the left-hand sides of \eqref{vNsmall} below are bounded by $C_\gamma\sum_j\tau_j^2(1+\log R_j)$ and by $C_\gamma\sup_j\tau_j(1+\log R_j)$, respectively. By \eqref{defvN}, $\sum_j\tau_j^2(1+\log R_j)=\frac{\alpha^2}{h_N^2}\sum_j\big(\frac1{j^2}+\frac{\log\Lambda_N}j\big)$ and $\sup_j\tau_j(1+\log R_j)\le\frac{\alpha}{h_N}(1+\log\Lambda_N)$, and since $\log\Lambda_N\ge1$,
\be
\label{vNsmall}
\begin{gathered}
\big\|(\log(2+|X|))^{\frac12}\pa^\gamma v_N\big\|_{L^2}^2\lesssim_\gamma\frac{\alpha^2\log\Lambda_N}{h_N},\\
\sup_X\jb X^{1+|\gamma|}\log(2+|X|)\,|\pa^\gamma v_N(X)|\lesssim_\gamma\frac{\alpha\log\Lambda_N}{h_N},
\end{gathered}
\ee
and both right-hand sides tend to zero as $N\to\infty$, since $h_N\ge\log(N+1)$. In particular $\|v_N\|_{L^2}\to0$. Nevertheless, the perturbation $v_N$ does not disappear from the scattering coefficient.

\begin{proposition}[Phase rotation]
\label{propphase}
Let $u\in C^\infty_c(\R^2)$ be real. Then, as $N\to\infty$,
\be
\label{tailphase}
a_{u+v_N}\longrightarrow\Theta\,a_u\quad\mbox{in}\ L^2(\mu),\qquad \Theta(k):=e^{2i\alpha\Im c(k)}\ \ (k\in\C^+),
\ee
where $\Theta$ is extended to $\C^-$ by $\Theta(\bar k):=\overline{\Theta(k)}$.
\end{proposition}

The factor $e^{\alpha c(k)}=\lim_N\prod_j(1+\tau_jc(k))$ is the product of the contributions of the terms $\tau_jw_{R_j}$ to the eigenfunction on the support of $u$; the same product, with $c(k)$ replaced by $-\overline{c(k)}$, appears in the solution of the integral equation with the kernel $H_k$, and the modulus cancels in the scattering coefficient.

\begin{proof}
{\bf Step 1:} Block systems. Let $K\subset\C^+$ be a compact set of regular points of $u$, so that $\sup_{k\in K}\|T_k^{-1}\|<\infty$ by Step 2 of the proof of Lemma \ref{lemmairregular}. Let $\Omega_0$ be a closed disc containing $\supp u$ and $\Omega_j:=R_j\supp w$, $1\le j\le N$; for $N$ large these sets are pairwise disjoint, and $|X|\ge R_j\ge16R_i$ for $X\in\Omega_j$, $i<j$. Let $P_i$ be the restriction to $\Omega_i$, $\Omega:=\bigcup_i\Omega_i$, and $q_0:=u$, $q_j:=\tau_jw_{R_j}$, so that $u+v_N=\sum_iq_i$. On $L^\infty(\Omega)$, the operator $\mathcal T:=I-P_\Omega G_k\ast((u+v_N)\,\cdot\,)$ has the blocks $\mathcal T_{ij}:=\delta_{ij}-P_iG_k\ast(q_j\,\cdot\,)$, $0\le i,j\le N$, and the norm of a block operator on $L^\infty(\Omega)$ is at most the maximum over $i$ of $\sum_j\|\mathcal T_{ij}\|$. We use three bounds, uniform in $k\in K$ and $N$ large. First, let $J:=G_k\ast(u\,\cdot\,)$, so that $JP_0=J$ and $T_k=I-J$. Then
$$\begin{gathered}(I-P_0J)(I+P_0T_k^{-1}J)=I+P_0\big((I-J)T_k^{-1}-I\big)J=I,\\ (I+P_0T_k^{-1}J)(I-P_0J)=I+P_0\big(T_k^{-1}(I-J)-I\big)J=I,\end{gathered}$$
so that the block $\mathcal T_{00}=I-P_0J$ is invertible on $L^\infty(\Omega_0)$, with the uniformly bounded inverse $I+P_0T_k^{-1}J$. Second, by Lemma \ref{lemmashell} (ii), the diagonal blocks satisfy $\|\mathcal T_{jj}-I\|=O(\tau_j)$, $j\ge1$, and the upper blocks satisfy $\|\mathcal T_{ij}\|=O(\tau_j)$, $i<j$. Third, for $i>j\ge1$ the distance from $\Omega_i$ to $\Omega_j$ is at least $\frac{R_i}2$, and Lemma \ref{lemmashell} (ii) gives $\|\mathcal T_{ij}\|\lesssim\tau_jR_i^{-1}\|w_{R_j}\|_{L^1}\lesssim\tau_jR_j/R_i$; similarly $\|\mathcal T_{i0}\|\lesssim\|u\|_{L^1}R_i^{-1}$. Since $R_j/R_i\le\Lambda_N^{-1}$ for $j<i$, we obtain $\sum_{j<i}\|\mathcal T_{ij}\|=O(\Lambda_N^{-1})$ for every $i$.

Let $\mathcal D$ be the diagonal part and $\mathcal U$ the strictly upper part of $\mathcal T$. By the first two bounds, $\|\mathcal D^{-1}\|\lesssim1$, and the blocks of $(\mathcal D^{-1}\mathcal U)^p$ in the row $i_0$ are sums over chains $i_0<i_1<\dots<i_p$ of products of operators of norms at most $C\tau_{i_1},\dots,C\tau_{i_p}$; since $\sum_{i_1<\dots<i_p}\tau_{i_1}\cdots\tau_{i_p}\le\frac{(\sum_j\tau_j)^p}{p!}$, we obtain $\|(\mathcal D^{-1}\mathcal U)^p\|\le\frac{(C\alpha)^p}{p!}$. Hence $\mathcal D+\mathcal U$ has an inverse of norm $O(e^{C\alpha})$, and by the third bound, $\mathcal T$ is invertible with uniformly bounded inverse for $N$ large. Since the potential $u+v_N$ is supported in $\Omega$, the operator $J_N:=G_k\ast((u+v_N)\,\cdot\,)$ satisfies $J_N=J_NP_\Omega$, where $J_N$ acts on functions on $\Omega$ extended by zero. Hence $P_\Omega(I-J_N)=\mathcal TP_\Omega$ and $(I-J_N)J_N=J_N\mathcal T$, and therefore $I-J_N$ is invertible on $L^\infty(\R^2)$, with inverse $I+J_N\mathcal T^{-1}P_\Omega$. Hence every $k\in K$ is a regular point of $u+v_N$, and $\chi_{u+v_N}=(I-J_N)^{-1}1$ coincides with $\mathcal T^{-1}1$ on $\Omega$. The same argument, with the indices $1\le i,j\le N$ only, applies to $v_N$, and to the operator $I-H_k\ast(v_N\,\cdot\,)$, for which Lemma \ref{lemmashell} provides the same bounds.

\noindent{\bf Step 2:} The eigenfunctions on $\Omega_0$. Let $\Pi_i:=\prod_{j>i}(1+\tau_jc(k))$, $0\le i\le N$, and consider the function $\tilde\chi$ on $\Omega$ equal to $\Pi_i$ on $\Omega_i$, $i\ge1$, and to $\Pi_0\chi_u$ on $\Omega_0$. Since $\Pi_{i-1}=\Pi_i+\tau_ic(k)\Pi_i$, we have $\Pi_i=1+c(k)\sum_{j>i}\tau_j\Pi_j$, and $|\Pi_i|\le e^{\alpha|c(k)|}$. On $\Omega_i$, $i\ge1$, by Lemma \ref{lemmashell} (i) and the last two bounds of Step 1,
$$(\mathcal T\tilde\chi)_i=\Pi_i-c(k)\sum_{j>i}\tau_j\Pi_j+O\Big(\sum_{j>i}\tau_j\frac{R_i}{R_j}+\tau_i+\Lambda_N^{-1}\Big)=1+O\big(\Lambda_N^{-1}+h_N^{-1}\big),$$
since $|X|\le3R_i\le\frac{R_j}4$ on $\Omega_i$ for $j>i$; here and below, the implicit constants depend on $\alpha$, $u$ and $K$. On $\Omega_0$, $\chi_u-J\chi_u=1$ and $|X|$ is bounded, so that $(\mathcal T\tilde\chi)_0=\Pi_0-c(k)\sum_{j\ge1}\tau_j\Pi_j+O(\sum_j\tau_jR_j^{-1})=1+O(\Lambda_N^{-1})$. Since $\mathcal T^{-1}$ is uniformly bounded, $\chi_{u+v_N}=\tilde\chi+o(1)$ uniformly on $K\times\Omega$. Moreover $\Pi_0=\exp\big(\sum_j\log(1+\tau_jc(k))\big)\to e^{\alpha c(k)}$ uniformly on $K$, since $\max_j\tau_j\to0$ and $\sum_j\tau_j=\alpha$. Hence
\be
\label{tailincoming}
\chi_{u+v_N}\longrightarrow e^{\alpha c(k)}\chi_u\qquad\mbox{uniformly on}\ K\times\Omega_0 .
\ee
The same argument with $H_k$ and the constant $-\overline{c(k)}$ of \eqref{tailshell} shows that the solution $\chi^\dagger_{v_N}$ of $\chi^\dagger_{v_N}=1+H_k\ast(v_N\chi^\dagger_{v_N})$ satisfies
\be
\label{tailoutgoing}
\chi^\dagger_{v_N}\longrightarrow e^{-\alpha\overline{c(k)}}\qquad\mbox{uniformly on}\ K\times\Omega_0 .
\ee

\noindent{\bf Step 3:} A two-potential identity. Let $U=u+v$ with $v:=v_N$, $k\in K$, $\mathcal E:=\chi_U-\chi_v$ and $g:=u\chi_U$. Subtracting the equations $\chi_U=1+G_k\ast(U\chi_U)$ and $\chi_v=1+G_k\ast(v\chi_v)$ gives $\mathcal E=G_k\ast(g+v\mathcal E)$. Let $\vartheta:=e_k\chi^\dagger_v$. By Lemma \ref{lemmagreen} (iv), $\overline{G_k(-Z)}=G_{\bar k}(-Z)=G_k(-Z)e^{iZ\cdot\zeta(k)}$, so that $e_k(Z)H_k(Z)=G_k(-Z)$; since $e_k(X)=e_k(X-X')e_k(X')$, multiplying the equation of $\chi^\dagger_v$ by $e_k$ gives $\vartheta-e_k=G_k^T\ast(v\vartheta)$, where $G_k^T(Z):=G_k(-Z)$. The functions $g$ and $v\vartheta$ have compact supports and $\mathcal E$, $\vartheta$ are bounded, so that Fubini's theorem applies and gives
$$\int(\vartheta-e_k)(g+v\mathcal E)\,dX=\int v\vartheta\,\big(G_k\ast(g+v\mathcal E)\big)\,dX=\int v\vartheta\,\mathcal E\,dX,$$
hence $\int e_k\,(g+v\mathcal E)\,dX=\int\vartheta g\,dX$.
Since $U\chi_U-v\chi_v=g+v\mathcal E$, \eqref{defa} gives
\be
\label{twopotential}
a_U(k)-a_v(k)=-\frac{i}{2\pi}\int_{\R^2}e_k(X)\,u(X)\,\chi_U(k;X)\,\chi^\dagger_v(k;X)\,dX .
\ee

\noindent{\bf Step 4:} Conclusion. By \eqref{tailincoming}, \eqref{tailoutgoing} and \eqref{twopotential}, $a_{u+v_N}-a_{v_N}\to e^{\alpha(c-\bar c)}a_u=\Theta a_u$ uniformly on $K$. By \eqref{plancherelsmooth}, $8\|a_{v_N}\|_{L^2(\mu)}^2=\|v_N\|_{L^2}^2\to0$. Hence $a_{u+v_N}\to\Theta a_u$ in $L^2(K)$, and by the symmetry $a(\bar k)=\overline{a(k)}$ also in $L^2(\bar K)$. By the proof of Lemma \ref{lemmairregular}, the irregular points of $u$ in $\C^+$ form the zero set of a nonzero real-analytic function on the connected set $\C^+$, which is a null set \cite{Mityagin}. Hence the compact sets $K$ as above exhaust $\C\setminus\R$ up to a null set, and since $(a_{u+v_N})$ is bounded in $L^2(\mu)$, $a_{u+v_N}\weakto\Theta a_u$ in $L^2(\mu)$. Finally, by \eqref{plancherelsmooth}, $8\|a_{u+v_N}\|_{L^2(\mu)}^2=\|u+v_N\|_{L^2}^2\to\|u\|_{L^2}^2=8\|\Theta a_u\|_{L^2(\mu)}^2$, since $|\Theta|=1$, and the convergence is strong.
\end{proof}

\subsection{The construction}

\begin{proof}[Proof of Theorem \ref{thmcounter}]
{\bf Step 1:} The profile and the phase. For the nonnegative example we take $w\ge0$, $w\neq0$. For the example with zero horizontal mean we take $w=\pa_x\Psi$ with $0\le\Psi\in C^\infty_c((1,2)^2)$, $\Psi\neq0$. At $k=i$, $G^\infty_i(-Y)=\frac1{2\pi}(-2y_2+iy_1)^{-1}$ has the imaginary part $-\frac1{2\pi}\frac{y_1}{4y_2^2+y_1^2}$. In the first case, $\Im c(i)<0$. In the second case, an integration by parts in $y_1$ gives
$$\Im c(i)=\frac1{2\pi}\int_{\R^2}\frac{4y_2^2-y_1^2}{(4y_2^2+y_1^2)^2}\,\Psi(Y)\,dY>0,$$
since $4y_2^2>4>y_1^2$ on $(1,2)^2$. In both cases $\Im c$ is real-analytic on $\C^+$ and not constant, since $c(ib)\to0$ as $b\to\infty$. Hence, for every $\alpha>0$, the function $|\Theta-1|^2=2-2\cos(2\alpha\Im c)$ is real-analytic on $\C^+$ and not identically zero, and its zero set is a null set \cite{Mityagin}: $\Theta\neq1$ almost everywhere on $\C$. We fix $\alpha>0$. For the opposite rotation we use the profile $w_-(X):=w(-X)$, supported in $(-2,-1)^2$. Lemmas \ref{lemmatail} and \ref{lemmashell} apply to $w_-$, with $z_2\ge\frac34$ in the proof of \eqref{tailshell} for $G_k$ and $z_2\le-\frac34$ for $H_k$, and since $G^\infty_k$ is odd, the constants in \eqref{tailshell} become $-c(k)$ and $\overline{c(k)}$. Therefore Proposition \ref{propphase} holds for the functions $v_N$ built from $w_-$, with $\Theta^{-1}$ in place of $\Theta$. We denote by $v_N^+$ and $v_N^-$ the functions \eqref{defvN} built from $w$ and from $w_-$; both satisfy \eqref{vNsmall}.

\noindent{\bf Step 2:} The recursive construction. Let $f_0\in C^\infty_c(\R^2)$ be real, with $f_0\ge0$, $f_0\neq0$ in the first case, and $f_0=\pa_x\Psi_0$ with $\Psi_0\in C^\infty_c(\R^2)$ real, $f_0\neq0$, in the second case. Its $L^2$ norm can be chosen arbitrarily small and positive. Let $A_0:=a_{f_0}$, and $T_{-1}:=0$. Assume that $f_n$ is constructed, and let $A_n:=a_{f_n}$. Since $f_n\in C^\infty_c$, the solution with datum $f_n$ scatters to $\cB A_n$ in both time directions, and we choose $T_n>T_{n-1}+n+1$ such that its interaction profile $S(-t)u^{(n)}(t)$ is within $2^{-n}$ of $\cB A_n$ in $L^2$ at $t=\pm T_n$. By the continuity of the $L^2$ flow \cite{Bourgain}, there is $d_n\in(0,1]$ such that every datum within $d_n$ of $f_n$ in $L^2$ generates a solution whose interaction profile at $t=\pm T_n$ is within $2^{-n}$ of that of $u^{(n)}$. We set
$$f_{n+1}:=f_n+v^{\sigma_n}_{N_n},\qquad \sigma_n:=(-1)^n,$$
where $N_n$ is so large that the first scale $R_1=\Lambda_{N_n}$ exceeds $n+1$ and twice the radius of a disc centered at $0$ containing $\supp f_n$, so that the support of $v^{\sigma_n}_{N_n}$ is disjoint from $\supp f_n$ and contained in $\{|X|\ge n+1\}$, and so large that, by Proposition \ref{propphase} with $u=f_n$ and by \eqref{vNsmall},
\be
\label{choiceN}
\begin{gathered}
\big\|A_{n+1}-\Theta^{\sigma_n}A_n\big\|_{L^2(\mu)}<2^{-n-3}\|A_0\|_{L^2(\mu)},\\ \|f_{n+1}-f_n\|_{L^2}<2^{-n-2}\min\big(d_0,\dots,d_n,\|f_0\|_{L^2}\big),\\
\sum_{|\gamma|\le n}\big\|(\log(2+|X|))^{\frac12}\pa^\gamma(f_{n+1}-f_n)\big\|_{L^2}<2^{-n},\\ \sum_{|\gamma|\le n}\sup_X\jb X^{1+|\gamma|}\log(2+|X|)\,|\pa^\gamma(f_{n+1}-f_n)(X)|<2^{-n} .
\end{gathered}
\ee
In the second case we also require $\|(\log(2+|X|))^{\frac12}\pa_x^{-1}\pa_y(f_{n+1}-f_n)\|_{L^2}<2^{-n}$, which is possible since $\pa_x^{-1}\pa_yw_R=R^{-1}(\pa_y\Psi)(\cdot/R)$ for $w=\pa_x\Psi$, and the computation \eqref{vNsmall} applies to it.

\noindent{\bf Step 3:} The limit datum. By \eqref{choiceN}, $f_n$ converges in every $H^m(\R^2)$, and in the weighted norms of \eqref{choiceN}, to a real function $u_0$, with $\|u_0-f_0\|_{L^2}\le\frac12\|f_0\|_{L^2}$; hence $\frac12\|f_0\|_{L^2}\le\|u_0\|_{L^2}\le\frac32\|f_0\|_{L^2}$. The functions $f_{n+1}-f_n=v^{\sigma_n}_{N_n}$, $n\ge0$, have disjoint compact supports, and the support of $f_{n+1}-f_n$ is contained in $\{|X|\ge n+1\}$. Hence, if $|X|$ is large, $X$ belongs to the support of at most one of these functions, whose index $n(X)$ tends to infinity as $|X|\to\infty$, and the last bound of \eqref{choiceN}, applied with $n=n(X)\ge|\gamma|$, gives $\pa^\gamma u_0(X)=o(|X|^{-1-|\gamma|}(\log|X|)^{-1})$ as $|X|\to\infty$, for every $\gamma$. In the first case $u_0\ge0$. In the second case, every horizontal line meets the supports of only finitely many of the functions $f_{n+1}-f_n$, since each term $\tau_jw_{R_j}$ or $\tau_jw_{R_j}(-\cdot)$ of $v^{\sigma_n}_{N_n}$ is supported in $\{R_j<|y|<2R_j\}$, with $R_j\ge n+1$, so that $u_0(\cdot,y)$ is a finite sum of derivatives of compactly supported functions, and $\int_\R u_0(x,y)\,dx=0$ for every $y$; moreover, $\pa_x^{-1}\pa_yf_n$ converges in $L^2$, and since the Fourier multiplier $\pa_x^{-1}\pa_y$ is a closed operator on $L^2$, $\pa_x^{-1}\pa_yu_0\in L^2$.

\noindent{\bf Step 4:} The two limits. Let $B_n:=\Theta^{-1}A_n$ for odd $n$ and $B_n:=A_n$ for even $n$; these coefficients are real, since $\Theta(\bar k)=\overline{\Theta(k)}$. By \eqref{choiceN} and $|\Theta|=1$, $\|B_{n+1}-B_n\|_{L^2(\mu)}=\|A_{n+1}-\Theta^{\sigma_n}A_n\|_{L^2(\mu)}<2^{-n-3}\|A_0\|_{L^2(\mu)}$. Hence $B_n\to A_\infty$ in $L^2(\mu)$, with $\|A_\infty-A_0\|_{L^2(\mu)}\le\frac14\|A_0\|_{L^2(\mu)}$; in particular $A_\infty\neq0$, and since $\Theta\neq1$ almost everywhere, $\Theta A_\infty\neq A_\infty$. Let $u(t)$ be the solution with datum $u_0$. By \eqref{choiceN}, $\|u_0-f_n\|_{L^2}\le\sum_{m\ge n}\|f_{m+1}-f_m\|_{L^2}<d_n$, so that, at $t=\pm T_n$,
$$\begin{aligned}\big\|S(-t)u(t)-\cB A_n\big\|_{L^2}&\le\big\|S(-t)u(t)-S(-t)u^{(n)}(t)\big\|_{L^2}\\&\quad+\big\|S(-t)u^{(n)}(t)-\cB A_n\big\|_{L^2}<2^{1-n} .\end{aligned}$$
By \eqref{bornisometry}, $\cB A_n\to\cB A_\infty$ along even $n$ and $\cB A_n\to\cB(\Theta A_\infty)$ along odd $n$. Hence $S(-t)u(t)$ has the two distinct strong limits $\cB A_\infty$ and $\cB(\Theta A_\infty)$ along the sequences $t=T_{2n}$ and $t=T_{2n+1}$, and also along $t=-T_{2n}$ and $t=-T_{2n+1}$. Both limits are real, and a real Schwartz function $\varphi$ with $(\cB A_\infty-\cB(\Theta A_\infty),\varphi)\neq0$ distinguishes them. This proves Theorem \ref{thmcounter}.
\end{proof}

The data of Theorem \ref{thmcounter} lie outside all the classes of Theorem \ref{thmmain}, as they must. For instance, the $L^{\frac32,4}$ quasi-norm of $\tau_jw_{R_j}$ equals $\tau_jR_j^{\frac13}\|w\|_{L^{3/2,4}}$, and for the last term $j=N$ of $v_N$, $\tau_NR_N^{\frac13}=\frac{\alpha}{Nh_N}\Lambda_N^{\frac N3}\to\infty$ as $N\to\infty$; the same divergence occurs for the norms of the mixed spaces, whose scaling exponent $\frac1{p_x}+\frac1{p_y}-1$ is positive. The same computation shows that $\jb X^\delta u_0\notin L^2$ for every $\delta>0$, and $u_0\notin L^p$ for $1\le p<2$.



\begin{thebibliography}{99}

\bibitem{ABF} Ablowitz, M. J.; Bar Yaacov, D.; Fokas, A. S., {\em On the inverse scattering transform for the Kadomtsev--Petviashvili equation}, Stud. Appl. Math. 69 (1983), 135--143.

\bibitem{AblowitzClarkson} Ablowitz, M. J.; Clarkson, P. A., {\em Solitons, nonlinear evolution equations and inverse scattering}, London Mathematical Society Lecture Note Series 149, Cambridge University Press, Cambridge, 1991.

\bibitem{AizenmanSimon} Aizenman, M.; Simon, B., {\em Brownian motion and Harnack inequality for Schr\"odinger operators}, Comm. Pure Appl. Math. 35 (1982), 209--273.

\bibitem{BealsCoifman} Beals, R.; Coifman, R. R., {\em Linear spectral problems, non-linear equations and the $\bar\partial$-method}, Inverse Problems 5 (1989), 87--130.

\bibitem{Benjamin67} Benjamin, T. B., {\em Internal waves of permanent form in fluids of great depth}, J. Fluid Mech. 29 (1967), 559--592.

\bibitem{Bers} Bers, L., {\em Theory of pseudo-analytic functions}, Institute for Mathematics and Mechanics, New York University, New York, 1953.

\bibitem{BlackstoneGGM24a} Blackstone, E.; Gassot, L.; G\'erard, P.; Miller, P. D., {\em The Benjamin--Ono initial-value problem for rational data with application to long time asymptotics and scattering}, Ann. Inst. H. Poincar\'e C Anal. Non Lin\'eaire (2025), published online, \href{https://doi.org/10.4171/AIHPC/169}{doi:10.4171/AIHPC/169}.

\bibitem{BorgheseJMcL18} Borghese, M.; Jenkins, R.; McLaughlin, K. D. T.-R., {\em Long time asymptotic behavior of the focusing nonlinear Schr\"odinger equation}, Ann. Inst. H. Poincar\'e Anal. Non Lin\'eaire 35 (2018), 887--920.

\bibitem{Bourgain} Bourgain, J., {\em On the Cauchy problem for the Kadomtsev--Petviashvili equation}, Geom. Funct. Anal. 3 (1993), 315--341.

\bibitem{ChenLiu21} Chen, G.; Liu, J., {\em Soliton resolution for the focusing modified KdV equation}, Ann. Inst. H. Poincar\'e Anal. Non Lin\'eaire 38 (2021), 2005--2071.

\bibitem{Chen25} Chen, X., {\em Scattering of the defocusing Calogero--Moser derivative nonlinear Schr\"odinger equation}, preprint (2025), \href{https://arxiv.org/abs/2511.06432}{arXiv:2511.06432}.

\bibitem{Chen26} Chen, X., {\em $L^2$ soliton resolution for the Benjamin--Ono equation}, preprint (2026), \href{https://hal.science/hal-05713824v1}{hal-05713824}.

\bibitem{CuccJenk16} Cuccagna, S.; Jenkins, R., {\em On the asymptotic stability of $N$-soliton solutions of the defocusing nonlinear Schr\"odinger equation}, Comm. Math. Phys. 343 (2016), 921--969.

\bibitem{BouardMartel} de Bouard, A.; Martel, Y., {\em Non existence of $L^2$-compact solutions of the Kadomtsev--Petviashvili II equation}, Math. Ann. 328 (2004), 525--544.

\bibitem{BouardSaut} de Bouard, A.; Saut, J.-C., {\em Solitary waves of generalized Kadomtsev--Petviashvili equations}, Ann. Inst. H. Poincar\'e Anal. Non Lin\'eaire 14 (1997), 211--236.

\bibitem{DeiftVZ94} Deift, P.; Venakides, S.; Zhou, X., {\em The collisionless shock region for the long-time behavior of solutions of the KdV equation}, Comm. Pure Appl. Math. 47 (1994), 199--206.

\bibitem{DeiftZhou93} Deift, P.; Zhou, X., {\em A steepest descent method for oscillatory Riemann--Hilbert problems. Asymptotics for the MKdV equation}, Ann. of Math. (2) 137 (1993), 295--368.

\bibitem{DonmazovLiuPerry} Donmazov, S.; Liu, J.; Perry, P., {\em Large-time asymptotics for the Kadomtsev--Petviashvili I equation}, preprint (2024), \href{https://arxiv.org/abs/2409.14480}{arXiv:2409.14480}.

\bibitem{Dryuma} Dryuma, V. S., {\em Analytic solution of the two-dimensional Korteweg--de Vries (KdV) equation}, JETP Lett. 19 (1974), 387--388.

\bibitem{DuyKM22} Duyckaerts, T.; Kenig, C.; Martel, Y.; Merle, F., {\em Soliton resolution for critical co-rotational wave maps and radial cubic wave equation}, Comm. Math. Phys. 391 (2022), 779--871.

\bibitem{DuyKM23} Duyckaerts, T.; Kenig, C.; Merle, F., {\em Soliton resolution for the radial critical wave equation in all odd space dimensions}, Acta Math. 230 (2023), 1--92.

\bibitem{EckSch83} Eckhaus, W.; Schuur, P., {\em The emergence of solitons of the Korteweg--de Vries equation from arbitrary initial conditions}, Math. Methods Appl. Sci. 5 (1983), 97--116.

\bibitem{Folland} Folland, G. B., {\em Real analysis. Modern techniques and their applications}, second edition, Pure and Applied Mathematics, John Wiley \& Sons, New York, 1999.

\bibitem{GGKM} Gardner, C. S.; Greene, J. M.; Kruskal, M. D.; Miura, R. M., {\em Method for solving the Korteweg--de Vries equation}, Phys. Rev. Lett. 19 (1967), 1095--1097.

\bibitem{GGM26} Gassot, L.; G\'erard, P.; Miller, P. D., {\em A proof of the soliton resolution conjecture for the Benjamin--Ono equation}, preprint (2026), \href{https://arxiv.org/abs/2601.10488}{arXiv:2601.10488}.

\bibitem{Gerard22} G\'erard, P., {\em An explicit formula for the Benjamin--Ono equation}, Tunis. J. Math. 5 (2023), 593--603.

\bibitem{GeLe-24} G\'erard, P.; Lenzmann, E., {\em The Calogero--Moser derivative nonlinear Schr\"odinger equation}, Comm. Pure Appl. Math. 77 (2024), 4008--4062.

\bibitem{Grafakos} Grafakos, L., {\em Classical Fourier analysis}, third edition, Graduate Texts in Mathematics 249, Springer, New York, 2014.

\bibitem{Grin} Grinevich, P. G., {\em Nonsingularity of the direct scattering transform for the KP II equation with a real exponentially decaying-at-infinity potential}, Lett. Math. Phys. 40 (1997), 59--73.

\bibitem{Hadac} Hadac, M., {\em Well-posedness for the Kadomtsev--Petviashvili II equation and generalisations}, Trans. Amer. Math. Soc. 360 (2008), 6555--6572.

\bibitem{HHK} Hadac, M.; Herr, S.; Koch, H., {\em Well-posedness and scattering for the KP-II equation in a critical space}, Ann. Inst. H. Poincar\'e Anal. Non Lin\'eaire 26 (2009), 917--941.

\bibitem{HHKcorr} Hadac, M.; Herr, S.; Koch, H., {\em Erratum to ``Well-posedness and scattering for the KP-II equation in a critical space''}, Ann. Inst. H. Poincar\'e Anal. Non Lin\'eaire 27 (2010), 971--972.
\bibitem{Hadama} Hadama, S., {\em Small-data $L^2$ theory for the intermediate NLS and the Calogero–Moser derivative NLS}, \href{https://arxiv.org/abs/2608.01138}{arXiv:2608.01138}.

\bibitem{Hormander} H\"ormander, L., {\em The analysis of linear partial differential operators I. Distribution theory and Fourier analysis}, reprint of the second (1990) edition, Classics in Mathematics, Springer-Verlag, Berlin, 2003.

\bibitem{HornJohnson} Horn, R. A.; Johnson, C. R., {\em Matrix analysis}, second edition, Cambridge University Press, Cambridge, 2013.

\bibitem{IfrimTataru2019} Ifrim, M.; Tataru, D., {\em Well-posedness and dispersive decay of small data solutions for the Benjamin--Ono equation}, Ann. Sci. \'Ec. Norm. Sup\'er. (4) 52 (2019), 297--335.

\bibitem{IsazaMejia} Isaza, P.; Mej\'ia, J., {\em Local and global Cauchy problems for the Kadomtsev--Petviashvili (KP-II) equation in Sobolev spaces of negative indices}, Comm. Partial Differential Equations 26 (2001), 1027--1054.

\bibitem{Janson} Janson, S., {\em On interpolation of multilinear operators}, in: Function spaces and applications (Lund, 1986), Lecture Notes in Math. 1302, Springer, Berlin, 1988, 290--302.

\bibitem{JenLaw23} Jendrej, J.; Lawrie, A., {\em Soliton resolution for the energy-critical nonlinear wave equation in the radial case}, Ann. PDE 9 (2023), Paper No. 18.

\bibitem{JenLaw25} Jendrej, J.; Lawrie, A., {\em Soliton resolution for energy-critical wave maps in the equivariant case}, J. Amer. Math. Soc. 38 (2025), 783--875.

\bibitem{JenkLPS18} Jenkins, R.; Liu, J.; Perry, P.; Sulem, C., {\em Soliton resolution for the derivative nonlinear Schr\"odinger equation}, Comm. Math. Phys. 363 (2018), 1003--1049.

\bibitem{KP} Kadomtsev, B. B.; Petviashvili, V. I., {\em On the stability of solitary waves in weakly dispersing media}, Soviet Phys. Dokl. 15 (1970), 539--541; Russian original: Dokl. Akad. Nauk SSSR 192 (1970), 753--756.

\bibitem{KLMV26} Killip, R.; Laurens, T.; Murphy, J.; Vi\c{s}an, M., {\em Dispersive decay and scattering for continuum Calogero--Moser models}, \href{https://arxiv.org/abs/2609.12985}{arXiv:2609.12985}.

\bibitem{KillipLaurensVisan23-BO} Killip, R.; Laurens, T.; Vi\c{s}an, M., {\em Sharp well-posedness for the Benjamin--Ono equation}, Invent. Math. 236 (2024), 999--1054.

\bibitem{KLV25} Killip, R.; Laurens, T.; Vi\c{s}an, M., {\em Scaling-critical well-posedness for continuum Calogero--Moser models on the line}, Commun. Amer. Math. Soc. 5 (2025), 284--320.

\bibitem{KimKwon} Kim, T.; Kwon, S., {\em Soliton resolution for Calogero--Moser derivative nonlinear Schr\"odinger equation}, J. Eur. Math. Soc., to appear, \href{https://arxiv.org/abs/2408.12843}{arXiv:2408.12843}.

\bibitem{Kis} Kiselev, O. M., {\em Asymptotic behaviour of a solution for Kadomtsev--Petviashvili-2 equation}, preprint (2000), \href{https://arxiv.org/abs/math-ph/0003014}{arXiv:math-ph/0003014}.

\bibitem{KleinS21} Klein, C.; Saut, J.-C., {\em Nonlinear dispersive equations---inverse scattering and PDE methods}, Applied Mathematical Sciences 209, Springer, Cham, 2021.

\bibitem{LannesSaut} Lannes, D.; Saut, J.-C., {\em Weakly transverse Boussinesq systems and the KP approximation}, Nonlinearity 19 (2006), 2853--2875.

\bibitem{LiebLoss} Lieb, E. H.; Loss, M., {\em Analysis}, second edition, Graduate Studies in Mathematics 14, American Mathematical Society, Providence, RI, 2001.

\bibitem{MMPP} Mendez, A. J.; Mu\~noz, C.; Poblete, F.; Pozo, J. C., {\em Long time asymptotics of large data in the Kadomtsev--Petviashvili models}, Nonlinearity 37 (2024), Paper No. 055017.

\bibitem{Mityagin} Mityagin, B. S., {\em The zero set of a real analytic function}, Math. Notes 107 (2020), 529--530.

\bibitem{Mizumachi} Mizumachi, T., {\em Stability of line solitons for the KP-II equation in $\R^2$}, Mem. Amer. Math. Soc. 238 (2015), no. 1125.

\bibitem{MizumachiTzvetkov} Mizumachi, T.; Tzvetkov, N., {\em Stability of the line soliton of the KP-II equation under periodic transverse perturbations}, Math. Ann. 352 (2012), 659--690.

\bibitem{NachmanRegevTataru} Nachman, A.; Regev, I.; Tataru, D., {\em A nonlinear Plancherel theorem with applications to global well-posedness for the defocusing Davey--Stewartson equation and to the inverse boundary value problem of Calder\'on}, Invent. Math. 220 (2020), 395--451.

\bibitem{Novikov92} Novikov, R. G., {\em The inverse scattering problem on a fixed energy level for the two-dimensional Schr\"odinger operator}, J. Funct. Anal. 103 (1992), 409--463.

\bibitem{ONeil} O'Neil, R., {\em Convolution operators and $L(p,q)$ spaces}, Duke Math. J. 30 (1963), 129--142.

\bibitem{Perry} Perry, P., {\em Global well-posedness and long-time asymptotics for the defocussing Davey--Stewartson II equation in $H^{1,1}(\C)$}, J. Spectr. Theory 6 (2016), 429--481.

\bibitem{Rudin} Rudin, W., {\em Fourier analysis on groups}, Interscience Tracts in Pure and Applied Mathematics 12, Interscience Publishers, New York--London, 1962.

\bibitem{Saut93} Saut, J.-C., {\em Remarks on the generalized Kadomtsev--Petviashvili equations}, Indiana Univ. Math. J. 42 (1993), 1011--1026.

\bibitem{Simon} Simon, B., {\em Trace ideals and their applications}, second edition, Mathematical Surveys and Monographs 120, American Mathematical Society, Providence, RI, 2005.

\bibitem{Stein} Stein, E. M., {\em Singular integrals and differentiability properties of functions}, Princeton Mathematical Series 30, Princeton University Press, Princeton, NJ, 1970.

\bibitem{Sung} Sung, L.-Y., {\em An inverse scattering transform for the Davey--Stewartson II equations. I, II, III}, J. Math. Anal. Appl. 183 (1994), 121--154, 289--325, 477--494.

\bibitem{TakaokaTzvetkov} Takaoka, H.; Tzvetkov, N., {\em On the local regularity of the Kadomtsev--Petviashvili-II equation}, Internat. Math. Res. Notices 2001, no. 2, 77--114.

\bibitem{Ukai} Ukai, S., {\em Local solutions of the Kadomtsev--Petviashvili equation}, J. Fac. Sci. Univ. Tokyo Sect. IA Math. 36 (1989), 193--209.

\bibitem{Vekua} Vekua, I. N., {\em Generalized analytic functions}, Pergamon Press, London; Addison-Wesley, Reading, MA, 1962.

\bibitem{Wick87} Wickerhauser, M. V., {\em Inverse scattering for the heat operator and evolutions in $2+1$ variables}, Comm. Math. Phys. 108 (1987), 67--89.

\bibitem{Wick} Wickerhauser, M. V., {\em Hamilton's form for the Kadomtsev--Petviashvili equation}, J. Math. Phys. 29 (1988), 2300--2302.

\bibitem{Wu} Wu, D., {\em The inverse scattering theory of Kadomtsev--Petviashvili II equations}, J. Nonlinear Waves 1 (2025), Paper No. e4.

\bibitem{WuLT} Wu, D., {\em Long time asymptotics for the KPII equation}, preprint (2025), \href{https://arxiv.org/abs/2509.02907}{arXiv:2509.02907}.

\bibitem{Yosida} Yosida, K., {\em Functional analysis}, reprint of the sixth (1980) edition, Classics in Mathematics, Springer-Verlag, Berlin, 1995.

\bibitem{ZakharovShabat} Zakharov, V. E.; Shabat, A. B., {\em A scheme for integrating the nonlinear equations of mathematical physics by the method of the inverse scattering problem. I}, Funct. Anal. Appl. 8 (1974), 226--235.

\end{thebibliography}
\end{document}